\documentclass[a4paper,11pt,english,reqno]{amsart}
\usepackage{amsmath,amssymb,amsfonts,dsfont,amsthm,upgreek,bm}    
\usepackage[utf8]{inputenc}
\usepackage[english]{babel}
\usepackage[T1]{fontenc} 
\usepackage{graphicx,xcolor}
\usepackage{float}
\usepackage[font=small,labelfont=bf]{caption}
\usepackage[a4paper, margin=2.4cm]{geometry}
\usepackage{tikz,pgfplots}
\usepackage{tkz-fct}
\usepgfplotslibrary{polar}

\usepackage{enumerate}
\newcommand{\stkout}[1]{\ifmmode\text{\sout{\ensuremath{#1}}}\else\sout{#1}\fi}
\usepackage{scalerel,stackengine}
\stackMath
\newcommand\reallywidecheck[1]{%
\savestack{\tmpbox}{\stretchto{%
  \scaleto{%
    \scalerel*[\widthof{\ensuremath{#1}}]{\kern-.6pt\bigwedge\kern-.6pt}%
    {\rule[-\textheight/2]{1ex}{\textheight}}
  }{\textheight}%
}{0.6ex}}%
\stackon[1pt]{#1}{\scalebox{-0.8}{\tmpbox}}%
}

\newcommand{\defeq}{\stackrel{\mathrm{def}}{=}}
\newcommand{\prob}{\mathbf{P}}
\newcommand{\erw}{\mathds{E}}
\newcommand{\e}{\mathrm{e}}

\newcommand{\cG}{\mathcal{G}}

\newcommand{\tr}{\mathrm{tr\,}}

\newcommand{\supp}{\mathrm{supp\,}}

\newcommand{\dist}{\mathrm{dist\,}}
\newcommand{\diag}{\mathrm{diag}}

\newcommand{\C}{\mathds{C}}
\newcommand{\R}{\mathds{R}}
\newcommand{\T}{\mathds{T}}
\newcommand{\N}{\mathds{N}}
\newcommand{\Z}{\mathds{Z}}

\newcommand{\wind}{d}
\newcommand{\wt}{\widetilde}
\newcommand{\ol}{\overline}
\newcommand{\ul}{\underline}

\newcommand{\wh}{\widehat}
\newcommand{\vep}{\varepsilon}

\newcommand{\gD}{\mathfrak{D}}
\newcommand{\gX}{\mathfrak{X}}
\newcommand{\gY}{\mathfrak{Y}}
\newcommand{\gx}{\mathfrak{x}}
\newcommand{\gy}{\mathfrak{y}}

\newcommand{\gF}{\mathfrak{F}}

\newcommand{\gT}{\mathfrak{T}}
\newcommand{\gP}{\mathfrak{P}}
\newcommand{\gC}{\mathfrak{C}}
\newcommand{\gp}{\mathfrak{p}}

\newcommand{\gI}{\mathfrak{I}}
\newcommand{\gd}{\mathfrak{d}}

\newcommand{\sE}{\mathsf{E}}
\newcommand{\sF}{\mathsf{F}}

\newcommand{\E}{\mathbf{E}}

\newcommand{\upa}{\upalpha}

\newcommand{\D}{\mathds{D}}

\newcommand{\cB}{\mathcal{B}}
\newcommand{\cN}{\mathcal{N}}
\newcommand{\cF}{\mathcal{F}}
\newcommand{\cP}{\mathcal{P}}
\newcommand{\cZ}{\mathcal{Z}}
\newcommand{\cS}{\mathcal{S}}
\newcommand{\cT}{\mathcal{T}}
\newcommand{\cL}{\mathcal{L}}
\newcommand{\cI}{\mathcal{I}}
\newcommand{\cJ}{\mathcal{J}}
\newcommand{\cA}{\mathcal{A}}
\newcommand{\cY}{\mathcal{Y}}
\newcommand{\cX}{\mathcal{X}}

\newcommand{\cC}{\mathcal{C}}
\newcommand{\cU}{\mathcal{U}}
\newcommand{\cW}{\mathcal{W}}

\newcommand{\sgn}{\operatorname{sgn}}

\newcommand{\COMP}[1]{ \check{#1} }

\renewcommand{\ge}{\geqslant}
\renewcommand{\geq}{\geqslant}
\renewcommand{\le}{\leqslant}
\renewcommand{\leq}{\leqslant}
\newcommand{\abbr}[1]{{\sc\lowercase{#1}}}

\newtheorem{thm}{Theorem}
\newtheorem{cor}[thm]{Corollary}

\newtheorem{prop}[thm]{Proposition}
\newtheorem{lem}[thm]{Lemma}
\newtheorem{assu}[thm]{Assumption}
\newtheorem{assump}[thm]{Assumption}

\newtheorem{defn}[thm]{Definition}
\newtheorem{rem}[thm]{Remark}
\theoremstyle{remark}

\numberwithin{equation}{section}
\numberwithin{thm}{section}

\usepackage{hyperref}
\hypersetup{pdfborder=0 0 0, 
	    colorlinks=true,
	    citecolor=blue,
	    linkcolor=blue,
	    urlcolor=blue,
	    pdfauthor={Anirban Basak, Martin Vogel, Ofer Zeitouni}
	   }
\title{Localization of eigenvectors of non-Hermitian banded noisy Toeplitz matrices}
\author{Anirban Basak}
\address[Anirban Basak]{International Centre for Theoretical Sciences, Tata Institute of Fundamental Research, Bangalore, India.}
\email{anirban.basak@icts.res.in}
\author{Martin Vogel}
\address[Martin Vogel]{Institut de Recherche Math{\'e}matique Avanc{\'e}e - UMR 7501, 
Universit{\'e} de Strasbourg et CNRS, 7 rue René-Descartes, 67084 Strasbourg Cedex, France.}
\email{vogel@math.unistra.fr}
\author{Ofer Zeitouni}
\address[Ofer Zeitouni]{Department of Mathematics, 
Weizmann Institute of Science, POB 26, Rehovot 76100, Israel
and Courant Institute, New York University,
251 Mercer St, New York, NY 10012, USA.}
\email{ofer.zeitouni@weizmann.ac.il}
 \date{\today}
\begin{document}
\begin{abstract}
We prove localization with high probability on sets of size of 
order $N/\log N$ for the eigenvectors 
of non-Hermitian finitely banded $N\times N$ Toeplitz matrices $P_N$ 
subject to small 
random perturbations, in a very general setting. As perturbation we consider 
$N\times N$ random matrices with independent entries of zero mean, finite 
moments, and which satisfy an appropriate
anti-concentration bound. We show via 
a Grushin problem that an eigenvector for a given eigenvalue $z$ is well 
approximated by a random linear combination of the singular vectors of  
$P_N-z$ corresponding to its small singular values. 
We prove precise probabilistic 
bounds on the local distribution of the eigenvalues of the perturbed matrix
and provide a detailed analysis of the 
singular vectors to conclude the localization result.
 \end{abstract}
 \maketitle
 \setcounter{tocdepth}{1}
%
\section{Introduction and statement of results}\label{int}
\subsection{The setting}\label{sec:setting}
The spectrum of non-Hermitian operators is inherently sensitive to tiny 
perturbations due to the fact that their resolvent may be large even far 
away from the spectrum. This is in stark contrast to the Hermitian case, where 
due to the spectral theorem, the {norm of the} resolvent is effectively controlled by the 
distance of the spectral parameter to the spectrum. This spectral instability 
of non-Hermitian operators, although traditionally an adversary for numerical 
analysis \cite{Tr97,TrEm05}, has recently shown itself at the origin of beautiful new results 
in a variety of contexts. For instance in the theory of non-linear partial differential equations 
for instance, spectral instability may help to explain the blow up in finite time of 
solutions to certain non-linear diffusion equations which, when solely studying 
the spectrum of the linearized operator, were expected to have stable solutions 
\cite{SaSc05,RaZw,Ga12}. 
\par
In mathematical physics non-Hermitian operators appear in a large variety 
of subjects, such as open quantum systems \cite{NoZw07,NoZw09,NoSjZw14}. 
In quantum mechanics, the study of scattering systems 
\cite{En83,HeSj86,SiSo87,SjZw91,Me95,DyZw19} naturally leads to the concept 
of \emph{quantum resonances} which can be described by the eigenvalues of a 
non-Hermitian operator obtained from
a complex deformation \cite{SjZw91} of a 
Hermitian quantum Hamiltonian. In physical models an ``ideal'' operator can 
be perturbed by many different sources, some of which are uncontrolled by 
experimentalists. To account for these error terms disorder is introduced and, 
in view of the phenomenon of spectral instability, 
it is therefore relevant to investigate the influence of 
random perturbations 
on the spectral data of non-Hermitian operators. 
The recent works \cite{Sj14,Kl16,Dr18} 
investigate for instance the distribution of the quantum resonances of 
random Schr\"odinger operators such as the 
celebrated Anderson model \cite{An58}.

\par
In this paper we consider large deterministic non-Hermitian $N\times N$ Toeplitz 
matrices $P_N$ {with} small {additive} random perturbations. It was 
shown in a series of recent results that the spectra of such matrices, apart from 
finitely many fluctuating outliers \cite{BPZ,BPZ1,BZ,SjVo19a,SjVo19b}, mimic the 
absolutely continuous spectra of 
the associated infinite dimensional Laurent operator on $\Z$. This is particularly 
striking since a perturbation of size {$O(N^{-\infty})$} is sufficient to produce this effect, 
whereas the spectrum of the unperturbed matrix is far from the spectrum of the Laurent 
operator.   
\par
The aim of this paper is to discuss the eigenvectors associated with the eigenvalues 
of such perturbed Toeplitz matrices. Are the eigenvectors \emph{localized} or 
\emph{delocalized}? The precise meanings of these notions vary over 
different subjects, however, they all serve to capture how much an $\ell^2$ normalized 
eigenvectors \emph{concentrates on} or \emph{spreads out over} certain parts of its support. 

In random matrix theory there are several ways of testing for localization or delocalization 
of $\ell^2$ normalized eigenvectors $\psi$. One way is by comparing their $\ell^p$ norms 
for $2<p\leq +\infty$ with $N^{1/p-1/2}$. Complete delocalization is said to occur when 
$\|\psi\|_p \, = O( N^{1/p-1/2})$ (up to some logarithmic factors) since $N^{1/p -1/2}$ 
is the $\ell^p$ norm of the fully delocalized vector 
$(N^{-1/2},\dots,N^{-1/2})$. Conversely localized eigenvectors have a large $\ell^p$ norm, 
as for instance the fully localized vector $(0,0,1,0,\dots,0)$ has $\ell^p$ norm equal to one. 
These notions were used for instance to prove delocalization via optimal $\ell^p$ bounds of 
the eigenvectors of Wigner matrices \cite{ESY09a,ESY09b}, for non-Hermitian random 
matrices \cite{RV}, and for the adjacency matrix of Erd{\H o}s-R\'enyi graphs \cite{EKYY13}. 
Recently,  localization and delocalization of eigenvectors for the adjacency matrix 
of \textit{critical} Erd{\H o}s-R\'enyi graphs were established in \cite{ADK}.
\par
There is a complementary notion of delocalization, known as \textit{no-gaps delocalization}, 
which asserts that for any subset $I \subset [N]$, with $|I|$ reasonably large, one has 
$\|\psi\|_{\ell^2(I)} \gtrsim |I|$ (again allowing for logarithmic factors). Recently, such delocalizations 
have been established for Wigner matrices and matrices with independent and identically 
distributed (i.i.d.) entries (cf.~\cite{RV, LO, LT}).
%
%
\par
In the field of \emph{quantum chaos} \cite{Sa92}, in the setting of Hermitian 
{pseudo-differential} operators, 
localization and delocalization of normalized eigenvectors are 
studied via their associated 
\textit{semiclassical defect measures}. 
Translated to the matrix setting \cite{AnSa19}, we note that $\sum_{x=1}^N|\psi(x)|^2\delta_x$, where $\delta_x$ denotes the Dirac measure at $x$, 
defines a probability measure. One says that \emph{quantum ergodicity} 
occurs when $\sum_{x=1}^Na(x)|\psi(x)|^2$ is close to $\frac{1}{N}\sum_{x=1}^Na(x)$ 
for most eigenvectors $\psi$, and \emph{uniquely quantum ergodicity} occurs when this holds 
for all eigenvectors. In contrast, \emph{scarring} occurs when we have concentration of the form 
$\sum_{x\in\Lambda}|\psi(x)|^2\geq 1-\varepsilon$, $\varepsilon>0$, of the eigenvector 
on some small set $\Lambda$. On the other hand, if $\|\psi\|_{p} \asymp N^{f(p)}$ for 
some $f(p) \ne 1/2 - 1/p$, then the eigenstate $\psi$ is termed to be \textit{non-ergodic} and 
\textit{multi-fractal} \cite{LAKS}. These notions were recently applied to 
the study
of the 
eigenfunctions of the discrete Laplacian on large regular graphs \cite{An17,AnMa15},  and to the proof of
delocalization of eigenvectors of generalized Wigner matrices \cite{BoYa17}. 
See also \cite{Be20} for results on deformed Wigner matrices.
\\
\par
\emph{
In this paper we prove that the eigenvectors of non-selfadjoint Toeplitz matrices subject 
to small random perturbations localize on a set of cardinality $N/\log N$ in the sense that 
they scar on a set of size $N/\log N$ with probability close to one.} 
\\
\par
To the best of our knowledge, this is the first instance where localization results are 
proved in the setting of noisy perturbations of non-Hermitian matrices.
\par
It will be seen below that for eigenpairs $(\lambda, \psi)$,
the length of the localizing 
set for $\psi$ and the rate of decay of the slowest decaying \textit{pure state} associated 
to $\lambda$ have the same order of magnitude (see Remark \ref{rem:pstate}). On the 
other hand, pure states can be related to the {\em Lyapunov spectra} of the associated 
{\em transfer matrices} (cf.~\cite[Section 1.2]{BPZ}). 
Therefore, the reader may note that our result on the localization has the same flavor 
as the one predicted in the case of the \textit{random Schr\"odinger operator} on a strip, where 
it is conjectured that the rate of decay of eigenfunctions is neither slower nor faster than 
the one prescribed by the slowest \textit{Lyapunov exponent} (see \cite{GS} and the 
references therein).
\subsection{The results}
Let $N_{\pm}\in\Z$ {be} such that  $-N_- \leq N_+$ and either 
$N_+\neq 0$ or $N_-\neq 0$. 
Let $a_i\in \C${, $i\in\Z$,} be such that $a_{N_+}\neq 0$, $a_{-N_-}\neq 0$, and $a_i=0$
for $i\not\in [-N_{-},N_+]$. 
Introduce the symbol $p(\zeta)=\sum_{-N_-}^{N_+} a_j \zeta^{-j}$ and
the associated
$N\times N$ \textit{Toeplitz} matrix $P_N$  with entries
$P_N(i,j)=a_{i-j}$, that is
\begin{equation}\label{eq-intro}
 	P_N = \begin{pmatrix}
	  a_0 & a_{-1} & \dots & a_{-N_{-}} & \dots \\
	  a_{1} & a_0 & a_{-1} & \dots& \dots\\
		\vdots & \ddots & \ddots & \ddots&\vdots\\
		a_{N_+}&\dots&\dots &\dots&\dots\\
		\vdots&\ddots&\ddots&\ddots&\vdots&\\
		\dots&\dots&a_{N_+}&\dots&a_0
	\end{pmatrix}. 
      \end{equation}
      (We refer to Section \ref{Sec:AnaTO} for an introduction to 
      the terminology, especially with respect to symbols.)
\par
We consider in this paper noisy perturbations of $P_N$ of the form
\begin{equation}
  \label{eq-noisy}
  P_{N,\gamma}^Q =P_N+N^{-\gamma} Q_N,\quad \gamma>1,
\end{equation}
with $Q_N$ an $N\times N$ (random) matrix satisfying Assumptions 
\ref{assump:mom} and \ref{assump:anticonc} below. The first assumption is on the
existence of finite moments. 
\begin{assump}\label{assump:mom}
Let $\{Q_{i,j}\}_{i,j=1}^N$ be the entries of the $N \times N$ noise matrix $Q=Q_N$. 
\begin{enumerate}
\item[(i)] The entries of $Q$ are jointly independent and have zero mean.
\item[(ii)]
 For any $h \in \N$ there exists an absolute constant $\mathfrak{C}_h < \infty$ such that
 \[\max_{i,j =1}^N \erw[|Q_{i,j}|^{2h}] \le \mathfrak{C}_h.\] 
\end{enumerate}
\end{assump}
For notation convenience, we take the sequence $\mathfrak{C}_h$ increasing in $h$. 
To introduce the second assumption, recall L\'{e}vy's concentration function, defined 
for any complex-valued random variable $\mathds{X}$ and $\varepsilon >0$ by 
\begin{equation}
  \label{eq-levy}
\cL(\mathds{X},\varepsilon) \defeq \sup_{w \in \C} \prob(|\mathds{X}-w| \le \varepsilon).
\end{equation}
\begin{assump}
\label{assump:anticonc}
Assume that there exist absolute constants $\upeta  \in (0,{1}]$ and $C_{\ref{assump:anticonc}}<\infty$, such that
\begin{equation}\label{eq:levy-bd}
{\cL(Q_{i,j},\varepsilon) \le C_{\ref{assump:anticonc}} \varepsilon^{1+\upeta}},
\end{equation}
 for all sufficiently small $\varepsilon>0$, uniformly for all $N$ and
 $i,j \in\{1,2,\ldots, N\}$.
\end{assump}
\noindent (The standard example of a noise matrix satisfying Assumptions 
 \ref{assump:mom} and \ref{assump:anticonc} is the complex Ginibre matrix, 
 i.e.~with i.i.d. entries that are standard complex Gaussian variables.)

It was recently shown in \cite{BPZ, BPZ1, SjVo19b, SjVo19a} that all
but $o(N)$ of the eigenvalues $\{\lambda_i^N\}$
of $P_{N,\gamma}^Q$ lie in a small neighborhood
of the curve $p(S^1)$, where $S^1:=\{z\in \C: |z|=1\}$; in fact, it was shown 
in those references that 
the empirical measure of eigenvalues of  $P_{N,\gamma}^Q$,
\begin{equation}\label{eq:L-N}
L_N:=N^{-1}\sum_{i=1}^N \delta_{\lambda_i^N},
\end{equation}
converges weakly
to the push forward of the uniform measure on $S^1$ by $p$. As part of our
study, we will obtain more precise information, and show (see Theorem
 \ref{theo-location} and Sections \ref{sec-separation}-\ref{sec:bulk-eig}) 
that most of the
eigenvalues lie in certain neighborhoods of width of order $\log N/N$ 
that are separated from $p(S^1)$ by distance of the same order.

Our goal in this paper is to study the eigenvectors associated with 
the latter (random) eigenvalues. Roughly speaking, we will show that those 
we will show that those eigenvalues $\hat{z}$ away from certain isolated bad points 
of $p(S^1)$ have corresponding eigenvectors which are close to a random linear 
combination of  the eigenvectors $e_j$ of $(P_N-\hat{z}I)^*(P_N-\hat{z}I)$ associated 
with its smallest eigenvalues. 
In particular we will show that this random linear 
combination of vectors localizes at scale $N/\log N$. 
To state our results precisely requires the introduction of some machinery, which we
now do.
\par
Sometimes,  the symbol $p$ possesses a natural contraction,
defined as follows. Set
\begin{equation}\label{eq:gpintro}
  {\sf g}(p):={\rm gcd}\{ |j| : j \ne 0 \text{ and } a_j \ne 0\}.
\end{equation}
If ${\sf g}_0:={\sf g}(p)>1$ then $p(\zeta)=q_p(\zeta^{ {\sf g}_0})$
for some Laurent polynomial $q_p$. If ${\sf g}_0=1$ then $q_p=p$. For $\vep>0$ 
and a set $\cB\subset \C$, $\cB^\vep$ denotes the $\vep$-blow up of $\cB$, that is 
the Minkowski sum of the sets $\cB$ and  $D(0,\vep)$, the open disc of radius 
$\vep$ centered at zero. 
\begin{defn}[Set of bad points]
  \label{def-badintro}
Let $\cB_1$ be the collection of self intersection points 
of $q_p(S^1)$, and let
$\cB_2$  be the set of {branch points}, 
i.e.~points $z$ 
where the Laurent polynomial $p(\cdot)-z$ has double roots. Set $\cB_p:=\cB_1\cup\cB_2$ and $\cG_{p,\vep}:=p(S^1)\setminus \cB_p^\vep$.
\end{defn}
\noindent 
In Definition \ref{def-badintro}, a point $w\in \C$ is a self intersection point of $q_p(S^1)$ 
if there exist $\zeta_1\neq \zeta_2\in S^1$ so that $q_p(\zeta_1)=q_p(\zeta_2)=w$.
\par
Throughout the paper, we make the following assumption on the symbol $p$.
\begin{assu}
  \label{assu-symbol}
  The symbol $p$ satisfies 
    $a_{-N_-},a_{N_+}\neq 0$, and $\cB_1$ is a finite set.
\end{assu}
Under Assumption \ref{assu-symbol}, $\cB_p$ is a finite set. Indeed, $\cB_2$ is precisely 
the set of all $z$'s such that the discriminant  of the polynomial 
$\zeta \mapsto \zeta^{N_+} p(\zeta)-z$ vanishes, and \cite[Lemma 11.4]{BoGr05} yields that 
$\cB_2$ is  a finite set. We note that by \cite{KK19}, unless $N_-=N_+$ and 
$|a_{-N_-}|=|a_{N_+}|$,  $\cB_1$ has cardinality bounded above by $(N_++N_--1)^2$, so 
symbols avoiding this situation satisfy Assumption \ref{assu-symbol}. For $z\in \C$ let 
\begin{equation*}
	\wind(z)={\rm ind}_{p(S^1)}(z)
\end{equation*}
denote the winding number of the curve $p(S^1)$ around $z$. 
We now describe the collection of eigenvalues of interest to us. 
For $0<\vep,C<\infty$ and $N$ large enough so that $2C\log N/N<\vep$, set
\begin{equation}
  \label{eq-Omegadef}
  \Omega(\vep,C,N):=\{z\in \C: C^{-1} \log N/N<
  {\mathrm{dist}}(z,\cG_{p,\vep})<C\log N/N, \wind(z)\neq 0\},
\end{equation}
where for a set $\cB \subset \C$ and $w \in \C$ we denote ${\mathrm{dist}}(w, \cB):=\inf_{w' \in \cB} |w-w'|$. 
Let $\mathcal{N}_{\Omega(\vep,C,N),N,\gamma}:=|\{\lambda_i^N\in\Omega(\vep,C,N)\}|$ 
denote the number of eigenvalues of $P_{N,\gamma}^Q$ that lie in $\Omega(\vep,C,N)$.
\par
The following theorem, a combination of the convergence of $L_N$ discussed above and 
Theorems \ref{thm:no-outlier}, \ref{thm:sep-spec-curve}, and \ref{thm-thintube} below, shows that
\textit{most} eigenvalues of $P_{N,\gamma}^Q$ lie in $\Omega(\vep,C,N)$ for appropriate $\vep,C$. 
\begin{thm}
  \label{theo-location}
  Let Assumptions \ref{assump:mom}, \ref{assump:anticonc}, and \ref{assu-symbol} hold. Fix $\mu>0$ and $\gamma>1$.
  Then there exist $ 0< \vep_{\ref{theo-location}}, C_{\ref{theo-location}} < \infty$ (depending on $\gamma,\mu$ and $p$ only) 
  so that 
  \begin{equation}
    \label{eq-finalloc}
    \prob\big(\mathcal{N}_{\Omega(\vep_{\ref{theo-location}},C_{\ref{theo-location}},N),N,\gamma}< (1-\mu) N\big)\to_{N\to\infty} 0.
  \end{equation}
\end{thm}
Theorem \ref{theo-location} may be of independent interest 
since it improves upon previous 
results \cite{BPZ, BPZ1, SjVo19b, SjVo19a} by providing a much sharper estimate
on the position of 
the eigenvalues of $P_{N,\gamma}^Q$. We refer the reader to Sections 
\ref{sec-separation}--\ref{sec:bulk-eig} for its proof. In what follows, we fix $\mu>0$, $\gamma>1$ and consider the
$\vep_{\ref{theo-location}}$ and $C_{\ref{theo-location}}$ determined by Theorem \ref{theo-location}. We then consider
eigenvalues $\lambda_i^N\in \Omega(\vep_{\ref{theo-location}},C_{\ref{theo-location}},N)$. By Theorem 
\ref{theo-location}, \textit{most} eigenvalues are of this type. Notice also that for any $z\in \Omega(\vep_{\ref{theo-location}},C_{\ref{theo-location}},N)$, we have that $d=\wind(z)\neq 0$.
\par
\begin{figure}[ht]
  \centering
   \begin{minipage}[b]{0.45\linewidth}
   \includegraphics[width=\textwidth]{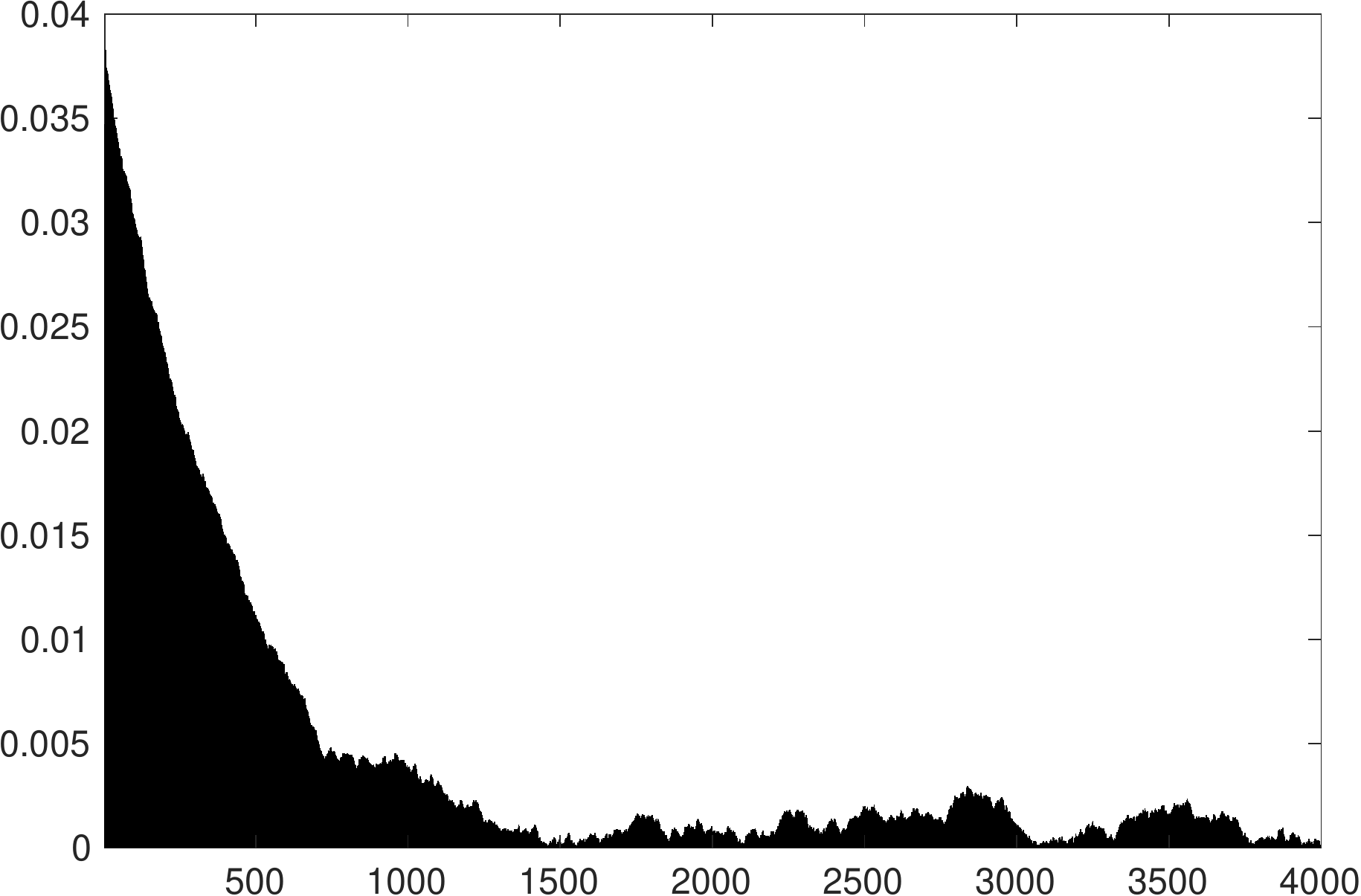}
 \end{minipage}
 \hspace{1cm}
 \begin{minipage}[b]{0.313\linewidth}
  \includegraphics[width=\textwidth]{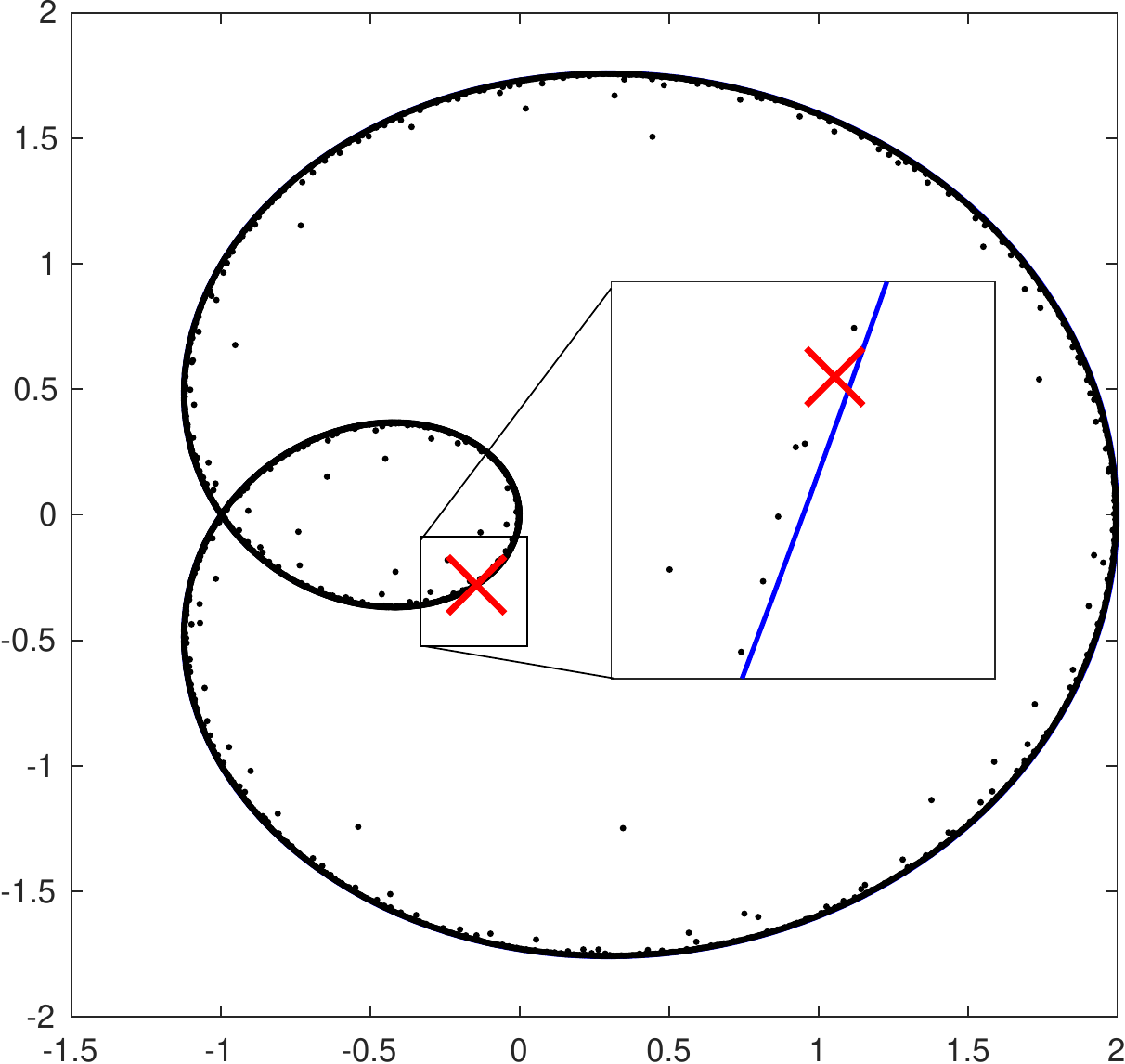}
 \end{minipage}
 \\[1.5ex]
 \begin{minipage}[b]{0.46\linewidth}
   \includegraphics[width=\textwidth]{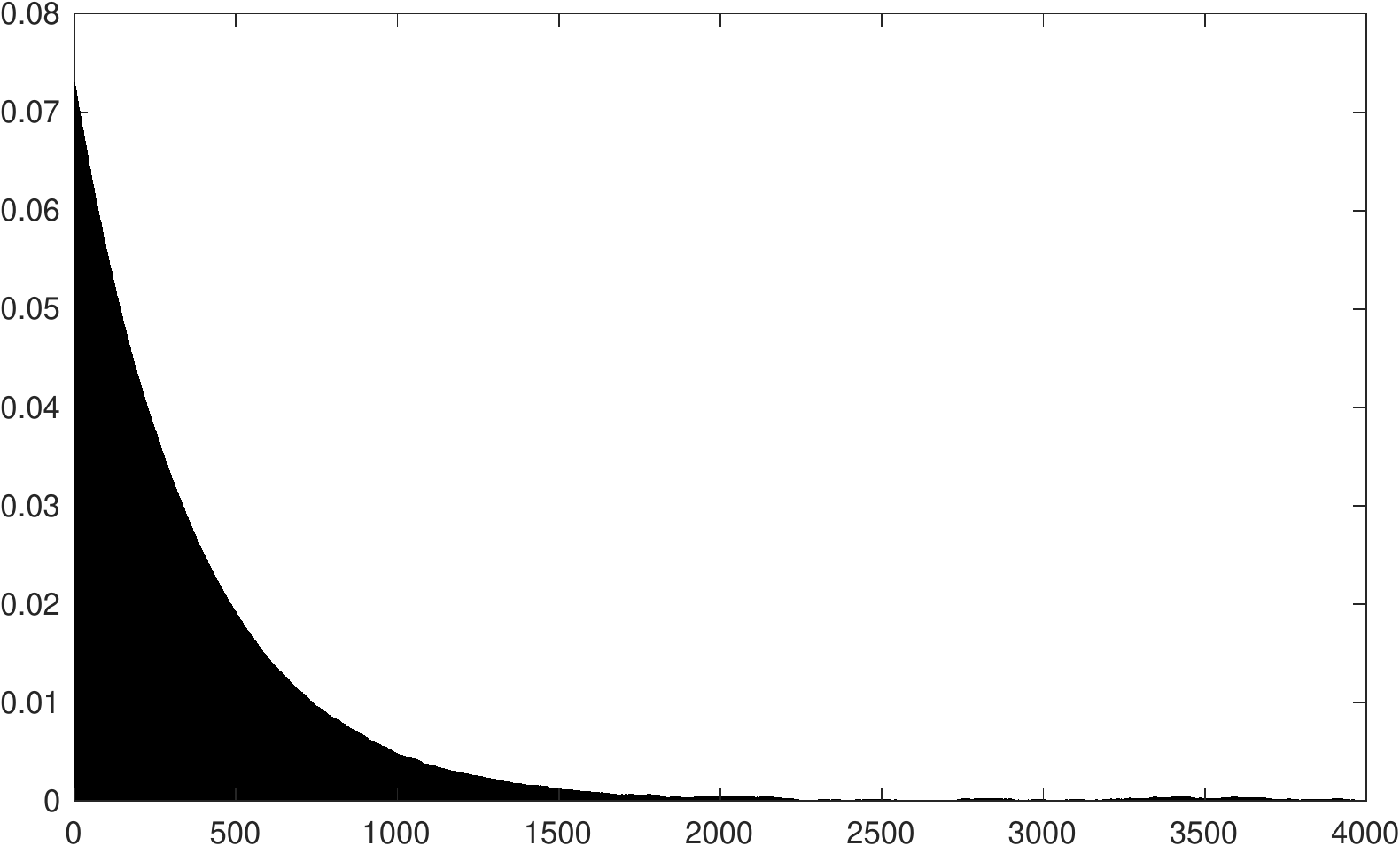}
 \end{minipage}
  \hspace{1cm}
 \begin{minipage}[b]{0.3\linewidth}
  \includegraphics[width=\textwidth]{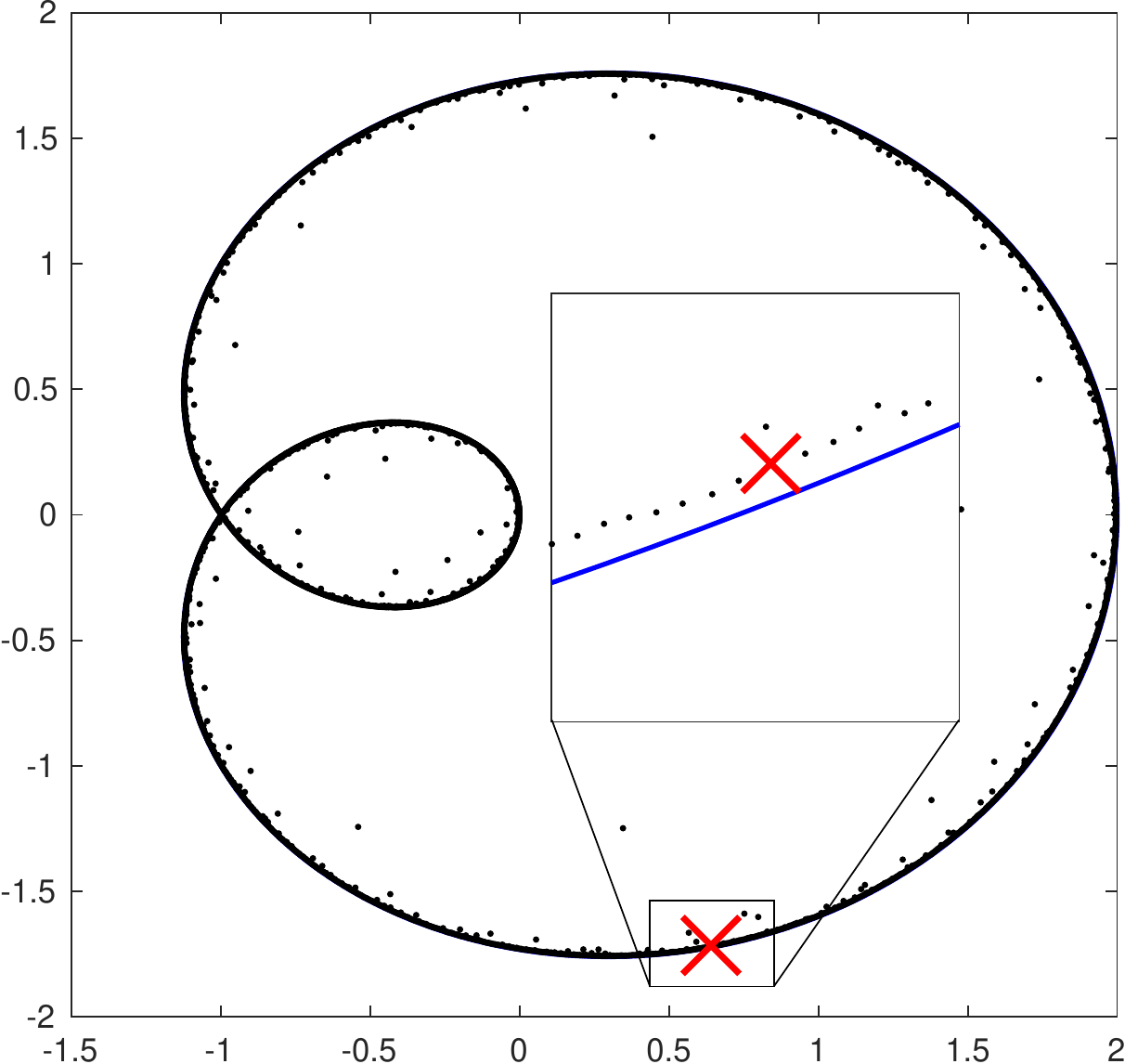}
 \end{minipage}
 \\[1.5ex]
 \begin{minipage}[b]{0.47\linewidth}
   \includegraphics[width=\textwidth]{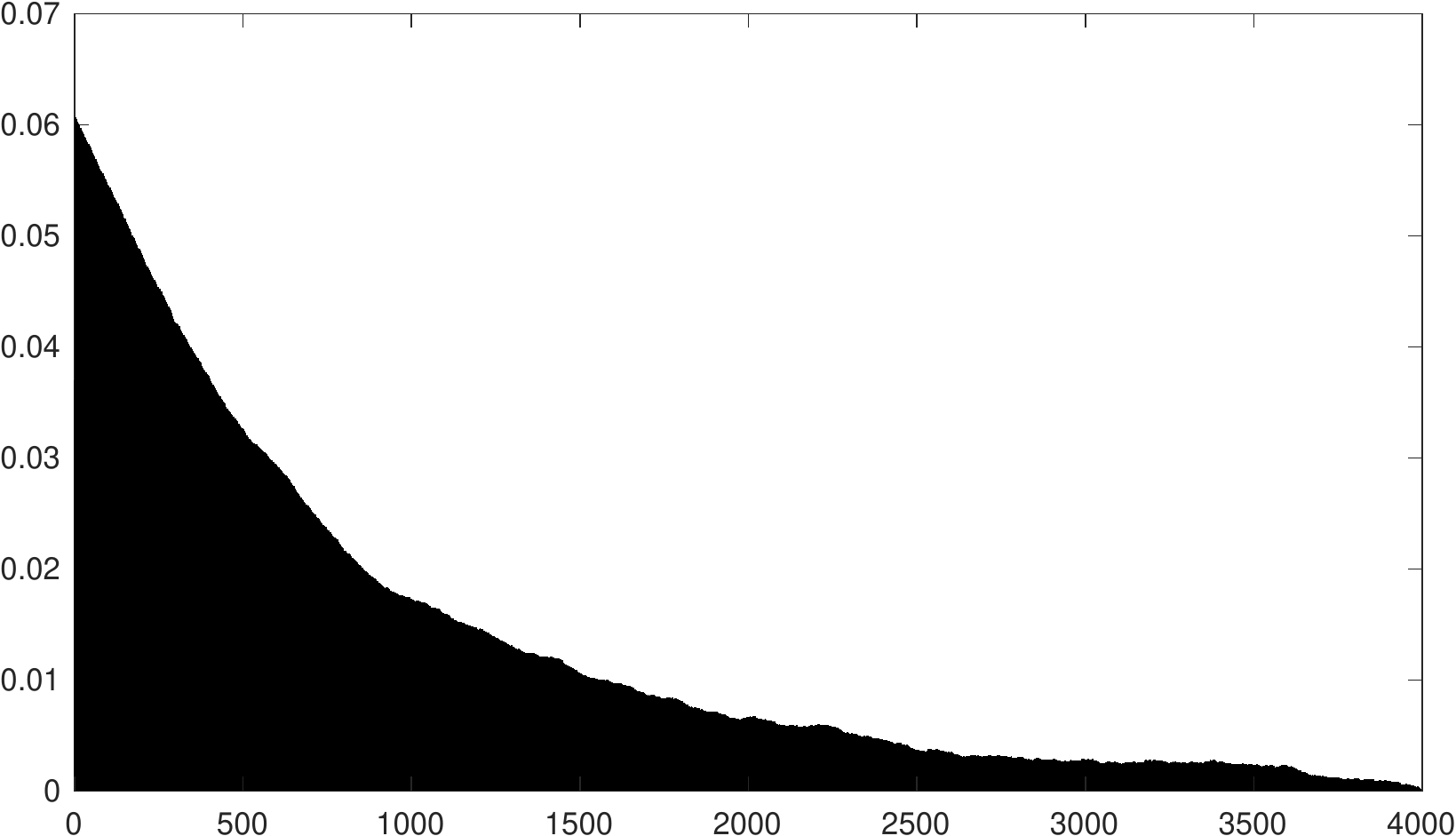}
 \end{minipage}
  \hspace{1cm}
 \begin{minipage}[b]{0.3\linewidth}
  \includegraphics[width=\textwidth]{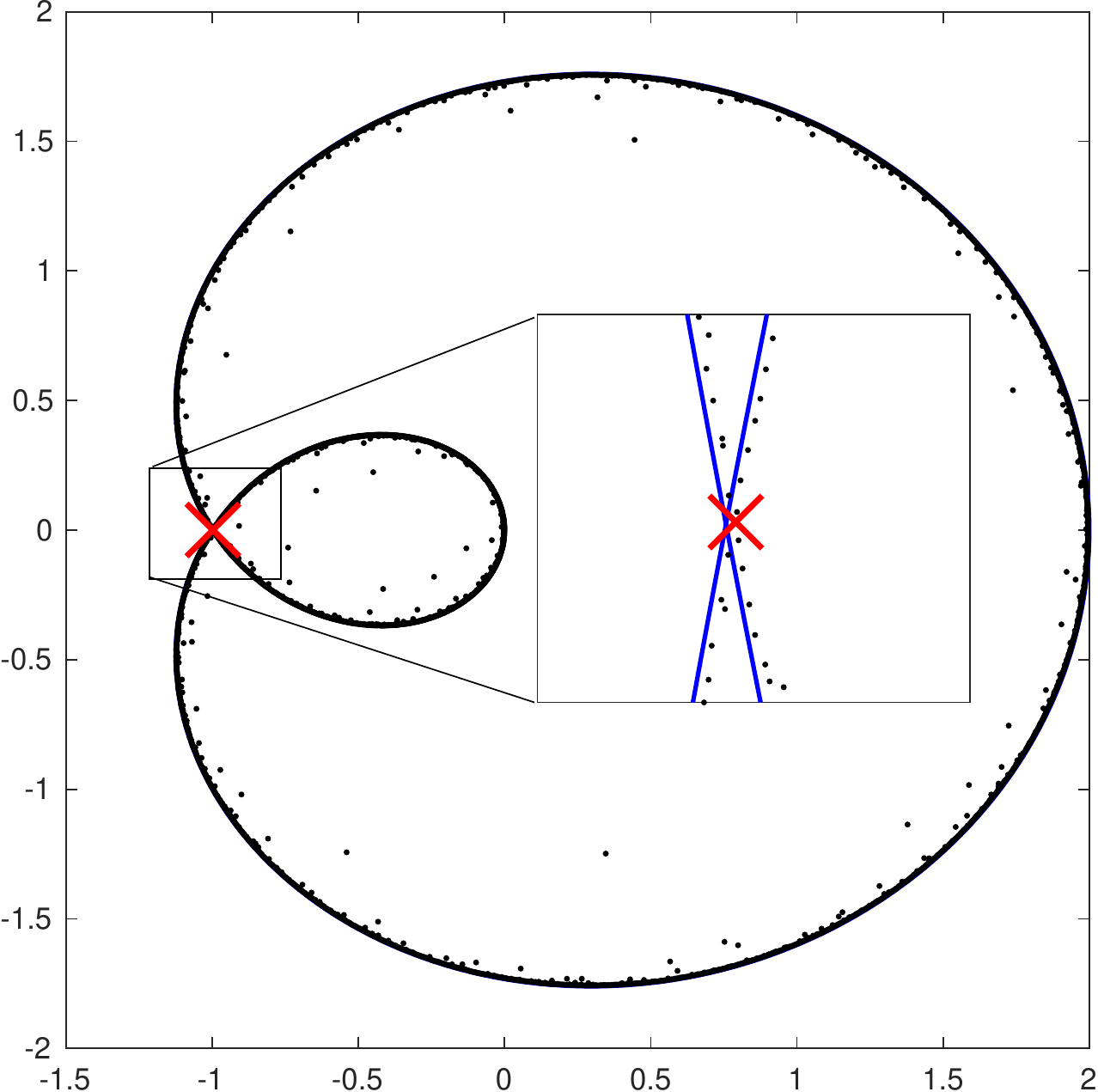}
 \end{minipage}
 \caption{Eigenvectors (left panel) and eigenvalues (right panel) for 
 $N=4000$, $\gamma=1.2$ and symbol $\zeta+\zeta^{2}$.
Plotted are the moduli of the entries of the eigenvector that
 corresponds to the eigenvalue marked with a red $\times$. 
 The top two rows correspond to situations covered by Theorem 
 \ref{thm:main};
 note the localization, which occurs at scale $N/\log N$.
 The bottom row is not covered by Theorem \ref{thm:main}, because
 the chosen eigenvalue is at vanishing distance from $\cB_1$.
}
  \label{fig1}
\end{figure}
The main result of this paper is the following description of the 
(right) eigenvectors of $P_{N,\gamma}^Q$.
\begin{thm}
  \label{thm:main}
  Fix $\vep_{\ref{theo-location}},C_{\ref{theo-location}}$ and the notation as above.\par
1.  The following occurs with probability approaching one as $N\to\infty$.
For each $\hat z\in \Omega(\vep_{\ref{theo-location}},C_{\ref{theo-location}},N)$ which is an  eigenvalue
  of $P_{N,\gamma}^Q$,  let $v=v(\hat z)$ denote the corresponding (right) 
  eigenvector, normalized so that $\|v\|_2=1$. Then there
  exists a vector $w$ with 
  $\|w\|_2=1$ such that
  \begin{equation}
    \label{eq-main1}
    \|v-w\|_2=o(1),
  \end{equation}
 and  a constant $c>0$, depending on $\gamma$, so that for any $\ell\in [N]$,
  \begin{equation}
    \label{eq-main2}
    \begin{array}{ll}
      \|w\|_{\ell^2([\ell,N])}
      \leq \e^{-c\ell\log N/N}/c 
      , & \mbox{\rm if $d>0$},\\
     \|w\|_{\ell^2([1,N-\ell])}
      \leq \e^{-c\ell\log N/N}/c 
      , & \mbox{\rm if $d<0$}.
    \end{array}
  \end{equation}
  The vector $w$ can be taken as a (random) linear combination of the 
  $|d|$ eigenvectors of $(P_N-\hat zI)^*(P_N-\hat zI)$ corresponding to the
  $|d|$ smallest eigenvalues. 
  \par
2. Fix $z_0=z_0(N)\in \Omega(\vep_{\ref{theo-location}},C_{\ref{theo-location}},N)$ deterministic, 
$C_0$, $\wt C_0$ large, and $\eta>0$ small. Then, there exist constants 
$c_1=c_1(\eta,C_0, \wt C_0)$ and $c_0=c_0(\gamma) \in (0,1)$, 
with $c_0 \to 1$ as $\gamma \to 1$ and $c_0 \to 0$ as $\gamma \to \infty$,  so that, with probability 
at least $1-\eta$, for every $\hat z=\lambda_i^N\in D(z_0,C_0\log N/N)$, any 
$0<\ell\leq \ell'\le \wt C_0 N/\log N$ satisfying $\ell'-\ell>N^{c_0}$ and all large $N$,
 \begin{equation}
    \label{eq-main3}
    \begin{array}{ll}
      \|w\|_{\ell^2([\ell,\ell'])}^2
      \geq c_1 (\ell'-\ell)\log N/N, & \mbox{\rm if $d>0$},\\
      \|w\|_{\ell^2([N-\ell',N-\ell])}^2
    \geq c_1 (\ell'-\ell)\log N/N, & \mbox{\rm if $d<0$}.
    \end{array}
  \end{equation}
 Further, for any $0<c'\leq \wt C_0$,
 \begin{equation}
    \label{eq-main4}
    \begin{array}{ll}
      \|v\|_{\ell^2([1,c'N/\log N])}^2
      \geq c'c_1/2, & \mbox{\rm if $d>0$},\\
      \|v\|_{\ell^2([N-c'N/\log N,N])}^2
    \geq c'c_1/2,  & \mbox{\rm if $d<0$}.
    \end{array}
  \end{equation}
\end{thm}
\vskip5pt
Theorem \ref{thm:main} shows a \textit{localization} phenomenon, 
numerically illustrated in the examples of Figure \ref{fig1}: 
for all eigenvalues in the good regions, the corresponding 
eigenvectors localize at scale $N/\log N$, and 
for most eigenvalues, this is the scale at which the 
eigenvector is ``spread out''. 
(Contrast the situation with the regime $\gamma<1$, where
delocalization is observed in simulations, see Figure \ref{fig1a};
we discuss predictions for that regime in Section \ref{sec:gam-l-1},
after we introduce relevant notions and in particular the relevant
{\em Grushin problem}.)
\par
Building on Theorem \ref{thm:main}, equipped with a local estimate 
on the number of eigenvalues $\lambda_i^N$ in regions of diameters $O(\log N/N)$ 
(see Theorem \ref{thm-thintube2}) and applying a Fubini type argument, 
one can show that except for an arbitrarily small fraction of the eigenvalues, 
the corresponding eigenvectors $v(\lambda_i^N)$ localize at 
scale $N/\log N$. In particular, we prove the following result.
\begin{cor}\label{cor:main} 
Let Assumptions \ref{assump:mom}, \ref{assump:anticonc}, and \ref{assu-symbol} hold. Then, for 
any $\mu>0$ there exists  $\mu_1,\mu_2>0$ so that
with $|\supp_{\mu_1}(v)|:=\min \{ |I|: \|v\|_{\ell^2(I)}>1-\mu_1\}$,
\[\limsup_{N \to \infty}\frac1N \E\#\{i: \supp_{\mu_1}(v(\lambda_i^N))<\mu_2 N/\log N\}\leq \mu.\]
\end{cor}

\begin{figure}[ht]
  \centering
  \begin{minipage}[b]{0.45\linewidth}
   \includegraphics[width=\textwidth]{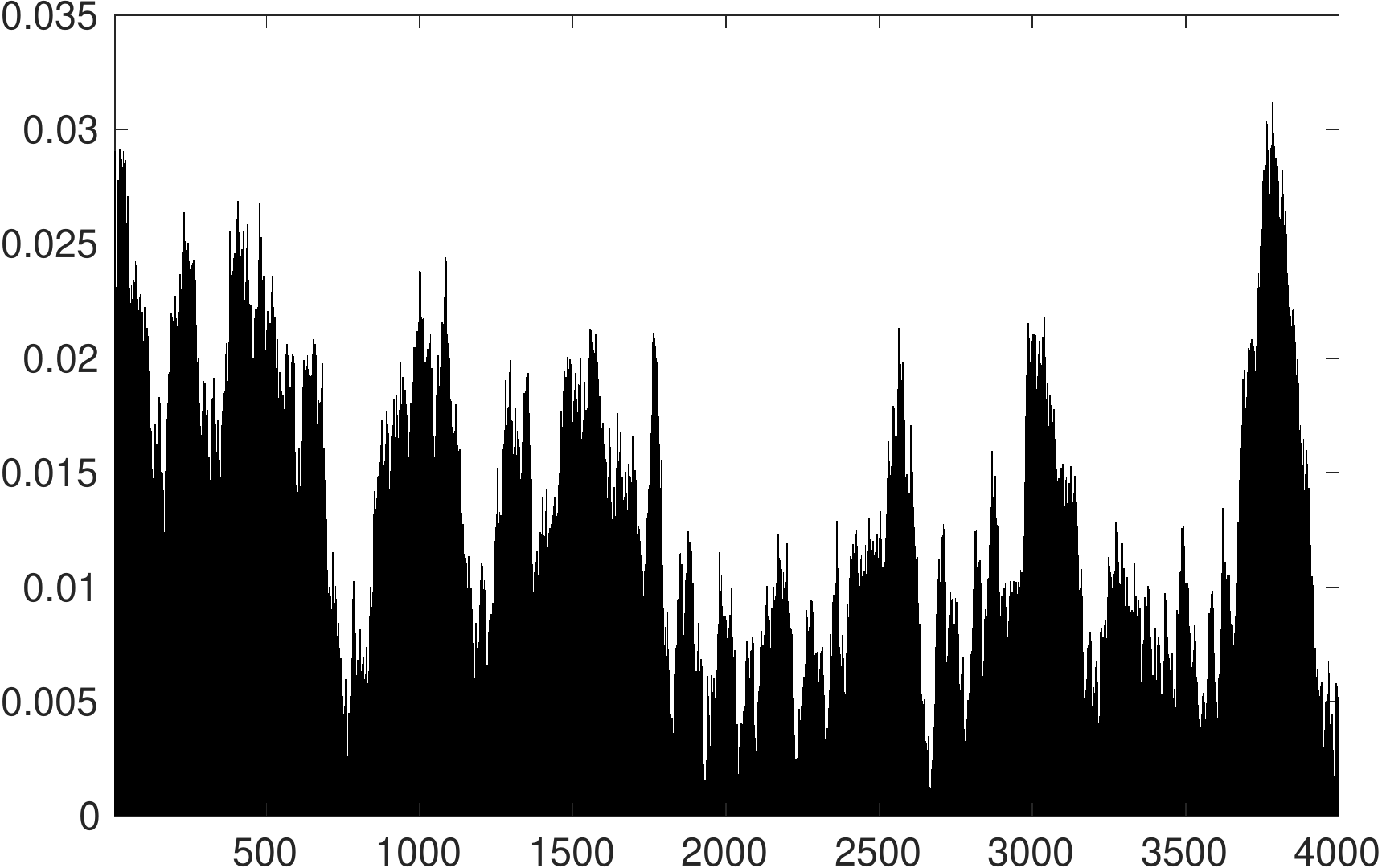}
 \end{minipage}
 \hspace{1cm}
  \begin{minipage}[b]{0.29\linewidth}
  \includegraphics[width=\textwidth]{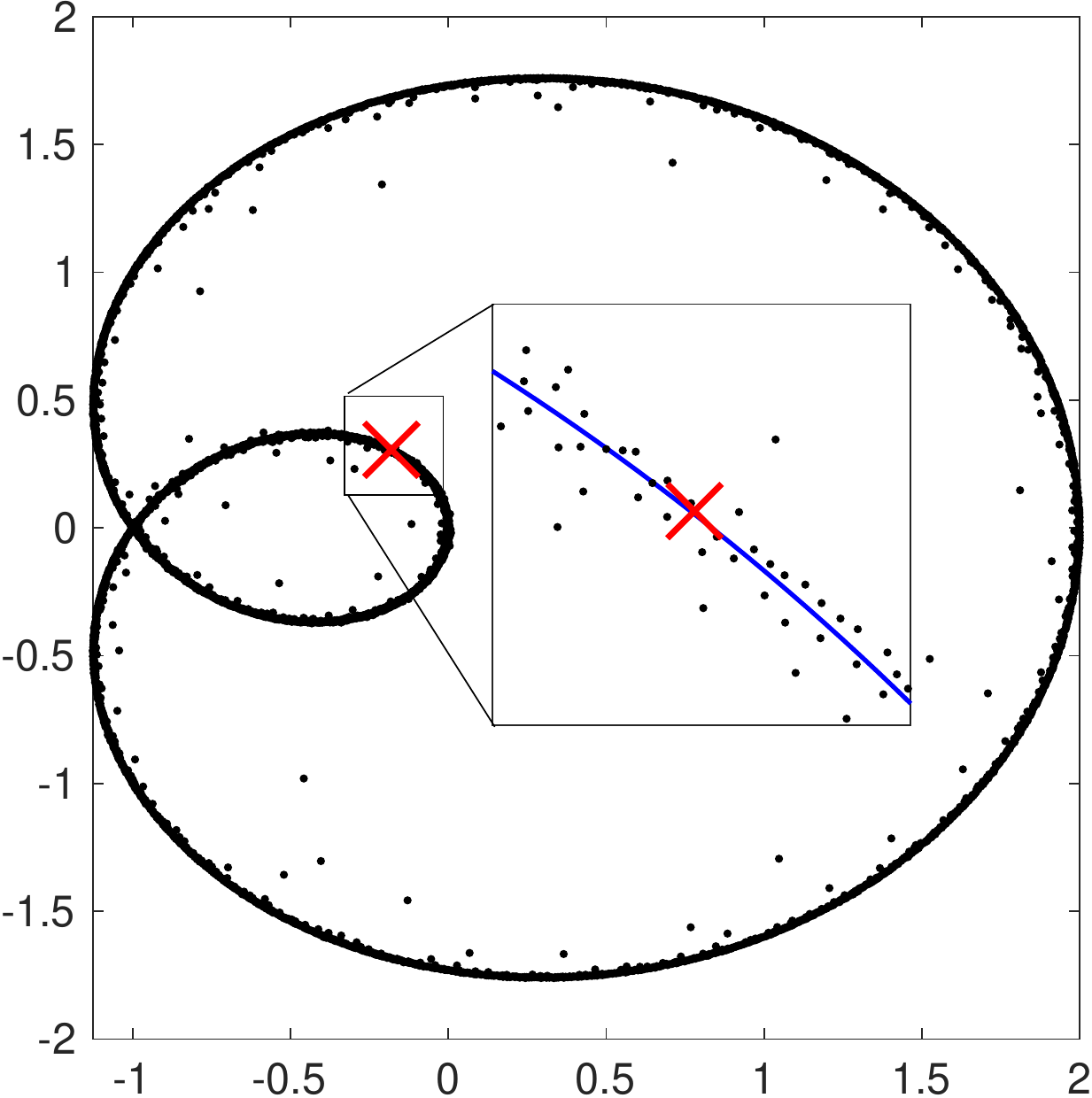}
 \end{minipage}
\\[1.5ex]
  \begin{minipage}[b]{0.455\linewidth}
   \includegraphics[width=\textwidth]{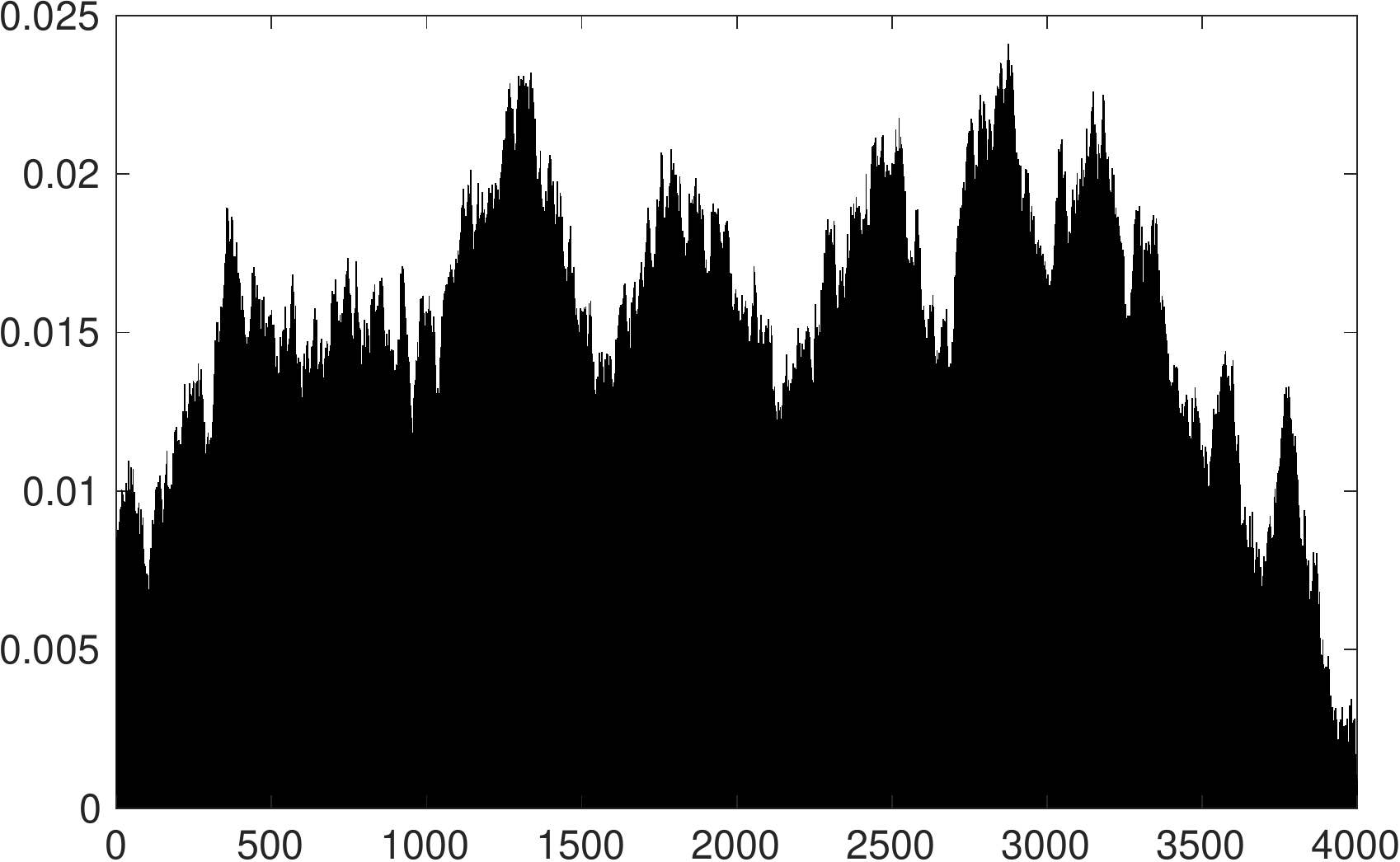}
 \end{minipage}
 \hspace{1cm}
  \begin{minipage}[b]{0.29\linewidth}
  \includegraphics[width=\textwidth]{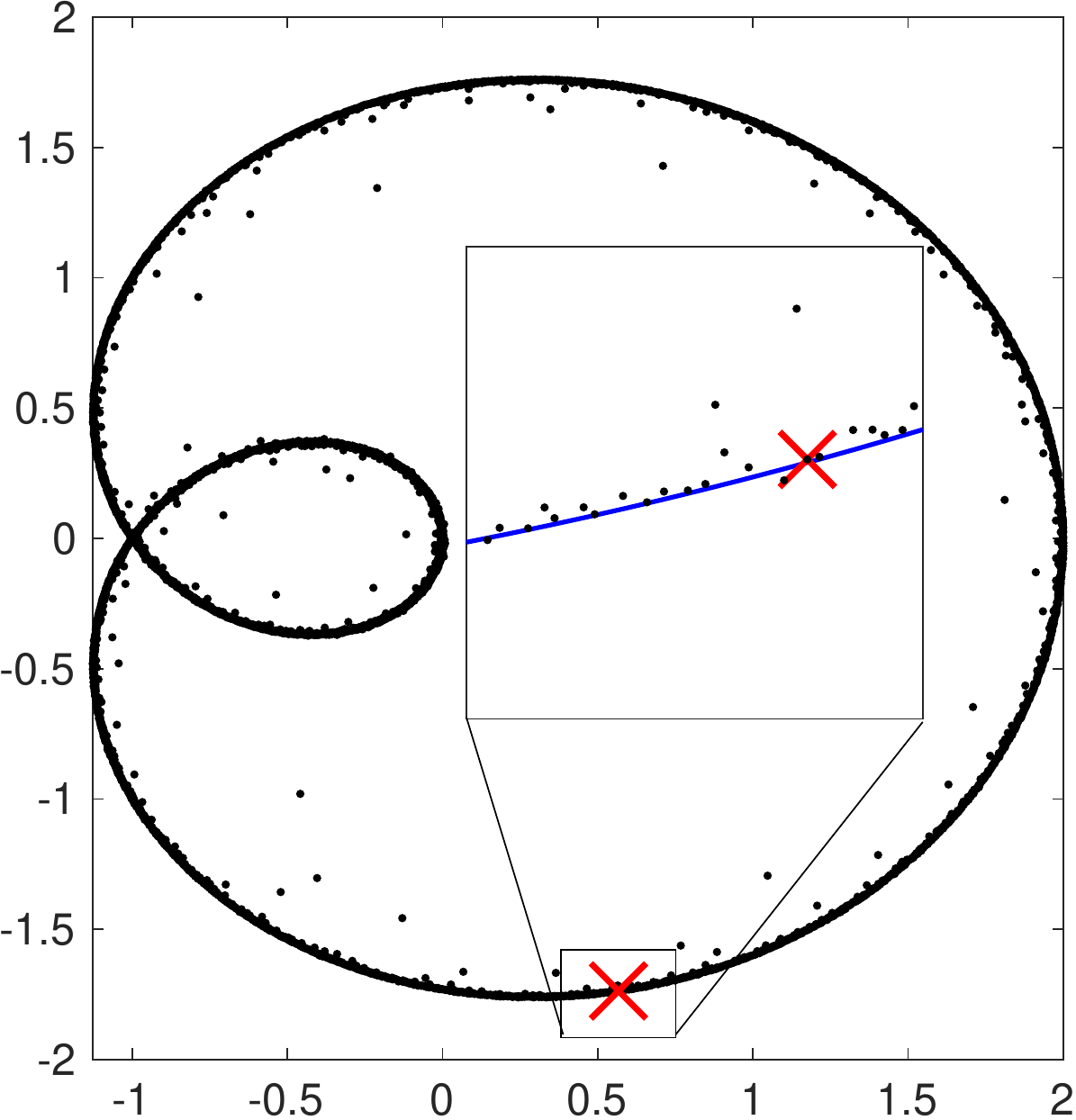}
\end{minipage}
 \caption{Eigenvectors (left panel) and eigenvalues (right panel) for 
 $N=4000$, $\gamma=0.8$ and symbol $\zeta+\zeta^{2}$.
 These cases are not covered by Theorem \ref{thm:main}.
 Note the  stark difference with the corresponding situations in Figure 
 \ref{fig1}, in both the location of eigenvalues relative to $p(S^1)$ and
 in the localization properties of eigenvectors.}
  \label{fig1a}
\end{figure}
$\phantom{.}$
\\[2ex]
\noindent
Theorem \ref{thm:main} states that the eigenvectors of $P_{N,\gamma}^Q$ corresponding 
to most eigenvalues $\hat{z}$ can be approximated by a random linear combination 
of the eigenvectors $e_j$ of $(P_N-\hat{z}I)^*(P_N-\hat{z}I)$ associated with its $|d|$ smallest 
eigenvalues. These are precisely the right {singular vectors} of $P_N-\hat{z}I$ associated with 
its $|d|$ smallest singular values. 
However, these singular vectors $e_j$ are (in general) difficult objects to study and 
do not admit an easy description. 
Therefore, we will
 approximate these singular vectors with certain {\em quasimodes} of 
 the operator 
$P_N$. The term quasimode for $P_N$ and a (quasi-)eigenvalue $z$ refers to 
a approximate 
$\ell^2$-normalized eigenvectors $\psi\in\C^N$ of $P_N-z$ in the sense that 
\begin{equation*}
	\| (P_N-z)\psi \| \to 0, \quad N\to \infty.
\end{equation*}
In the literature quasimodes are also referred to as 
\emph{pseudo-eigenvectors} 
or \emph{pseudomodes}. In Section \ref{sec:QM}, we describe 
a $|d|$-dimensional
space of quasimodes associated with $P_N-zI$, and in Section \ref{sec-singular}, we show 
that the eigenvectors $e_j$ of $(P_N-zI)^*(P_N-zI)$ associated with its $|d|$ smallest 
eigenvalues are close to these quasimodes. We refer to Propositions 
\ref{prop:QM}  and \ref{prop:sgvSpan} for the precise construction, which we do not repeat 
here. We emphasize however that the \emph{construction} of the $e_j$-s depends on $p$, $N$  and 
$z$ only and not on $Q$ (even if eventually the value of $z$ to which it will be applied will 
depend on $Q$).
\par
The upper bound \eqref{eq-main2} is  due to the decaying nature of these $|d|$ linearly independent quasimodes either to the left (when $d<0$) or to the right 
(when $d>0$). 
These quasimodes decay exponentially quickly, $|u(n)| \asymp \e^{-rn}$ or 
$|u(n)| \asymp \e^{-r(N-n)}$, however, at different rates $r>0$. Out of these $|d|$ 
quasimodes the first $(|d|-{\sf g}_0)$ (recall \eqref{eq:gpintro}) quasimodes decay 
at a constant rate $r>0$, resulting in them being completely localized to a point, i.e.~either on the left or right hand side of the interval $[1,N]$. In contrast, the rest decay at a rate 
$r\asymp \log N/N$, which implies that they localize at a scale $N/\log N$. In contrast,
the lower bound \eqref{eq-main3} follows upon showing that $w$ has a non-negligible 
projection (in $\ell^2$) onto $\cS$, the subspace spanned by the last ${\sf g}_0$ 
quasimodes that decay precisely at rate $\log N/N$.
\par
Note that Theorem \ref{thm:main} and Corollary \ref{cor:main} establish absences 
of quantum ergodicity and no-gaps delocalization, and show that the semiclassical 
defect measure in this setting is the Dirac measure at zero or one, depending on 
whether $d$ is positive of negative.
\par
One may wonder whether the assumption $\hat z  \in D(z_0, C_0 \log N/N)$ 
in the second part of Theorem \ref{thm:main} is optimal. It will be clear from the proof 
that to derive \eqref{eq-main3} one needs to control the supremum of the {random field} 
$z \mapsto \wt f(z)^t Q \wt e(z)$ for some $\wt e(z), \wt f(z) \in \C^N$ such that the 
Lipschitz norms of the functions $z \mapsto \wt e(z)$ and $z \mapsto \wt f(z)$ are 
$O(N/\log N)$ and $O(1)$, respectively. It is then standard to check that the supremum of 
the field $\{\wt e(z)^t Q \wt f(z)\}$ can only be bounded by an $O(1)$ quantity 
in discs of radius $O(\log N/N)$. 
The boundedness of this random field is crucial in deriving that $w$ has non-negligible 
projection onto $\cS$. Repeating the same reasoning one can also observe that 
for all $\hat z$ simultaneously in the good region, with probability approaching one, the 
$\ell^2$ norm of the projection of $w$ onto $\cS$ is at least of the order $\log N/N$, and 
we believe that this is the correct picture. 

\subsection{Connection to pseudospectra and pseudo-eigenvectors} Roughly speaking, the pseudospectrum of an operator represent the spectrum of the operator when subjected to the worst-case perturbation. However, here we study the spectrum under a {\em typical perturbation}.
Nevertheless, in many scenarios the spectrum of random perturbations of a non-self-adjoint operator closely resembles the {\em pseudospectral level lines} of the unperturbed operator (cf.~\cite[Section 1.3]{BPZ} and the references therein). There have been attempts to understand pseudospectral properties of Toeplitz matrices. It was proved that the $\vep$-pseudospectrum of  an unperturbed Toeplitz matrix converges to an $\vep$-neighborhood of the spectrum of the limiting Toeplitz operator \cite{RT}. On the other hand, in \cite{BoGr04} it is shown that for $P_N$ any {\em asymptotically good} pseudo-eigenvector must be {\em asymptotically strongly localized}, meaning that almost all its mass is carried either by the subset $[1,j_N]$ or the subset $[N-j_N, N]$ for any sequence $\{j_N\}$ such that $j_N \to \infty$ as $N \to \infty$. Note the contrast with the localization behavior of the eigenvectors of the noisy version of $P_N$:~they localize with $j_N \asymp N/\log N$. 

For a more general model, the twisted Toeplitz matrices, a special case of the Berezin-Toeplitz quantization of the two dimensional torus, it has been shown in \cite{TC} that if certain {\em (anti-)twist condition} is satisfied by the 
\emph{symbol} of the operator, then the pseudo-eigenvectors are localized in the form of {\em localized wave packets}. This has been generalized in \cite{BU}
to the
Berezin-Toeplitz quantizations of compact symplectic K\"ahler manifolds. 
In the field of microlocal analysis of pseudo-differential operators,
the analogue of this (anti-)twist condition is the 
non-vanishing of the Poisson bracket of the real and imaginary parts of the \emph{principal symbol} of the operator. This condition, also known as
H\"ormander's
commutator condition \cite{Ho85}, is a corner stone in the now classical theory of local (non-)solvability of partial differential equations. In the works \cite{D1, D2, D3, Zw01,NSjZw04,Pr08} it was linked via a WKB construction to the construction of wave packet pseudomodes of Schr\"odinger operators with complex potentials, 
and more generally to nonselfadjoint pseudo-differential operators. 
\subsection{Extensions} 
A natural extension of our results would be to the region $\gamma\in (1/2,1]$, see
Section \ref{sec:gam-l-1} for a discussion.
We mention a couple of additional 
potentially interesting extensions, of broad interest.
\\[2ex]
\noindent
\textbf{Sup-norm delocalization versus no-gaps delocalization.} As already mentioned in Section \ref{sec:setting},
there are two complementary notions of delocalizations in the random matrix literature. It is shown in \cite{BLo, RVnogap} that both these notions of delocalization hold for Wigner matrices under various assumptions on it entries. There is no reason to believe that these two notions of delocalizations should hold simultaneously in any given setting.

Indeed, from the proof of Theorem \ref{thm:main} it follows that when $P_N$ is a Jordan block, i.e.~its symbol is $p(\zeta)=\zeta$, and $\gamma >3/2$ we have that $\|v\|_\infty \asymp \sqrt{{\log N}/{N}}$ for an eigenvector $v$ corresponding to a bulk eigenvalue. See Remark \ref{rem:sup-v-nogap} for further details. {\em Thus, in this simple setting, the eigenvectors corresponding to most of the eigenvalues are completely delocalized according to the sup-norm criterion. However, they do not satisfy no-gaps delocalization. }
It is worth investigating whether one indeed has that $\|v\|_\infty \asymp \sqrt{{\log N}/{N}}$ for all $\gamma >1$ and any finitely banded $P_N$. 
\\[2ex]
\noindent
\textbf{Localization for the outlier eigenvalues and multi-fractal structure.} 
Based on simulations and some heuristic arguments,
we predict that the eigenvector $\psi$ corresponding to an eigenvalue 
residing at a distance of order $N^{\upalpha -1}$, $\upa \in (0,1]$,
from the spectral curve localizes
at scale $N^{1-\upa}$. This shows in particular
that for such a $\psi$ one has 
$\|\psi\|_p \asymp N^{(\upa-1) \cdot (\frac12 -\frac1p)}$, 
establishing that such eigenvectors are multi-fractal. 
The same reasoning shows that the eigenvectors corresponding to 
outlier eigenvalues, i.e.~those are at a distance of order one from $p(S^1)$,
would be completely localized. That is, most of their
mass is carried by  finitely many entries. 
It seems plausible that the methods of our current work
could be adapted to prove these results.

\subsection{Structure of the paper}
In Section \ref{sec-Grushin}, we introduce a Grushin problem for 
a matrix $P$, which allows us to represent  the null space of $P$, 
if non-empty,
in terms of a certain resolvent expansion, 
see Lemma \ref{lem-basicinverse}. Section \ref{res} then provides a 
sketch of the proof of Theorem \ref{thm:main}.
Section \ref{Sec:AnaTO} discusses results on Toeplitz and related operators that are used later in the paper.  Sections \ref{sec-separation}-\ref{sec:bulk-eig}
are devoted to estimates on the location of ``most'' 
eigenvalues of $P_{N,\gamma}^Q$, which together yield Theorem 
\ref{theo-location}. The following Section \ref{sec-resolvent} 
derives estimates on the resolvent of of $P_N-zI$, for appropriate $z$ 
in the ``good'' region. Section \ref{sec:QM} is devoted to 
the study of quasimodes of Toeplitz matrices, which are the building blocks for the eigenvectors of $P_{N,\gamma}^Q$ and eventually give the vector $w$
in the statement of Theorem \ref{thm:main}. Section 
\ref{sec-sing} gives estimates on the singular values and vectors of $P_N-z$, 
and bounds on various norms of matrices appearing in the Grushin problem.
Finally, Section \ref{sec-mainproof} collects all preparatory material and provides the proof of Theorem \ref{thm:main}.

\subsection{Notation}\label{sec:notation}
We use the following set of notation throughout this paper. The notation  
 $a \ll b$ means that 
$Ca \leq  b$ for some sufficiently large constant $C>0$. 
Writing $a \asymp b$ means 
that there exists a constant $C>1$ such that $C^{-1} a \leq b \leq Ca$. The notation 
$a \gtrsim b$ is used to denote that $a \ge C^{-1} b$, while we write $a = o(b)$ to denote 
$a \le C^{-1}b$ for {\em all} $C< \infty$. Constants $C, \wt C>0$ 
denoted explicitly in the proofs may change value from line to line,
while constants such as $C_{\ref{assump:anticonc}}$ will be kept fixed throughout. 
The notation $f = O(N)$ 
means that there exists a constant $C>0$ (independent of $N$) such that $|f| \leq C N$. 
When we want to emphasize that the constant $C>0$ depends on some parameter 
$k$, then we write $C_k$, or with the above notation $O_k(N)$. 
\par
We use the standard notation of ceiling and floor: for $x \in \R$ we write $\lceil x\rceil  := \max\{ n\in \Z; n\geq x\}$ and $\lfloor x\rfloor  := \max\{ n\in \Z; n\leq x\}$, respectively. For $m,n \in \Z$ the notation $[m,n]$ is used to denote the discrete interval $\{m, m+1, \ldots, n\}$ and we use the shorthand $[n]$ to denote the discrete interval $[1,n]$.
\par
We identify $\ell^2(K) \simeq \ell^2_K := \{u \in \ell^2(\Z); \supp u \subset K\}$
and we will frequently identify $\ell^2_{[0,N-1]}\simeq \C^N$ 
so the norm $\| \cdot \|$ and Hermitian scalar product $\langle \cdot | \cdot \rangle$ on $\C^N$ 
will be the ones of $\ell^2_{[0,N-1]}$. When the Euclidean norm and inner products are considered over some specific discrete interval $I$ we write $\| \cdot\|_{\ell^2(I)}$ and $\langle \cdot \mid \cdot \rangle_{\ell^2(I)}$. The notation $\|\cdot\|_\infty$ is used to denote the supremum norm for a vector. For a matrix $A$ the notations $\|A\|$ and $\|A\|_{{\rm HS}}$ denote its operator norm, and Hilbert-Schmidt norms, respectively. For a vector $v\in \C^N$ we write $v^*\in(\C^N)^*$ for its natural dual with respect to the Hermitian scalar product, i.e.  
$v^*u = \langle u | v \rangle$ for any $v\in\C^N$.

For a $N\times N$ matrix $A$ we denote by 
	$t_1(A) \leq \dots \leq t_N(A)$ the eigenvalues of $\sqrt{A^*A}$. 
	The singular values 
	$s_1(A) \geq \cdots \geq s_N(A)$ of $A$ are then given by 
	\begin{equation}\label{qm24.0}
		s_{N-n+1}(A) = t_n(A), \quad n=1,\dots, N. 
	\end{equation}
For brevity, we sometimes write $s_{\max}(A)$ and $s_{\min}(A)$ for the maximum and the minimum singular values of $A$. We let $I_d$ denote the identity matrix, viewed as element in $\C^{d\times d}$.

The following standard conventions are followed:
\begin{equation}\label{p0.0}
	a \lor b \defeq \max\{a,b\}, ~a,b\in\R, \qquad \prod_1^0 a_j =1, \qquad \text{ and } \qquad \sum_1^0 a_j = 0.\notag
\end{equation}
We also use the Dirac notation
$\delta_{n,m} ={\bf I}(m=n)$. 
The notation $\Theta(\cdot)$ denotes the Heaviside function, that is $\Theta(x) =0$ when $x\leq 0$, and $\Theta(x)=1$ for $x>0$.
  For $z\in \C$ and $r>0$, we write $D(z,r)\subset \C$ 
  for the (open) disc of radius $r$ centered at $z$. For $\vep>0$ and a set $\cB\subset \C$,
   $\cB^\vep$ denotes the $\vep$-blow up of the set $\cB$, that is
   the Minkowski sum of the sets $\cB$ and  $D(0,\vep)$. The notation $\cB \dot{\cup} \cB'$ is 
   used to denote a disjoint union. For a set $D \subset \C$ we use both ${\rm cl}(D)$ and $\ol{D}$ to denote its closure.
  The complement of  an event $\cA$ is denoted $\cA^\complement$.
   
 \subsection*{Acknowledgements} We thank Elliot Paquette and Nick Trefethen for useful discussions 
 at the beginning of this project, and thank Nicholas Cook for comments on an early draft of this paper. 
 The research of AB was partially supported by DAE Project no.~RTI4001 via ICTS, 
 the Infosys Foundation via the Infosys-Chandrashekharan Virtual Centre 
 for Random Geometry, an Infosys--ICTS Excellence Grant, 
 a Start-up Research Grant (SRG/2019/001376) and a 
 MATRICS grant (MTR/2019/001105) from Science and Engineering Research Board. The
 research of MV was partially supported by a CNRS Momentum 2017 grant and by the Agence Nationale de la Recherche under the grant ANR-20-CE40-0017. 
The research of OZ was partially supported by an Israel Science Foundation grant \# 421/20
and by the European Research Council (ERC) under the European Union's Horizon 2020 research and innovation programme (grant agreement No.~692452).  Part of this work was carried out when AB and MV visited the Mathematics department of the Weizmann Institute of Science whose hospitality is gratefully acknowledged. 

\section{The Grushin Problems}
\label{sec-Grushin}
We begin by setting up, in some generality, a well-posed Grushin problem, 
based on \cite{Vo19c,HaSj08}, see also \cite{SjVo19a,SjVo19b}.  It, and its 
behavior under perturbations, will play a crucial role in our analysis. 
\par
Roughly speaking, the Grushin problem amounts to replacing an operator of 
interest by an enlarged bijective system. In the context of 
linear partial differential equations, the
study of such enlarged system of operators 
can be traced back to Grushin \cite{Gr},
where it was used to study hypoelliptic 
operators. In a different setting, such an enlarged system was used
by Sj\"ostrand 
\cite{Sj73}, whose notation we use.
It has also been quite useful in bifurcation 
theory, numerical analysis, and for treatments of spectral problems arising in 
electromagnetism and quantum mechanics. See the review paper \cite{SZ}.
\subsection{Grushin problem for the unperturbed operator}\label{sec:GrPr1}
Let $P$ be a  complex $N\times N$-matrix.
(In our application, we will often take $P=P_N-zI$ where $P_N$ is 
the (deterministic) 
Toeplitz 
matrix with symbol $p$ and $z$ is a \textit{random} parameter close to the
spectral curve $p(S^1)$. Then, all objects implicitly depend on $z$, 
and we supress this dependence in notation when not needed.)
Let 
 \begin{equation}\label{gp1}
	0\leq t_1^2 \leq \cdots \leq t_{N}^2 
\end{equation}
denote the eigenvalues of $P^*P$ with associated orthonormal basis 
of eigenvectors $e_1,\dots,e_{N}\in \C^N$. The spectra of $P^*P$ 
and $PP^*$ are equal and we can find an orthonormal basis 
$f_1,\dots,f_{N}\in \C^N$ of eigenvectors of $PP^*$ associated with the 
eigenvalues \eqref{gp1} such that 
 \begin{equation}\label{gp2}
	P^* f_i = t_i e_i, \quad Pe_i = t_i f_i, \quad i=1,\dots, N.
\end{equation}
Let $0 < \alpha \ll 1$ and let $M>0$ be the number of singular values $t_i
\in 
[0,\alpha]$, i.e. 
 \begin{equation}\label{gp2.1}
   0 \leq t_1 \leq \cdots \leq t_M \leq \alpha < t_{M+1}\leq \cdots\leq t_N.
\end{equation}
Let $\delta_i$, $1\leq i \leq M$, denote an orthonormal basis of $\C^M$.
Put 
\begin{equation}\label{gp5}
R_+:=\sum_{i=1}^M \delta_i\circ  e_i^*, \quad  R_-:=\sum_{i=1}^M f_i\circ\delta_i^*,
\end{equation}
Then the Grushin problem 
\begin{equation}\label{gp6}
\mathcal{P}: = \begin{pmatrix}
			P & R_- \\ R_+ & 0 \\
		\end{pmatrix} : \C^N\times \C^M \longrightarrow  \C^N\times \C^M 
\end{equation}
is bijective. To see this we take $(v,v_+)\in \C^N\times \C^M$ and proceed to solve
 \begin{equation}\label{gp7}
		\mathcal{P}
		 \begin{pmatrix} u \\ u_- 
		\end{pmatrix} =
		 \begin{pmatrix} v \\ v_+
		\end{pmatrix}.
\end{equation}
We write $u= \sum_1^{N} u(j) e_j$ and $v= \sum_1^{N} v(j) f_j$. Similarly, we express 
$u_-,v_+$ in the basis $\delta_1,\dots,\delta_M$. 
The relation \eqref{gp2} then shows 
that \eqref{gp7} is equivalent to 
\begin{equation*}
		\begin{cases}
			\sum_1^{N} t_i u(i) f_i + \sum_1^M u_-(j)f_j = \sum_1^{N} v(j) f_j, \\ 
			u(j) = v_+(j), \quad j =1, \dots, M,
		\end{cases}
\end{equation*}
which can be written as 
\begin{equation}\label{gp7.1}
		\begin{cases}
		 \qquad  \quad t_i u(i)   = 
		v(i), \qquad \qquad\qquad i=M+1,\ldots,N,  \\[1.5ex] 
			 \begin{pmatrix} t_i & 1 \\ 1 & 0 \\
			\end{pmatrix}
			 \begin{pmatrix} u(i) \\ u_-(i)\\
			\end{pmatrix}
			= 
			 \begin{pmatrix} v(i) \\ v_+(i)
			\end{pmatrix}, \quad i=1,\dots, M.
		\end{cases}
\end{equation}
Since 
\begin{equation*}
	 \begin{pmatrix} t_i & 1 \\ 1 & 0 \\
	\end{pmatrix}^{-1} 
	=
	 \begin{pmatrix} 0& 1 \\ 1 & -t_i  \\
	\end{pmatrix},
\end{equation*}
we see that 
\begin{equation}\label{gp7.2}
		\mathcal{P}^{-1}= :\mathcal{E} =
		\begin{pmatrix} 
		E  & E_+\\ 
		E_- & E_{-+}\\
		\end{pmatrix} 
\end{equation}
with
\begin{equation}\label{gp8}
\begin{split}
		&E = \sum_{M+1}^{N} \frac{1}{t_i} e_i \circ f_i, \quad 
	         E_+ = \sum_1^M e_i \circ \delta_i^*,  \quad
		E_- =\sum_1^M  \delta_i\circ f_i^*, \quad \text{ and } \quad
		 E_{-+} = - \sum_1^M t_j \delta_j\circ\delta_j^*.
\end{split}
\end{equation}
From \eqref{gp2.1} and \eqref{gp8} it follows that we have the following norm estimates
\begin{equation}\label{gp9}
	\|E(z) \| \leq \frac{1}{\alpha}, \quad \| E_{\pm } \| =1, \quad 
	\| E_{-+}\| \leq \alpha.
\end{equation}
Next, we recall a general fact on well-posed Grushin problems. 
\begin{lem}
  \label{lem-basicinverse}
	Let $\mathcal{H}$ be an $N$-dimensional complex Hilbert space, and 
	let $N\geq M>0$. Suppose that 	
	 \begin{equation*}
		\mathcal{P} = \begin{pmatrix} 
		P & R_- \\ 
		R_+ & 0\\
		\end{pmatrix} :  \mathcal{H} \times \C^M \longrightarrow  \mathcal{H} \times \C^M
	\end{equation*}
	is a bijective matrix of linear operators, with inverse 
	 \begin{equation*}
		\mathcal{E}
		=
		 \begin{pmatrix} 
		E  & E_+\\ 
		E_-& E_{-+}\\
		\end{pmatrix}.
	\end{equation*}
	Then, $E_+:  \mathcal{N}(E_{-+}) \to \mathcal{N}(P)$ is bijective with inverse 
	$R_+ \!\upharpoonright_{\mathcal{N}(P)}$, and 
	$E_-^*: \mathcal{N}(E_{-+}^*) \to \mathcal{N}(P^*)$ is bijective with inverse 
	$R_-^* \!\upharpoonright_{\mathcal{N}(P^*)}$.
\end{lem}
\begin{proof}
	From $\mathcal{P}\mathcal{E}=1$, we get that $PE_+ + R_-E_{-+} =0$ and so 
	 \begin{equation}\label{gp:lem1.1}
		E_+:\mathcal{N}(E_{-+}) \to \mathcal{N}(P).
	\end{equation}
	Similarly, we get from $\mathcal{E} \cP=1$ the equation $E_- P + E_{-+}R_+ =0$, and hence
	 \begin{equation}\label{gp:lem1.2}
		R_+:  \mathcal{N}(P) \to \mathcal{N}(E_{-+}).
	\end{equation}
	The identity $EP + E_+R_+=1$ yields that $E_+R_+=1$ on $\mathcal{N}(P)$, 
	which, together with $R_+E_+=1$, shows that \eqref{gp:lem1.1}, \eqref{gp:lem1.2}, are 
	bijective and inverser to each other. 
	The proof of the second claim is similar, one can follow the same arguments applied to 
	$\mathcal{P}^*\mathcal{E}^*=\mathcal{E}^* \cP^*=1$.
\end{proof}
\subsection{Grushin problem for the perturbed operator}
Now we turn to the perturbed operator 
\begin{equation}\label{gpp1}
	P^\delta:=P+\delta Q, \quad 0 \leq \delta \ll 1.
\end{equation}
where $Q$ is a complex $N\times N$-matrix (eventually, random). 
Let $R_{\pm}$ be 
as in \eqref{gp5}, and put 
\begin{equation}\label{gpp2}
\mathcal{P}^{\delta}: = \begin{pmatrix}
			P^{\delta} & R_- \\ R_+ & 0 \\
		\end{pmatrix} : \C^N\times \C^M \longrightarrow  \C^N\times \C^M 
\end{equation}
with $\mathcal{P}=\mathcal{P}^0$. Applying $\mathcal{E}$ (see \eqref{gp7.2}) from
the right to \eqref{gpp2} yields 
\begin{equation}\label{gpp3}
  \mathcal{P}^{\delta}\mathcal{E} = I_{N+M} + 
\begin{pmatrix}
			\delta Q E &  \delta Q E_+\\ 0 & 0 
		\end{pmatrix}. 
\end{equation}
Suppose that $(I+\delta Q E)$ is invertible. It is then 
straightforward to check that $\mathcal{P}^\delta$ is invertible, with
inverse
\begin{equation}\label{gpp5a}
	(\mathcal{P}^{\delta})^{-1} 
=: \mathcal{E}^{\delta} = \begin{pmatrix} 
		E^{\delta}  & E^{\delta}_+\\ 
		E^{\delta}_-& E^{\delta}_{-+}
		\end{pmatrix},
\end{equation}
where
\begin{equation}
  \label{eq-Edelta}
E^\delta = E (I+ \delta Q E)^{-1}, \quad E_-^\delta= E_- (I+ \delta Q E)^{-1},
\end{equation}
\begin{equation}
\label{eq-march1a}
E^\delta_{-+}= E_{-+} -E_- (I+\delta Q E)^{-1} \delta Q E_+,
\end{equation} 
and
\begin{equation}
\label{eq-march1b}
E^\delta_+= E_+-E(I+\delta Q E)^{-1}\delta Q E_+.
\end{equation}
We note that if one takes $P=P_N-zI_N$ with $z$ an eigenvalue of $P_N+\delta
Q$, then Lemma \ref{lem-basicinverse}  applied to $P^\delta$ gives a
convenient description of the null-space of $P$, which is precisely 
the eigenspace of $P_N+\delta Q$ corresponding to the eigenvalue $z$. 
This observation  will be a crucial part of our analysis, see \eqref{eq:algeb-id} below.
\begin{rem}
 Under the additional assumption that
\begin{equation}\label{gpp4}
	2\delta \|Q\| \alpha^{-1} \leq 1,
\end{equation}
which will occur in our setup of $Q$ as in Assumption 
\ref{assump:mom} if $\delta=N^{-\gamma}$ with $\gamma>3/2$ 
(and with $\alpha\,\asymp\, N^{-1}$),
we obtain by a Neumann series argument that
\begin{equation}\label{gpp5}
	\mathcal{E}^{\delta}= \begin{pmatrix} 
		E^{\delta}  & E^{\delta}_+\\ 
		E^{\delta}_-& E^{\delta}_{-+}
		\end{pmatrix} 
	= \mathcal{E} + \sum_{n=1}^{\infty}(-\delta)^n 
	\begin{pmatrix}
			E(Q E)^n &   (EQ )^{n}E_+\\  E_-(Q  E)^n& E_-(Q E)^{n-1}Q E_+ \\
	\end{pmatrix},
\end{equation}
where by \eqref{gpp4}, \eqref{gp9}, 
 \begin{equation}\label{gpp6}
 \begin{split}
		&\| E^{\delta} \| = \|  E( 1+ \delta Q E)^{-1} \| \leq 2 \|E\| \leq 2 \alpha^{-1}, \\
		&\| E_+^{\delta} \| = \|  ( 1+ \delta Q E)^{-1}E_+ \| \leq 2 \|E_+\| \leq 2,  \\
		&\| E_-^{\delta} \| = \| E_- ( 1+ \delta Q E)^{-1} \| \leq 2\|E_-\| \leq 2,  \\
		&\| E_{-+}^{\delta} -E_{-+}\| = 
		\| E_- ( 1+ \delta Q E)^{-1}\delta Q E_+ \| \leq 2 \|\delta Q  \| \leq  \alpha.  \\
\end{split}
\end{equation}
In particular,  in that case,
\begin{equation}\label{gpp7}
	E_+^{\delta} = E_+ - \delta Q  E( 1 + O(\delta \| Q \| \alpha^{-1} ) )E_+
\end{equation}
and 
\begin{equation}\label{gpp8}
	E_{-+}^{\delta} = E_{-+} - \delta E_-( 1 + O(\delta \| Q \| \alpha^{-1} )) Q E_+.
\end{equation}
\end{rem}
\section{Structure of the proof of Theorem \ref{thm:main}}
\label{res}
A key ingredient for the proof Theorem \ref{thm:main} is Theorem \ref{theo-location}. The proof of the latter result splits into two parts:~In the first part we show that all eigenvalues must be separated from $p(S^1)$ by a distance of the order $\log N/N$. At a very high level it involves an expansion of the determinant of $P_{N, \gamma}^Q-zI_N$, with $z \in \C$, identifying the dominant term in that expansion, and showing that the dominant cannot be equal to zero (with probability approaching one) when $z$ is in the vicinity of the spectral curve. We refer the reader to Section \ref{sec-separation}-\ref{sec:thm-sep-spec-curve} for further details on these steps. The second part of Theorem \ref{theo-location} requires us to show that most of the eigenvalues must be within a distance $O(\log N/N)$ from spectral curve, again with probability approaching one. This is achieved by an application of Jensen's formula together with upper and lower bounds on the log-potential of $L_N$ (see \eqref{eq:L-N}). See Section \ref{sec:bulk-eig} for details.
\par
In the remainder of this section we describe the structure of the proof
of Theorem \ref{thm:main}, 
taking for granted Theorem \ref{theo-location} and various 
technical estimates. 
The proof of Theorem \ref{thm:main} 
begins with the Grushin problem 
for $P^\delta_z=P_{N,\gamma}^Q-zI$, see \eqref{gpp5a}, 
for $\delta=N^{-\gamma}$,
$z$ which is
roughly an eigenvalue, and $M=|{\rm ind}_{p(S^1)}(z)|$ (this will lead to 
$t_{M+1}\, \gtrsim \, \log N/N$ and
$\alpha$ bounded below by a constant multiple  of $\log N/N$, see Proposition \ref{prop:SmallSG}).
To keep track of the dependence on $z$, throughout this section we write $E(z), E_+(z)$, etc. 
To 
relate the null-space of $P^\delta_z$ with the null space of $E^{\delta}_{-+}(z)$ we will use 
Lemma \ref{lem-basicinverse} in an indirect manner: As in its proof note that from \eqref{gpp2} 
and \eqref{gpp5a},
\begin{equation}\label{eq:algeb-id}
E^\delta(z)P^\delta_z+E_+^\delta(z)R_+(z)=I \quad \text{ and } \quad E_-^\delta(z) P_z^\delta + E_{-+}^\delta(z) R_+(z) =0.
\end{equation}
If $z=\hat z$ were an eigenvalue of $P_{N,\gamma}^Q$ with corresponding normalized eigenvector
$v$ 
then, with notation
as in Section \ref{sec-Grushin} and recalling the definition of $E_+^\delta(z)$,
we would obtain from \eqref{eq:algeb-id} that
since $\{e_i(\hat z)\}$ forms an orthonormal basis of $\C^N$, 
\begin{eqnarray}
\sum\nolimits_{i=M+1}^N (e_i(\hat z)^* v) \cdot e_i(\hat z)&= &  (I-E_+(\hat z)R_+(\hat z)) v\nonumber\\
& =&(I-E_+^\delta(\hat z)R_+(\hat z)) v - E(\hat z)(I+\delta QE(\hat z))^{-1} \delta QE_+(\hat z) R_+(\hat z)v \nonumber\\
&=&E(\hat z) (I+\delta QE(\hat z))^{-1} \delta QE_+(\hat z) R_+(\hat z)v. \label{eq:ef-to-pm-1}
\end{eqnarray}
Consider first the case where $\gamma$ is large ($\gamma>3/2$ will do). Since 
$\|Q\|=O(N^{1/2+\upepsilon})$, for any $\upepsilon >0$, with high probability, we obtain that $N^{-\gamma} \|Q\|\alpha^{-1}=o(1)$ and therefore \eqref{gpp4} holds. 
Using then \eqref{gpp7}-\eqref{gpp8}, 
the projection of $v$ on $\mbox{\rm span}(e_i(z), i\geq M+1)$ is negligible,
which yields 
the first
part of Theorem \ref{thm:main}.

To see the second part, still in the case of
large $\gamma$ (here we will need $\gamma >2$) and $z=\hat z$, we obtain from \eqref{eq:algeb-id} that 
\begin{eqnarray}\label{eq:coeff-eM}
0 &=& - E_{-+}^\delta(\hat z) R_+(\hat z) v 
 =
 - E_{-+}(\hat z)R_+(\hat z) v + \delta E_-(\hat z) Q E_+(\hat z) R_+(\hat z) x \\
 &&-
 \delta^2 E_-(\hat z) (I+\delta Q E(\hat z))^{-1} Q E(\hat z) Q E_+(\hat z) R_+(\hat z) v, 
 \nonumber
\end{eqnarray}
where we also have used the resolvent expansion. By the same reasoning as above, the third term in \eqref{eq:coeff-eM}
turns out to be of order
$o(\delta)$, hence negligible compared to the first two terms. Therefore, recalling the definitions of $E_\pm(z), E_{-+}(z)$, and $R_+(z)$ we obtain that,
with $a_j= (e_j(\hat z)^* v)$,
\begin{multline}
- E_{-+}(\hat z)R_+(\hat z) v + \delta E_-(\hat z) Q E_+(\hat z) R_+(\hat z) v\\
 = \sum_{i=1}^M a_i t_i \delta_i + \delta \sum_{i=1}^M \left[\sum_{j=1}^M 
 a_j  (f_i(\hat z)^* Q e_j(\hat z))\right] \delta_i
 =o(\delta). \label{eq:coeff-eM-dom}
\end{multline}
Now, again by Proposition \ref{prop:SmallSG}, there exists $M>M_0\geq 0$ so that $t_j$ decay exponentially in $N$ for $j\in [M_0]$. Thus, 
we obtain from \eqref{eq:coeff-eM-dom}
that for $\gamma$ large, 
\[  \left\|\sum_{i=1}^{M_0}
\left[\sum_{j=1}^M a_j \cdot (f_i(\hat z)^* Q e_j(\hat z))\right] \delta_i\right\|=o(1).\]
Assume now that $a_j=o(1)$ for $j=M_0+1,\ldots,M$. Using a basic chaining argument we would then conclude that
\[
\left\|\sum_{i=1}^{M_0}
\left[\sum_{j=1}^{M_0} a_j \cdot (f_i(\hat z)^* Q e_j(\hat z))\right] \delta_i\right\|= \| a^{{\sf T}} A\|=o(1),\]
where $A$ is the $M_0\times M_0$ matrix with entries  $A_{i,j}=  f_i(\hat z)^* Q e_j(\hat z)$. If $\hat z$ were deterministic, 
we would have that
the smallest singular value of $A$ is $o(1)$ 
and this would lead to a contradiction.
Since $\hat z$ is actually  random, we will proceed by using the fact that the functions $f_i,e_i$ are localized, which makes the minimal singular value of $A$ continuous in $z$.  

When $\gamma\in (1,2]$, we cannot use in \eqref{eq:ef-to-pm-1} and \eqref{eq:coeff-eM} an a-priori bound of the form \eqref{gpp4}. Instead, we use a lower bound on the minimum singular
value of  $I+\delta Q E$, see \eqref{lem-10.7}, and the resolvent expansion to replace $(I+\delta QE)^{-1}$ by
$\sum_{i=0}^L (-\delta QE)^i + (-\delta QE)^{L+1}(I-\delta QE)^{-1}$ for an appropriate $L$. Proposition \ref{prop:bd-terms-resolvent-exp} and
the a-priori bounds on the minimal singular value of $I+\delta QE$ suffice to control the sum, and again we use a net in order to work with deterministic $z$'s.

To carry out this program necessitates a fair amount of auxillary results. We need a-priori estimates for the singular values and associated quasimodes of $P_N -zI$ for $z$ close to the location
of eigenvalues of $P_{N,\gamma}^Q$. So we begin in Section \ref{Sec:AnaTO} with a fairly detailed analysis of the singular values and vectors of $P_N-zI$ (for general $z$), and relate the 
singular values to the winding number of $p$ around $z$ (we will stay away of  the bad set $\cB_{p}$ of Definition \ref{def-badintro}, so that the winding number is locally constant and
the distance of $z$ from $p(S^1)$ is controlled).  In Sections 
\ref{sec-separation}-\ref{sec:bulk-eig} we provide precise estimates for the range of $z$s that we need to consider, that is, for the location of the 
eigenvalues of $P_{N,\gamma}^Q$.  Those precise
estimatess are then used in Section \ref{sec:QM} in constructing 
 the quasimodes of $(P_N-zI)(P_N-zI)^*$, in terms of the roots of the symbol $p-z$. These quasimodes are then used in Section \ref{sec-sing}, where we show that the 
eigenvectors are appropriate linear combinations of the localized quasimodes.

\subsection{The case $\gamma<1$ - discussion and speculations}\label{sec:gam-l-1} 
We end this section with some brief remarks concerning $\gamma<1$.
In that regime, the single entries of $\delta Q=
N^{-\gamma} Q$ are larger than $N^{-1}$,
and in particular are asymptotically larger than the distance of the eigenvalue
from the spectral curve. In particular, when writing the Grushin problem
\eqref{gpp2}, one is forced to take $M$ growing with $N$ (in fact, essentially
$M\, \asymp \, N^{2(1-\gamma)}$; This is forced by the requirement
that $\|E_-(\delta Q) E_+\|< M/N$). The resulting eigenvector of $P^\delta$
are expected to 
be a combination of the $M$ bottom quasimodes, with random coefficients.
Since the quasimodes oscillate at scale $N/M=N^{2\gamma-1}$, 
the combination is expected to converge to a $\gamma$-dependent
Gaussian process with 
correlation length of that scale. The simulations in Figure \ref{fig1a} are in 
line with this picture, although proving it require ideas going beyond the methods of this paper.

\section{Analysis of Toeplitz matrices}\label{Sec:AnaTO}
In this section we begin with presenting some fundamental results about the calculus of 
Toeplitz matrices. In particular we will focus on symbols 
given by Laurent polynomials 
and we will discuss the quantization procedure which maps such a 
symbol to an operator acting on functions on $\Z$, $\ell^2(\Z)$, $\ell^2([0,\infty[)$, and  
$\ell^2(\Z/N\Z)$. 
\subsection{Toeplitz matrices}
We begin by recalling some well known facts about Toeplitz matrices, see for instance 
\cite{BoGr05} and the references therein. 
\par
For a $u\in \ell^2(\Z)$ we define the Fourier transform by 
\begin{equation}\label{a1.0}
	\mathcal{F}u(\xi) = \sum_{n\in\Z} u_n \e^{-in\xi}, \quad \xi \in \R/ 2\pi \Z,
\end{equation}
so that $\mathcal{F}: \ell^2(Z)\to L^2(\R/ 2\pi \Z, \frac{d\xi}{2\pi})$ is unitary, 
and 
\begin{equation}\label{a1.1}
	\mathcal{F}^{-1}(f)(n) = f_n=  \frac{1}{2\pi} \int_{\R/ 2\pi \Z}f(\xi )\e^{in\xi} d\xi. 
\end{equation}
\par
We consider the symbol class of continuous functions on the unit circle $S^1$ 
with absolutely convergent Fourier series called the Wiener algebra 
\begin{equation}\label{a1}
	W:= \{ a \in C(S^1); ~ \sum_{n\in\Z} |a_n| < + \infty \}. 
\end{equation}

Using the Fourier transform we  can represent a symbol $p \in W$ by 
\begin{equation}\label{a2}
	p(\e^{i\xi}) = \sum_{n\in \Z} p_n \, \e^{-in\xi}, \quad \xi\in \R/ 2\pi \Z. 
\end{equation}
We can \textit{quantize}
the symbol $p\in W$ by 
\begin{equation}\label{a5}
	\mathrm{Op}(p) \defeq \mathcal{F}^{-1}\, p\,\mathcal{F}.
\end{equation}
Using Parseval's equality it is easy to see that $p(\tau): \ell^2(\Z) \to \ell^2(\Z)$ is a bounded 
operator; it is also straight forward to check that
the $\ell^2$ adjoint of $\mathrm{Op}(p)$ is given by 
\begin{equation}\label{a6.1}
	\mathrm{Op}(p)^*= \mathrm{Op}(\overline{\widetilde{p}}), \quad \widetilde{p}(\zeta)  \defeq 
	 {p}(1/\zeta).
\end{equation}
Let $\tau$ denote the right shift operator $(\tau\psi)(n) = \psi(n-1)$ on $\ell^2(\Z)$ or more 
generally on functions $\psi : \Z \to \C$. Notice that $\tau \e^{i n \theta} = \e^{-i\theta} \e^{i n \theta} $, so 
the symbol of $\tau$ is given by $\e^{-i\xi} = 1/\zeta$, $\zeta \in S^1$. So we can express \eqref{a5} 
as well as 
\begin{equation}\label{a6}
	\mathrm{Op}(p)= \sum_{n\in \Z} p_n \, \tau^{n}.
\end{equation}
which can be seen to act more generally on function $\psi : \Z \to \C$. 
\par
Let $K\subset \Z$ be a finite subset or an infinite interval. We identify 
$\ell^2(K) \simeq \ell^2_K := \{u \in \ell^2(\Z); \supp u \subset K\}$, and 
we define 
\begin{equation}\label{ev0}
 	P_K \defeq P_K(p)\defeq \mathbf{1}_K \mathrm{Op}(p) \mathbf{1}_K : \ell^2_K \to \ell^2_K. 
\end{equation}
When $K$ is finite we call $P_K$ a \textit{Toeplitz matrix} and when $K$ is infinite, 
for example $K=]-\infty,0]$ or $K=[0,\infty[$ we call $P_K$ an \textit{infinite Toeplitz matrix}. 
For simplicity we will sometimes write $P_N=P_N(p)$ when $K=[0, N-1]$ and $P(p)$ 
when $K=[0,\infty[$. 
\\
\par
Notice that the matrix elements of $P_N$ are given by 
\begin{equation}\label{a7.4}
 	P_N(\nu,\mu) = p_{\nu-\mu}, \quad \nu,\mu \in [0,N-1].
\end{equation}
A counter part to the Toeplitz matrices are the \textit{Hankel matrices}. Let $R$
be the reflection  operator defined by $(R\psi)(n) = \psi(-n)$ on functions 
$\psi: \Z\to \C$. Let 
\begin{equation}
  \label{eq-defchi} \chi_n(\zeta) = \zeta^n, \qquad  \zeta\in S^1, n\in \Z. 
\end{equation}
For $a\in W$ we define the Hankel matrix of $a$ by 
\begin{equation}\label{a7.2}
 	H(a) \defeq \sum_{1}^{\infty} a_n H(\chi_n), \quad H(\chi_n) 
	= \mathbf{1}_{[0,\infty[} \tau^{n-1}R \mathbf{1}_{[0,\infty[}.
\end{equation}
Represented as infinite matrices we see that Toeplitz matrices carry the 
same entry on the diagonals, whereas Hankel matrices carry the same 
entry on the anti-diagonals, for instance 
\begin{equation}\label{a7.3}
 	T(a) = \begin{pmatrix}
		a_0 & a_{-1} & a_{-2} & \dots \\
		a_{1} & a_0 & a_{-1} & \dots \\ 
		a_2 & a_1 & a_0 & \dots \\
		\dots & \dots & \dots & \dots \\
	\end{pmatrix}, 
	\quad 
	H(a) = \begin{pmatrix}
		a_1 & a_{2} & a_{3} & \dots \\
		a_{2} & a_3 & \dots & \dots \\ 
		a_3 & \dots & \dots & \dots \\
		\dots & \dots & \dots & \dots \\
	\end{pmatrix}.
\end{equation}
Note that $H(\chi_n)=0$ for $n\leq 0$, and that $H(\chi_n)$ has rank $n$ for $n\geq 1$.
\\
\par
The quantization procedures $p \mapsto \mathrm{Op}(p)$ and $p \mapsto P_K(p)$ are clearly linear. 
Composition of such operators is however more complicated. Given $a,b \in W$ we have that 
\begin{equation}\label{a7.1}
 	\mathrm{Op}(ab) = \mathrm{Op}(a) \mathrm{Op}(b).
\end{equation}
However, the composition of two Toeplitz matrices is in general not a Toeplitz matrix but only a Toeplitz 
matrix modulo two products of Hankel matrices, see for instance \cite[Proposition 3.10]{BoGr05}. 
Given $a,b\in W$ we have,
with $\Pi_N = \mathbf{1}_{[0,N-1]}$ and $\widetilde{\Pi}_N = \Pi_N H(\chi_N) \Pi_N$,  that 
\begin{equation}\label{a7}
 	T_N( ab) = T_N(a)T_N(b) + \Pi_N H(a) H(\widetilde{b}) \Pi_N 
	+ \widetilde{\Pi}_N H(\widetilde{a}) H(b)\widetilde{\Pi}_N.
\end{equation}
In this paper we mainly consider the operator 
 \begin{equation}\label{int0}
	\mathrm{Op}(p)= \sum_{-N_-}^{N_+} a_j \tau^j, \quad a_{-N_-}, a_{-N_-+1},\dots,a_{N_+}\in \C, 
	~a_{\pm N_\pm}\neq 0
\end{equation}
acting on $\ell^2(\Z)$ or, more generally, on functions $\psi : \Z \to \C$, whose coefficients 
satisfy Assumption \ref{assu-symbol}, and with symbol $p$ given by 
 \begin{equation}\label{int1}
	\C\backslash\{0\} \ni \zeta \mapsto p(\zeta) = \sum_{-N_-}^{N_+} a_j \zeta^{-j}.
\end{equation}
\subsection{Circulant matrices}
In this section we discuss circulant matrices,  which are
close relatives
of Toeplitz matrices that play an important role in our analysis.

For $N\geq 1$ and $p\in W$ we consider $\mathrm{Op}(p)$ \eqref{a6} acting on $\ell^2(\Z/ N\Z)$ which 
we identify with the space of $N$-periodic functions on $\Z$. To distinguish this case from 
the other operators considered in this paper, we write 
\begin{equation}\label{c1}
 		P_{\Z/N\Z} \defeq \mathrm{Op}(p): \ell^2(\Z/ N\Z) \to \ell^2(\Z/ N\Z).
\end{equation}
For $\nu \in \Z/N\Z$ 
we get that 
\begin{equation*}
 		(P_{\Z/N\Z} u)(\nu) = \sum_{n\in\Z} u(\nu - n) 
		= \sum_{\mu \in \Z/N\Z} \sum_{k\in \Z} p_{\nu - \mu + kN}\, u(\mu),
\end{equation*}
so the matrix elements of $P_{\Z/N\Z} $ are given by 
\begin{equation}\label{c3}
 		P_{\Z/N\Z}(\nu,\mu) =  \sum_{k\in \Z} p_{\nu - \mu + kN}, \quad \mu,\nu \in \Z/N\Z.
\end{equation}
Identifying $[0,N-1] \simeq \Z/N\Z$, we see by \eqref{a7.4} that for $p$ as in \eqref{int1}, 
\begin{equation}\label{c5}
 		P_{\Z/N\Z}(\nu,\mu) - P_{N}(\nu,\mu) =  p_{\nu - \mu + N} +  p_{\nu - \mu- N} 
		\defeq B(\nu,\mu), \quad 
		\mu,\nu \in [0,N-1], 
\end{equation}
where $B$ is of rank $N_+ + N_-$. Using the discrete Fourier transform, with 
$$\mathcal{F}_N = N^{-1/2} (\exp(-2\pi i n m/ N))_{0\leq n,m \leq N-1}$$
 we may 
represent $P_{\Z/N\Z}$ by 
\begin{equation}\label{c4}
 		P_{\Z/N\Z} = \mathcal{F}_N^* \,\diag ( p(e^{2\pi i l/N});0 \leq l \leq N-1) \, \mathcal{F}_N,
\end{equation}
which immediately shows that the spectrum of $P_{\Z/N\Z} $ is 
given by $\sigma(P_{\Z/N\Z} ) = \{p(e^{2\pi i l/N});0 \leq l \leq N-1) \}$. 
\subsection{Roots of a Laurent polynomial}
In this section we discuss the roots 
of the Laurent polynomial 
 \begin{equation}\label{lp1}
	p(\zeta)= \sum_{-\mathfrak{N}_-}^{\mathfrak{N}_+} a_j \zeta^{-j}, \quad -\mathfrak{N}_- \leq \mathfrak{N}_+, 
	\quad  a_{-\mathfrak{N}_-}, a_{-\mathfrak{N}_-+1},\dots,a_{\mathfrak{N}_+}\in \C, 
	\quad a_{\pm\mathfrak{N}_\pm }\neq 0.
\end{equation}
It will be convenient to extend the symbol $p$ to a holomorphic map 
$\widetilde{p}: \widehat{\C} \to \widehat{\C}$ 
on the extended complex plane $\widehat{\C} \defeq \C \,\dot{\cup}\, \{ \infty\}$. Here $\widehat{\C}$ is 
a complex manifold, equipped with the topology given by $U\in\widehat{\C}$ open if either $U\subset \C$ open 
or $\widehat{\C}\backslash U \subset \C$ is compact, and with the equivalence class of holomorphic 
atlases represented by $\{ \mathrm{id}:\C\to\C, ~\phi: \widehat{\C} \backslash\{0\} \to \C\}$ where 
$\phi(\infty) = 0$ and $\phi(z) =1/z$ for $z\notin\{ 0,\infty\}$. In what follows we will drop the tilde notation and 
denote by $p$ also the extension.

We exclude the case of constant Laurent polynomials, i.e. we assume that 
 \begin{equation}\label{lp1.1}
	 \mathfrak{N}_{\pm} >0 \text{ when } \mathfrak{N}_{\mp} =0.
\end{equation}
We say that $\zeta \in  \C \, \dot{\cup} \{\infty\}$ is a root of \eqref{lp1} when 
 \begin{equation}\label{p0}
	p(\zeta)= \sum_{-\mathfrak{N}_-}^{\mathfrak{N}_+} a_j \zeta^{-j} =0,
\end{equation}
and, by using the change of coordinates $\zeta=1/\omega$, we have that $\zeta = \infty $ 
is a root if $\omega=0$ is a root of $p(1/\omega)$. We keep in mind that in the sequel we will 
be interested in symbols whose coefficients $a_j$ may depend on a spectral parameter $z\in\C$. 
For instance the coefficient of order zero of the symbol $p(\zeta)-z$ is $a_0 -z$, and all the other 
coefficients remain independent of $z$.

\begin{lem}\label{lem:root1} The Laurent polynomial \eqref{lp1}, assuming \eqref{lp1.1}, has 
$\mathfrak{N}=\mathfrak{N}_+\lor 0 + \mathfrak{N}_-\lor 0$ roots 
in $ \C\, \dot{\cup} \,\{ \infty\} $.
Assume that $0\notin p(S^1)$, and let
   \begin{equation}\label{at2}
	\zeta_1^+,\dots,\zeta_{m_+}^+ \text{ denote the roots in } D(0,1),
\end{equation}
and
\begin{equation}\label{at3}
	\zeta_1^-,\dots,\zeta_{m_-}^- \text{ denote the roots in } (\C\,\dot{\cup} \,\{\infty\} )\backslash \overline{D(0,1)},
\end{equation}
where $m_+ + m_- = \mathfrak{N}$ (working with the convention that when $m_+=0$ resp. $m_-=0$, 
we have no roots \eqref{at2} resp. \eqref{at3}).
We can distinguish the following three cases:

\par
\textbf{Case 1} If $\mathfrak{N}_{\pm} \geq 0$, then we have that 
\begin{equation}\label{at4.2.0}
0< |\zeta_1^+|\leq \cdots \leq |\zeta_{m_+}^+| < 1 < |\zeta_1^-| \leq \cdots \leq |\zeta_{m_-}^-| < +\infty
\end{equation}
\par
\textbf{Case 2} When $\mathfrak{N}_+ <0$, then we have 
that $\infty$ is not a root but $0$ is a root of \eqref{p0} of multiplicity 
$m_0\in[1\lor|\mathfrak{N}_+|,|\mathfrak{N}_-|]$ . In this case, we have that
\begin{equation}\label{at4.2}
	0 =|\zeta_1^+| = \cdots = |\zeta_{m_0}^+| < |\zeta_{m_0 +1}^+|\leq \cdots \leq |\zeta_{m_+}^+|
	 < 1 < |\zeta_1^-| \leq \cdots \leq |\zeta_{m_-}^-| < +\infty.
\end{equation}
\par
\textbf{Case 3} When $\mathfrak{N}_- <0$, then $0$ is no root, but $\infty$ is a root of multiplicity 
$m_\infty\in[1\lor|\mathfrak{N}_-|,|\mathfrak{N}_+|],$ and we may order the roots as 
\begin{equation}\label{at4.1}
	0< |\zeta_1^+|\leq \cdots \leq |\zeta_{m_+}^+| < 1 < |\zeta_1^-| \leq \cdots \leq |\zeta^-_{m_--m_{\infty}}| < +\infty
=|\zeta_{m_- - m_{\infty}+1}^-| = \dots = |\zeta_{m_-}^-|.
\end{equation}
\end{lem}
\begin{proof}
1. If $\mathfrak{N}_{\pm} \geq 0$, we see that neither $0$ nor $\infty$ can be a root of \eqref{p0}, 
which therefore has the same roots as
 \begin{equation}\label{p0.1}
	\zeta^{\mathfrak{N}_+} p(\zeta)= \sum_{0}^{\mathfrak{N}_++\mathfrak{N}_-} a_{\mathfrak{N}_+-j} \zeta^{j} =0,
\end{equation}
a polynomial of degree $\mathfrak{N}_+ + \mathfrak{N}_-$. We order 
its $\mathfrak{N}_+ + \mathfrak{N}_-$ roots, counted with 
their multiplicities, as in \eqref{at2}, \eqref{at3}, with 
$m_+ + m_- = \mathfrak{N}_+ + \mathfrak{N}_-$, and we conclude \eqref{at4.2.0}.
\\
\par
2. When $\mathfrak{N}_+ < 0$ then \eqref{p0} is given by
\begin{equation}\label{p0.3}
	p(\zeta)= \sum_{|\mathfrak{N}_+|}^{\mathfrak{N}_-} a_{-j} \zeta^{j} =0,
\end{equation}
a polynomial of degree $\mathfrak{N}_-$ and so we have $\mathfrak{N}_-$ roots in $\C$ and 
$\infty$ cannot be a root. It follows from \eqref{p0.3} that $0$ is a root with multiplicity 
\begin{equation}\label{c2.3}
	m_{0} = \min\{ j \in \{ |\mathfrak{N}_+|,\dots, \mathfrak{N}_-\} ; a_j -z\delta_{j,0} \neq 0\},
\end{equation}
where we recall the Dirac  notation from Section \ref{sec:notation}. 
Notice that 
$\max\{1,|\mathfrak{N}_+|\} \leq m_{0} \leq |\mathfrak{N}_-| $. 
In this case, we denote the roots as in \eqref{at2}, \eqref{at3} and \eqref{at4.2}.
\\
\par
3. If $\mathfrak{N}_-<0$ then \eqref{p0} is given by
\begin{equation}\label{p0.5}
	p(\zeta)= \sum_{|\mathfrak{N}_-|}^{\mathfrak{N}_+} a_{j} \zeta^{-j}  =0.
\end{equation} 
So $0$ is not a root, however, $\infty$ maybe be one. Performing the change of variables 
$\omega = 1/\zeta$ we see that $p(1/\omega)$ satisfies the 
assumptions of \text{Step 2} with the roles $\mathfrak{N}_-$ and $\mathfrak{N}_+$ exchanged and roots $\omega_j^{\pm} = 1/\zeta_j^{\pm}$. Hence, we conclude \eqref{at4.1} when Case 3 holds. 
\end{proof}
\subsection{Kernels of $P_{[0,\infty[}-z$ and $P_{]-\infty,0]}-z$}
In this section we discuss how to construct the elements of the kernels of 
$P_{[0,\infty[}-z$ and $P_{]-\infty,0]}-z$, see \eqref{ev0}, where $p$ is as in \eqref{int1}, 
\eqref{int0}, $z\in \C\backslash p(S^1)$. 
\\
\par
Recall  that the index of a Fredholm operator $A$ on a Hilbert space is given by 
\begin{equation*}
		\mathrm{Ind}(A) = \dim \mathcal{N}(A) - \dim \mathcal{N}(A^*).
\end{equation*}
Here $\mathcal{N}(A)$ denotes the kernel of $A$, and note that 
the kernel of $A^*$ is isomorphic to the cokernel of $A$. 
We know from 
\cite[Theorem 1.9]{BoGr05}, \cite{SjVo19a} that $P_{K} -z$, for $K= ]-\infty,0]$ and $[0,\infty[$, 
are Fredholm operators, and that the winding 
number $\mathrm{ind}_{p(S^1)}(z)$ of the curve $p(S^1)$ around $z$  is related to the 
Fredholm index of $P_K-z$  as follows:
\begin{equation}\label{at9.1}
		\mathrm{Ind}(P_{[0,\infty[} -z) = \mathrm{ind}_{p(S^1)}(z), 
		\quad
		\mathrm{Ind}(P_{]-\infty,0]} -z) =  -\mathrm{ind}_{p(S^1)}(z).
\end{equation}
%
%
%
Using Lemma \ref{lem:root1} one can express the winding number of $p(S^1)$ around $z$ by
\begin{equation}\label{at9.0}
	 \mathrm{ind}_{p(S^1)}(z) =  \frac{1}{2\pi i} \int_{S^1} \frac{d}{d\eta} \log ( p(\eta) -z ) d\eta =
	 m_+ - (N_+\lor 0) =( N_-\lor 0) - m_- .
\end{equation}
\begin{rem}\label{rem1} 
If the symbol $p-z$ satisfies Case $2$ of Lemma \ref{lem:root1}, then $ \mathrm{ind}_{p(S^1)}(z) = m_+>0$, and if it satisfies Case $3$, then $ \mathrm{ind}_{p(S^1)}(z) = - m_- < 0$.
\par
When comparing with the literature, for instance \cite{BoGr05}, then the right hand sides of the equations in \eqref{at9.0} and \eqref{at9.1} have the opposite sign. This discrepancy comes from our choice 
to write the Fourier series \eqref{a1.0} with a minus instead of a plus sign in the exponential.
\end{rem}
Our aim is now to give explicit expressions for the kernel vectors 
of $P_{[0,\infty[}-z$ and $P_{]-\infty,0]}-z$ depending on the roots of 
the symbol $p-z$. We will only treat the case when the symbol satisfies 
the assumptions of Case 1 of Lemma \ref{lem:root1} since the other 
two cases are non-generic. Furthermore, we will only work in 
the case where all roots of $p-z$ are simple. We will see that under 
these assumptions the kernel vectors will be given by exponential 
solutions as in \eqref{at9} below, corresponding to the roots in 
$\C\backslash \{0\}$.
\subsubsection{Exponential solutions}
We begin with a slightly more general discussion on exponential solutions. 
Let $\zeta\in\C\backslash\{0\}$ be a root of \eqref{p0} with multiplicity $\mathrm{mult(\zeta)}$, 
then the exponential functions 
\begin{equation}\label{at8}
\Z\ni \nu \mapsto f_{\zeta ,k}(\nu )\defeq \nu^k(\zeta ^\nu ),\ 0\leq k\leq
\mathrm{mult\,}(\zeta )-1
\end{equation}
are solutions to 
\begin{equation}\label{exp.1}
\mathrm{Op}(p)f_{\zeta ,k}=0, \text{ on } \Z,
\end{equation}
and these functions are linearly independent, by the following proposition. 
\begin{prop}[{\cite[Proposition 3.1]{SjVo19a}}]\label{at:prop1}
Let $\zeta _1,...,\zeta _{\mathfrak{N}}\in \C\setminus \{0 \}$
be distinct numbers and let $1\le m_j<\infty $, $1\le j \le \mathfrak{N}$.
The functions $f_{\zeta _j,k}\,: \Z\to \C$, $1\le j\le \mathfrak{N}$, $0\le
k\le m_j-1$ are linearly independent. More precisely, if $K\subset \Z$
is an interval with $\# K=m_1+m_2+...+m_{\mathfrak{N}}$, then ${{f_{\zeta
      _j,k}}_\vert}_{K}$ form a basis in $\ell^2(K)$. 
\end{prop}
This result together with Lemma \ref{lem:root1} immediately yield the 
following result. 
\begin{prop}[{\cite[Proposition 3.7]{SjVo19a}}]\label{prop:expSol}
Let $p$ be as in \eqref{lp1}, \eqref{lp1.1} with $\mathfrak{N}_{\pm} \geq 0$. Suppose 
that $0 \notin p(S^1)$, that all roots of $p(\zeta)$ are simple and ordered as in 
\eqref{at4.2.0}.
\par
Then, the space of exponential solutions to $\mathrm{Op}(p)u =0$ is of dimension 
$m_+ + m_-$, and the general  solution is of the form 
\begin{equation}\label{at9}
		u(\nu) = \sum_{j=1}^{m_+}a_{j}^+(\zeta_j^+)^\nu 
		+ \sum_{j=1}^{m_-}a_{j}^-(\zeta_j^-)^\nu, 
		\quad a^{\pm}_{j} \in \C. 
\end{equation}
The subspace of solutions decaying as $\nu\to +\infty$ is given by 
\begin{equation}\label{at10}
		a_{j}^{-} =0, \quad \text{for } ~ 
		0 \leq j \leq m_-
\end{equation}
and the subspace of solutions decaying as $\nu\to -\infty$ is given by 
\begin{equation}\label{at10b}
		a_{j}^{+} =0, \quad \text{for } 1 \leq j\leq m_{+}
\end{equation}
\end{prop}
\begin{proof}
	The statement on the dimension and \eqref{at9} are an immediate consequence of 
	Lemma \ref{lem:root1} and Proposition \ref{at:prop1}. The last statement is a consequence 
	of the fact that $|\zeta_j^+|<1$ and $|\zeta_j^-|>1$. 
\end{proof}
\subsubsection{Eigenvectors in Case 1}
We recall \cite[Proposition 3.6]{SjVo19a} in a slightly modified form to fit our somewhat different 
notation.
\begin{prop}\label{ev:prop1}
	Let $p$ be as in \eqref{lp1}, \eqref{lp1.1} with $\mathfrak{N}_{\pm} \geq 0$. 
	Let $K\subset\Z$ be an interval of length $\leq \mathfrak{N}_+ +\mathfrak{N}_-$. 
	Any function $u:~K\to \C$ can be extended to a solution $\widetilde{u}:~\Z \to \C$ 
	to $\mathrm{Op}(p)\widetilde{u}=0$. The space of such extensions is affine of 
	dimension $\mathfrak{N}_+ +\mathfrak{N}_- - \# K$. In particular the extension is unique 	
	when $\mathfrak{N}_+ +\mathfrak{N}_- =\# K$. 
\end{prop}
Let $p$ be as in \eqref{int0}, \eqref{int1} and let $z\in \C\backslash p(S^1)$. We assume that the 
symbol $p(\zeta) -z$ satisfies Case 1 of Lemma \ref{lem:root1}
(with $p(\zeta)-z$ replacing $p(\zeta)$ in \eqref{lp1}).
This translates to 
\begin{equation}\label{eqn:NC1}
\begin{split}
	&\bullet \quad N_{\pm} > 0, \\
	&\bullet \quad \text{or } z\neq a_0 \text{ when } N_- = 0, \text{ or } z\neq 0 \text{ when }N_- < 0, \\
	&\bullet \quad \text{or }  z\neq a_0 \text{ when } N_+ = 0, \text{ or }  z\neq 0 \text{ when } N_+ < 0.
\end{split}
\end{equation}
By Lemma \ref{lem:root1}, we have $m_++m_- = N_+\lor 0 + N_-\lor 0$ roots of $p(\zeta)-z$.
\\
\par 
%
We now  turn to the operators $P_{[0,\infty[}-z$ and $P_{]-\infty,0]}-z$. The two cases 
are similar, so we focus on the first one and  identify $[0,\infty[ \simeq \N$ whenever 
convenient. The following discussion is a modified version of the one in \cite{SjVo19a}, 
presented here for the reader's convenience.
\\
\par
Let $u\in\ell^2(\N)$ be so that $(P_{[0,\infty[} -z) u = 0$ on $\N$. When $N_+ >0$ we put 
\begin{equation}\label{at11.1}
		u(-{N_+})= \dots = u(-1)=0, 
\end{equation}
and when $N_+\leq 0$, we do not put \eqref{at11.1}. Then we see that, for $\nu =0,1,\dots$,
\begin{equation}\label{at11}
\left( \sum_{-N_-}^{N_+} a_j \tau^j -z\right) u(\nu) = 0.
\end{equation}
\par
Continuing, we then know how to extend $u\!\upharpoonright_{[-(N_+\lor 0 ),\infty[}$ to a function 
$\widetilde{u}:\Z\to \C$ satisfying  $(\mathrm{Op}(p) -z)\widetilde{u}=0$ on $\Z$, by  solving 
\eqref{at11} with $u$ replaced by $\widetilde{u}$ and for $\nu =-1,-2,\dots$. More precisely, 
the equation for $\nu =-1$ defines uniquely $\widetilde{u}(-N_+\lor 0  -1)$ and the 
next one gives $\widetilde{u}(-N_+\lor 0  -2)$. Continuing in this way we get a solution 
$\widetilde{u}$ of $(\mathrm{Op}(p) -z)\widetilde{u} =0$ on $\Z$. Since we are in Case 1, Proposition 
\ref{prop:expSol} implies that $\widetilde{u}$ is
of the form \eqref{at9}. 
\par
Since $u\in \ell^2(\N)$, it follows by \eqref{at10} that 
\begin{equation}\label{at13}
		\widetilde{u}(\nu) = \sum_{j=1}^{m_+}a_{j}^+(\zeta_j^+)^\nu.
\end{equation}
 When $N_+ >0$ we have by construction that 
 $\widetilde{u}(\nu) = u(\nu)=0$ for $\nu \in [-N_+,-1]$, which implies 
that the coefficients  $a_j^+$ in \eqref{at13} are determined by 
\begin{equation}\label{at14}
\begin{split}
	&0=\mathfrak{A} \begin{pmatrix} a_{1}^+ \\ \vdots \\ a_{m_+}^+ \end{pmatrix},  
	\quad \mathfrak{A}=(\mathfrak{A}_1^+,\dots,\mathfrak{A}_{m_+}^+) \in \C^{N_+\times m_+},
	\\
	& \mathfrak{A}_j^+ = ( (\zeta_j^+)^\nu)_{-N_+ \leq \nu\leq -1}, \quad 
	\text{for } j=1,\dots, m_+.
\end{split}
\end{equation}
Notice that $\mathfrak{A}$ is a rectangular matrix of size $N_+\times m_+$. Recall from Proposition \ref{at:prop1} 
that the exponential functions \eqref{at8} restricted to an interval $K\Subset \Z$ of 
length $|K| = m_+$ are linearly independent. Thus, the linear system in the first line of \eqref{at14} 
has $m_+ - N_+$ linearly independent solutions if $m_+ - N_+\geq 0$, and none when $m_+ - N_+ <0$. 
This implies that $\dim \mathcal{N}(P_{[0,+\infty[}-z) = (m_+-N_+)\lor 0$.
\par
Similarly, when $N_+\leq 0$, we have 
no constraints on the coefficients in \eqref{at13}, which yields 
$m_+\lor 0$ linearly independent solutions, so  $\dim \mathcal{N}(P_{[0,+\infty[}-z) = m_+\lor 0$. Thus 
\begin{equation}
	\dim \mathcal{N}(P_{[0,+\infty[}-z)  = (m_+-N_+\lor 0)\lor 0.
\end{equation}
Similarly one also obtains the corresponding statements for the kernel of $(P_{]-\infty,0]}-z)$ with the $N_+$, $m_+$ replaced by $N_-$, $m_-$. Summing up what we have proven so far, we get in view of \eqref{at9.1}, \eqref{at9.0}, see \cite{SjVo19a} for similar statements:
\begin{prop}\label{at:prop3}
Let $p$ be as in \eqref{int0}, \eqref{int1}. Let $z\in\C\backslash p(S^1)$. Assume 
that the assumptions of Case 1 in Lemma \ref{lem:root1} hold for $p(\zeta)-z$, see also 
\eqref{eqn:NC1}. Then 
\begin{equation*}
		\dim \mathcal{N}\left(P _{[0,+\infty[}-z\right) = \mathrm{ind}_{p(S^1)}(z) \lor 0
\end{equation*}
and $u\in \mathcal{N}\left(P _{[0,+\infty[}-z\right) $ if and only if $u\in\ell^2([0,\infty[)$ is such that $\widetilde{u}$ defined above is of the form
\begin{equation}\label{at:prop3.1}
		\widetilde{u}(\nu) = \sum_{j=1}^{m_+}a_{j}^+(\zeta_j^+)^\nu, \quad a^+_{j} \in\C,
\end{equation}
with $a^+_{j}$ additionally satisfying \eqref{at14} when $N_+ >0$. Furthermore, 
\begin{equation*}
		\dim \mathcal{N}\left(P_{]-\infty,0]}-z\right) = (-\mathrm{ind}_{p(S^1}(z) )\lor 0.
\end{equation*}
and $u\in \mathcal{N}\left(P_{]-\infty,0]}-z\right) $ if and only if $u\in\ell^2(]-\infty,0]$ is of the form
\begin{equation}\label{at:prop3.1b}
		\widetilde{u}(\nu) = \sum_{j=1}^{m_-}a_{j}^-(\zeta_j^-)^\nu, \quad a^-_{j} \in\C, 
\end{equation}
with $a^-_{j}$ additionally satisfying the analog of \eqref{at14} (with $m_+,N_+,\zeta_j^+$ 
replaced by $m_-,N_-,\zeta_j^-$ and the corresponding vectors $\mathfrak{A}_j^-$ indexed 
over $1\leq \nu \leq N_-$), when $N_- >0$. 
\end{prop}
\section{Separation of the eigenvalues from the spectral curve}
\label{sec-separation}
As discussed in the introduction, in our setup all but $o(N)$ of the
eigenvalues of $P_{N,\gamma}^Q$ lie in a small neighborhood of $p(S^1)$.
In this section we obtain  complementary,
 precise estimates on the location of (all) the eigenvalues for the finitely banded case:
 we will show that if the perturbation is small
 in the sense that $\gamma >1$, then the eigenvalues avoid certain $N$-dependent regions, referred to as \textit{forbidden tubes} (see Definition \ref{dfn:tubes} below), around the spectral curve. 
 These tubes are determined in terms of the roots of the Laurent polynomial $p_z(\cdot) =0$, where 
 \begin{equation}\label{eq:p_z}
p_z(\zeta) \defeq 	p(\zeta)-z= \sum_{-N_-}^{N_+} a_j \zeta^{-j} -z =0,
\end{equation}
with $\zeta \in \C \dot{\cup} \{ \infty\}$ and $z \in \C$. 
As before, we will exclude the trivial case $N_\pm =0$. Furthermore, if $N_+ <0$ then we 
modify $p$ by setting $a_j =0$ for $j= N_++1, \ldots, 0$. This allows us, without loss of generality, to assume that $N_+ \ge 0$.

To state the main result of this section we need to 
introduce some notation.  With the goal of being consistent
with the notation of \cite{BZ}, 
we write $\{\eta_j(z)\}_{j=1}^{\wt m}$ for 
the negative of the roots of $p_z(\cdot)=0$ 
arranged in a non-increasing order of their moduli. 
For $z \notin p(S^1)$ we let $m_+=m_+(z)$ denote
the number of roots of $p_z(\cdot)$ that are in ${D(0,1)}$, 
while the number of roots in 
$(\C \cup \{\infty\}) \setminus \ol{D(0,1)}$ is denoted 
$m_-=m_-(z)$. Therefore,
\[
\wt m \defeq m_+ + m_- = N_+ + \max \{N_-, 0\},
\]
and \[
|\eta_1(z)| \ge |\eta_2(z)| \ge \cdots \ge |\eta_{m_-}(z) | > 1 > |\eta_{m_-+1}(z)| \ge \cdots \ge  |\eta_{\wt m}(z)|.
\]
Compared with the previous notation, see e.g.~Lemma \ref{lem:root1},
we have that $\eta_i=-\zeta^-_{m_--i+1}$ for $i=1,\ldots,m_-$ and 
$\eta_{m_-+i}=-\zeta^+_{m_+-i+1}$ for $i=1,\ldots,m_+$.
Recall the notation ${\sf g}(\cdot)$ from \eqref{eq:gpintro}. The parameter ${\sf g}(\cdot)$  will determine the width of forbidden tubes,
see Remark \ref{rmk:g_0}.

The definition of the forbidden 
tubes also involves 
the winding number of $p(S^1)$. For $d \in \Z$ we set
\begin{equation}\label{eq:cS-d}
\cS_d \defeq \{z \in \C \setminus p(S^1): m_+(z) - N_+ = d\}. 
\end{equation}
Note that for $z \in \cS_d$,
the winding number ${\rm ind}_{p(S^1)}(z)$ equals  $d$. Also observe that on $\cS_d$ the functions $z \mapsto m_\pm(z)$ are constants and those common values are determined by $d$ and $N_\pm$. 

 Recall that 
${\rm cl}(A)$ denotes the closure of a set $A\subset \C$.
We can now define the forbidden tubes.
\begin{defn}[Forbidden tubes]\label{dfn:tubes}
Let $p(\cdot)$ be a Laurent polynomial. 
For $z \in \C$, let $p_z(\cdot)$ be as in \eqref{eq:p_z}, and write
${\sf g}(p)={\sf g}_0\in \Z$. Fix $\varepsilon_0 >0$, $\gamma' >1$, and $\varepsilon_0' >0$ such that $\varepsilon_0' < \vep_0$. 

\begin{itemize}
  \item For $d\in [0,\wt m]\cap \Z$ we set $ \cT_{\varepsilon_0', \vep_0}^d \defeq \cT_{\vep_0',\varepsilon_0}^{d, (1)}  \cup \cT_{ \vep_0', \varepsilon_0}^{d, (2)}$, where for $d >0$ 
\begin{multline*}
\cT_{\vep_0', \varepsilon_0}^{d, (1)} \defeq {\rm cl}\Big({\Big\{z \in {\cS_d}:  \max\left\{{ |\eta_{m_-+{\sf g}_0+1}(z)|,  |\eta_{m_-}(z)|^{-1}} \right\} \le 1 -\varepsilon_0,}\\ 
 {1 - \vep_0'} \le { |\eta_{m_-+{\sf g}_0}(z)| \le    |\eta_{m_-+1}(z)| }< 1\Big\}\Big),
\end{multline*}
\noindent
$\cT_{\vep_0', \varepsilon_0}^{0, (1)}=\emptyset$, and 
\begin{multline*}
\cT_{\vep_0', \varepsilon_0}^{d, (2)} \defeq \Big\{z \in {\cS_d}:  \max\left\{ |\eta_{m_-+1}(z)|,  |\eta_{m_--{\sf g}_0}(z)|^{-1} \right\} \le 1 -\varepsilon_0, \\
1-\vep_0'  \le  |\eta_{m_--{\sf g}_0+1}(z)|^{-1} \le    |\eta_{m_-}(z)|^{-1} < 1\Big\}.
\end{multline*}
\item
For $d \in[-\wt m,0)\cap \Z$ we set $ \cT_{ \varepsilon_0', \vep_0}^d \defeq \cT_{\vep_0', \varepsilon_0}^{d, (1)}  \cup \cT_{ \vep_0', \varepsilon_0}^{d, (2)}$, where
\begin{multline*}
\cT_{\vep_0', \varepsilon_0}^{d, (1)} \defeq {\rm cl}\Big(\Big\{z \in {\cS_d}:  \max\left\{ |\eta_{m_-+1}(z)|,  |\eta_{m_- - {\sf g}_0}(z)|^{-1}  \right\} \le 1 -\varepsilon_0, \\
1-\vep_0'  \le   |\eta_{m_- - {\sf g}_0+1}(z)|^{-1} \le   
|\eta_{m_-}(z)|^{-1} < 1 \Big\},
\end{multline*}
\begin{multline*}
\cT_{ \vep_0',\varepsilon_0}^{d, (2)} \defeq \Big\{z \in {\cS_d}: \max\left\{|\eta_{m_-+{\sf g}_0+1}(z)|,  |\eta_{m_-}(z)|^{-1} \right\} \le 1 -\varepsilon_0, \\
1 - \vep_0' \le   |\eta_{m_-+{\sf g}_0}(z)| \le    |\eta_{m_-+1}(z)| < 1 \Big\}.
\end{multline*}
\end{itemize}
For all $d$, 
set 
$
\wh \cT^{d, (s)}_{\gamma',\vep_0} \defeq \cT^{d, (s)}_{(\gamma'-1) \log N/N, \vep_0},  s \in \{1,2\},$
and write
$ \wh \cT^d_{\gamma', \vep_0} \defeq \wh \cT^{d, (1)}_{\gamma',\vep_0} \cup \wh \cT^{d, (2)}_{\gamma',\vep_0}.  $
\end{defn}
Note that $\wh \cT^d_{\gamma',\vep_0}$ is a (union of)
tubes of vanishing width, whereas the width of 
$\cT^d_{\vep_0',\vep_0}$ is small but fixed.
When the choices of the parameters $\gamma', \vep_0, \vep_0'$ are clear from the context we will suppress the dependence  of the tubes on these parameters, and
 write $\cT^d$, $\wh \cT^d$, $\cT^{d,(s)}$, and $\wh \cT^{d, (s)}$, for $s \in \{1,2\}$.

\begin{rem}
The map $z \mapsto \{\eta_j(z)\}_{j=1}^{\wt m}$ is continuous in the {\em symmetric product topology} \cite[Appendix 5, Theorem 4A]{W72}. Therefore, the map $z \mapsto (|\eta_1(z)|, |\eta_2(z)|, \ldots, |\eta_{\wt m}(z)|)$ is continuous as well. This implies in particular that for $d>0$ and any $z \in \cT^{d,(1)}_{\vep_0,\vep_0}$, one has that
\[
\max\left\{{ |\eta_{m_-+{\sf g}_0+1}(z)|,  |\eta_{m_-}(z)|^{-1}} \right\} \le 1 -\varepsilon_0, \quad \text{ and } \quad 
 {1 - \vep_0'} \le { |\eta_{m_-+{\sf g}_0}(z)| \le    |\eta_{m_-+1}(z)| } \le  1,
\]
where $m_\pm$ are determined by $d$ and $N_\pm$. When $d < 0$ a similar assertion holds for $z \in \cT^{d, (1)}_{\vep_0, \vep_0'}$. 
\end{rem}

\begin{rem}
While proving the results in this section, and Sections \ref{sec:thm-sep-spec-curve} and \ref{sec:bulk-eig} we will need to consider $\vep_0$ and $\vep_0'$ such that $\vep_0' /\vep_0$ is a small fraction depending only on the degree of the Laurent polynomial $p(\cdot)$. The specific choices will be spelled out during the proofs. 
\end{rem}

We  illustrate Definition \ref{dfn:tubes} in some  examples. Consider 
first the simplest setup when the Toeplitz matrix $P_N$ is the 
Jordan block $J_N$ given by
$(J_N)_{i,j} = {\bf 1}_{j=i+1}$ for $i, j \in [N]$. In this case, $\cT^{1,(1)}= \ol{D(0,1)}\setminus D(0, 1-\vep_0')$, whereas 
$\cT^{0,(2)}$ is $\ol{D(0,1+\vep_0')}\setminus \ol{D(0,1)}$. 
The rest of the tubes are empty. 
When the spectral curve $p(S^1)$ is a Lima\c{c}on, i.e.~$P_N=J_N+J_N^2$ has
symbol 
$p(\zeta)=\zeta+\zeta^2$,
the union of the tubes is a region around the entire spectral curve,
except for a small neighborhood, determined by $\vep_0$, 
around the point on $p(S^1)$ where the curve intersects itself. 
See Figure \ref{fig:1}.
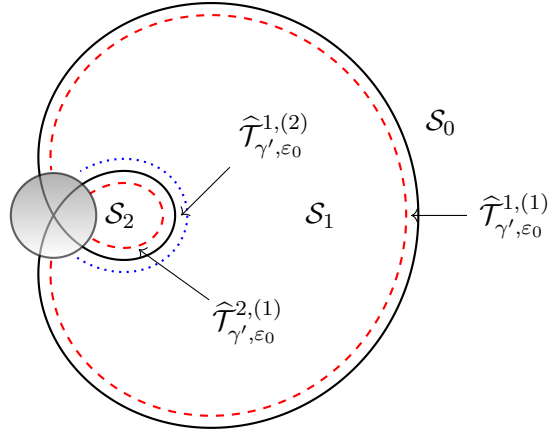
\begin{figure}
\begin{tikzpicture}[scale=1.6]
\draw[thick,variable=\t,domain=0:360,samples=360]
  plot ({cos(\t)+cos(2*\t)},{sin(\t)+sin(2*\t)});
\draw[thick,red,dashed,variable=\t,domain=252:468,samples=360]
plot({cos(\t)+cos(2*\t)- 0.1*(cos(\t)+2*cos(2*\t))/sqrt(5+4*cos(\t)},{sin(\t)+sin(2*\t)- 0.1*(sin(\t)+2*sin(2*\t))/sqrt(5+4*cos(\t)});
\draw[thick,red,dashed,variable=\t,domain=134:227,samples=100]
plot({cos(\t)+cos(2*\t)- 0.1*(cos(\t)+2*cos(2*\t))/sqrt(5+4*cos(\t)},{sin(\t)+sin(2*\t)- 0.1*(sin(\t)+2*sin(2*\t))/sqrt(5+4*cos(\t)});
\draw[thick,blue,dotted,variable=\t,domain=133:226,samples=120]
plot({cos(\t)+cos(2*\t)+0.1*(cos(\t)+2*cos(2*\t))/sqrt(5+4*cos(\t)},{sin(\t)+sin(2*\t)+0.1*(sin(\t)+2*sin(2*\t))/sqrt(5+4*cos(\t)});
\shadedraw[thick,opacity=0.6](-1,0) circle (10pt);
\node[right, xshift=110pt] {$\wh \cT_{\gamma',\vep_0}^{1,(1)}$};

\node[right, xshift=20pt, yshift=30pt] {$\wh \cT_{\gamma',\vep_0}^{1,(2)}$};

\node[right, xshift=10pt, yshift=-40pt] {$\wh \cT_{\gamma',\vep_0}^{2,(1)}$};

\node[right, xshift=90pt,yshift=35pt] {$\cS_0$};

\node[right, xshift=45pt] {$\cS_1$};

\node[right, xshift=-30pt] {$\cS_2$};
\draw[->] (2.4,0)--(1.94,0);
\draw[->] (0.45,0.4)--(0.05,0);
\draw[->] (0.28,-0.7)--(-0.3,-0.28);
\end{tikzpicture}
\caption{The different tubes for the  Lima\c{c}on, where
$p(\zeta)=\zeta+\zeta^2$, $N_+=0$, and $d=0,1,2$ according to whether $z$ is outside the curve, inside the
larger loop, or inside the inner loop. The shaded disc 
is $\cB_1^{\wt \vep_0}$, for some $\wt \vep_0 >0$ (recall Definition \ref{def-badintro}). 
Theorem \ref{thm:no-outlier} precludes the appearance of
eigenvalues
of small 
random additive perturbations of $P_N=J_N+J_N^2$ in a region away of the shaded area that includes
$\cS_0\cup\wh \cT_{\gamma',\vep_0}^{1,(1)}$,
while Theorem  \ref{thm:sep-spec-curve} handles $\wh \cT_{\gamma',\vep_0}^{1,(2)}\cup \wh \cT_{\gamma',\vep_0}^{2,(1)}$.}
\label{fig:1}
\end{figure}

Below is the first main result of this section. Hereafter, for any $z \in \C$ we use the shorthand $P_z^\delta:= P_{N,\gamma}^Q- zI_N$.

\begin{thm}\label{thm:no-outlier}
Fix $\gamma>1$. Set $\delta = N^{-\gamma}$. Let $p(\cdot)$ be a Laurent polynomial and $P_N$ be the $N \times N$ Toeplitz matrix with symbol $p(\cdot)$, and the entries of the noise matrix $Q$ satisfies Assumption \ref{assump:mom}. Then, for any $\wt \vep_0>0$ there exists some constant $\wh c_{\gamma, \wt \vep_0} >0$, depending on $\wt \vep_0$ and $\gamma$, such that  
\[
\lim_{N \to \infty} \prob  \left( \exists z \in \cS_0^{\wh c_{\gamma, \wt\vep_0} \log N/N}\setminus (\cB_1^{\wt \vep_0} \cup \cB_2^{\wt \vep_0}): \det(P_z^\delta)=0 \right) = 0.
\] 
\end{thm}
\begin{rem}\label{rem:c-gamma-explain}
The proof of Theorem \ref{thm:no-outlier} yields that one can take
\begin{equation}
  \label{eq-cgammaeps}
\wh c_{\gamma, \wt \vep_0} = \frac{\gamma'-1}{\sup_{(p(S^1))^{\wt \vep_0}\setminus \cB_2^{\wt \vep_0}} \max_{j \in [\wt m]} |\frac{d}{dz}\eta_j(z)|},
\end{equation}
for any $\gamma '$ such that $1 < \gamma' < \gamma$. By the Implicit Function Theorem we have that
\begin{equation}\label{eq:root-implicit}
\sup_{\D} \max_{j \in [\wt m]} \left|\frac{d}{dz}\eta_j(z)\right| = O(1),
\end{equation}
for any bounded domain $\D$ such that $\dist(\D, \cB_2) >0$.
Therefore, if we {\em impose the additional assumption} 
that $\dist(p(S^1), \cB_2) >0$, 
then for $\wt \vep_0 >0$ sufficiently small the 
constant $\wh c_{\gamma, \wt \vep_0}$ in Theorem \ref{thm:no-outlier} can be chosen to be free of $\wt \vep_0$. 
In particular, in the simplest case when $P_N=J_N$, the
elementary Jordan block,
Theorem \ref{thm:no-outlier} holds with any $\wh c_{\gamma, \wt \vep_0}=\wh c_\gamma < (\gamma -1)$. 
We note that in that example
with $Q_N$ a complex Ginibre matrix,
Davies and Hager \cite{DaHa09} showed that for $\gamma \ge 7$, with high probability
all eigenvalues are contained in  $D(0,1- (\gamma -3) {\log N}/{N})$.
Thus, Theorem \ref{thm:no-outlier} provides a sharper 
result for general finitely banded Toeplitz matrices, under a fairly generic assumption on $Q_N$.   
\end{rem}

\begin{rem}\label{rmk:g_0}
Let ${\sf g}_0 \in \N$ and consider the symbol 
$p(\zeta)= \zeta^{{\sf g}_0}$. Then  the empirical measures of the eigenvalues of
random perturbations of the corresponding Toeplitz matrices
converge to the uniform measure on $S^1$, see
\cite{BPZ, BPZ1, SjVo19a}. However, using \eqref{eq-cgammaeps},
one obtains from Theorem 
\ref{thm:no-outlier} that for $\gamma >1$, with high
probability the eigenvalues must be inside the disk of radius $1- {\sf g}_0 \cdot{(\gamma' -1) \log N}/{N}$, for any $\gamma' < \gamma$. The radius of the disk can be shown to tight upto  $o(\log N/N)$. 
Thus, although the parameter ${\sf g}(p)$ plays no role in determining the limit of the empirical measure of the eigenvalues of randomly perturbed Toeplitz matrices, it indeed plays a role in determining the separation of the eigenvalues from the limiting spectral curve. 
\end{rem}

Theorem \ref{thm:no-outlier} shows that all eigenvalues of $P^\delta$ must be at a distance at least $\wh c_{\gamma,\wt \vep_0} \log N/N$ from the portion of $p(S^1)$ that intersects ${\rm cl}({S}_0)$. To prove the localization of the eigenvectors we need to extend this picture to the regions with non-zero winding number. This is done in the result below.  

\begin{thm}\label{thm:sep-spec-curve}
Fix $\varepsilon_0 >0$, $\gamma'+1>\gamma > \gamma' >1$ and $d \in [-\wt m,\wt m] \cap \Z$. Assume that 
the  noise matrix $Q$ satisfies Assumptions \ref{assump:mom} and \ref{assump:anticonc}. Then, for all $\wt \vep_0>0$, we have 
\[
\lim_{N \to \infty} \prob \left( \exists z \in \wh \cT_{\gamma',\varepsilon_0}^d\setminus \cB_2^{\wt \vep_0}: \det(P_z^\delta)=0 \right) = 0. 
\]
\end{thm}

While proving Theorems \ref{thm:no-outlier} and 
  \ref{thm:sep-spec-curve}, for convenience, we will avoid neighborhoods of bad points as defined in
Definition \ref{def-badintro}. With additional efforts, it seems possible to extend our results also to neighborhoods of some 
bad points, in particular those in $\cB_1$. We do not purse that direction here. 

\begin{rem}
  Theorems \ref{thm:no-outlier} and 
  \ref{thm:sep-spec-curve} hold for $\gamma >1$. 
 We believe, and simulations suggest (see Figure \ref{fig1a}), 
 that this threshold for $\gamma$ is sharp. 
 That is, for $\gamma <1$ 
 one may find eigenvalues in ${\rm cl}({\cS}_0)$ with non-negligible 
 probability. Furthermore, the width of the tube, namely a constant multiple of
 ${\log N}/{N}$, is also sharp. See also Theorem \ref{thm-thintube}. 
\end{rem}
\begin{rem}
The papers
\cite{SjVo16, SjVo19b, SjVo19a} provide complementary estimates on the location of the eigenvalues around the curve $p(S^1)$, showing
that for the finitely banded case, $\gamma$ sufficiently large and any $\updelta_0>0$, all but $O(N^{2\updelta_0})$ eigenvalues reside inside a tube around the spectral curve of width $O(N^{-\updelta_0})$, with high probability. 
\end{rem}

\begin{rem}
The proofs of Theorems \ref{thm:no-outlier} and \ref{thm:sep-spec-curve} yield quantitative bounds on probabilities of the events considered. From the proof it follows that, for any $K_\star \in \N$, the probability that there is no eigenvalue in $\cS_0^{\wh c_{\gamma, \wt \vep_0} \log N/N}\setminus (\cB_1^{\wt \vep_0} \cup \cB_2^{\wt \vep_0})$ can be bounded by $O(N^{-K_\star})$. If additional assumptions on the moments of entries of $Q$ are imposed then this bound can be further improved. For example, the assumption that the entries of $Q$ have a uniform stretched exponential tail decay would imply a stretched exponential tail decay of the probability as well. 
On the other hand, the proof of Theorem \ref{thm:sep-spec-curve} yields that the probability that there are no eigenvalues of $P_z^\delta$ in $\wh \cT_{\gamma',\varepsilon_0}^d\setminus \cB_2^{\wt \vep_0}$ can be bounded by $O(N^{-(\gamma-\gamma')\upeta/8})$, for $\upeta>0$ as in Assumption 
\ref{assump:anticonc}. 
\end{rem}

\begin{rem}
The determinant of a matrix and its transpose are the same. Furthermore, if $Q$ satisfies Assumptions \ref{assump:mom} and \ref{assump:anticonc} then so does $Q^{\sf T}$. Therefore, whenever $P_N$ is a lower triangular Toeplitz matrix, to prove Theorems \ref{thm:no-outlier} and \ref{thm:sep-spec-curve} we can work with $P_N^{\sf T}$. Thus,
while proving these two theorems, without loss of generality, we may and will assume that $N_- \ge 1$. 
\end{rem}

\subsection{Decomposition of the determinant}  In this section
we decompose the determinant of $P_z^\delta$ into the sum of
homogeneous polynomials $\{{\det}_k(z)\}_{k=0}^N$ in
the entries of the noise matrix $Q$, see
\eqref{eq;spec-rad-new} and \eqref{eq:P_k-z} below. We
then suitably preprocess ${\det}_k(z)$  
so that various bounds on those terms 
can be derived,
leading to a proof of Theorems 
\ref{thm:no-outlier} and \ref{thm:sep-spec-curve}. 
At a high level, this decomposition and some of the 
preprocessing steps are as in
\cite{BPZ1, BZ}, where these were used
to study the empirical measure of the eigenvalues 
and the process of
outlier eigenvalues, respectively. However, in \cite{BPZ1, BZ} 
it was sufficient to derive bounds 
when $z$ avoids an {\em $N$-independent} 
region around the spectral curve $p(S^1)$. Here, we need to consider
an {\em $N$-dependent} region around it, and this requires 
better bounds and a more sophisticated analysis
than in \cite{BPZ1, BZ}, especially in the proof of Theorem  
\ref{thm:sep-spec-curve}.

For $k \in [N]$ we set
 \begin{equation}\label{eq:P_k-z}
{\det}_k(z)={\det}_{k,N}(z):= \sum_{\substack{X, Y \subset [N]\\ |X|=|Y|=k}} (-1)^{\sgn(\sigma_{X}) \sgn(\sigma_{Y})} \det (P_z[{X}^c; {Y}^c]) N^{-\gamma k} \det (Q[X; Y]),
\end{equation}
where for a matrix $A$ the notation 
 $A[X; Y]$ denotes the sub-matrix of $A$ induced by the rows in $X$ and columns in $Y$, ${X}^c:= [N]\setminus X$, ${Y}^c := [N]\setminus Y$, and for $Z \in \{X, Y\}$, $\sigma_Z$ is the permutation on $[N]$ which places all the elements of $Z$ before all the elements of ${Z}^c$, but preserves the order of the elements within the two sets. Additionally denote 
 ${\det}_0(z):= \det (P_z)$. 
 From \cite[Lemma A.1]{BPZ1} it follows that 
\begin{equation}\label{eq;spec-rad-new}
\det(P_z+N^{-\gamma} Q) = \sum_{k=0}^N {\det}_k(z).
\end{equation}

We next represent
$\{{\det}_k(z)\}_{k=1}^N$ as linear combinations of products of determinants of certain bidiagonal and upper triangular Toeplitz matrices with coefficients that are determinants of sub-matrices of $Q$. 
If $P_z$ is an upper triangular matrix, as $\{-\eta_j(z)\}_{j=1}^{\wt m}$  are the roots of the equation $p_z(\cdot)=0$, it is immediate 
to check that 
\begin{equation}
  \label{eq:uTprod}
P_z =a_{-N_-} \cdot  \prod_{j=1}^{\wt m} (J_N+\eta_j(z)I_N), 
\end{equation}
where we recall that
$J_N$ is the elementary Jordan block
given by $(J_N)_{i,j} = \delta_{i+1,j}$ for $i, j \in [N]$,  using the
Dirac notation from Section \ref{sec:notation}. Then the desired representation is a consequence of the Cauchy-Binet theorem. 
For a general Toeplitz matrix, \eqref{eq:uTprod}
does not hold. However, as in \cite{BPZ1}, 
$P_z= P_{N}(z)$ can be viewed as a certain sub-matrix of an upper triangular finitely banded Toeplitz matrix with a slightly larger dimension, and 
 one can use a product form for the latter 
coupled with the Cauchy-Binet theorem. 
To use this idea efficiently  we borrow the following definition from \cite{BPZ1}. As already alluded to above, without loss of generality, we will assume that $N_+ \ge 0$ and $N_- \ge 1$. 

\begin{defn}[Toeplitz with a shifted symbol]\label{dfn:toep-shifted}
Let $P_N(p)$ be a Toeplitz matrix with the finite symbol 
\[
p(\tau)=\sum_{j=-N_-}^{N_+} a_j \tau^j, \quad N_+\geq 0, N_-\geq 1. 
\]
For $n> N_+ + N_-$, $z \in \C$  and
$\bar N_+, \bar N_- \in \N_0$ such that $\bar N_+ + \bar N_-=N_+ +N_-\defeq \widetilde N$,
let $P_n(p, z; \bar N_-)$ denote the $n \times n$ Toeplitz matrix with symbol
\[
\widetilde p(\tau) \defeq \sum_{j=-N_-}^{N_+} a'_j \tau^{j+( N_-  - \bar N_-)},
\]
where $a'_j := a_j - z \delta_{j,0}$, $j=-N_-, -N_-+1, \ldots, N_+$. 
\end{defn}

From Definition \ref{dfn:toep-shifted}
it follows that
\begin{equation}\label{eq:P-wtN}
  P_{\wh N_{+}}(p,z;\widetilde N)=\begin{bmatrix}
a_{N_+} &\cdots &a_0-z& \cdots & a_{-N_-}&0&\cdots& 0\\
0& a_{N_+} & & a_0-z & &\ddots& & \vdots\\
\vdots & \ddots &\ddots &&\ddots&&\ddots & \vdots\\
\vdots & & \ddots & \ddots &&\ddots& & a_{-N_-}\\
\vdots & &  & \ddots & \ddots & &\ddots & \vdots\\
\vdots && &  &\ddots& \ddots & & a_0 -z\\
\vdots & & & &&  \ddots & \ddots & \vdots \\
 0 & \cdots & \cdots & \cdots&\cdots &\cdots & 0 & a_{N_+}
\end{bmatrix}
, \qquad \wh N_{+}:=N+N_+.
\end{equation}
Therefore, 
\[
P_N(p_z)=P_N(p, z;N_-) = P_{\wh N_{+}}( p, z; \widetilde N) [[N]; [\wh N_+]\setminus [N_+]]. 
\]
Recalling 
\eqref{eq:P_k-z} and writing $S+ \ell$ for the 
Minkowski sum of the sets $S$ and $\{\ell\}$, we obtain that 
\begin{eqnarray}\label{eq:det-decompose-1}
{\det}_k(z) 
&= &\sum_{\substack{X, Y \subset [N]\\ |X|=|Y|=k}} (-1)^{\sgn(\sigma_{X}) \sgn(\sigma_{Y})} \gD(X,Y)  \cdot N^{-\gamma k} \cdot \det (Q[X; Y]),
\end{eqnarray}
where 
\begin{equation}\label{eq:gD-def}
\gD(X,Y)=\gD(X,Y,z) \defeq \det (P_{\wh N_+}(p, z; \widetilde N)[{X}^c; {Y}^c+N_+]).
\end{equation}
To derive upper 
bounds on $\{{\det}_k(z)\}_{k=0}^N$ we will use
Assumption \ref{assump:mom} to compute high moments  of the 
former (and also use Markov's inequality). By \eqref{eq:det-decompose-1},
this requires the computation of high moments of
$\gD(X,Y)$ for any $X, Y \subset [N]$.
To this end, we first simplify the expression for $\gD(X,Y)$. Because
$P_{\wh N_+}(p, z;\widetilde N)$ is an upper triangular Toeplitz matrix and
$\{\eta_j(z)\}_{j=1}^{\wt m}$ are the negative of the roots of the 
equation $p_z(\cdot)=0$,  we obtain as in \eqref{eq:uTprod} that
\begin{align}\label{eq:P_N-prod-0}
P_{\wh N_+}(p, z;\widetilde N)  & = \sum_{\ell=0}^{\widetilde N} (a_{N_+ - \ell} - z \delta_{\ell, N_+}) J_{\wh N_+}^\ell 
= a_{-N_-} \prod_{j=1}^{\widetilde m} (J_{\wh N_+}+ \eta_j(z) I_{\wh N_+}).
\end{align}
We split the product in \eqref{eq:P_N-prod-0} into three blocks:~in the first block we consider the product over the roots that are greater than one in moduli and separated from one, in the second block we will have product over the roots that are close to one in moduli, and in the third we consider the product over the rest. 
To write this decomposition efficiently we set 
\begin{equation}\label{eq:wh-m1}
\wh m_1 =\wh m_1(z) \defeq \left\{\begin{array}{ll} m_- & \mbox{ for } z \in \cT^{d, (1)},\\
m_- - {\sf g}_0 & \mbox{ for } z \in \cT^{d, (2)},
\end{array}
\right.
\end{equation}
and introduce, recalling \eqref{eq:gpintro} and that ${\sf g}(p)={\sf g}_0$,
\begin{equation}\label{eq:cP-z}
\cP_z =\cP_{\wh N_+}(z) \defeq \prod_{j=\wh m_1+1}^{\wh m_1+{\sf g}_0}
(J_{\wh N_+}+ \eta_j(z) I_{\wh N_+}).
\end{equation}
(Note that if ${\sf g}_0=1$ then 
$\cP_z$ is a bidiagonal matrix. The case ${\sf g}_0>1$, in which 
$\cP_z$ is an upper triangular finitely banded Toeplitz matrix with
bandwidth ${\sf g}_0$, requires a 
special combinatorial
analysis contained in Section \ref{sec:toep-minor-skew-schur} below.)
We now obtain from \eqref{eq:P_N-prod-0} that
\begin{equation}\label{eq:P_N-prod}
P_{\wh N_+}(p, z;\widetilde N) = a_{-N_-} \prod_{j=1}^{\wh m_1} (J_{\wh N_+}+ \eta_j(z) I_{\wh N_+}) \cdot \cP_z \cdot \prod_{j=\wh m_1 +{\sf g}_0+1}^{\widetilde m} (J_{\wh N_+}+ \eta_j(z) I_{\wh N_+}),
\end{equation}
where, as is our convention, an empty product equals one. 
Notice that in  the special case ${\sf g}_0=1$,
the decompositions \eqref{eq:P_N-prod-0} and \eqref{eq:P_N-prod} 
are identical. 
Before going further let us explain the advantage of the decomposition \eqref{eq:P_N-prod} over \eqref{eq:P_N-prod-0}.

When $z$ avoids a {\em $N$-independent} 
region around the spectral curve all the roots 
of $p_z(\cdot)=0$ are also bounded 
away from the unit circle. 
For such $z$'s using the decomposition \eqref{eq:P_N-prod-0} appropriate bounds on quantities analogous to $\gD(X,Y,z)$ were derived in \cite{BPZ1, BZ}.  Under the setup of Theorem \ref{thm:sep-spec-curve}, 
when $\gamma>{\sf g}_0$, suitable adaptations of the arguments 
in \cite{BPZ1, BZ} yield good
bounds on $\gD(X,Y,z)$, and these adaptations fail when $\gamma\leq
{\sf g}_0$.
At a high level this is due to the fact that the arguments in \cite{BPZ1, BZ}, while applying the Cauchy-Binet theorem to obtain a bound on $\gD(X,Y,z)$, bound each of the roots $\{-\eta_j(z)\}_{j=1}^{\wt m}$ by their moduli. Since, for any $z \in \cT^d$, there are ${\sf g}_0$ many roots close to one in moduli, this requires 
$\gamma > {\sf g}_0$. To push the analysis to the case
${\sf g}_0\geq 
\gamma>1$, we will use that the roots of the symbol of the
Toeplitz matrix $\cP_z$  are 
those of $p_z(\cdot)=0$
that are close to the unit circle. 
Instead of bounding each of these roots by their moduli,
obtaining a bound on minors  of $\cP_z$ by combinatorial means yields
cancellations.
Therefore, \eqref{eq:P_N-prod} is more useful than
\eqref{eq:P_N-prod-0} in the proof of 
Theorem \ref{thm:sep-spec-curve}, when $1<\gamma\leq {\sf g}_0$. 

We now turn to deriving a tractable representation of $\gD(X,Y,z)$. Set 
\begin{equation}\label{eq:whm-2}
\wh m_2 =\wh m_2(z) := \wt m - \wh m_1(z) -{\sf g}_0. 
\end{equation}
Fix $k \in [N]$, and sets $\{X_i\}_{i=1}^{\wh m_1+1}$ and $\{Y_{i'}\}_{i'=1}^{\wh m_2+1}$ such that $|X_i| =k+N_+$, for $i \in [\wh m_1+1]$, and $|Y_{i'}| =k+N_+$, for $i' \in [\wh m_2+1]$.
Write 
\begin{equation}
\label{eq-Lnew}
X_i :=\left\{ x_{i,1} < x_{i,2} < \cdots < x_{i,k+N_+} \right\}, \quad \cX_k :=(X_1, X_2, \ldots, X_{\wh m_1 +1}),
\end{equation}
\begin{equation}
\label{eq-Lnew-1}
Y_{i'} :=\left\{ y_{i',1} < y_{i',2} < \cdots < y_{i',k+N_+} \right\}, \quad \text{ and } \quad \cY_k :=(Y_1, Y_2, \ldots, Y_{\wh m_2 +1}).
\end{equation}
Set
\begin{equation}
\gD_1(\cX_k)= \gD_1(\cX_k, z) :=\prod_{i=1}^{\wh m_1} \det\left((J_{\wh N_+} + \eta_i(z) I_{\wh N_+})[\COMP{X}_i; \COMP{X}_{i+1}]\right)
\end{equation}
and
\begin{equation}
\gD_2(\cY_k)= \gD_2(\cY_k, z) :=  \prod_{i'=1}^{\wh m_2} \det\left((J_{\wh N_+} + \eta_{i'+\wh m_1+{\sf g}_0}(z) I_{\wh N_+})[\COMP{Y}_{i'}; \COMP{Y}_{i'+1}]\right),
\end{equation}
where $\COMP{Z}:=[\wh N_+]\setminus Z$ for any $Z \subset [\wh N_+]$. We emphasize the notational difference between $\COMP{Z}$ and $Z^c$. The former will be used to write the complement of $Z$ when viewed as a subset of $[\wh N_+]$, whereas for the latter $Z$ will be viewed as a subset of $[N]$.

Now  using the Cauchy-Binet theorem and the representation \eqref{eq:P_N-prod} we find that 
\begin{equation}\label{eq:gD}
\gD(X,Y)=\sum_{\cX_k, \cY_k} a_{-N_-}^{N-k} \cdot \gD_1(\cX_k) 
\cdot \det(\cP_z [\COMP{X}_{\wh m_1+1};\COMP{Y}_1 ]) \cdot \gD_2(\cY_k),
\end{equation}
where the sum above is taken over $(\cX_k, \cY_k)$ such that 
\begin{equation}\label{eq:X-Y-constraint}
X_1=  X \cup [\wh N_+]\setminus [N] \qquad \text{ and } \qquad Y_{\wh m_2+1}=  (Y+N_+) \cup [N_+]. 
\end{equation}
Next using  \cite[Lemma A.3]{BPZ1} we find that
\begin{multline}\label{eq:det-bidiagonal}
\det( (J_{\wh N_+} + \eta_i(z) I_{\wh N_+}) [\COMP{X}_i; \COMP{X}_{i+1}]) = \eta_i(z)^{x_{i+1,1}-1} \cdot \left(\prod_{\ell=2}^{k+N_+} \eta_i(z)^{x_{i+1,\ell}-x_{i,\ell}-1} \right) \\
\cdot \eta_i(z)^{\wh N_+ - x_{i,k+N_+}} 
\cdot {\bf 1} \left\{ x_{i+1,\ell} \le x_{i,\ell} < x_{i+1, \ell+1}, \ell \in [k+N_+]\right\},
\end{multline}
where we have set $x_{i+1,k+N_+ +1}=\wh N_++1$ for convenience. Thus, in light of \eqref{eq:det-bidiagonal}, we may restrict the sum in \eqref{eq:gD} over $\cX_k$ belonging to
\begin{equation}\label{eq:L-k}
\wh L_k^{(1)} := \{ \cX_k:  \, 1 \le x_{i+1,1} \le x_{i,1} < x_{i+1,2} \le x_{i,2} < \cdots < x_{i+1,k+N_+} \le x_{i,k+N_+} \le \wh N_+\}. 
\end{equation}
It will also be convenient to further partition the set of all $\cX_k \in \wh L_k^{(1)}$ so that $\gD_1(\cX_k)$ is constant 
in each block of that partition. 
To this end, 
for any ${\bm \ell}:= (\ell_1,\ell_2,\ldots,\ell_{\wh m_1+1})$ with $0 \le \ell_i \le \wh N_+$ for 
$i \in [\wh m_1]$, and $k\in[N]$, we let
\begin{multline*}
L_{{\bm \ell},k}^{(1)}:=\{ \cX_k \in \wh L_k^{(1)}:  \\
\, x_{i+1,1}+ \sum_{j=2}^{k+N_+}(x_{i+1,j}-x_{i,j-1})+ (\wh N_+-x_{i,k+N_+}) =\ell_i+k+N_+, \text{ for all } i=1,2,\ldots,\wh m_1 \}.
\end{multline*}
If  $\cX_k \in L_{{\bm \ell}, k}^{(1)}$ for some ${\bm \ell}$, by \eqref{eq:det-bidiagonal}, we have that 
\begin{equation}\label{eq:prod-decomp}
\prod_{i=1}^{\wh m_1} \det\left((J_{\wh N_+} + \eta_i(z)I_{\wh N_+})[\COMP{X}_i; \COMP{X}_{i+1}]\right) = \prod_{i=1}^{\wh m_1} \eta_i(z)^{\ell_i}.
\end{equation}
Recall that for $z \in \cT^d$ the roots $\{\eta_i(z)\}_{i=1}^{\wh m_1}$ are greater than one in moduli and bounded away from one. Thus, for large values of $\{\ell_i\}_{i=1}^{\wh m_1}$ the \abbr{RHS} of \eqref{eq:prod-decomp} will be exponentially (in $N$) large. It would be useful to factor
out this exponential factor. So, using the observation that 
\[
x_{i+1,1}+  \sum_{j=2}^{k+N_+}(x_{i+1,j}-x_{i,j-1}) +  \sum_{j=1}^{k+N_+}(x_{i,j}-x_{i+1,j}) + (\wh N_+- x_{i,k+N_+}) = \wh N_+,
 \]
we have the following equivalent representation of $L_{{\bm \ell}, k}^{(1)}$:
  \begin{equation}\label{eq:L-ell-k}
L_{{\bm \ell},k}^{(1)}= \Big\{ \cX_k \in \wh L_k:  \, \sum_{j=1}^{k+N_+}(x_{i,j}-x_{i+1,j}+1) =\hat \ell_i; \ i \in [\wh m_1]\Big\},
\;\mbox{\rm where} \;\hat \ell_i := \wh N_+- \ell_i.
\end{equation}
Since $x_{i+1,j} \le x_{i,j}$ for any $i \in [\wh m_1] $, and $j \in [k+N_+]$, we further note from \eqref{eq:L-ell-k} that
\begin{equation}\label{eq:hat-ell-lbd}
\hat \ell_i \ge k+N_+, \qquad \text{ for } i=1,2,\ldots, \wh m_1. 
\end{equation}
This lower bound will be used later in the proof. 

Equipped with the above set of notation we now find that, for any $\cX_k \in L_{{\bm \ell}, k}^{(1)}$, 
\begin{equation}\label{eq:det-prod-decompose}
\gD_1(\cX_k) = \prod_{i=1}^{\wh m_1} \eta_i(z)^{\wh N_+} \cdot \prod_{i=1}^{\wh m_1} \eta_i(z)^{-\hat \ell_i}.
\end{equation}
Thus we indeed have that $\gD_1(\cX_k)$ is constant 
for $\cX_k \in L_{{\bm \ell}, k}^{(1)}$.
We now adopt a similar decomposition of the set of indices $\cY_k$ so that again inside each of the blocks $\gD_2(\cY_k)$ will remain the same. This necessitates the following notation:~By a similar reasoning to the above the sum over $\cY_k$ in \eqref{eq:gD} can be restricted to the following set
\begin{equation}\label{eq:L-k-Y}
\wh L_k^{(2)} := \{ \cY_k:  \, 1 \le y_{i'+1,1} \le y_{i',1} < y_{i'+1,2} \le y_{i',2} < \cdots < y_{i'+1,k+N_+} \le y_{i',k+N_+} \le \wh N_+\}. 
\end{equation}
Next, for ${\bm q}:= (q_1, q_2, \ldots, q_{\wh m_2})$ with $0 \le q_{i'} \le \wh N_+$ for $i' \in [\wh m_2]$ we write 
\begin{multline}\label{eq:L-ell-k2}
L_{{\bm q},k}^{(2)}:= \{ \cY_k \in \wh L_k^{(2)}:  \\
\, y_{i'+1,1}+ \sum_{j=2}^{k+N_+}(y_{i'+1,j}-y_{i',j-1})+ (\wh N_+-y_{i',k+N_+}) =q_{i'}+k+N_+, \text{ for } i'\in [\wh m_2] \}.
\end{multline}
Using \cite[Lemma A.3]{BPZ1} we observe that for $\cY_k \in L_{{\bm q},k}^{(2)}$,
\begin{equation}\label{eq:det-prod-decompose-Y}
\gD_2(\cY_k) = \prod_{i'=1}^{\wh m_2} \eta_{i'+\wh m_1 +{\sf g}_0}(z)^{q_{i'}}.
\end{equation}
Now upon recalling the restriction on $X_1$ and $Y_{\wh m_2+1}$ from   \eqref{eq:X-Y-constraint}, we deduce from \eqref{eq:gD}, \eqref{eq:det-prod-decompose}, and \eqref{eq:det-prod-decompose-Y} that, for any $X, Y \subset [N]$ such that $|X|=|Y|=k$, 
we have  the representation
\begin{multline}\label{eq:P-k-decompose}
\gD(X,Y, z)= a_{-N_-}^{N-k} \cdot 
 \Bigg[ \sum_{\cX_k \in L_{k}^{(1)}} \sum_{\cY_k \in L_{k}^{(2)}}  \gD_1(\cX_k) 
\cdot \det(\cP_z [\COMP{X}_{\wh m_1+1};\COMP{Y}_1 ]) \cdot \gD_2(\cY_k) \Bigg]\\
= a_{-N_-}^{N-k} \cdot \prod_{i=1}^{\wh m_1} \eta_i(z)^{\wh N_+} \cdot \Bigg[\sum_{{\bm \ell}, {\bm q}} \sum_{\cX_k \in L_{{\bm \ell},k}^{(1)}} \sum_{\cY_k \in L_{{\bm q},k}^{(2)}} \prod_{i=1}^{\wh m_1} \eta_i(z)^{-\hat \ell_i} \\
 \cdot \det(\cP_z[\COMP{X}_{\wh m_1+1}; \COMP{Y}_1]) \cdot \prod_{i'=1}^{\wh m_2} \eta_{i'+\wh m_1 + {\sf g}_0}(z)^{q_{i'}}\cdot {\bf 1}_{X_1=  X \cup [\wh N_+]\setminus [N] } \cdot {\bf 1}_{Y_{\wh m_2+1}=  (Y+N_+) \cup [N_+]}\Bigg].
\end{multline}

\subsection{Bound on $\gD(X,Y)$}\label{sec:bd-gD}
In this section we state bounds on $\gD(X,Y,z)$ which will be used in Section \ref{sec:high-mom} to compute high moments of the non-dominant terms in the expansion $\det(P_z^\delta)$. While stating (and deriving) such bounds it will be convenient to scale $\gD(X,Y,z)$ (and $\det_k(z)$) by the order of magnitude of the dominant term in the expansion of $\det(P_z^\delta)$. Therefore, for $z \in \cT^d$ we define
 \begin{equation}
  \label{eq:fancyK-1}
 \mathfrak{K}(z) = \mathfrak{K}(z, d, \vep_0', \vep_0):= a_{-N_-}^N \cdot N^{-\gamma \wh d} \cdot \prod_{j=1}^{\wh m_1 +{\sf g}_0} \eta_j(z)^{N},
 \end{equation}
 where  $\wh m_1$ is as in \eqref{eq:wh-m1} and
 \begin{equation}\label{eq:wh-d}
\wh d = \wh d(z) := \left\{ \begin{array}{ll}
d - {\sf g}_0 & \mbox{ if } z \in \cT^{d, (1)},\\
d & \mbox{ if } z \in \cT^{d, (2)}.
\end{array}
\right.
\end{equation}
Recalling \eqref{eq:wh-m1} and \eqref{eq:whm-2},
this implies  that 
\begin{equation}\label{eq:whdm}
\wh d(\cdot) = \wh m_2(\cdot) - N_+ \qquad \text{ on } \cT^d.
\end{equation}
For $z \in \cT^d$, further set 
\begin{equation}\label{eq:hatP-k}
\widehat {\det}_k (z):= {\det}_k (z)/\mathfrak{K}(z), \qquad k=0, 1,2,\ldots, N,
\end{equation}
and for any $X, Y \subset[N]$ such that $|X|=|Y|=k$, 
\begin{equation}\label{eq:wh-gD-dfn}
\wh \gD =\wh \gD(X,Y,z) := \frac{\gD(X,Y,z)}{N^{\gamma(\wh m_2 - N_+)} \cdot \mathfrak{K}(z)} = \frac{\gD(X,Y,z)}{a_{-N_-}^N \cdot \prod_{j=1}^{\wh m_1 +{\sf g}_0} \eta_j(z)^N},
\end{equation}
where the last equality is due to 
\eqref{eq:fancyK-1} and \eqref{eq:whdm}.

The following is the first main result of this subsection. Its proof spans over this and the next subsection. 
\begin{lem}\label{lem:comb-bound-tube-1}
Let $p(\cdot)$ be a Laurent polynomial  with $N_+\geq 0, N_-\geq 1$ 
and ${\sf g}(p)={\sf g}_0\in \N$. Fix $\vep_0', \vep_0 >0$ such that $\vep_0'/\vep_0$ is sufficiently small. Fix $X, Y \subset [N]$ such that $|X|=|Y|= k$ so that
$X =\{x_1 < x_2 < \cdots <  x_k\}$ and
$Y= \{y_1 < y_2 < \cdots < y_k\}.$
Assume $\wh d \ge 0$.
Then we have the following bounds on $\wh\gD(X,Y,z)$,
for all $z \in \cT_{\vep_0', \varepsilon_0}^{d}$. 
\begin{enumerate}
\item[(i)] There exists a constant  $C_{\ref{lem:comb-bound-tube-1}} < \infty$ depending only on $\vep_0$ and $p$ so that for any
  $k \in \N$ satisfying $ k + N_+ \ge \wh m_2$,
\begin{multline}\label{eq:gD-bd-1}
 |\wh\gD (X,Y,z)| \le  C_{\ref{lem:comb-bound-tube-1}}^{\wh m_1(k+N_+)} \cdot |\eta_{\wh m_1+1}(z)|^{-N  G(k,z)} \\
 \cdot \left\{\sum_{q \ge 0 }\binom{q(\wh m_2+1) + \wh m_2 (k+N_++1) }{\wh m_2 (k+N_++1)}  \left(1-\frac{\vep_0}{2}\right)^{q} \cdot \gI({X, Y, q}) \right\},
\end{multline}
where
\begin{equation}\label{eq:Gkz}
G(k,z) := \left\{\begin{array}{ll}
 k+N_+ - \wh m_2 & \mbox{ if } \wh m_2 \le k +N_+ < \wh m_2 +{\sf g}_0 \mbox{ and } z \in \cT^{d, (1)},\\
 {\sf g}_0 & \mbox{ if }   k +N_+ \ge \wh m_2 +{\sf g}_0 \mbox{ and } z \in \cT^{d, (1)},\\
 0 & \mbox{ if } z \in \cT^{d, (2)},
\end{array}
\right.
\end{equation}
and
\begin{equation}\label{eq:gI-XY}
  {\gI(X, Y, q)} := \prod_{j =1}^{\wh m_2 - N_+} {\bf 1}_{\{y_{j} \le q\}} \cdot \prod_{j =k+N_+ - \wh m_2+1}^{k} {\bf 1}_{\{N - x_{j} \le q\}}.
\end{equation}
\item[(ii)] Let $k \in \N$ be such that $k+N_+ < \wh m_2$. Then,
  for all large $N$,
\[
 |\wh\gD(X,Y,z)|  \le \left(1-\frac{\vep_0}{4}\right)^N.
\]
\end{enumerate}
\end{lem}

\begin{rem}
The roots of $p_z(\cdot)=0$ can be partitioned into blocks of cardinality 
${\sf g}(p)={\sf g}_0$ such that all ${\sf g}_0$ roots 
in any of those blocks have the same modulus. One also has that $N_+$ is a multiple of ${\sf g}_0$. Therefore, $d = m_+ - N_+ \ge 1$ implies that 
$d \ge {\sf g}_0$. On the other hand, for $d=0$ we have that $\cT^{d, (1)} = \emptyset$. Hence, whenever needed, we will use the bound on $\wh \gD(X, Y, z)$ derived in Lemma \ref{lem:comb-bound-tube-1} for all $z \in \cT^{d, (1)}$ and all $d \ge 0$.  
\end{rem}

Some explanations of the bounds stated in Lemma \ref{lem:comb-bound-tube-1} are in order. We will see in Section \ref{sec:thm-sep-spec-curve} that, for $k_0 \le N$ such that $k_0+N_+=\wh m_2(z)$, ${\det}_{k_0}(z)$ will be the dominant term in the expansion \eqref{eq;spec-rad-new}, 
and 
we will show  in Corollary \ref{cor:lbd-dom-term} that it has the same order of magnitude as that of $\mathfrak{K}(z)$. The other ${\det}_k(z)$ will
be negligible compared to ${\det}_{k_0}(z)$. 
Since the entries of the noise matrix are independent and of zero mean,
one gets from \eqref{eq:det-decompose-1} that 
\begin{equation}\label{eq:det-k-sec-mom-0}
\E[|{\det}_k(z)|^2] = N^{-2\gamma k} \sum_{\substack{X, Y \subset [N]\\|X|=|Y|=k}} |\gD(X,Y,z)|^2 \cdot \E(|\det(Q[X;Y])|^2).
\end{equation}
Therefore, the bound in Lemma \ref{lem:comb-bound-tube-1}(ii) indeed shows that, for $k$ such that $k+N_+ < \wh m_2(z)$ one has that ${\det}_k(z)$ is {\em exponentially small} compared to ${\det}_{k_0}(z)$. 

The implication of the bound in Lemma \ref{lem:comb-bound-tube-1}(i) is a bit more delicate. Notice from \eqref{eq:gD-bd-1}-\eqref{eq:gI-XY} that the number of choices of the 
largest $(\wh m_2 -  N_+)$ elements of $X$ is bounded by $O(q^{\wh m_2 - N_+})$. Same holds for the smallest $(\wh m_2 - N_+)$ elements of $Y$. Since there is a factor in \eqref{eq:gD-bd-1} which is exponential in $q$ those elements of $X$ and $Y$ are {\em essentially fixed} for the purpose of computation of the 
second moment of ${\det}_k(z)$. This observation, as well as the factor $ |\eta_{\wh m_1+1}(z)|^{-N G(k,z)}$ in \eqref{eq:gD-bd-1}, 
are crucial in determining the correct order of magnitude for ${\det}_k(z)$ for $k > \wh m_2(z) - N_+$. Now summing over the allowable ranges of the rest of the elements of $X$ and $Y$, i.e.~those which are {\em free}, using \eqref{eq:det-k-sec-mom-0} one obtains that, for $k > \wh m_2(z)- N_+$ and $z \in \wh \cT^d_{\gamma', \vep_0}$ with $1 < \gamma' < \gamma$, ${\det}_k(z)$ is {\em polynomially small} compared to ${\det}_{k_0}(z)$. 
See  the proof of Lemma \ref{lem:sec-mom-large-k} for details. 

 The proof of Lemma \ref{lem:comb-bound-tube-1} requires two auxiliary lemmas. Before stating them we introduce the notation
\begin{equation}\label{eq:bm-x}
{\bm x}_i ={\bm x}_i(\cX_k):= \sum_{j=1}^{k+N_+} x_{i,j};  \text{ for } i \in [\wh m_1+1], \, \text{ and } \,  {\bm y}_{i'} ={\bm y}_{i'}(\cY_k):= \sum_{j=1}^{k+N_+} y_{i',j};   \text{ for } i' \in [\wh m_2+1].
\end{equation}

Below is the first auxiliary lemma. In its statement we use the notation 
of Lemma \ref{lem:comb-bound-tube-1}.
\begin{lem}\label{lem:gD12-bd}
Fix ${\bm \ell} \in [\wh N_+]^{\wh m_1}$ and ${\bm q} \in [\wh N_+]^{\wh m_2}$. 
For $\cX_k \in L^{(1)}_{{\bm \ell}, k}$ and 
$\cY_k \in L^{(2)}_{{\bm \ell}, k}$, set
\begin{equation}\label{eq:wt-gD-1-def}
  \gF_1= \gF_1(\cX_k, {\bm \ell}, z) :=  \prod_{i=1}^{\wh m_1} \eta_i(z)^{-\hat \ell_i}, 
\end{equation}
 \begin{equation}\label{eq:wt-gD-2-def}
\gF_2 =  \gF_2(\cY_k, {\bm q}, z) := 
\prod_{i=1}^{\wh m_2}{\eta_{\wh m_1+{\sf g}_0+i}(z)}^{q_i} \quad \text{ and } \quad  \wt \gF_2 = \wt \gF_2(\cY_k, {\bm q}, z) := 
\prod_{i=1}^{\wh m_2} \left( \frac{\eta_{\wh m_1+{\sf g}_0+i}(z)}{\eta_{\wh m_1+{\sf g}_0}(z)}\right)^{q_i}. 
 \end{equation}
 Assume $\vep_0' \le \vep_0/3$. Then, for any $z \in \cT_{\vep_0', \vep_0}^{d}$ we have
\begin{equation}\label{eq:wt-gD-bd-1}
\sum_{{\bm \ell}} \sum_{\cX_k \in L_{{\bm \ell}, k}^{(1)}} \left|\gF_1 \right|\cdot {\bf 1}_{\{X_1= X \cup ([\wh N_+]\setminus [N])\}}  
 \le \left(\frac{1}{\vep_0}\right)^{\wh m_1(k+N_+)},
\end{equation}
\begin{align}\label{eq:wh-gD-bd-21}
\sum_{{\bm q}} \sum_{\cY_k \in L_{{\bm q}, k}^{(2)}}\left|   \gF_2\right|  \cdot {\bf 1}_{\{Y_{\wh m_2+1}=  (Y+N_+) \cup [N_+]\}} 
 & \le \left(\frac{2 (\wh m_2+1)}{\vep_0}\right)^{\wh m_2(k+N_++1)+1},
\end{align}
and 
\begin{equation}\label{eq:wh-gD-bd-2}
 \sum_{{\bm q}: \, \ol{\bm q}=q} \sum_{\cY_k \in L_{{\bm q}, k}^{(2)}}\left|  \wt \gF_2 \right|  \cdot {\bf 1}_{\{Y_{\wh m_2+1}=  (Y+N_+) \cup [N_+]\}} 
 \le  \binom{q(\wh m_2+1) + \wh m_2 (k+N_++1) }{\wh m_2 (k+N_++1)}  \left(1-\frac{\vep_0}{2}\right)^{q},
\end{equation}
where
\begin{equation}\label{eq:ol-bm-q}
\ol{\bm q} = \ol{\bm q}({\bm q}) := \sum_{i'=1}^{\wh m_2} q_{i'}.
\end{equation}
\end{lem}
\begin{proof}
We first prove  \eqref{eq:wt-gD-bd-1}. We will iteratively sum over the collection of indices $\cX_k$ such that $X_1= X \cup ([\wh N_+]\setminus [N])$, starting from $X_{\wh m_1+1}$. For $\cX_k \in L^{(1)}_{{\bm \ell}, k}$ we note that $ \gF_1$ depends on $\cX_k$ only through the values of ${\bm \ell}$. Thus, to compute the sum over $\cX_k$ such that $\cX_k \in L^{(1)}_{{\bm \ell}, k}$ it is enough to find a bound on the number of possible choices for all those indices. 

In the first step we keep the collection of indices $(X_1, X_2, \ldots X_{\wh m_1})$ frozen. Then, upon recalling \eqref{eq:L-ell-k} we observe that the number of choices of $\{x_{\wh m_1+1,j}\}_{j=1}^{k+N_+}$ is bounded above by $ \binom{\hat \ell_{\wh m_1}-1}{k+N_+-1}$. We iterate this argument for any $i \in [\wh m_1]$. The number of choices of the indices $\{x_{i+1,j}\}_{j=1}^{k+N_+}$, upon keeping the collection of indices $(X_1, X_2, \ldots, X_i)$ frozen, is at most  $ \binom{\hat \ell_{i}-1}{k+N_+-1}$.
This yields that
\begin{equation}\label{eq:gD-1-sum-2}
 \sum_{\cX_k\in L_{{\bm \ell}, k}^{(1)}} \left|\gF_1 \right|\cdot {\bf 1}_{X_1= X \cup ([\wh N_+]\setminus [N])}\le
\prod_{i=1}^{\wh m_1 } \binom{\hat \ell_{i}-1}{k+N_+-1} \cdot  \left( 1-{\vep_0}\right)^{\hat \ell_i},
\end{equation}
where we recalled \eqref{eq:wh-m1} and used that for   $z \in \cT^d$,
\begin{equation}\label{eq:tube-1-bd-1}
\max\left\{\max_{j=\wh m_1+{\sf g}_0+1}^{\wt m} |\eta_j(z)|, \max_{j=1}^{\wh m_1} |\eta_j(z)|^{-1} \right\} 
=\max\left\{ |\eta_{\wh m_1+{\sf g}_0+1}(z)|,  |\eta_{\wh m_1}(z)|^{-1}  \right\}
\le 1 -\varepsilon_0. 
\end{equation}
Finally, to compute the sum over $\hat \ell_1, \hat \ell_2, \ldots, \hat \ell_{\wh m_1}$ we use the following combinatorial identity:~For any $\uplambda \in (0,1)$ and $m \in \N$ 
\begin{equation}\label{eq:comb-negbin}
\sum_{s \ge m} \binom{s-1}{m-1} \cdot \uplambda^{s-m} = (1-\uplambda)^{-m}.
\end{equation}
Indeed, applying the above identity with $\uplambda = (1-\vep_0)$, $m= k+N_+$ and $s=\hat \ell_i$ for $i=1,2,\ldots, \wh m_1$ we deduce \eqref{eq:wt-gD-bd-1}. Note also that we also used the lower bound $\hat \ell_i \ge k+N_+$ (see \eqref{eq:hat-ell-lbd}). 

We now turn to the proof of \eqref{eq:wh-gD-bd-21}. We first prove the following intermediate step.
\begin{align}\label{eq:wh-gD-bd-211}
\sum_{{\bm q}: \ol{\bm q} =q} \sum_{\cY_k \in L_{{\bm q}, k}^{(2)}}\left|   \gF_2\right|  \cdot {\bf 1}_{\{Y_{\wh m_2+1}=  (Y+N_+) \cup [N_+]\}} 
 & \le \binom{q(\wh m_2+1) + \wh m_2 (k+N_++1) }{\wh m_2 (k+N_++1)}  \left(1-{\vep_0}\right)^{q}.  
\end{align}

To obtain \eqref{eq:wh-gD-bd-211}, we sum over the indices $\{Y_{i'}\}_{i'=1}^{\wh m_2}$ iteratively, starting with $Y_{1}$. It is straightforward to see that, upon keeping the other indices frozen, the number of choices for the indices $\{y_{1, j}\}_{j=1}^{k+N_+}$ such that $\cY_k \in L_{{\bm q}, k}^{(2)}$ is bounded above by $ \binom{q_{1}+k+N_+}{k+N_+}$. We use the same reasoning to successively bound the number of allowable choices of $Y_{2}, \ldots, Y_{\wh m_2}$. It yields that the number of choices of $\cY_k$ such that $\cY_k \in L_{{\bm q}, k}^{(2)}$ and $Y_{\wh m_2+1}=  (Y+N_+) \cup [N_+]$ is bounded above by 
\[
\prod_{i'=1}^{\wh m_2}  \binom{q_{i'}+k+N_+}{k+N_+}  \le \binom{\ol{\bm q}+k+N_+}{k+N_+}^{\wh m_2} \le \binom{\wh m_2 (\ol{\bm q}+k+N_+)}{\wh m_2 (k+N_+)}.
\]
Thus, by \eqref{eq:tube-1-bd-1}, we have that the \abbr{LHS} of \eqref{eq:wh-gD-bd-211} is bounded above by 
\begin{equation}\label{eq:bm-q-to-q}
\sum_{{\bm q}: \, \ol{\bm q} = q} \sum_{\cY_k \in L_{{\bm q}, k}^{(2)}}\left|   \gF_2\right|  \cdot {\bf 1}_{\{Y_{\wh m_2+1}=  (Y+N_+) \cup [N_+]\}} \le \binom{\wh m_2 (q+k+N_+)}{\wh m_2 (k+N_+)} \cdot |\{{\bm q}: \, \ol{\bm q}=q\}| \cdot (1-\vep_0)^q. 
\end{equation}
Noting that, for any $q \ge 0$, 
\begin{equation}\label{eq:ol-bmq-to-q}
 |\{{\bm q}: \, \ol{\bm q}=q\}| =  \binom{q+\wh m_2-1}{\wh m_2-1} \le \binom{q+\wh m_2}{\wh m_2},
\end{equation}
that $\binom{a}{b}\binom{c}{d}\leq \sum_{k'=0}^{b+d}\binom{a}{k'}
\binom{c}{b+d-k'}=\binom{a+c}{b+d}$,
and substituting in
\eqref{eq:bm-q-to-q} we obtain \eqref{eq:wh-gD-bd-211}. To derive \eqref{eq:wh-gD-bd-21} from \eqref{eq:wh-gD-bd-211} we observe that 
\begin{equation}\label{eq:comb-negbin-2}
(1-\vep) \le \left(1- \frac{\vep}{2 m_\star}\right)^{m_\star}, \qquad \text{ for } \vep \in (0,1/2) \text{ and } m_\star \in \N,
\end{equation}
and use \eqref{eq:comb-negbin} again. 

Finally, to prove \eqref{eq:wh-gD-bd-2} applying \eqref{eq:tube-1-bd-1} we find that whenever $\vep_0'\le \vep_0/3$,
\[
\sup_{z \in \cT^d}\left\{ \max_{i=\wh m_1+{\sf g}_0+1}^{\wt m} \left| \frac{\eta_i(z)}{\eta_{\wh m_1+{\sf g}_0}(z)}\right|\right\} \le 1-\frac{\vep_0}{2},
\]
 Repeating the proof of \eqref{eq:wh-gD-bd-211}, the bound in \eqref{eq:wh-gD-bd-2} follows. 
\end{proof}
The next lemma is 
the second auxiliary result to be used in the proof of Lemma \ref{lem:comb-bound-tube-1}, and uses its notation. 
\begin{lem}\label{lem:wh-gD-3}
 Fix $X_{\wh m_1+1},Y_1 \subset [N+ N_+]$ such that $|X_{\wh m_1+1}|=|Y_1|= k+N_+$, where $k \le N$. 
If 
\begin{equation}\label{eq:xy-rel}
x_{\wh m_1+1, j} \ge y_{1,j}; \, j \in [k+N_+], \quad \text{ and } \quad x_{\wh m_1+1, j} < y_{1,j+{\sf g}_0}; \, j \in [k+N_+-{\sf g}_0],
\end{equation}
then, for all   $z \in \cT^{d}$,
\begin{equation}\label{eq:wh-gD-bd-3}
\left| \det(\cP_z[\COMP{X}_{\wh m_1+1}; \COMP{Y}_1]) \right|  \le  
{\sf g}_0^{(2+{\sf g}_0)(k+N_+)}
\cdot |\eta_{\wh m_1+1}(z)|^{{\sf g}_0 (N-k)+{\bm y}_1-{\bm x}_{\wh m_1+1}}. 
\end{equation}
On the other hand, if \eqref{eq:xy-rel} does not hold then the \abbr{LHS} of \eqref{eq:wh-gD-bd-3} vanishes. 
 \end{lem}
 
 The proof of Lemma \ref{lem:wh-gD-3} is postponed to Section \ref{sec:toep-minor-skew-schur}. In that section a Toeplitz minor would be represented as a certain {\em skew Schur polynomial} and using that representation Lemma \ref{lem:wh-gD-3} will be proved. We now prove Lemma \ref{lem:comb-bound-tube-1}.

\begin{proof}[Proof of Lemma \ref{lem:comb-bound-tube-1}(i) for $z \in \cT^{d, (1)}$]
Fix $k \in \N$ such that $k+N_+ \ge \wh m_2$. Fix ${\bm \ell}, {\bm q}, \cX_k$, and $\cY_k$ such that $\cX_k \in L_{{\bm \ell}, k}^{(1)}$, $\cY_k \in L_{{\bm q}, k}^{(2)}$, and \eqref{eq:X-Y-constraint} holds. We also assume that  \eqref{eq:xy-rel} holds, for otherwise by the second half of Lemma \ref{lem:wh-gD-3} there is no contribution to
the sum in \eqref{eq:P-k-decompose}.
 
We first show that for such pairs $(\cX_k, {\bm \ell})$ and $(\cY_k, {\bm q})$ one must have
\begin{equation}\label{eq:gI-0}
\gI(X,Y, \ol{\bm q})=1,
\end{equation}
where we recall \eqref{eq:gI-XY} and \eqref{eq:ol-bm-q} for the definitions of $\gI$ and $\ol{\bm q}$.  
To see \eqref{eq:gI-0}, as $\cY_k \in L_{{\bm q}, k}^{(2)}$, we note that for any $j \in [\wh m_2]$
\begin{equation}\label{eq:gI-1}
y_{\wh m_2+1, j} = \sum_{i=0}^{j-2} (y_{\wh m_2+1 -i, j-i} - y_{\wh m_2-i, j-i-1})  + y_{\wh m_2 - j+2, 1} \le \sum_{i=0}^{j-1} q_{\wh m_2 -i} = \sum^{\wh m_2 }_{\lambda= \wh m_2 - j+1} q_\lambda \le \ol{\bm q}.
\end{equation}
Similarly, as $\cX_k \in L_{{\bm \ell},k}^{(1)}$ and $\cY_k \in L_{{\bm q}, k}^{(2)}$, using that \eqref{eq:xy-rel} holds, we also observe that
\begin{multline}\label{eq:gI-2}
(N+ N_+ -x_{1,j}) \le (\wh N_+ -x_{\widehat m_1+1,j})  \stackrel{\eqref{eq:xy-rel}}{\le} (\wh N_+ -y_{1,j})\\
= (\wh N_+ - y_{  k+N_+ +1 -j, k+N_+}) +  \sum_{i=0}^{k+N_+ -j-1} (y_{ i+2,j+i+1} - y_{i+1,j+i}) \le \sum_{\lambda= 1}^{ k +N_+ + 1 - j} q_\lambda \le \ol{\bm q},
\end{multline}
where in the last step we also used that  
\[
j \ge k +N_+ - \wh m_2 +1 \ge 1 \qquad \Longrightarrow \qquad  k +N_++ 1 -j \le \wh m_2. 
\]
The inequalities \eqref{eq:gI-1}-\eqref{eq:gI-2} together with \eqref{eq:X-Y-constraint}  now establish \eqref{eq:gI-0}. This implies that while summing over the indices in \eqref{eq:P-k-decompose}, without loss of generality we can assume that \eqref{eq:gI-0} holds. 

Before evaluating that sum we do some more simplifications to the representation in \eqref{eq:P-k-decompose}. As $\cX_k \in L_{{\bm \ell}, k}^{(1)}$, recalling the definition of $\{{\bm x}_i\}_{i=1}^{\wh m_1+1}$ from \eqref{eq:bm-x} we see that 
\begin{equation*}
{\bm x}_{\wh m_1+1} = \sum_{i=1}^{\wh m_1} ({\bm x}_{i+1} - {\bm x}_i) + {\bm x}_1 = \wh m_1 (k+N_+) -\sum_{i=1}^{\wh m_1} \hat \ell_i + {\bm x}_1.
\end{equation*}
Thus, for any $\cX_k \in L_{{\bm \ell}, k}^{(1)}$,
\begin{multline}\label{eq:hat-ell-1}
 \wh \gD_1 = \wh \gD_1(\cX_k, {\bm \ell}, z) :=\prod_{i=1}^{\wh m_1} \eta_i(z)^{-\hat \ell_i}  \cdot \eta_{\wh m_1+1}(z)^{-{\bm x}_{\wh m_1+1}}  \\
 = \prod_{i=1}^{\wh m_1} \left( \frac{\eta_i(z)}{\eta_{\wh m_1+1}(z)}\right)^{-\hat \ell_i} \cdot \eta_{\wh m_1+1}(z)^{-\wh m_1(k+N_+)} \cdot \eta_{\wh m_1+1}(z)^{-{\bm x}_1}.
\end{multline}
On the other hand, for any $\cY_k \in L_{{\bm q}, k}^{(2)}$,
\[{\bm y}_{1}   = {\bm y}_{2} - q_{1} + (N-k)
= {\bm y}_{3} + 2(N-k) - q_{1} - q_{2}  = \cdots = {\bm y}_{\wh m_2+1} + \wh m_2 \cdot (N-k)   - \sum_{i=1}^{\wh m_2} q_i.\]
This yields that
\begin{align}\label{eq:hat-ell-2}
\wh \gD_2 = \wh \gD_2(\cY_k, {\bm \ell}, z) &  := \prod_{i=1}^{\wh m_2} \eta_{\wh m_1 +{\sf g}_0+i}(z)^{q_i}  \cdot \eta_{\wh m_1+1}(z)^{{\bm y}_{1}} \notag \\
 & =\prod_{i=1}^{\wh m_2} \left( \frac{\eta_{\wh m_1+{\sf g}_0+i}(z)}{\eta_{\wh m_1+1}(z)}\right)^{q_i} \cdot \eta_{\wh m_1+1}(z)^{\wh m_2(N-k) +{\bm y}_{\wh m_2+1}}. 
\end{align}
Fix $k$ such that $k+N_+ \ge \wh m_2 +{\sf g}_0$. For $\cX_k \in L_{{\bm \ell}, k}^{(1)}$ and $\cY_k \in \wh L_{k}^{(2)}$, upon using that $\wh m_2 \ge 0$ and \eqref{eq:xy-rel}, we find that  
\begin{align}\label{eq:wt-bmy}
  {\bm y}_{\wh m_2+1} \stackrel{\eqref{eq:bm-x}}{\ge} \sum_{j=\wh m_2 +{\sf g}_0+1}^{k+N_+} y_{\wh m_2 +1, j} &\stackrel{\eqref{eq:L-k-Y}}{\ge}  \sum_{j={\sf g}_0+1}^{k+N_+- \wh m_2} y_{1, j}   
  \stackrel{\eqref{eq:xy-rel}}{\ge} \sum_{j=1}^{k+N_+ - \wh m_2 - {\sf g}_0} x_{ \wh m_1 +1, j}  \notag\\
  & \stackrel{\eqref{eq:hat-ell-lbd}}{\ge} \sum_{j=1}^{k+N_+ - \wh m_2 - {\sf g}_0} x_{1, j} - \sum_{i=1}^{\wh m_1} \hat \ell_i +(k+N_+) \wh m_1 \notag\\
& \stackrel{\eqref{eq:bm-x}}{\ge}  {\bm x}_1 - \sum_{i=1}^{\wh m_1} \hat \ell_i +(k+N_+) \wh m_1 - \wh N_+ (\wh m_2 +{\sf g}_0).  
\end{align}
Note that for $k \in \N$ such that $\wh m_2 \le k+N_+ \le \wh m_2 +{\sf g}_0$, the sum over $j$ in the left side of \eqref{eq:wt-bmy} is
empty. Therefore, by \eqref{eq:hat-ell-lbd} and as $X_1 \subset [\wh N_+]$ we obtain that for such values of $k$,
\begin{align}\label{eq:wt-bmy1}
{\bm y}_{\wh m_2+1}& \ge  {\bm x}_1 - \sum_{i=1}^{\wh m_1} \hat \ell_i +(k+N_+) \wh m_1 - \wh N_+  (k+N_+).  
\end{align}

Now as $z \in \cT^{d, (1)}$ we have $|\eta_{\wh m_1+1}(z)| \le 1$. Thus, recalling \eqref{eq:wt-gD-1-def}-\eqref{eq:wt-gD-2-def} and \eqref{eq:Gkz}, and using \eqref{eq:wt-bmy}-\eqref{eq:wt-bmy1} we deduce that 
\begin{equation}\label{eq:wh-wt-D13}
|\wh \gD_1| \cdot |\wh \gD_2| \le |\gF_1| \cdot |\wt \gF_2| \cdot |\eta_{\wh m_1+1}(z)|^{- NG(k,z) - \wt m (k+N_+)}.
\end{equation}
On the other hand, from \eqref{eq:hat-ell-1}-\eqref{eq:hat-ell-2}, \eqref{eq:P-k-decompose}, and \eqref{eq:wh-gD-dfn} it follows that 
\begin{align}\label{eq:wh-gD-bd-10}
\left|\wh \gD(X, Y,z) \right| 
&\le  |a_{-N_-}|^{-k} \prod_{i=1}^{\wh m_1} |\eta_i(z)|^{N_+}   \prod_{i=\wh m_1+1}^{\wh m_1+{\sf g}_0} |\eta_i(z)|^{-N}\cdot
\Bigg[\sum_{{\bm \ell}, {\bm q}} \sum_{\cX_k \in L_{{\bm \ell},k}^{(1)}} \sum_{\cY_k \in L_{{\bm q},k}^{(2)}} |\wh \gD_1|\\
&\qquad\cdot \left|\frac{\det(\cP_z[\COMP{X}_{\wh m_1+1}; \COMP{Y}_1])}{\eta_{\wh m_1+1}(z)^{{\bm y}_1 - {\bm x}_{\wh m_1 +1}}} \right|\cdot |\wh \gD_2|  
\cdot {\bf 1}_{X_1=  X \cup [\wh N_+]\setminus [N] } \cdot {\bf 1}_{Y_{\wh m_2+1}=  (Y+N_+) \cup [N_+]} \cdot \gI(X, Y, \ol{\bm q}) \Bigg],
\nonumber
\end{align}
where we also used that ${\sf g}(p)={\sf g}_0$ (recall
\eqref{eq:gpintro} for the
definition of ${\sf g}(\cdot)$) implies that
\begin{equation}\label{eq:mod-eq}
|\eta_{\wh m_1+1}(z)| = |\eta_{\wh m_1+2}(z)| = \cdots = |\eta_{\wh m_1+{\sf g}_0}(z)|.
\end{equation}  
It is clear that if $z \in \cT^{d}$ for some $d$ then $z$ lies in compact domain in $\C$. As the map $z \mapsto \max_{j=1}^{\wt m} |\eta_j(z)|$ is 
continuous, we therefore
have that 
\[
\max_d \sup_{z \in \cT^{d}} \max_{j=1}^{\wt m} |\eta_j(z)| < \infty. 
\]
Therefore, continuing from above, applying Lemma \ref{lem:wh-gD-3}, from \eqref{eq:gI-0} and \eqref{eq:wh-wt-D13} we deduce that, there exists some constant $\wt C_{\ref{lem:comb-bound-tube-1}} < \infty$ such that  
\begin{multline*}
\left|\wh \gD(X, Y,z) \right| \le  \wt C_{\ref{lem:comb-bound-tube-1}}^{\wh m_1 (k+N_+)} \cdot  |\eta_{\wh m_1+1}(z)|^{- N G(k,z) - \wt m (k+N_+)} \cdot\\
\Bigg[\sum_{{\bm \ell}, {\bm q}} \sum_{\cX_k \in L_{{\bm \ell},k}^{(1)}} \sum_{\cY_k \in L_{{\bm q},k}^{(2)}} | \gF_1| \cdot |\wt \gF_2| \cdot {\bf 1}_{X_1=  X \cup [\wh N_+]\setminus [N] } \cdot {\bf 1}_{Y_{\wh m_2+1}=  (Y+N_+) \cup [N_+]} \cdot \gI(X, Y, \ol{\bm q})\Bigg].
\end{multline*}
Finally we use Lemma \ref{lem:gD12-bd} and 
\begin{equation}\label{eq:tube-1-bd-2}
1- \vep_0' \le \min^{\wh m_1+{\sf g}_0}_{j=\wh m_1+1}  |\eta_{j}(z)| =  \max^{\wh m_1+{\sf g}_0}_{j=\wh m_1+1}  |\eta_{j}(z)| \le 1,
\quad \mbox{\rm for  $z \in \cT^{d,(1)}$},
\end{equation}
to derive the desired bound. This completes the proof of 
Lemma \ref{lem:comb-bound-tube-1}(i) for $z \in \cT^{d, (1)}$. 
\end{proof}

\begin{proof}[Proof of Lemma \ref{lem:comb-bound-tube-1}(ii) for $z \in \cT^{d,(1)}$]
As $k +N_+ < \wh m_2$, for any $\cY_k \in L_{{\bm q}, k}^{(2)}$ we have that 
\begin{align}
\ol{\bm q} & \stackrel{\eqref{eq:ol-bm-q}}{\ge} \sum_{i=1}^{k+N_+ +1} q_i \notag\\
& \stackrel{\eqref{eq:L-ell-k2}}{\ge} y_{2, 1} + \sum_{j=2}^{k+N_+} (y_{j+1,j} - y_{j,j-1}) + (\wh N_+ - y_{k+N_++1,k+N_+}) - (k+N_+) (k+N_++1) \notag\\
& \ge \wh N_+ - \wh m_2^2, \label{eq:sum-q-lbd}
\end{align}
where in the second inequality we used that
\[y_{2,1}\leq q_1+k+N_+, y_{3,2}-y_{2,1}\leq q_2+k+N_+,\ldots,\wh N_+-y_{k+N_++1,k+N_+}\leq q_{k+N_++1}+k+N_+,\]
and in the last we telescoped the sum and used again that $k+N_+<\wh m_2$.
This implies that the set $L_{{\bm q}, k}^{(2)}$ is empty unless \eqref{eq:sum-q-lbd} holds.

On the other hand, recalling \eqref{eq:wt-gD-1-def}-\eqref{eq:wt-gD-2-def},  and using \eqref{eq:P-k-decompose} and \eqref{eq:wh-gD-dfn} we note that  
\begin{multline*}
\left|\wh \gD(X, Y,z) \right| \le  |a_{-N_-}|^{-k} \cdot \prod_{i=1}^{\wh m_1} |\eta_i(z)|^{N_+} \cdot  \prod_{i=\wh m_1+1}^{\wh m_1+{\sf g}_0} |\eta_i(z)|^{-N}\cdot \\
\Bigg[\sum_{{\bm \ell}, {\bm q}} \sum_{\cX_k \in L_{{\bm \ell},k}^{(1)}} \sum_{\cY_k \in L_{{\bm q},k}^{(2)}} | \gF_1| \cdot \left|{\det(\cP_z[\COMP{X}_{\wh m_1+1}; \COMP{Y}_1])} \right|\cdot |\gF_2| \cdot {\bf 1}_{X_1=  X \cup [\wh N_+]\setminus [N] } \cdot {\bf 1}_{Y_{\wh m_2+1}=  (Y+N_+) \cup [N_+]}\Bigg].
\end{multline*}
From Lemma \ref{lem:wh-gD-3} we have that
\begin{equation}\label{eq:toep-minor-bd-21}
\left|\frac{\det(\cP_z[\COMP{X}_{\wh m_1+1}; \COMP{Y}_1])}{\prod_{i=\wh m_1+1}^{\wh m_1+{\sf g}_0} \eta_i(z)^{N}} \right| \le
{\sf g}_0^{({\sf g}_0+2)(k+N_+)}
|\eta_{\wh m_1+1}(z)|^{{\bm y}_1 - {\bm x}_{\wh m_1 +1}-k{\sf g}_0} \le 
(1+2 \vep_0')^{\wh m_2 N},
\end{equation}
for all large $N$, where in the first step we have used \eqref{eq:mod-eq} and in the second nequality we used \eqref{eq:tube-1-bd-2},
that ${\bm y}_1 \ge 0$, ${\bm x}_{\wh m_1+1} \le \wh m_2 \wh N_+$,
and $k+N_+\leq \wh m_2$.

Therefore, applying Lemma \ref{lem:gD12-bd}, 
and \eqref{eq:sum-q-lbd}-\eqref{eq:toep-minor-bd-21}, we deduce  that
\begin{multline}\label{eq:gD-exp-bd}
\left|\wh \gD(X, Y,z) \right|  
\stackrel{\eqref{eq:toep-minor-bd-21},\eqref{eq:wt-gD-bd-1}}{\le} \left(\frac{1}{\vep_0}\right)^{\wh m_1(k+N_+)} \cdot (1+2 \vep_0')^{\wh m_2 N} \cdot \left[\sum_{{\bm q} } \sum_{\cY_k \in L_{{\bm q}, k}^{(2)}}   |\gF_2| \cdot  {\bf 1}_{Y_{\wh m_2+1}=  (Y+N_+) \cup [N_+]}\right]\\
\stackrel{\eqref{eq:wh-gD-bd-2}}{\le} \left(\frac{1}{\vep_0}\right)^{\wh m_1(k+N_+)} \cdot
 (1+2 \vep_0')^{\wh m_2 N}\cdot  \left[\sum_{\ol{\bm q} } 
 \binom{\ol{\bm q}(\wh m_2+1) + \wh m_2 (k+N_++1) }{\wh m_2 (k+N_++1)}  \left(1-\frac{\vep_0}{2}\right)^{\ol{\bm q}}
 \cdot {\bf 1}_{\{\ol{\bm q} \ge \wh N_+ - \wh m_2^2\}}\right].
\end{multline}
Since 
for any fixed $n_1\in \N$ and $n_2 \in \N$ large enough, depending only on $\vep_0$ and $n_1$, one has 
\[
\sum_{\ol{\bm q} \ge n_2} \ol{\bm q}^{n_1} \left(1-\frac{\vep_0}{2}\right)^{\ol{\bm q}}
\le \left(1-\frac{3\vep_0}{8}\right)^{n_2},
\]
and $\vep_0' / \vep_0$ is sufficiently small, upon using \eqref{eq:ol-bmq-to-q}, the claimed bound on $\wh \gD(X,Y,z)$ now follows from \eqref{eq:gD-exp-bd}. This completes the proof of part (ii) for $z \in \cT^{d, (1)}$.  
\end{proof}

Next we prove Lemma \ref{lem:comb-bound-tube-1} for $z \in \cT^{d,(2)}$. It requires some minor modifications compared to the case $z\in \cT^{d,(1)}$. 
\begin{proof}[Proof of Lemma \ref{lem:comb-bound-tube-1} for $z \in \cT^{d, (2)}$]
First let us prove part (i). We begin by noting that \eqref{eq:gI-0} continues to holds even for $z \in \cT^{d, (2)}$ and $k \le N$ such that $k+N_+ \ge \wh m_2$. Recall \eqref{eq:wt-gD-1-def}-\eqref{eq:wt-gD-2-def}. We use  \eqref{eq:P-k-decompose} and \eqref{eq:wh-gD-dfn} to  note that 
\begin{multline}\label{eq:wh-gD-2-bd-10}
\left|\wh \gD(X, Y,z) \right| \\
\le  |a_{-N_-}|^{-k} \prod_{i=1}^{\wh m_1} |\eta_i(z)|^{N_+}   \prod_{i=\wh m_1+1}^{\wh m_1+{\sf g}_0} |\eta_i(z)|^{-N}\cdot
\Bigg[\sum_{{\bm \ell}, {\bm q}} \sum_{\cX_k \in L_{{\bm \ell},k}^{(1)}} \sum_{\cY_k \in L_{{\bm q},k}^{(2)}} | \gF_1| \cdot \left|{\det(\cP_z[\COMP{X}_{\wh m_1+1}; \COMP{Y}_1])} \right|\cdot | \gF_2| \\
\cdot {\bf 1}_{X_1=  X \cup [\wh N_+]\setminus [N] } \cdot {\bf 1}_{Y_{\wh m_2+1}=  (Y+N_+) \cup [N_+]} \cdot \gI(X, Y, {\bm q}) \Bigg].
\end{multline}
As $z \in \cT^{d,(2)}$ implies that $|\eta_{\wh m_1 +1}(z)| \ge 1$, applying Lemma \ref{lem:wh-gD-3}, and using that ${\bm y}_1 \le {\bm x}_{\wh m_1+1}$, we obtain that
\begin{equation}\label{eq:toep-minor-bd-22}
  \left|\frac{\det(\cP_z[\COMP{X}_{\wh m_1+1}; \COMP{Y}_1])}{\eta_{\wh m_1+1}(z)^{N {\sf g}_0}}\right| \le {\sf g}_0^{{\sf g}_0+2(k+N_+)}.
\end{equation}
Thus, from \eqref{eq:wh-gD-2-bd-10}, we now derive that 
\begin{multline*}
\left|\wh \gD(X, Y,z) \right| \le  |a_{-N_-}|^{-k} \cdot{\sf g}_0^{{\sf g}_0+2(k+N_+)} \cdot \prod_{i=1}^{\wh m_1} |\eta_i(z)|^{N_+}\cdot\\
\Bigg[\sum_{{\bm \ell}, {\bm q}} \sum_{\cX_k \in L_{{\bm \ell},k}^{(1)}} \sum_{\cY_k \in L_{{\bm q},k}^{(2)}} | \gF_1| \cdot | \gF_2| 
\cdot {\bf 1}_{X_1=  X \cup [\wh N_+]\setminus [N] } \cdot {\bf 1}_{Y_{\wh m_2+1}=  (Y+N_+) \cup [N_+]} \cdot \gI(X, Y, {\bm q}) \Bigg].
\end{multline*}
The claimed bound on $\wh \gD(X, Y,z)$ now follows upon using Lemma 
\ref{lem:gD12-bd}, in the same way as for $z\in \cT^{d,(1)}$. This completes the proof of part (i) of the lemma for $z \in \cT^{d, (2)}$. 

We now turn to the proof of the second part. To this end, we note that all the steps of the proof of Lemma \ref{lem:comb-bound-tube-1}(ii) continues to hold under the current setup, except for the bound \eqref{eq:toep-minor-bd-21}. In fact, under our current setup, the bound \eqref{eq:toep-minor-bd-21} can be improved to \eqref{eq:toep-minor-bd-22}. Thus, repeating rest of the arguments in the proof of Lemma \ref{lem:comb-bound-tube-1}(ii) used for $z\in \cT^{d,(1)}$,
the proof of the second part of this lemma completes for $z \in \cT^{d, (2)}$. 
\end{proof}

\subsection{Toeplitz minors as skew Schur polynomials}\label{sec:toep-minor-skew-schur}
The goal of this section is to prove Lemma \ref{lem:wh-gD-3}. This will follow from the following bound on minors of finitely banded Toeplitz matrices. 
\begin{prop}\label{prop:toep-minor}
Let ${\sf g}_\star, \wh N \in \N$. Fix a compact domain $\D \subset \C$ and let $z \in \D$. Let $\wh T_{\wh N}(z) = \wh T_{\wh N}(\wh p_z)$ be the $\wh N \times \wh N$ Toeplitz matrix with symbol 
\begin{equation}\label{eq:wh-p}
\wh p_z(\tau) := \sum_{j=0}^{{\sf g}_\star} \wh a_{-j}(z) \tau^{-j}, \quad a_0(z), a_{-1}(z), \ldots, a_{-{\sf g}_\star}(z) \in \C; \, a_0(z) \ne 0, \, a_{-{\sf g}_\star}(z)=1\, \forall z \in \D. 
\end{equation}
Denote $\{\wh \eta_j(z)\}_{j=1}^{{\sf g}_\star}$ to be the roots of the polynomial $\wh p_z(\cdot)=0$. Assume that, for some constant $c_1 \in (0,1)$, 
\begin{equation}\label{eq:minor-c1}
\inf_{z \in \D} \min_{k \ne j} \left| 1 - \frac{\wh \eta_j(z)}{\wh \eta_k(z)}\right| \ge c_1.
\end{equation}

Then for any 
\[
\gX := \{\gx_1 < \gx_2 < \cdots < \gx_n\} \subset [\wh N] \qquad \text{ and } \qquad  \gY := \{\gy_1 < \gy_2 < \cdots < \gy_n\} \subset [\wh N]
\]
we have that  
\begin{equation}\label{eq:toep-minor}
\left|\det (\wh T_{\wh N}(z) [\COMP{\gX}, \COMP{\gY}]) \right| \\
\le c_1^{-{\sf g}_\star n}  {\sf g}_\star^{2n} \cdot \prod_{j=1}^{{\sf g}_\star} |\wh \eta_j(z)|^{\wh N -n} \cdot \left(\min_{j=1}^{{\sf g}_\star}|\wh \eta_j(z)|\right)^{\bar \gy - \bar \gx} \cdot \prod_{i=1}^n {\bf 1}_{\{\gx_i \ge \gy _i\}} \cdot \prod_{i=1}^{n-{\sf g}_\star} {\bf 1}_{\{\gx_i < \gy _{i+{\sf g}_\star}\}},
\end{equation}
where $\COMP{\gX} := [\wh N] \setminus \gX$, $\COMP{\gY} := [\wh N] \setminus \gY$,
\[
\bar \gx := \sum_{i=1}^n \gx_i, \qquad \text{ and } \qquad  \bar \gy := \sum_{i=1}^n \gy_i.
\]
\end{prop}
\begin{rem}
The bound on the order of magnitude of Toeplitz minors derived via Proposition \ref{prop:toep-minor} can be seen to be optimal when the roots of $\wh p_z(\cdot)=0$ have the same moduli. However, it is suboptimal when the roots have different moduli. For example, consider the case ${\sf g}_\star =2$ and choose $z \in \C$ such that $|\wh \eta_1(z)| / |\wh \eta_2(z)| \ge (1+\vep')$ for some $\vep'>0$. 
To see the sub optimality of the bound now let $\gX=\{\wh N-1, \wh N\}$ and $\gY=\{1,2\}$.

Since we will apply Proposition \ref{prop:toep-minor} for a Toeplitz matrix for which the roots of its symbols have the
same moduli, we have not tried to derive a version of Proposition \ref{prop:toep-minor} that captures the optimal order of magnitude even when roots have different moduli. 
\end{rem}

As already mentioned in Section \ref{sec:bd-gD},
to prove Proposition \ref{prop:toep-minor} (and hence Lemma \ref{lem:wh-gD-3}) we will use the fact a Toeplitz minor can be expressed as certain {\em skew Schur polynomial}, 
see Remark \ref{rem-history} below for historical
background. To state the relevant result and carry out the proof of Proposition \ref{prop:toep-minor} we need to borrow some notations from the theory of symmetric functions. The references \cite{Macd, Stan2} are excellent resources for this purpose.

We start with the definition of {\em complete homogeneous symmetric polynomials}.  

\begin{defn}[Complete homogeneous symmetric polynomial]\label{def:h}
Fix $w, {\sf r} \in \N$. The ${\sf r}$-th complete homogeneous polynomial in $w$ variables, to be denoted by ${\sf h}_{\sf r}(\cdot)$, is the sum of all monomials of total degree ${\sf r}$. That is, for $(t_1, t_2, \ldots, t_w) \in \R^n$, 
\[
{\sf h}_{\sf r}(t_1,t_2, \ldots, t_{w}) := \sum_{1 \le s_1 \le s_2 \le \cdots \le s_{{\sf r}} \le w} \prod_{u=1}^{{\sf r}} t_{s_u}.
\] 
We further set ${\sf h}_0 := 1$ and ${\sf h}_{{\sf r}} := 0$ for ${\sf r} <0$. 
\end{defn}

We next proceed to define the notion of skew Schur polynomials.

\begin{defn}[Skew partitions and skew Schur polynomials]\label{def:sckew-schur}
Fix $w,m \in \N$. A vector
${\bm \lambda} = (\lambda_1, \lambda_2, \ldots, \lambda_m)$ with
$\lambda_i \in \Z$, $i \in [m]$, and $\lambda_1 \ge \lambda_2 
\ge \cdots \ge \lambda_m \ge 0$ is called an integer partition (or in short a partition), of length 
$\ell({\bm \lambda})=m$ and weight 
$|{\bm \lambda}|=\sum_{i=1}^m \lambda_i$.

For two partitions ${\bm \mu}$ and ${\bm \lambda}$ we write ${\bm \mu} \subseteq {\bm \lambda}$ if $\ell({\bm \mu}) \le \ell({\bm \lambda})$ and $\mu_s \le \lambda_s$ for all $s \in [\ell({\bm \mu})]$. In that case, the pair $({\bm \lambda}, {\bm \mu})$ is called a skew partition and is often denoted by ${\bm \lambda}/{\bm \mu}$. 

Given a skew partition ${\bm \lambda}/{\bm \mu}$ we define the skew Schur polynomial in $w$ variables as follows:
\begin{equation}\label{eq:skew-schur}
{\bm s}_{{\bm \lambda}/{\bm \mu}}(\cdot) := \det \left[({\sf h}_{\lambda_u - \mu_v - u+v}(\cdot))_{u,v=1}^{\ell({\bm \lambda})}\right],
\end{equation}
where $\{{\sf h}_{\cdot} (\cdot)\}$ are the complete homogeneous symmetric polynomials in $w$ variables as in Definition \ref{def:h}, and the partition ${\bm \mu}$ is extended to have length $\ell({\bm \lambda})$ by appending zeros. When ${\bm \mu}$ is an empty partition the polynomial ${\bm s}_{{\bm \lambda}/\emptyset}$ is called a Schur polynomial and is written as ${\bm s}_{{\bm \lambda}}$. Thus, for any given partition ${\bm \lambda}$, we have 
\begin{equation}\label{eq:schur-poly}
{\bm s}_{{\bm \lambda}} = \det \left[({\sf h}_{\lambda_u - u+v})_{u,v=1}^{\ell({\bm \lambda})}\right]. 
\end{equation}
\end{defn}

\begin{rem}
Skew Schur polynomials admit a combinatorial description in terms
of {\em skew semistandard Young tableaux}. 
The equivalence of these two definitions is due to the {\em first Jacobi-Trudi identity} (see \cite[Theorem 7.16.1]{Stan2}). Since we do not require the notion of Young tableaux elsewhere in this paper we have chosen \eqref{eq:skew-schur} as the definition of the skew Schur polynomial ${\bm s}_{{\bm \lambda}/{\bm \mu}}$. 
We refer the reader to \cite[Chapter 7.10]{Stan2} for a 
detailed overview of these matters.
\end{rem}

The following identity, known already to Jacobi (see \cite[Exercise 7.4]{Stan2} or \cite[Theorem 3.2]{Cornelius})
provides an efficient representation of complete homogeneous symmetric polynomials.
\begin{lem}\label{lem:h-alt-exp}
Let ${\sf r}, w \in \N$ and $(t_1, t_2, \ldots, t_w) \in \R^n$. If $\{t_s\}_{s=1}^w$ are pairwise distinct then 
\[
{\sf h}_{{\sf r}}(t_1, t_2, \ldots, t_w) = \sum_{j=1}^w \frac{t_j^{{\sf r}+w-1}}{\prod_{k \in [w]\setminus \{j\}} (t_j -t_k)}.
\]
\end{lem}

The next lemma, which is a key to the proof of Proposition \ref{prop:toep-minor}, states that a Toeplitz minor can be expressed in terms of skew Schur polynomials. 

\begin{lem}[{\cite[Theorem 2.1]{MM}}]\label{lem:toep-minor}
Consider the setup of Proposition \ref{prop:toep-minor}. Then 
\[
\det (\wh T_{\wh N}(z)[\COMP{\gX}; \COMP{\gY}]) = (-1)^{|{\bm \lambda}_0| + |{\bm \mu}_0|} \cdot \wh a_0(z)^{\wh N-n}  \cdot {\bm s}_{{\bm \lambda_0}/{\bm \mu}_0}(\wh \eta_1(z)^{-1}, \wh \eta_2(z)^{-1}, \ldots, \wh \eta_{{\sf g}_*}(z)^{-1}),
\]
where
\[
{\bm \lambda_0} := (\wh N -n +1- \gy_1, \wh N - n +2-\gy_2, \ldots, \wh N - \gy_n)  \text{ and }   {\bm \mu_0} := (\wh N - n +1- \gx_1, \wh N - n+ 2-\gx_2, \ldots, \wh N - \gx_n).
\]
\end{lem}

\begin{rem}
The reader may note that the expression for the Toeplitz minor stated in Lemma \ref{lem:toep-minor} is somewhat different than the one in \cite[Theorem 2.1]{MM}. This is due to the fact that \cite{MM} expresses the minor in terms of the roots of the polynomial $\ol{p}_z(\tau) := \tau^{{\sf g}_\star} \cdot \wh p_z(\tau)$, where $\wh p_z(\cdot)$ is as in \eqref{eq:wh-p}, while  $\{\wh \eta_j(z)\}_{j=1}^{{\sf g}_\star}$ are the roots of $\wh p_z(\cdot)=0$ which are indeed the reciprocals of those of $\ol{p}_z(\cdot)=0$. 
\end{rem}

\begin{rem}
  \label{rem-history}
It was first noted by Bump and Diaconis \cite{BD} that any Toeplitz minor can be expressed in terms of skew Schur polynomials.  They expressed such minors as an integral of certain functions involving skew Schur polynomials over the unitary group (see \cite[Theorem 3]{BD}). Later Alexandersson in \cite{Al} expressed minors of triangular Toeplitz matrices as certain skew Schur polynomials (see \cite[Proposition 3]{Al}). Maximenko and Moctezuma-Salazar \cite{MM} gave 
a different proof of such a representaion
for non-triangular banded Toeplitz matrices. Here we use the formulation of \cite{MM} which is the most convenient for our setup. 
\end{rem}

Before proving Proposition \ref{prop:toep-minor} we make the following elementary observation. 

\begin{lem}\label{lem:toep-minor-0}
Consider the setup of Proposition \ref{prop:toep-minor}. If $\gX, \gY \subset [\wh N]$ are such that either there exists an $i \in [n]$ for which $\gx_i < \gy_i$ or there exists an $i \in [n - {\sf g}_\star]$ such that $\gx_i \ge \gy_{i+{\sf g}_\star}$, then 
\[
\det (\wh T_{\wh N}(z)[\COMP{\gX}; \COMP{\gY}])=0. 
\]
\end{lem}

\begin{proof}
We note that $\wh T_{\wh N}(z)$ is an upper triangular Toeplitz matrix with the symbol $\wh p_z(\cdot)$ (see \eqref{eq:wh-p}). As $\{\wh \eta_j\}_{j=1}^{{\sf g}_\star}$ are the roots of $\wh p(1/\cdot)=0$, it follows that
\[
\wh T_{\wh N}(z) = \prod_{j=1}^{{\sf g}_\star} (J_{\wh N} - \wh \eta_j(z) I_{\wh N}).
\] 
Therefore, by the Cauchy-Binet theorem we have that
\begin{equation}\label{eq:minor-non-zero}
\det (\wh T_{\wh N}(z) [\COMP{\gX}; \COMP{\gY}])= \sum_{\ell=2}^{{\sf g}_\star} \sum_{\substack{Z_\ell \subset [\wh N]\\ |Z_\ell| =n}} \prod_{j=1}^{{\sf g}_\star} \det( (J_{\wh N} -\wh \eta_j(z) I_{\wh N})[ \COMP{Z}_j; \COMP{Z}_{j+1}]), 
\end{equation}
where we set $Z_1 = \gX$ and $Z_{{\sf g}_\star+1} = \gY$. For $\ell \in [{\sf g}_\star +1]$, we  index the elements of $Z_\ell$ as follows:
\[
Z_\ell := \{ z_{\ell,1} < z_{\ell,2} < \cdots < z_{\ell,n} \}.
\]
By \cite[Lemma A.3]{BPZ1} we note that
\[
\det( (J_{\wh N} -\wh \eta_\ell(z) I_{\wh N})[ \COMP{Z}_\ell; \COMP{Z}_{\ell+1}]) \ne 0
\]
if and only if 
\begin{equation}\label{eq:interlace}
z_{\ell+1, 1} \le z_{\ell,1} < z_{\ell+1,2} \le z_{\ell,2} < \cdots < z_{\ell+1, n} \le z_{\ell,n} \le n.
\end{equation}
Therefore, in \eqref{eq:minor-non-zero} we can restrict the sum over $\{Z_2, Z_3, \ldots, Z_{{\sf g}_\star}\}$ such that \eqref{eq:interlace} holds for all $\ell \in [{\sf g}_\star]$. On the other hand, if \eqref{eq:interlace} holds for all $\ell \in [{\sf g}_\star]$, as $Z_1=\gX$ and $Z_{{\sf g}_\star+1}=\gY$, we derive that for all $i \in [n]$
\begin{equation*}\label{eq:interlace-1}
\gy_i = z_{{\sf g}_\star+1,i} \le z_{{\sf g}_\star, i} \le \cdots \le z_{1,i} = \gx_i,
\end{equation*}
and for any $i \in [n - {\sf g}_\star]$
\begin{equation*}\label{eq:interlace-2}
\gx_i = z_{1,i} < z_{2,i+1} < z_{3,i+2} < \cdots < z_{{\sf g}_\star+1, i+{\sf g}_\star} = \gy_{i+{\sf g}_\star}.
\end{equation*}
Therefore, if either of the above two conditions are violated for some $i \in [n]$, then there must be an $\ell \in [{\sf g}_\star]$ such that \eqref{eq:interlace} is violated. In that case, the sum $\{Z_2, Z_3, \ldots, Z_{{\sf g}_\star}\}$ in \eqref{eq:minor-non-zero} is an empty sum and hence $\det (\wh T_{\wh N}(z) [\COMP{\gX}; \COMP{\gY}])=0$. 
\end{proof} 

We are now ready to prove Proposition \ref{prop:toep-minor}. 

\begin{proof}[Proof of Proposition \ref{prop:toep-minor}]
 From \eqref{eq:skew-schur} we see that
\begin{equation}\label{eq:skew-schur-1}
{\bm s}_{{\bm \lambda_0}/{\bm \mu}_0} (\cdot)= \det \left[\left( {\sf h}_{\gx_v - \gy_u}(\cdot)\right)_{u,v=1}^n \right], 
\end{equation}
where ${\bm \lambda_0}$ and ${\bm \mu_0}$ are as in Lemma \ref{lem:toep-minor}. Set $\ul{\wh \xi} := (\wh \xi_1, \wh \xi_2, \ldots, \wh \xi_{{\sf g}_\star})$, where $\wh \xi_i := \wh \eta_i^{-1}(z)$ for $i \in [n]$. Using
\eqref{eq:skew-schur-1} in the first equality and 
Lemma \ref{lem:h-alt-exp} in the second, 
we have that 
\begin{multline}\label{eq:skew-schur-2}
{\bm s}_{{\bm \lambda_0}/{\bm \mu}_0} (\ul{\wh \xi}) = \sum_{\uppi} \sgn(\uppi) \prod_{v=1}^n {\sf h}_{\gx_v - \gy_{\uppi(v)}}(\ul{\wh \xi}) \\
=\sum_{j_1, j_2, \ldots, j_n=1}^{{\sf g}_\star} \prod_{u=1}^n \frac{\wh \xi_{j_u}^{{\sf g}_*-1}}{\prod_{k_u \in [{\sf g}_\star]\setminus \{j_u\}} (\wh \xi_{j_u} - \wh \xi_{k_u})} \cdot \left[\sum_{\uppi} \sgn(\uppi) \cdot \prod_{v=1}^n \wh \xi_{j_v}^{\gx_v - \gy_{\uppi(v)}} \cdot \prod_{v=1}^n {\bf 1}_{\{\gx_v \ge \gy_{\uppi(v)}\}}\right],
\end{multline}
where the sum in the first step and the innermost sum in the second step are over all permutations $\uppi$ on $[n]$, and in the last step we also used the fact that ${\sf h}_{\sf r}(\cdot)=0$ for ${\sf r} <0$. 

We claim that there are at most ${\sf g}_\star^n$ many permutations $\uppi$ such that the summand in the innermost sum in \eqref{eq:skew-schur-2} is nonzero. That is, there are at most ${\sf g}_\star^n$ many permutations $\uppi$ such that
\begin{equation}\label{eq:uppi}
\gx_v \ge \gy_{\uppi(v)}, \quad \text{ for all } v \in [n]. 
\end{equation}
Equipped with this claim, upon crudely bounding each term appearing in the sums in \eqref{eq:skew-schur-2} by its maximum, it now follows that 
\begin{equation}\label{eq:skew-schur-3}
\left| {\bm s}_{{\bm \lambda_0}/{\bm \mu}_0} (\ul{\wh \xi}) \right| \le {\sf g}_\star^{2n} \left(\max_{j=1}^{{\sf g}_\star} \left\{ \prod_{k \in [{\sf g}_\star]\setminus \{j\}} \left|\frac{1}{1 - (\wh \xi_k /\wh \xi_j)}\right| \right\}\right)^n \cdot \left(\max_{j=1}^{{\sf g}_\star} |\wh \xi_j|\right)^{\bar \gx - \bar \gy} \le  {\sf g}_\star^{2n} \cdot c_1^{-{\sf g}_\star n} \cdot \left(\max_{j=1}^{{\sf g}_\star} |\wh \xi_j|\right)^{\bar \gx - \bar \gy}, 
\end{equation}
where the last inequality is due to \eqref{eq:minor-c1}. 
Now, using Lemmas \ref{lem:toep-minor} and \ref{lem:toep-minor-0}, and noting that
\begin{equation}\label{eq:prod-root}
\wh a_0(z) = (-1)^{{\sf g}_\star} \prod_{j=1}^{{\sf g}_\star} \wh \eta_j(z), 
\end{equation}
we obtain \eqref{eq:toep-minor} 
from \eqref{eq:skew-schur-3}.

Thus, to finish the proof it remains to prove the claim regarding the number of permutations $\uppi$ satisfying \eqref{eq:uppi}. To do that, by Lemma \ref{lem:toep-minor-0}, without loss of generality, we may assume that $\gx_v < \gy_{v+{\sf g}_\star}$ for all $v \in [n - {\sf g}_\star]$. 

Now we bound the number of choices of $\uppi$ satisfying \eqref{eq:uppi} as follows: 

\begin{itemize}
\item As $\gx_1 < \gy_{1+{\sf g}_\star}$, $\gy_1 < \gy_2< \cdots < \gy_n$,  and $\gx_1 \ge \gy_{\uppi(1)}$ we find that $\uppi(1) \in [{\sf g}_\star]$. That is, the number of choices for $\uppi(1)$ is ${\sf g}_\star$.

\item Having chosen $\uppi(1)$ we now choose $\uppi(2)$. Using the same argument as above we note that $\uppi(2) \in [{\sf g}_\star+1]$. As $\uppi(1) \in [{\sf g}_\star] \subset [{\sf g}_\star+1]$ the number of choices for $\uppi(2)$ is
  again at most ${\sf g}_\star$. 

\item Continuing from the above, we find that for any given $v \in [n - {\sf g}_\star]$ we must have $\uppi(v) \in [v+ {\sf g}_\star-1]$. On the other hand $\uppi(1), \uppi(2), \ldots, \uppi(v-1) \in [v+{\sf g}_\star -2]$. Thus the number of choices for $\uppi(v)$ is at most ${\sf g}_\star$.

\item Finally $\uppi(n-{\sf g}_\star+1), \uppi(n-{\sf g}_\star+2), \ldots, \uppi(n)$ can be chosen in at most ${\sf g}_\star ! (\le {\sf g}_\star^{{\sf g}_\star})$ ways.
\end{itemize}
This proves the claim and hence we have the desired bound. 
\end{proof}

We end this section with the proof of Lemma \ref{lem:wh-gD-3}. This is a direct consequence of Proposition \ref{prop:toep-minor}.

\begin{proof}[Proof of Lemma \ref{lem:wh-gD-3}]
We recall from \eqref{eq:cP-z} that $\cP_z$ is an upper triangular Toeplitz matrix of dimension $\wh N_+$ such that the roots of its symbols are $\{-\eta_j(z)\}_{j=\wh m_1+1}^{\wh m_1 +{\sf g}_0}$. Note that ${\sf g}(p)={\sf g}_0$ (see \eqref{eq:gpintro} for a definition of ${\sf g}(\cdot)$) implies that there exists $\wt \eta(z) \in \C$ such that
\begin{equation}\label{eq:ze-t-o-wt-ze}
\eta_{\wh m_1+j}(z) = \wt \eta(z) \cdot e^{2 \pi \sqrt{-1} j/{\sf g}_0}, \qquad j \in [{\sf g}_0].
\end{equation}
Upon recalling the definition of the tube $\cT^d$ we see that 
\begin{equation*}
1-\frac{\vep_0}{3}  \le \inf_{z \in \cT^d}   |\eta_{\wh m_1+{\sf g}_0}(z)| \le  \sup_{z \in \cT^d}   |\eta_{\wh m_1+1}(z)| \le 1+\frac{\vep_0}{3}.
\end{equation*}
This further implies that for any $z \in \cT^d$ we have that $\wt \eta(z) \ne 0$. Therefore, we indeed have that
\begin{equation}\label{eq:ratio-away-from-1}
  \min_{i \ne j \in [{\sf g}_0]}\left| \frac{\eta_{\wh m_1+i}(z)}{\eta_{\wh m_1+j}(z)} -1\right| \ge {\sf g}_0^{-1}. 
\end{equation}
Recall also, see e.g. \eqref{eq:mod-eq}, that 
\begin{equation*}
|\eta_{\wh m_1+1}(z)| = |\eta_{\wh m_1+2}(z)| = \cdots = |\eta_{\wh m_1+{\sf g}_0}(z)|.
\end{equation*}  
 As $|X_{\wh m_1+1}| = |Y_1|=k+N_+$ the bound \eqref{eq:wh-gD-bd-3}, as well as the fact that that the \abbr{LHS} of \eqref{eq:wh-gD-bd-3} equals zero when \eqref{eq:xy-rel} is violated, is now immediate from Proposition \ref{prop:toep-minor}.
 This completes the proof. 
\end{proof}

\section{Proofs of the separation Theorems \ref{thm:no-outlier} and \ref{thm:sep-spec-curve}}\label{sec:thm-sep-spec-curve}

In this section, we complete the proofs of Theorems \ref{thm:no-outlier} and \ref{thm:sep-spec-curve}, based on 
Lemma \ref{lem:comb-bound-tube-1}.
The proof is split into four parts:~first, upon employing Lemma \ref{lem:comb-bound-tube-1}, we derive correct order of magnitudes (of moments) of 
${\det}_k(z)$, per fixed $z$ and $k$ corresponding to non-dominant terms,
in the regions of interest. Next, in order to control 
$\sup_z \det_k(z)$, we  
derive bounds on the moments on the supremum of the derivative of $\det_k(z)$. 
Combining these steps
and
a covering argument, we get the desired bounds for $\sup_z \det_k(z)$ 
for non-dominant terms in the expansion \eqref{eq;spec-rad-new}. The third part of the proof derives a uniform lower bound on the dominant term, using
Assumption \ref{assump:anticonc}. In the last step,
we combine the above ingredients with
geometric information on the forbidden regions
to deduce Theorems \ref{thm:no-outlier} and \ref{thm:sep-spec-curve}.

\subsection{Step 1: moment 
bounds on the non-dominant terms (per fixed $z$)}\label{sec:high-mom}
In this section we derive bounds on the moments of ${\det}_k(z)$.
Recall \eqref{eq;spec-rad-new},
\eqref{eq:hatP-k}, \eqref{eq:wh-m1}, and \eqref{eq:whm-2}, and note that $\wh m_1$ and $\wh m_2$ (which depend on $z$)
are constant on the tubes $\cT^{d, (1)}$ and $\cT^{d, (2)}$ with fixed $d$. 
We first consider the case of {\em small} $k$. 
\begin{lem}\label{lem:sec-mom-small-k}
Let Assumption \ref{assump:mom} hold.
Fix $ \vep_0' ,\vep_0>0$ such that $\vep_0' / \vep_0$ is
sufficiently small. Fix $d \ge 0$ and $s \in \{1,2\}$. Then, for all large $N$,
\begin{equation}\label{eq:exp-decay-mom}
\max_{k < \wh m_2 - N_+} \sup_{z \in \cT^{d, (s)}_{\vep_0', \vep_0}} \E\left[ \left|\wh{\det}_k (z)\right|^2 \right] \le \left(1 -\frac{\vep_0}{8}\right)^N.
\end{equation}
\end{lem}


\begin{proof}
Fix $s \in \{1,2\}$. Since the entries of $Q$ are independent with zero mean and bounded variance, we get  that for any $\mathbb{X}_*, \mathbb{Y}_*, \mathbb{X}', \mathbb{Y}' \subset [N]$
\begin{equation}\label{eq:det-non-zero}
\E \left[\det(Q[\mathbb{X}_*; \mathbb{Y}_*] )\cdot \overline{\det(Q[\mathbb{X}'; \mathbb{Y}'] )}\right] \leq  \mathfrak{C}_1  \left\{\begin{array}{ll}
 k! & \mbox{if } \mathbb{X}_*= \mathbb{X}', \, \mathbb{Y}_*=\mathbb{Y}', \text{ and } |\mathbb{X}_*|= |\mathbb{Y}_*| =k,\\
0 & \mbox{ otherwise}.
\end{array}
\right.
\end{equation}
Therefore, upon recalling \eqref{eq:det-decompose-1}, \eqref{eq:hatP-k}-\eqref{eq:wh-gD-dfn}, and \eqref{eq:det-k-sec-mom-0} we find that 
\begin{equation}\label{eq:sec-mom-expand}
\E\left[ \left|\wh{\det}_k(z)\right|^2 \right] \leq \mathfrak{C}_1 N^{2\gamma (\wh m_2 - N_+ - k)} \cdot k! \sum_{\substack{X, Y \subset [N]\\ |X|=|Y|=k}} |\wh \gD(X,Y,z)|^2. 
\end{equation}
Since there are at most $\binom{N}{k}^2$ choices for the sets $X, Y \subset [N]$ such that $|X|=|Y|=k$, the claimed upper bound now follows from Lemma \ref{lem:comb-bound-tube-1}(ii). This completes the proof. 
\end{proof}

Next we derive bound on the second moment of ${\det}_k(z)$ for $k\ge \wh m_2 - N_+$. 

\begin{lem}\label{lem:sec-mom-large-k}
In the setup of Lemma \ref{lem:sec-mom-small-k}, we have the following bounds, for all large $N$.
\begin{enumerate}

\item[(i)]  Fix $\gamma'$ with $1 < \gamma' < \gamma$. Then for any $k \in \N$ such that $ \wh m_2 - N_+ \le k \le N$ we have
\[
 \sup_{z \in \wh \cT^{d, (s)}_{\gamma',\vep_0}} \E\left[ \left|\wh{\det}_k(z)\right|^2 \right] =O\left( N^{-(\gamma - \gamma') \cdot (k +N_+ - \wh m_2)}\right), \quad s=1,2.
\]

\item[(ii)]For  $k_0= \wh m_2 - N_+$ we have 
\[
 \sup_{z \in  \cT^{d, (s)}_{\vep_0',\vep_0}} \E\left[ \left|\wh{\det}_{k_0}(z)\right|^2 \right]  = O(1), \quad s=1,2.
\]

\item[(iii)] For any $k \in \N$ such that $ \wh m_2 - N_+ \le k \le N$ we have
\[
 \sup_{z \in  \cT^{d, (2)}_{\vep_0',\vep_0}} \E\left[ \left|\wh{\det}_k(z)\right|^2 \right] =O\left( N^{-(\gamma - 1) \cdot (k +N_+ - \wh m_2)}\right).
\]

\end{enumerate}
\end{lem}

\begin{rem}
Lemma \ref{lem:sec-mom-large-k}(ii) shows that the supremum of the second moment of $\wh \det_k(z)$, for $k=k_0$,  is well controlled in the tubes $\cT^d$ that are of small but fixed width,
 whereas Lemma \ref{lem:sec-mom-large-k}(i) shows that the same can be said for any $k >k_0$ only in the tubes $\wh \cT^d$ that are of of vanishing width. In part (iii) we see that one has such a control for $z \in \cT^{d, (2)}$ but, as will be seen during the course of the proof, those bounds fail for $z \in \cT^{d, (1)}\setminus \wh \cT^{d, (1)}$. 
\end{rem}


\begin{proof}[Proof of Lemma \ref{lem:sec-mom-large-k}]
We start with the proof of part (i), beginning with $s=1$. Denote
\begin{equation}\label{eq:gd-q}
\gd_q := \binom{q(\wh m_2+1) + \wh m_2 (k+N_++1) }{\wh m_2 (k+N_++1)}  \left(1-\frac{\vep_0}{2}\right)^{q}. 
\end{equation}
As
\begin{equation}\label{eq:tube-1-bd-2n}
|\eta_{\wh m_1+1}(z)| \ge 1 - (\gamma'-1)\frac{\log N}{N}, \qquad \mbox{ for } z \in \wh \cT^{d, (1)}_{\gamma', \vep_0}, 
\end{equation}
by Lemma \ref{lem:comb-bound-tube-1}(i) we have that
\begin{equation}\label{eq:sec-mom-high-k-1}
\sum_{\substack{X, Y \subset [N]\\ |X|=|Y|=k}} |\wh \gD(X,Y,z)|^2 \le 2 \wh  C_{\ref{lem:comb-bound-tube-1}}^{2\wt m (k+N_+)} \cdot N^{2(\gamma'-1) G(k,z)} \left[\sum_{\substack{X, Y \subset [N]\\ |X|=|Y|=k}} \sum_{q, q'  \ge 0} \gd_q \gd_{q'}  \cdot \gI(X, Y, q) \right].
\end{equation}
We claim that, for any $q \ge 0$,
\begin{equation}\label{eq:XY-q-bd}
\left| \{X, Y \subset [N]: |X|=|Y|=k \text{ and } \gI(X, Y,q )=1\}\right| \le (2\wh m_2)! \binom{q+2 \wh m_2}{2 \wh m_2} \cdot \binom{N}{k+N_+-\wh m_2}^2. 
\end{equation}
Indeed, upon recalling \eqref{eq:gI-XY}, we note that for a given $q\ge 0$, 
\[
\gI(X, Y, q) = 1 \qquad \Longrightarrow \qquad \max\left\{  \max_{j =k+N_+ - \wh m_2+1}^{k} (N - x_{j}), \max_{j=1}^{\wh m_2 -N_+} y_j  \right\} \le q. 
\] 
Thus the number of choices of $\{x_j\}_{j =k+N_+ - \wh m_2+1}^{k}$ and $\{y_j\}_{j=1}^{\wh m_2 - N_+}$ is at most 
\[
q^{2(\wh m_2 -N_+)} \le q^{2\wh m_2} \le (2\wh m_2)! \binom{q+2 \wh m_2}{2 \wh m_2}. 
\]
Now choose the remaining  elements of $X$ and $Y$ to obtain \eqref{eq:XY-q-bd}.

Next applying \eqref{eq:comb-negbin} and \eqref{eq:comb-negbin-2} we have 
\[
\sum_{q' \ge 0} \gd_{q'} \le \left( \frac{4 (\wh m_2 +1)}{\vep_0}\right)^{\wh m_2 (k+N_++1) +1}.
\]
On the other hand, \eqref{eq:comb-negbin}, \eqref{eq:comb-negbin-2}, and \eqref{eq:XY-q-bd} yield that
\[
\left[\sum_{\substack{X, Y \subset [N]\\ |X|=|Y|=k}} \sum_{q \ge 0} \gd_q  \cdot \gI(X, Y, q) \right] \le (2 \wh m_2)! \cdot \left( \frac{4 (\wh m_2 +2)}{\vep_0}\right)^{\wh m_2 (k+N_++3) +1} \cdot N^{2(k+N_+ - \wh m_2)} k^{2\wh m_2 } (k!)^{-2}. 
\]
Plugging in the last two bounds in \eqref{eq:sec-mom-high-k-1} and using that $G(k,z) \le k+N_+ - \wh m_2$, we deduce that
\[
\sum_{\substack{X, Y \subset [N]\\ |X|=|Y|=k}} |\wh \gD(X,Y,z)|^2 \le \ol{C}_1^{\wt m (k+N_+)} \cdot N^{2\gamma' (k+N_+ -\wh m_2)}\cdot k^{2\wh m_2 } \cdot (k!)^{-2},
\]
for some constant $\ol{C}_1 < \infty$. Upon using \eqref{eq:sec-mom-expand} part(i) of  the lemma follows for $s=1$. 

Turning to $s =2$ we observe that, as
$G(k,\cdot) \equiv 0$ on $\cT^{d, (2)}_{\vep_0', \vep_0}$ and $\cT_{\vep_0', \vep_0}^{d, (2)} \supset \wh \cT_{\gamma', \vep_0}^{d, (2)}$ for all large $N$, the upper bound \eqref{eq:sec-mom-high-k-1} continues to hold in  with $N^{2(\gamma'-1) G(k,z)}$ replaced by one. Thus, repeating the rest of the arguments we derive part (i) for $s=2$, as well as part (iii). The proof of part (ii) is exactly the same, where we again note that $G(k_0, \cdot) \equiv 0$ on $\cT^{d, (1)}_{\vep_0',\vep_0}$ and use that $k_0 +N_+ - \wh m_2=0$. This completes the proof of the lemma.
\end{proof}

While applying Theorem \ref{thm:sep-spec-curve} to derive the localization of the eigenvectors we will choose $1 < \gamma' < \gamma$ such that $(\gamma - \gamma')$ is (fixed but) arbitrarily close to zero. For such choice of parameters, if $k+N_+ - \wh m_2$ is {\em small} then Lemma \ref{lem:sec-mom-large-k} does not yield a sufficiently strong probability bound to carry out the covering argument. To overcome this caveat we control {\em high moments} of ${\det}_k(z)$ for such choices of $k$, which when applied together with Markov's inequality produce desired probability bounds. 

\begin{lem}\label{lem:high-mom-small-k}
Consider the same setup as in Lemma \ref{lem:sec-mom-large-k}. 
There exists  a constant $\wh C_{\ref{lem:high-mom-small-k}}< \infty$ (depending on $p(\cdot)$ and $\vep_0$  only) so that,
for all $h, k, K_0 \in \N$  such that $\wh m_2 - N_+ \le k \le K_0 -N_+$, we have
\begin{equation}\label{eq:high-mom-t1}
 \sup_{z \in \wh \cT^{d, (1)}_{\gamma',\vep_0}}\E\left[\left|\wh{\det}_k(z)\right|^{2h} \right] \le C_{\ref{lem:high-mom-small-k}}    N^{-2h (\gamma-\gamma')\cdot (k+N_+ - \wh m_2)}  
\end{equation}
and
\begin{equation}\label{eq:high-mom-t2}
 \sup_{z \in  \cT^{d, (2)}_{\vep_0',\vep_0}}\E\left[\left|\wh{\det}_k(z)\right|^{2h} \right] \le C_{\ref{lem:high-mom-small-k}}   N^{-2h (\gamma-1)\cdot (k+N_+ - \wh m_2)},
\end{equation}
where
\[
C_{\ref{lem:high-mom-small-k}}=C_{\ref{lem:high-mom-small-k}}(h,K_0,\wt m, \wh m_2) := \wh C_{\ref{lem:high-mom-small-k}}^{\wt m K_0 h} \cdot (K_0h)^{8 K_0h}\cdot  \gC_{2K_0h}.
\]
\end{lem}

\begin{proof}
We start with the proof of \eqref{eq:high-mom-t1}. Note that it suffices to show that for any $z \in \wh \cT^{d, (1)}_{\gamma',\vep_0}$,
\begin{equation}\label{eq:high-mom-gT}
\mathfrak{T}_k(z) := N^{2\gamma h ( k+N_+ -\wh m_2)}  \E\left[\left|\wh{\det}_k(z)\right|^{2h} \right] \le  \wh C_{\ref{lem:high-mom-small-k}}^{\wt m K_0 h} \cdot (K_0h)^{8 K_0h}\cdot  \gC_{2K_0h} \cdot N^{2h\gamma' (k+N_+ - \wh m_2)}. 
\end{equation}
\
Turning to prove \eqref{eq:high-mom-gT}, using \eqref{eq:det-decompose-1}, \eqref{eq:fancyK-1}, and  \eqref{eq:hatP-k}-\eqref{eq:wh-gD-dfn} we find that
\begin{multline}\label{eq:high-mom-1}
\gT_k(z) = \E\Bigg[\sum_{{\bm X}} \sum_{\bm Y} \prod_{i=1}^{2h} (-1)^{\sgn (\sigma_{X^{(i)}}) \sgn(\sigma_{Y^{(i)}})} \cdot \prod_{i=1}^h \wh \gD(X^{(i)}, Y^{(i)}, z) \cdot \prod_{i=1}^h \det (Q[X^{(i)}; Y^{(i)}]) \\
\cdot \prod_{i=h+1}^{2h} \ol{\wh \gD(X^{(i)}, Y^{(i)}, z)} \cdot \prod_{i=h+1}^{2h} \ol{\det (Q[X^{(i)}; Y^{(i)}])}\Bigg],
\end{multline}
where the sum is taken over 
\[
{\bm X} := \{X^{(1)}, X^{(2)}, \ldots, X^{(2h)}\} \qquad \text{ and } \qquad {\bm Y}  := \{Y^{(1)}, Y^{(2)}, \ldots, Y^{(2h)}\},
\] 
and for $i \in [2h]$,
\[
X^{(i)} := \{x^{(i)}_1 < x^{(i)}_2 < \cdots < x^{(i)}_k\} \subset [N] \qquad \text{ and } \qquad Y^{(i)} := \{y^{(i)}_1 < y^{(i)}_2 < \cdots < y^{(i)}_k\} \subset [N]. 
\] 

Associate to ${\bm X}$  the partition $\mathfrak{P}_{{\bm X}}$ determined by the equivalence relation
 $x^{(i')}_{j'}\sim x^{(i'')}_{j''}$  iff $x^{(i')}_{j'} = x^{(i'')}_{j''}$, for some $i' \ne i'' \in [2h]$ and $j', j'' \in [k]$. 
Similarly, associate to ${\bm Y}$ the partition $\gP_{{\bm Y}}$.

As the entries of $Q$ are independent and are of zero mean it is straightforward to observe that for any $({\bm X}, {\bm Y})$ such that either $\gP_{{\bm X}}$ or $\gP_{{\bm Y}}$ has an equivalence class of size one, the expectation of the summand on the \abbr{RHS} of \eqref{eq:high-mom-1} equals zero. 
Hence, to compute a bound on the \abbr{RHS} of \eqref{eq:high-mom-1} we only need to consider partitions $\gP_{{\bm X}}$ and $\gP_{{\bm Y}}$ such that each of their equivalence classes has size at least two. For brevity we term them {\em pair partitions}. Fix one such pair of partitions $(\gP, \wt \gP)$ and a pair $({\bm X}, {\bm Y})$ such that $\gP_{{\bm X}}= \gP$ and $\gP_{{\bm Y}}= \wt \gP$. 

For any $i \in [2h]$ the determinant $\det(Q[X^{(i)}; Y^{(i)}])$ is a linear combination of $k!$ terms each of which are products of $k$ independent entries of the sub matrix $Q[X^{(i)}; Y^{(i)}]$. Therefore denoting 
\begin{equation*}
\mathfrak{Q} := \E\left[ \left|\prod_{i=1}^{h} \det (Q[X^{(i)}; Y^{(i)}]) \cdot \ol{\prod_{i=h+1}^{2h} \det (Q[X^{(i)}; Y^{(i)}])}\right|\right],
\end{equation*}
by the triangle inequality, we note that $\mathfrak{Q}$ is a sum of at most $(k!)^{2h}$ terms of the form
\begin{equation}\label{eq:pre-holder-1}
\E\left[ \left|\prod_{w=1}^{b} Q_{u_w,v_w}^{c_w} \cdot \ol{Q_{u_w,v_w}^{c_w'}}\right|\right],
\end{equation}
for some $b \in \N$ and collections of positive integers $\{c_w\}_{w=1}^b$ and $\{c'_w\}_{w=1}^b$ such that 
\[
\sum_{w=1}^b c_w = \sum_{w=1}^b c'_w = kh,
\]
where $Q_{u_w, v_w}$ denotes the $(u_w, v_w)$-th entry of $Q$ and the collection $\{(u_w, v_w)\}_{w=1}^b$ is pairwise disjoint. Thus $\{Q_{u_w,v_w}\}_{w=1}^b$ are  jointly independent. 
Upon using H\"{o}lder's inequality and Assumption \ref{assump:mom}(ii)  it follows that \eqref{eq:pre-holder-1} is bounded above by 
\[
\prod_{w=1}^{b} \left(\E \left[|Q_{u_w,v_w}|^{2kh}\right]\right)^{\frac{c_w+c'_w}{2kh}} \le  \gC_{2kh} \le \gC_{2K_0h},
\]
where in the last step we use that $k \le K_0$ and  that $\gC_w$ is increasing in $w$. This, in turn implies that $\mathfrak{Q} \le K_0^{2K_0 h} \cdot \gC_{2K_0h}$. Plugging this bound in \eqref{eq:high-mom-1}, and applying Lemma \ref{lem:comb-bound-tube-1}(i), \eqref{eq:tube-1-bd-2n}, and that $G(k,z) \le k+N_+ - \wh m_2$, 
we now deduce that 
\begin{multline}\label{eq:high-mom-2}
\gT_k (z) \le 2 \wh  C_{\ref{lem:comb-bound-tube-1}}^{2h\wt m (k+N_+)} \cdot N^{2h (\gamma'-1) \cdot (k+N_+ -\wh m_2)} \cdot K_0^{2K_0 h} \cdot \gC_{2K_0h} \cdot \\
 \left\{\sum_{\gP, \wt \gP} \left[\sum_{{\bm q}_\#}  \sum_{\bm Y: \gP_{\bm Y} = \wt \gP}\sum_{\bm X: \gP_{\bm X} = \gP} \prod_{i=1}^{2h} \gd_{q^{(i)}} \cdot \gI(X^{(i)}, Y^{(i)}, q^{(i)})\right]\right\},
\end{multline}
where ${\bm q}_\# := (q^{(1)}, q^{(2)}, \ldots, q^{(2h)})$, and the outer sum is over all pair partitions $\gP$ and $\wt \gP$. 
To evaluate the sum in \eqref{eq:high-mom-2} we will compute the sum in a specific order. To execute this step we need a few definitions. 

We split the elements of ${\bm X}$ and ${\bm Y}$ into generating and non-generating elements as follows:~For any $i \in [2h]$ we term the collection of elements $\{x^{(i)}_j\}_{j=k +N_+ -\wh m_2+1}^k$ and $\{y^{(i)}_j\}_{j=1}^{\wh m_2 - N_+}$ {\em non-generating elements}. The rest of the elements of ${\bm X}$ and ${\bm Y}$ are termed {\em generating elements}. Let $(\gp_1, \gp_2, \ldots, \gp_b)$, for some $b >0$, be the equivalence classes of $\gP_{\bm X}$, the partition associated with ${\bm X}$. For $w \in [b]$, we say that $\gp_w$ is non-generating if it contains {\em at least one  non-generating} element of ${\bm X}$. Otherwise $\gp_w$ will be said to be generating. The same convention is adopted for ${\bm Y}$. 

We are now ready to derive the desired upper bound on \eqref{eq:high-mom-2}. We proceed as follows:

\begin{itemize}
\item Fix a couple of pair partitions $\gP$ and $\wt \gP$. Note that, as $k \le K_0$, the number of choices of $(\gP, \wt \gP)$ is bounded above by $(K_0 h)^{4 K_0 h}$.

\item We next evaluate the innermost sum in \eqref{eq:high-mom-2}, fixing ${\bm q}_\#$. Since $\gP_{\bm X} = \gP$ one only needs to sum over the possible (common) values of the equivalence classes of $\gP$. 

\item Let $(\gp_1, \gp_2, \ldots, \gp_b)$ be the equivalence classes of $\gP$ and without loss of generality assume that $(\gp_1, \gp_2, \ldots, \gp_{b_0})$ are generating for some $b_0 \le b$. Notice that the number of elements of ${\bm X}$ is $2kh$. If $\gP_{\bm X} = \gP$ for some ${\bm X}$ and $\gP$ is a pair partition then $b_0 \le (k+N_+ -\wh m_2) h$. 

\item As ${\bm X} \subset[N]$ the total number of possible values of $\{\gp_w\}_{w=1}^{b_0}$ is bounded above by $N^{b_0} \le N^{(k+N_+ - \wh m_2) h}$. 

\item On the other hand, upon recalling \eqref{eq:gI-XY} we find that 
\[
\prod_{i=1}^{2h} \gI(X^{(i)}, Y^{(i)}, q^{(i)}) =1 \qquad \Longrightarrow \qquad \gp_w \ge N - \ol{\bm q}_\#, \quad w \in [b]\setminus [b_0+1],
\]
where $\ol{\bm q}_\# := \sum_{i=1}^{2h} q^{(i)}$. Thus, the number of choices for the of non-generating equivalences classes is bounded above by $\ol{\bm q}_\#^b \le \ol{\bm q}_\#^{K_0h}$. 

\item We repeat the same idea as above to compute the second sum in \eqref{eq:high-mom-2}.

\end{itemize}

Putting all the above pieces together we now obtain from \eqref{eq:high-mom-2} that
\begin{equation}\label{eq:high-mom-3}
\gT_k(z) \le 2 \wh  C_{\ref{lem:comb-bound-tube-1}}^{2h\wt m K_0} \cdot (2K_0h)^{8K_0 h} \cdot \gC_{2K_0h} \cdot N^{2(k+N_+ - \wh m_2) \gamma' h}\left[\sum_{{\bm q}_\#}  \ol{\bm q}_\#^{K_0h}\prod_{i=1}^{2h}  \cdot \gd_{q^{(i)}} \right].
\end{equation}
Applying \eqref{eq:comb-negbin} and \eqref{eq:comb-negbin-2} one can note that the sum over ${\bm q}_\#$ in \eqref{eq:high-mom-3} is bounded by $\wt C^{\wt m K_0 h}$, for some constant $\wt C < \infty$, depending only on $\wh m_2$ and $\vep_0$. Together, this yields \eqref{eq:high-mom-gT} and completes the proof of
 \eqref{eq:high-mom-t1}.

To prove \eqref{eq:high-mom-t2} we simply note that, again by Lemma \ref{lem:comb-bound-tube-1}, \eqref{eq:high-mom-2} continues to hold for any $z \in \cT^{d, (2)}_{\vep_0',\vep_0}$ with $N^{2h (\gamma'-1) \cdot (k+N_+ -\wh m_2)}$ replaced by one. Therefore, repeating the rest of the arguments we obtain  \eqref{eq:high-mom-t2} and complete  the proof of the lemma.
\end{proof}

\subsection{Step 2: uniform upper bounds on non-dominant terms}\label{sec:der-bd}
In this section we derive uniform upper bounds on the non-dominant terms in the expansion \eqref{eq;spec-rad-new}. The following is the main result of this section. 

\begin{thm}\label{thm:prob-non-dom-small} Let Assumption \ref{assump:mom} hold.
Fix parameters $\gamma', \vep_0, \vep_0' , \wt \vep_0>0$ such that $(\gamma -1) \vee 1 < \gamma' < \gamma$ and $\vep_0' / \vep_0$ is sufficiently small. Let $d \ge 0$ . Then,
for all large $N$,
\begin{equation}\label{eq:prob-non-dom-small1}
\max_{s\in \{1,2\}} \prob\left[\sup_{z \in \wh \cT^{d, (s)}_{\gamma',\vep_0}} \left|\sum\limits_{k \ne \wh m_2 - N_+}\wh{\det}_k(z)\right| \ge N^{-(\gamma- \gamma')/4} \right] \le 1/N.
\end{equation}
 Furthermore, for all large $N$,
\begin{equation}\label{eq:prob-non-dom-small2}
 \prob\left[\sup_{z \in  \cT^{d, (2)}_{\vep_0',\vep_0}} \left|\sum\limits_{k \ne \wh m_2 - N_+}\wh{\det}_k(z)\right| \ge N^{-(\gamma- \gamma')/4} \right] \le 1/N. 
\end{equation}
\end{thm}

The difference between \eqref{eq:prob-non-dom-small1} and \eqref{eq:prob-non-dom-small2} is that the former yields bounds for tubes with diminishing width (in $N$), while the latter provides bounds for certain tubes of fixed width. The bound \eqref{eq:prob-non-dom-small2} will be used in the proof of Theorem \ref{thm:no-outlier}. 

To prove Theorem \ref{thm:prob-non-dom-small} we will use the uniform bounds on the moments of $\wh \det_k(z)$ that were derived in Section \ref{sec:high-mom}. We will also need a bound on the second moment of the supremum of the derivatives of $\wh{\det}_k(z)$. 

\begin{lem}\label{lem:der-bd}
Consider the same setup as in Theorem \ref{thm:prob-non-dom-small}. Then there exist constants $0 < c_{\ref{lem:der-bd}}, C_{\ref{lem:der-bd}} <\infty$, depending only on $ \gamma'$ and $\wt \vep_0$, such that for any $d \ge 0$ and $k \in  [N]$ we have that
\begin{equation}\label{eq:lem-der-ubd2}
 \max_{s \in \{1,2\}} \E\left[\sup_{z \in \left(\wh \cT^{d, (s)}_{\gamma', \vep_0}\setminus \cB_2^{\wt \vep_0}\right)^{c_{\ref{lem:der-bd}} \log N/N}} \left|\frac{d}{dz} \wh{\det}_k(z)\right|^2 \right] \le C_{\ref{lem:der-bd}} \left(\frac{N}{\log N}\right)^3 
\end{equation}
and
\begin{equation}\label{eq:lem-der-ubd}
 \E\left[\sup_{z \in \left( \cT^{d, (2)}_{\vep_0', \vep_0}\setminus \cB_2^{\wt \vep_0}\right)^{c_{\ref{lem:der-bd}} \log N/N}} \left|\frac{d}{dz} \wh{\det}_k(z)\right|^2 \right] \le C_{\ref{lem:der-bd}} \left(\frac{N}{\log N}\right)^4.  
\end{equation}
\end{lem}


The proof of Theorem \ref{thm:prob-non-dom-small} will also require estimates on the {\em non-random} term in the expansion of $\det(P_z^\delta)$, as follows. 
Recall the notation $\wh d$, see \eqref{eq:wh-d}, and  that $P_z^\delta= P_{N, \gamma}^Q - z I_N$.

\begin{lem}\label{lem:det-0}
Consider the same setup as in Lemma \ref{lem:sec-mom-small-k}. We have the following bounds.
\begin{enumerate}
\item[(a)] If $\wh d >0 $ then, for all large $N$, 
\begin{equation}\label{eq:det-0-ubd}
\sup_{z \in \cT^{d, (s)}_{\vep_0',\vep_0}\setminus\cB_2^{\wt \vep_0}} |\wh{\det}_0(z)|  \le \left(1-\frac{\vep_0}{4}\right)^N, \quad s\in \{1,2\}.
\end{equation}
\item[(b)] Fix $d \ge 0$ and $\wt \vep>0$. Let $\wh d  \ge 0$. Then there exists a constant $c_\star\in (0,1)$ so that,
with $X_\star := [N]\setminus [N-\wh m_2 +N_+]$ and $Y_\star := [\wh m_2 - N_+]$,
\begin{equation}\label{eq:det-0-lubd}
 c_\star \le \inf_{z \in \cT^{d, (s)}_{\vep_0',\vep_0}\setminus\cB_2^{\wt \vep_0}} {|\wh \gD(X_\star, Y_\star, z)|}  \le  \sup_{z \in \cT^{d, (s)}_{\vep_0',\vep_0}\setminus\cB_2^{\wt \vep_0}} {|\wh \gD(X_\star, Y_\star, z)|}  \le c_\star^{-1}, \quad s\in \{1,2\}.
 \end{equation}
\end{enumerate}
\end{lem}

\begin{rem}\label{rem:det-0}
Note that if $\wh d=0$ then $X_\star = Y_\star = \emptyset$. Therefore, in that case $\wh \det_0(z) = \wh \gD(X_\star, Y_\star, z)$ (see \eqref{eq:det-decompose-1}-\eqref{eq:gD-def} and \eqref{eq:whdm}-\eqref{eq:wh-gD-dfn}). Thus, for $\wh d=0$ Lemma \ref{lem:det-0}(ii) provides a lower bound on the (dominant term) $\wh \det_0(z)$. 
\end{rem}

Equipped with Lemmas \ref{lem:der-bd} and \ref{lem:det-0} we now prove Theorem \ref{thm:prob-non-dom-small}.

\begin{proof}[Proof of Theorem \ref{thm:prob-non-dom-small} (assuming  Lemmas \ref{lem:der-bd} and \ref{lem:det-0}) ]
Fix $s \in \{1,2\}$ and
$K_0=K_0(\gamma, \gamma', \wt m) = \lceil 30/(\gamma - \gamma') \rceil +\wt m$.
We first show that for any $k$ such that $k+N_+ \ge K_0$,
\begin{equation}\label{eq:unif-bd-prob-1}
\max_{s\in \{1,2\}} \prob\left[\sup_{z \in \wh \cT^{d, (s)}_{\gamma', \vep_0} \setminus \cB_2^{\wt \vep_0}} \left|\wh{\det}_k(z)\right| \ge N^{-2} \right] \le 2/N^3.
\end{equation}
To see \eqref{eq:unif-bd-prob-1},  denote
\[
\cF(z) := \left\{ |\wh{\det}_k(z)| \le N^{-\frac{(k+N_+ -\wh m_2) (\gamma - \gamma')}{4}} \right\} \]
and 
\[
\cF_0  := \left\{\sup_{z \in \left(\wh \cT^{d, (s)}_{\gamma', \vep_0}\setminus \cB_2^{\wt \vep_0}\right)^{c_{\ref{lem:der-bd}} \log N/N}} \left|\frac{d}{dz} \wh{\det}_k(z)\right| \ge N^{3} \right\}.
\]
Applying Lemma \ref{lem:der-bd} (in particular \eqref{eq:lem-der-ubd2}) and Markov's inequality we obtain that 
\begin{equation}\label{eq:unif-non-dom-prob-bd-1}
\prob(\cF_0) \le N^{-3}.
\end{equation}
From our choice of $K_0$ and Lemma \ref{lem:sec-mom-large-k}(i) it follows that
\begin{equation}\label{eq:unif-non-dom-prob-bd-2}
\sup_{z \in \wh \cT^{d,(s)}_{\gamma', \vep_0}} \prob(\cF(z)^c) \le N^{-15}.
\end{equation}
Let $\cN$ be a net of $\wh \cT^{d, (s)}_{\gamma', \vep_0}$ of mesh size $N^{-6}$, which has 
cardinality at most $O(N^{12})$. Then, \eqref{eq:unif-non-dom-prob-bd-1}-\eqref{eq:unif-non-dom-prob-bd-2} yield that 
\begin{equation}\label{eq:unif-non-dom-prob-bd-3}
\prob(\cup_{z \in \cN} \cF(z)^c \cup \cF_0) \le 2/N^3.  
\end{equation}
On the other hand, we find that for all large $N$ and $z' \in \cN$ the ball $D(z', N^{-6})$ is contained in the $(c_{\ref{lem:der-bd}} \log N/N)$-blow up of $\wh \cT^{d, (s)}_{\gamma', \vep_0}\setminus \cB_2^{\wt \vep_0}$. Thus, using the triangle inequality and the first order Taylor series expansion, we see that on the event $\cap_{z' \in \cN} \cF(z') \cap \cF_0^c$, 
\begin{align}\label{eq:triangle-non-dom}
|\wh{\det}_k(z)|  & \le \sup_{z' \in \cN} |\wh{\det}_k(z')| + \dist(z, \cN) \cdot \sup_{z' \in \left(\wh \cT^{d, (s)}_{\gamma', \vep_0}\setminus \cB_2^{\wt \vep_0}\right)^{c_{\ref{lem:der-bd}} \log N/N}} \left|\frac{d}{dz}\wh{\det}_k(z')\right| \\
& \le N^{-\frac14 (k+N_+ -\wh m_2) \cdot (\gamma - \gamma')} + N^{-3} \le N^{-2}, \notag
\end{align}
for all large $N$, where in the last step follows from the facts that $k+N_+ \ge K_0$ and our choice of $K_0$. Therefore, from \eqref{eq:unif-non-dom-prob-bd-3} we now have \eqref{eq:unif-bd-prob-1}.

Next we claim that
\begin{equation}\label{eq:unif-bd-prob-2}
 \prob\left[\sup_{z \in \wh \cT^{d, (s)}_{\gamma', \vep_0}\setminus \cB_2^{\wt \vep_0}} \left|\wh{\det}_k(z)\right| \ge N^{-2} \right] \le 2/N^3,
\end{equation}
for any $k \in \N$ such that $k+N_+ < \wh m_2$. Indeed, by Lemma \ref{lem:sec-mom-small-k}, we see that \eqref{eq:unif-non-dom-prob-bd-2} continues to hold in this case. Thus, repeating the same proof as for \eqref{eq:unif-bd-prob-1}, we obtain \eqref{eq:unif-bd-prob-2}. 

Now we aim to show that  
\begin{equation}\label{eq:unif-bd-prob-3}
 \prob\left[\sup_{z \in \wh \cT^{d, (s)}_{\gamma', \vep_0}\setminus \cB_2^{\wt \vep_0}}\left|\wh{\det}_k(z)\right| \ge 2  N^{-\frac12 (\gamma - \gamma')} \right] \le 2/N^3,
\end{equation}
for any $k$ such that $ \wh m_2 +1 \le k +N_+ \le K_0$. 
To this end, we set $h_0=h_0(\gamma, \gamma', \wt m, {\sf g}_0)=\lceil 15 /(\gamma -\gamma')\rceil$. We apply Lemma \ref{lem:high-mom-small-k} with this $h_0$ and Markov's inequality to find that \eqref{eq:unif-non-dom-prob-bd-2} continues to hold for any $\wh m_2 - N_+ +1\le k \le K_0 - N_+$. Thus, repeating the same argument as above, yet again, we further observe that \eqref{eq:triangle-non-dom} also holds, except for the last step. Since $\gamma' > \gamma -1$ the last step there can be replaced by $2  N^{-\frac12 (\gamma - \gamma')}$. Therefore we have \eqref{eq:unif-bd-prob-3}.

Combining the probability estimates \eqref{eq:unif-bd-prob-1}, \eqref{eq:unif-bd-prob-2}, and \eqref{eq:unif-bd-prob-3}, and using a union bound we deduce that
\[
\prob\left[\sup_{z \in \wh \cT^{d, (s)}_{\gamma', \vep_0}\setminus \cB_2^{\wt \vep_0}} \left|\sum\limits_{k \notin \{\wh m_2 - N_+, 0\}}\wh{\det}_k(z)\right| \ge \frac12 N^{-(\gamma- \gamma')/4} \right] \le 3/N^2,
\]
for all large $N$. If $\wh d =0$ then, by \eqref{eq:ds}, $\wh m_2 - N_+=0$, and hence we have the desired probability bound. If $\wh d >0$ then upon using Lemma \ref{lem:det-0}(a) the proof of \eqref{eq:prob-non-dom-small1} completes. The proof of \eqref{eq:prob-non-dom-small2}, being similar, is omitted. 
\end{proof}

We now proceed to the proof of Lemma \ref{lem:der-bd}, which uses the bounds derived in Lemmas \ref{lem:sec-mom-small-k} and \ref{lem:sec-mom-large-k}, and  Cauchy's integral formula for smooth functions. Cauchy's integral formula allows us to control the second moment of the supremum of  a random analytic function (and its derivative) on a nice domain by controlling the supremum of the second moment of the same analytic function on a slightly enlarged domain. To carry out this scheme, we first show that the blow up of any tube $\cT^{d, (s)}$, $s =1,2$, away from the bad set $\cB_2^{\wt \vep_0}$,  is again contained in a union of the tubes with slightly modified parameters.

\begin{lem}\label{lem:tube-blow-up}
Fix $0 < \vep_0' < \vep_0$, and $\wt \vep_0>0$. We have the following geometric properties of the tubes.

\begin{enumerate}

\item[(i)] There exists a constant $C_{\ref{lem:tube-blow-up}} < \infty$ such that 
\begin{equation}\label{eq:dist-to-pS1}
\cT^{d}_{\vep_0',\vep_0} \subset (p(S^1))^{C_{\ref{lem:tube-blow-up}} \vep_0'}. 
\end{equation}
In particular,  $\cT^{d}_{\vep_0',\vep_0}$ is a bounded set.

\item[(ii)] Fix $\vep_0'' >0$ such that $\vep_0'' \le \vep_0'$. There exists some constant $\vep_{\ref{lem:tube-blow-up}}$, depending only on $\wt \vep_0$, such that, for any $\beta \in (0,1/2)$, and $\vep \le \vep_0''  \vep_{\ref{lem:tube-blow-up}}$, we have 
\begin{equation}\label{eq:bl-up-tube-1}
\left(\cT_{\vep_0', \vep_0}^{d, (1)}\setminus \cB_2^{\wt \vep_0} \right)^{\beta \varepsilon} \subset \left(\cT_{(1+\beta) \vep_0', (1-\beta)\vep_0}^{d, (1)} \cup \cT_{\beta\vep_0'', (1-\beta) \vep_0}^{d-{\sf g}_0, (2)}\right) \setminus \cB_2^{\wt \vep_0/2},
\end{equation}
for $d \ge 1$, while for $d \ge 0$ we have 
\begin{equation}\label{eq:bl-up-tube-2}
\left(\cT_{\vep_0', \vep_0}^{d, (2)}\setminus \cB_2^{\wt \vep_0}\right)^{\beta \varepsilon} \subset \left(\cT_{(1+\beta) \vep_0', (1-\beta)\vep_0}^{d, (2)} \cup \cT_{\beta \vep_0'', (1-\beta)\vep_0}^{d+{\sf g}_0, (1)}\right)\setminus \cB_2^{\wt \vep_0/2}.
\end{equation}
\end{enumerate}
\end{lem}

\begin{proof}
We begin with the proof of part (i). We observe that $z \in \cT^{d}_{\vep_0', \vep_0}$ implies that there exists a root $\eta(z)$  of $p_z(\cdot)=0$ such that $||\eta(z)| - 1| \le \vep_0'$. Set  $z_0 := p(\eta(z)/|\eta(z)|) \in p(S^1)$. By the triangle inequality and uniform boundedness of $p'(\cdot)$ in a compact neighborhood of $S^1$, one has that $|z-z_0| \le C_{\ref{lem:tube-blow-up}} \vep_0'$ for some $C_{\ref{lem:tube-blow-up}}< \infty$ depending on $p$,
 yielding \eqref{eq:dist-to-pS1}.

Turning to the proof of  part (ii) we pick any $z_\star$ belonging to the set on the \abbr{LHS} of  \eqref{eq:bl-up-tube-1}. By definition, there exists a $z_0 \in \cT_{\vep_0', \vep_0}^{d, (1)}\setminus \cB_2^{\wt \vep_0}$ such that $|z_\star - z_0| \le \beta \vep$.  Thus, for $\vep \le \wt \vep_0/2$, we have that $z_\star \notin \cB_2^{\wt \vep_0/2}$. On the other hand, 
noting that for $j \in [\wt m]$, the maps $z \mapsto \eta_j(z)$ are analytic outside $\cB_2^{\wt \vep_0/2}$ and using that $\cT^{d}$ is a bounded set, we deduce that
\begin{equation}\label{eq:root-lip}
\max_{j \in [\wt m]} |\eta_j(z_0) - \eta_j(z_\star)| \le O_{\wt \vep_0}(1) \cdot |z_\star - z_0| \le \beta \vep_0''/4, 
\end{equation}
where in the last inequality we chose $\vep \le \vep_{\ref{lem:tube-blow-up}} \vep_0''$ with
$\vep_{\ref{lem:tube-blow-up}}$  a sufficiently small constant, depending only on $\wt \vep_0$. 

Since $z_0 \in \cT^{d, (1)}_{\vep_0', \vep_0}$, by \eqref{eq:root-lip} we deduce that 
\begin{equation}\label{eq:blowup-eq1}
\max\left\{ |\eta_{m_-+{\sf g}_0+1}(z_\star)|,  |\eta_{m_-}(z_\star)|^{-1}\right\}  \le 1 -\varepsilon_0(1 -\beta/2), 
\end{equation}
and 
\begin{equation}\label{eq:blowup-eq2}
{1 - \vep_0'(1+ \beta/2)} \le   |\eta_{m_-+1}(z_\star)| =    |\eta_{m_-+{\sf g}_0}(z_\star)|  \le 1+ \vep_0'' \beta/2,
\end{equation}
where we also use that the roots can be partitioned into blocks, each of size ${\sf g}_0$, such that the roots in each of those block have the same moduli.

Now there are three possibilities:  If $|\eta_{m_-+1}(z_\star)| < 1$, in which case $z_\star \in \cS_d$, and therefore by \eqref{eq:blowup-eq1}-\eqref{eq:blowup-eq2} we have that
 $z_\star \in \cT^{d, (1)}_{(1+\beta)\vep_0', (1-\beta)\vep_0}$. If  $|\eta_{m_-+{\sf g}_0}(z_\star)| > 1$ then $z_\star \in \cS^{d-{\sf g}_0}$, and consequently $z_\star \in \cT^{d-{\sf g}_0, (2)}_{\beta\vep_0'', (1-\beta)\vep_0}$. 
Finally, if $|\eta_{m_-+1}(z_\star)| = 1$, by the maximum modulus principle, there exists a sequence $\{z_n\}_{n \in \N}$ such that $|z- z_n| \le 1/n$  and $|\eta_{m_-+1}(z_n)| < 1$ for all $n \in \N$. It is clear that $z_n \in \cS_d$, and by the continuity of the roots the  bound in \eqref{eq:blowup-eq1} and the lower bound in \eqref{eq:blowup-eq2} holds for $z_n$, for all large $n$, with $\beta/2$ replaced by $\beta$. This, in particular, shows that $z_n \in \cT^{d, (1)}_{(1+\beta)\vep_0', (1-\beta)\vep_0}$ for all large $n$. Therefore, $\cT^{d, (1)}_{(1+\beta)\vep_0', (1-\beta)\vep_0}$ being a closed set the limit $z_\star$ must also be in $\cT^{d, (1)}_{(1+\beta)\vep_0', (1-\beta)\vep_0}$. This completes the proof of
 \eqref{eq:bl-up-tube-1}. The proof of \eqref{eq:bl-up-tube-2} being similar, details are omitted. 
\end{proof}

The following lemma is proved by a standard volumetric estimate, that we omit.
\begin{lem}\label{lem:net-bd}
For any $\upepsilon,  \beta \in (0,1]$ there exists a net of $p(S^1)^\upepsilon $ of mesh size $\beta \upepsilon$ and  of  cardinality  at most $O(\beta^{-2} \upepsilon^{-1})$.
\end{lem}

We are now ready to provide the proof of Lemma \ref{lem:der-bd}.

\begin{proof}[Proof of Lemma \ref{lem:der-bd}]
Fix $d \ge 0$ and $s\in \{1,2\}$. This fixes $\wh m_1$ and $\wh d$ (see \eqref{eq:wh-m1} and \eqref{eq:wh-d}). 
Fix $\vep_0, \vep_0',\vep_0'' >0$ such that $\vep_0' < \vep_0$ is sufficiently small and $\vep_0'' \le \vep_0'$. The precise choice of $\vep_0''$ will be specified later. 
Set $\vep = \vep_{\ref{lem:tube-blow-up}} \vep_0''$ and let $\beta \in (0,1/16)$.
Pick any $z \in \T^{\beta} := (\cT^{d,(s)}_{\vep_0', \vep_0}\setminus \cB_2^{\wt \vep_0})^{\beta \vep}$. 
From Lemmas \ref{lem:tube-blow-up}(i)-(ii) and \ref{lem:net-bd} we deduce that there exist a finite collection $\{z_1,z_2,\ldots, z_b\} \subset \T^\beta$ such that 
\begin{equation}\label{eq:T1beta}
\cup_{i=1}^b D(z_i, \beta \vep) \supset \T^\beta
\end{equation}
with 
\begin{equation}\label{eq:bound-b}
b = O(\vep_0' \cdot \vep_0''^{-2}). 
\end{equation}
Hence, it suffices to derive bounds on the derivative of $\wh \det_k(\cdot)$ on each $\D_i^\beta := D(z_i, \beta \vep)$ for $i \in [b]$.

Turning to do that we let  $\D := D(0,R) \setminus \cB_2^{\wt \vep_0}$, where $R := 2 \max_{\eta \in S^1} |p(1/\eta)|$, and consider a smooth cutoff function $\upchi : \D \mapsto \C$ such that
\[
\upchi \equiv \left\{\begin{array}{ll}
1 & \mbox{ on } \D_i^{2\beta}, \\
0 & \mbox {on } \D \setminus \D_i^{3 \beta}.
\end{array}
\right.
\]
This implies that the derivative of $\upchi(\cdot)$ is non-zero only on $\D_i^{3\beta}\setminus \D_i^{2\beta}$. Since $\dist(\D^{2\beta}_i, \partial \D^{3\beta}_i) = \beta \vep$ one can choose $\upchi$ such that the absolute value of its anti-holomorphic derivative $\partial_{\bar w} \upchi(w)$ is at most $O(\beta^{-1} \vep^{-1})$ in $\D$.  
Further let $\Xi : \D \mapsto \C$ be some (possibly random)  holomorphic function. Then, applying Cauchy's integral formula for the smooth function $z \mapsto \Xi(z) \cdot \upchi(z)$ on the domain $\D$ (e.g.~see \cite[Theorem 1.2.1]{H73}) we obtain that for any $z \in \D_i^\beta$,
\[
\Xi(z) = -\frac{1}{ \pi} \int_{\D_i^{3\beta}\setminus \D_i^{2\beta}} \frac{\Xi(w) \cdot\partial_{\bar w} \upchi(w)}{w - z} dL(w),
\]
where $L(\cdot)$ is the two-dimensional Lebesgue measure. By the bounded convergence theorem we also have that
\[
\frac{d}{dz} \Xi(z) = -\frac{1}{\pi} \int_{\D_i^{3\beta}\setminus \D_i^{2\beta}} \frac{\Xi(w) \cdot \partial_{\bar w} \upchi(w)}{(w - z)^2} dL(w).
\] 
 Now, by the Cauchy-Schwarz inequality and Fubini's Theorem we further deduce that 
\begin{align}\label{eq:cauchy-der-l2}
\E \left[ \sup_{z \in \D_i^\beta} \left|\frac{d}{dz} \Xi(z)\right|^2 \right]  & \le \frac{L(\D_i^{3\beta})}{\pi^2 \cdot \dist(\D^\beta_i, \partial \D_i^{2\beta})^4} \cdot \int_{\D_i^{3\beta}\setminus \D_i^{2\beta}} \E \left[ |\Xi(w)|^2\right]\cdot \left|\partial_{\bar w} \upchi(w)\right|^2 dL(w) \notag\\
& = O(\beta^{-2} \vep^{-2}) \cdot \sup_{w \in \D^{3\beta}_i} \E\left[|\Xi(w)|^2\right]. 
\end{align}
To complete the proof of the lemma we now proceed to apply \eqref{eq:cauchy-der-l2} with appropriate choices of $\Xi(\cdot)$ and $\vep_0''$. 

We note that the roots of $p_z(\cdot)=0$ are analytic in $z$ for $z \in \D$. Therefore, there exists a reordering of the indices of the roots $\{\eta_j(z)\}_{j \in [\wt m]}$, to be denoted by $\{\wh \eta_j(z)\}_{j \in [\wt m]}$, such that the maps $z \mapsto \wh \eta_j(z)$ are holomorphic on $\D$.  We set 
\begin{equation}\label{eq:xik}
 \Xi(z)=\Xi_k(z) := \frac{\det_k(z)}{a_{-N_-}^N  N^{-\gamma \wh d} \prod_{j=1}^{\wh m_1+{\sf g}_0} \wh \eta_j(z)^N}. 
\end{equation}
We claim that the holomorphic functions $\{\wh \eta_j(z)\}_{j \in [\wt m]}$ can be chosen in such a way so that
\begin{equation}\label{eq:root-holo}
\prod_{j=1}^{\wh m_1+{\sf g}_0} \wh \eta_j(z) = \prod_{j=1}^{\wh m_1+{\sf g}_0}  \eta_j(z) \qquad \mbox{ for } z \in \D^{3\beta}_i. 
\end{equation}
Indeed, pick any $\wh z \in \D_i^{3\beta}$. By the Implicit function theorem it is immediate that $\{\wh \eta_j(z)\}_{j \in [\wt m]}$ can be chosen such that \eqref{eq:root-holo} holds for $z = \wh z$. Also, note that by Lemma \ref{lem:tube-blow-up}(ii) we have that
\begin{equation}\label{eq:T3beta}
\D^{3\beta}_i \subset \T^{4\beta} \subset \left(\cT_{(1+4\beta) \vep_0', (1-4\beta)\vep_0}^{d, (s)} \cup \cT_{ \vep_0'', (1-4\beta)\vep_0}^{d+(-1)^s {\sf g}_0, (3-s)}\right)\setminus \cB_2^{\wt \vep_0/2}.
\end{equation}
The inclusion \eqref{eq:T3beta} and $\beta<1/16$ imply
that there are no roots of $p_z(\cdot)=0$ in $\D^{3\beta}_i$
with moduli between  $1+  2\vep_0''$ and $1 + \vep_0/2$ 
(choose $\vep_0$ and $\vep_0'$ such that $\vep_0' /\vep_0 < 1/8$).  
Hence, \eqref{eq:root-holo} must continue to hold for all $z$ in the connected domain $\D_i^{3 \beta}$, for otherwise the image of the continuous map 
$|\wh \eta_j (z)|$ from $\D^{3\beta}_i$ to $\R_+$ would be disconnected for
some $j$.
It further follows from \eqref{eq:T3beta} that $\wh d(\cdot) \equiv \wh d$ and $\wh m_1(\cdot) \equiv \wh m_1$ on $\D_i^{3\beta}$. 
This means that $\wh \det_k(\cdot) \equiv  \Xi_k(\cdot)$ on $\D_i^{3\beta}$ (recall \eqref{eq:fancyK-1}-\eqref{eq:hatP-k}). Since this holds for every $i \in [b]$ we use \eqref{eq:T1beta} and \eqref{eq:T3beta} to deduce from \eqref{eq:cauchy-der-l2} that
\begin{equation}\label{eq:cauchy-der-l21}
\E \left[ \sup_{z \in \T^\beta} \left|\frac{d}{dz}\wh \det_k(z)\right|^2 \right]   =  O(\vep_0'/(\vep_0'')^4) \cdot \sup_{w \in \T^{4\beta}} \E\left[|\wh \det_k(w)|^2\right]. 
\end{equation}
To complete the proofs of \eqref{eq:lem-der-ubd2}-\eqref{eq:lem-der-ubd} it remains to find an upper bound on the \abbr{RHS} of \eqref{eq:cauchy-der-l21}. 


We first 
consider the proof of \eqref{eq:lem-der-ubd}. Set $\vep_0'' = (\gamma'-1)\log N/N$. With this choice of $\vep_0''$, as $\beta \le 1/16$, we obtain from \eqref{eq:T3beta} that $\T^{4\beta} \subset \cT^{d, (2)}_{2\vep_0', \vep_0/2} \cup \wh \cT^{d+{\sf g}_0, (1)}_{\gamma', \vep_0/2}$. Apply 
Lemmas \ref{lem:sec-mom-small-k}, \ref{lem:sec-mom-large-k}(i) and 
\ref{lem:sec-mom-large-k}(iii) with $\vep_0'$ and $\vep_0$ replaced by 
$2\vep_0'$ and $\vep_0/2$ to derive that 
\begin{equation}\label{eq:sup-exp}
 \sup_{w \in \T^{4\beta}} \E\left[ |\wh{\det}_k(z)|^2\right] = O\left(1\right). 
\end{equation}
Plugging this bound in \eqref{eq:cauchy-der-l21} we obtain \eqref{eq:lem-der-ubd}. 


The proof of \eqref{eq:lem-der-ubd2} is similar. Set
$\vep_0'=\vep_0'' = (\gamma'-1)\log N/N$ and choose $\beta$ such that $(1+3\beta) \vep_0' = (\gamma''-1)\log N/N$ for some $\gamma'' < \gamma$. Apply Lemmas \ref{lem:sec-mom-small-k}, \ref{lem:sec-mom-large-k}(i) and \ref{lem:sec-mom-large-k}(iii) with $\gamma'', \vep_0'',$ and $\vep_0/2$ instead of $\gamma', \vep_0',$ and $\vep_0$, respectively to deduce that \eqref{eq:sup-exp} continues to hold for $\T^\beta =  (\cT^{d,(1)}_{\vep_0', \vep_0}\setminus \cB_2^{\wt \vep_0})^{\beta \vep}$. The desired bound now follows from \eqref{eq:cauchy-der-l21}. 
\end{proof}

We end this section with the proof of Lemma \ref{lem:det-0}. 

\begin{proof}[Proof of Lemma \ref{lem:det-0}]
We first prove
part (a). The proof is a direct consequence of Widom's formula for the determinant of a finitely banded Toeplitz matrix. Indeed, by \cite[Theorem 2.8]{BoGr05}, for $s \in \{1,2\}$, we have, recalling that $\wt N = N_+ + N_-$, that
\[
\wh{\det}_0(z) = N^{\gamma(\wh m_2 - N_+)} \cdot \prod_{j=1}^{\wh m_1 +{\sf g}_0} \eta_j(z)^{-N} \cdot(-1)^{NN_-}\cdot \left[\sum_{\cI \in \binom{[\wt N ]}{N_-}} \cC_\cI(z) \cdot \prod_{i \in \cI} \eta_i(z)^N\right],
\]
where for any $\cI \subset [\wt N]$
\[
\cC_\cI(z) :=  \prod_{\substack{j_1 \in \cI\\ j_2 \in [\wt N]\setminus \cI}} \frac{\eta_{j_1}(z)}{\eta_{j_1}(z) - \eta_{j_2}(z)}.
\]
Observe that that $\inf_{z \in \cT^d} |\eta_{\wh m_1 + {\sf g}_0}(z)| \ge 1/2$. Thus, now by the continuity of the maps $z \mapsto \eta_j(z)$, for $j \in [\wt m]$, as $\cT^d$ is a bounded set and we work off $\cB_2^{\wt \vep_0}$, it follows that 
\begin{equation}\label{eq:cC-cI-bd}
0 <\inf_{z \in \cT^d \setminus \cB_2^{\wt \vep_0}} |\cC_{[\wh m_1+{\sf g}_0]}(z)|  \le \sup_{z \in \cT^d \setminus \cB_2^{\wt \vep_0}} |\cC_{[\wh m_1+{\sf g}_0]}(z)| \le \max_{\cI \subset [N_++N_-]} \sup_{z \in \cT^d \setminus \cB_2^{\wt \vep_0}} |\cC_\cI(z)| < \infty. 
\end{equation}
On the other hand, recalling the definition of the tube $\cT^d$,
it is evident that
\begin{equation}\label{eq:ratio-small-zeta}
\sup_{z \in \cT^d} \left\{\max_{{j_1 \in [\wh m_1 +{\sf g}_0], \, j_2 \notin [\wh m_1 +{\sf g}_0]}} \left(\left| \frac{\eta_{j_2}(z)}{\eta_{j_1}(z)}\right| \vee |\eta_{j_2}(z)|\right)\right\} \le \left(1-\frac{\vep_0}{3}\right). 
\end{equation}
Notice that, for $s \in \{1,2\}$
\begin{equation}\label{eq:ds}
\wh d > 0 \quad \Longleftrightarrow \quad \wh m_2 - N_+ > 0 \quad \Longleftrightarrow \quad \wh m_1 +{\sf g}_0 < N_-. 
\end{equation}
Therefore, as $\wh d >0$, we have that $\cI \setminus [\wh m_1 +{\sf g}_0] \ne \emptyset$ for any set $\cI \subset [\wt N]$ of cardinality $N_-$. 
Thus, \eqref{eq:ratio-small-zeta} now implies that 
\[
\sup_{z \in \cT^d} \left|\prod_{j=1}^{\wh m_1 +{\sf g}_0} \eta_j(z)^{-N} \cdot \prod_{i \in \cI} \eta_i(z)^N\right| \le \left(1-\frac{\vep_0}{3}\right)^N,
\]
for all such subsets $\cI$. Using this together with \eqref{eq:cC-cI-bd} we now have part (a).

  Turning to the proof of
  part (b), we recall \eqref{eq:P-wtN} to observe that $P_{\wh N_+}(p, z; \wt N)[X_\star^c; Y_\star^c + N_+]$ is  a $(\wh N_+ - \wh m_2)$-dimensional Toeplitz matrix with symbol 
  \[
   p_\star (\tau) := \tau^{\wh m_2 - N_+} \left[\sum_{j=-N_-}^{N_+} a_j \tau^j - z\right]. 
  \]
 Notice that the roots of $p_\star(1/\cdot)$ coincide with those of
  $p_z(\cdot)$. Therefore, again by \cite[Theorem 2.5]{BoGr05}, using that $\wt N = \wh m_1 +\wh m_2+{\sf g}_0$, from \eqref{eq:gD-def} and \eqref{eq:wh-gD-dfn} we deduce that
 \begin{equation}\label{eq:wh-gD-lbd}
{\wh \gD(X_\star, Y_\star, z)} =  (-1)^{\wh N  (\wh m_1+{\sf g}_0)} a_{-N_-}^{-(\wh m_2 - N_+)}\cdot \prod_{j=1}^{\wh m_1 +{\sf g}_0} \eta_j(z)^{-(\wh m_2 - N_+)} \sum_{\cI \in \binom{[\wt N]}{\wh m_1 +{\sf g}_0}}  \cC_\cI(z) \cdot    \frac{\prod_{i \in \cI} \eta_i(z)^{\wh N}}{\prod_{j=1}^{\wh m_1 +{\sf g}_0} \eta_j(z)^{\wh N}},
 \end{equation}
 where we use the shorthand $\wh N := \wh N_+ -\wh m_2 \, = N+N_+ - \wh m_2$. 
 
 Upon using \eqref{eq:cC-cI-bd} one finds that the summand in the \abbr{RHS} of \eqref{eq:wh-gD-lbd} for $\cI = [\wh m_1 + {\sf g}_0]$ is uniformly bounded below by $3 c_\star$ for some constant $c_\star >0$.  On the other hand, 
 for $\cI \ne [\wh m_1 +{\sf g}_0]$,
 using \eqref{eq:cC-cI-bd}-\eqref{eq:ratio-small-zeta} we deduce that it is exponentially small (in $\wh N$ and hence in $N$) compared to $c_\star$. Therefore, the sum in \eqref{eq:wh-gD-lbd} is uniformly bounded below by $2 c_\star$. Using 
 now that $\sup_{z \in \cT^d} \max_{j \in [\wt m]} |\eta_j(z)| < \infty$, 
 and shrinking $c_\star$ if needed, complete
 the proof of the lower bound in \eqref{eq:det-0-lubd} completes. The proof of the upper bound is immediate from \eqref{eq:wh-gD-lbd}, the above discussion, and the fact that $\inf_{z \in \cT^d} |\eta_{\wh m_1 + {\sf g}_0}(z)| \ge 1/2$. We omit further details. 
\end{proof}

\subsection{Step 3: uniform lower bound on the dominant term}\label{sec:dom-lbd}  
In this section we prove a uniform lower bound on the dominant term in the expansion of $\det(P_z^\delta)$. 
The following is the main result of this section.

\begin{thm}\label{thm:dom-lbd}
Fix $\vep_0, \wt \vep_0>0$.
Assume that $\wh d > 0$ and that the 
entries of $Q$ satisfy Assumption \ref{assump:anticonc}.
Then, there exists 
$C_{\ref{thm:dom-lbd}} =C_{\ref{thm:dom-lbd}}(p,\upeta,\vep_0,
\wt \vep_0)< \infty$ so that for all $\vep_0'$ sufficiently small,
\[
\prob\left(\inf_{z \in  \cT^{d, (s)}_{\vep_0', \vep_0} \setminus \cB_2^{\wt \vep_0}} |\wh \det_{\wh m_2 - N_+}^{(s)}(z)| \le \frac12(\vep_0')^{3\upeta/4}\right) \le C_{\ref{thm:dom-lbd}} (\vep_0')^{\upeta/3}, \quad s\in \{1,2\}.
\]
\end{thm}

The proof of
Theorem \ref{thm:dom-lbd} uses
the following anti-concentration bound for certain polynomials in
independent random variables such that the degree of each of those random variables in those polynomials is at most one. This is a generalization of \cite[Proposition 4.1]{BZ}.

\begin{lem}\label{lem:anti-conc}
Fix $k \in \N$ and let $\{U_i\}_{i=1}^k$ be a sequence of
independent complex-valued random variables, whose L\'{e}vy concentration functions satisfy the bound \eqref{eq:levy-bd} with $\upeta\in (0,1]$.
Let $\{Z_\cI; \cI \in \binom{[n]}{k}\}$ be another collection of random variables which is jointly independent of the collection $\{U_i\}_{i=1}^k$. Define 
\[
\mathcal{U}_k := \sum_{\cI \subset [k]} Z_{\cI} \prod_{i \in \cI} U_{i}.
\]
Then, for any constant $c_\star >0$ and $\vep \in (0, c_\star {e^{-1}}]$ we 
have that
\[
\prob\left( |\mathcal{U}_k| \le \vep\right) \le \bar C_{\ref{lem:anti-conc}} \cdot \left(\frac{\vep}{c_\star}\right)^{(1+\upeta)} \cdot \left(\log\left(\frac{c_\star}{\vep}\right)\right)^{k-1} + \prob\left(|Z_{[k]}| \le c_\star \right),
\]
where 
$\bar C_{\ref{lem:anti-conc}} < \infty$ is some large constant 
depending only on $\upeta$ and $C_{\ref{assump:anticonc}}$.
\end{lem}

\begin{proof}[Proof of Lemma \ref{lem:anti-conc}]
We will show that
\begin{align}\label{eq:levynew-1}
& \, \prob\left( |\mathcal{U}_k| \le \vep\big| Z_{[k]}  \right) \cdot {\bf 1}_{\{|Z_{[k]}| > c_\star\}} \notag\\
& \quad \le  \, \prob\left( |\wh {\mathcal{U}}_k| \le c_\star^{-1} \vep\big| Z_{[k]}  \right) \cdot {\bf 1}_{\{|Z_{[k]}| > c_\star\}} \le \bar C_{\ref{lem:anti-conc}} \cdot \left(\frac{\vep}{c_\star}\right)^{(1+\upeta)} \left(\log\left(\frac{c_\star}{\vep}\right)\right)^{k-1},
\end{align}
where
\[
\wh\cU_k := \sum_{\cI \subset [k]} \wh Z_{\cI} \prod_{i \in \cI} U_{i} \qquad \text{ and } \qquad \wh Z_\cI := \frac{Z_\cI}{Z_{[k]}}, \text{ for } \cI \subset [k]. 
\]
(Notice that $\{\wh Z_\cI; \cI \subset [k]\}$ are well defined on the event $\{|Z_{[k]}| > c_\star\}$, and hence so is $\wh \cU_k$.)
Taking the expectation over $Z_{[k]}$ in \eqref{eq:levynew-1} together with a union bound immediately yields Lemma \ref{lem:anti-conc}.  The first inequality in \eqref{eq:levynew-1} is straightforward from the definition of $\wh \cU_k$. So, we only need to prove the second inequality in \eqref{eq:levynew-1}.

Turning to that task, for $j \in [k]$,  set $\cJ_j:= \{j, j+1, \ldots, k\}$. Define
\[
 \wh \cU_j := \sum_{\cI: \, \cI  \supset \cJ_{j+1}} \wh Z_{\cI} \prod_{i \in \cI\setminus \cJ_{j+1}} U_{i} \qquad \text{ and } \qquad \wt \cU_j := \sum_{\substack{\cI: \, \cI \supset \cJ_{j+2}\\ (j+1) \notin \cI}} \wh Z_{\cI} \prod_{i \in \cI\setminus \cJ_{j+2}} U_{i}, \quad \mbox{ for } j \in [k-1] \cup\{0\},
\]
where $\cJ_{k+1}=\emptyset$. We will prove inductively that, for any $j \in [k]$, and all $t \in (0, e^{-1}],$
\begin{equation}\label{eq:anti-conc-induction}
\prob\left( |\wh \cU_j| \le t \big| Z_{[k]} \right) \le {(6 (C_{\ref{eq:levy-bd}} \vee 1) e^{1+\upeta})}^{j} t^{1+\upeta} \left(\log\left(\frac{1}{t}\right)\right)^{j-2}.
\end{equation}
Plugging $t = c_\star^{-1} \vep$ and $j=k$ in \eqref{eq:anti-conc-induction} yields \eqref{eq:levynew-1}. 

Proceeding  to the proof of \eqref{eq:anti-conc-induction}, we start with $j=1$. As $\wh Z_{[k]}=1$ we have that 
$
\wh \cU_1 = U_1 + \wt \cU_0. 
$
Thus, as $U_1$ is independent of $\{Z_{\cI}, \cI \subset [k]\}$, by \eqref{eq:levy-bd} we have that
\[
\prob\left(|\wh \cU_1| \le t \big| Z_{[k]} \right) \le \cL (U_1, t) \le C_{\ref{assump:anticonc}} t^{1+\upeta}. 
\]
To prove \eqref{eq:anti-conc-induction} for an arbitrary $j \in [k]$, upon using induction we note that 
\begin{equation}\label{eq:levynew-2}
\wh \cU_{j} = U_{j} \cdot \wh \cU_{j-1} + \wt \cU_{j-1}, \qquad \text{ for } j \in [k]. 
\end{equation}
Observe that by our assumption the random variables $\{\cU_{j-1}, \wt \cU_{j-1}\}$ are independent of $U_j$ for any $j \in [k]$. This, in particular, allows us to use the anti-concentration property of $U_j$ to derive the same for $\wh \cU_j$.

To complete the proof of \eqref{eq:anti-conc-induction} by induction, assume it holds for $j=j_*$, and set  $C_j={(6 (C_{\ref{eq:levy-bd}} \vee 1) e^{1+\upeta})}^{j}$. The induction hypothesis yields that 
\begin{align}\label{eq:anti-conc-split}
& \prob\left( \left|\wh \cU_{j_*+1} \right| \le t \big| Z_{[k]} \right) \notag\\
 &\le  \prob \left( \left| \wh \cU_{j_*}\right| \le t \big| Z_{[k]} \right) + \E\left[ \prob\left(  \left| U_{{j_*}} + \frac{\wt \cU_{j_*} }{\wh \cU_{j_*} }\right| \le \frac{t}{|\wh \cU_{j_*} |} \bigg| \, \wh \cU_{j_*}, \wt \cU_{j_*}, Z_{[k]}\right)\cdot {\bm 1} \left( \left|\wh \cU_{j_*}\right| \ge t\right) \bigg| Z_{[k]} \right] \notag\\
&\stackrel{\eqref{eq:levy-bd}}{ \le} C_{j_*} t^{1+\upeta} \left(\log \left( \frac{1}{t}\right)\right)^{j_*-2} + C_{\ref{assump:anticonc}} {t^{1+\upeta}} \cdot \E \left[ |\wh \cU_{j_*}|^{-(1+\upeta)} {\bm 1} \left( \left|\wh \cU_{j_*}\right| \ge t\right) \big| Z_{[k]}  \right],
\end{align}
where we have used that the triplet $\{\wh \cU_{j_*}, \wt \cU_{j_*}, Z_{[k]}\}$ is independent of $U_{{j_*}}$.
Using integration by parts, for any probability measure $\mu$ supported on $[0,\infty)$ we have that 
\[
\int_t^{{e^{-1}}} x^{-(1+\upeta)}  d\mu(x) = {e^{1+\upeta}} \mu([t, 1])  + (1+\upeta) \cdot  \int_t^{{e^{-1}}} \frac{\mu([t,s])}{s^{2+\upeta}} ds.
\]
Therefore, using the induction hypothesis and the fact that $\upeta \in (0,1]$, we have that
\begin{multline*}
\E \left[ |\wh \cU_{j_*}|^{-(1+\upeta)} {\bm 1} \left( \left|\wh \cU_{j_*}\right| \ge t \right) \big| Z_{[k]} \right] \le {e^{1+\upeta}}+ \E \left[ |\wh\cU_{j_*}|^{-(1+\upeta)} {\bm 1} \left( \left|\wh \cU_{j_*}\right| \in [t,{e^{-1}}]\right)\big|Z_{[k]} \right]  \\ 
\le 2{e^{1+\upeta}} + 2 \int_t^{{e^{-1}}} \frac{ \prob \left(|\wh \cU_{j_*}| \le s | Z_{[k]}\right)}{s^{2+\eta}} ds \le 2 {e^{1+\upeta}}+  2 C_{j_*} \int_t^{{e^{-1}}} s^{-1} \left(\log\left(\frac{1}{s}\right)\right)^{j_*-2} ds \\
 \le 2 {e}^{1+\upeta}+ \frac{2 C_{j_*}}{(j_*-1) \vee 1} \left(\log\left(\frac{1}{t}\right)\right)^{j_*-1}.
\end{multline*}
Combining the above with \eqref{eq:anti-conc-split} and using that $\log (1/t)\ge 1$ for $t \le e^{-1}$,
we establish \eqref{eq:anti-conc-induction} for $j=j_*+1$. This concludes the induction argument and hence the proof of the lemma is complete.
\end{proof}

Lemmas \ref{lem:det-0}(b) and \ref{lem:anti-conc} yield  easily an anti-concentration bound for $\det_{\wh m_2 -N_+}(\cdot)$ per fixed $z$, as follows. 


\begin{cor}\label{cor:lbd-dom-term}
Consider the same setup as in Theorem \ref{thm:dom-lbd}. Then for $s \in \{1,2\}$ and $\vep \in (0, c_\star {e^{-1}}]$ we have
\[
\sup_{z \in \cT^{d, (s)}_{\vep_0', \vep_0} \setminus \cB_2^{\wt \vep_0}} \prob \left(|\wh{\det}_{\wh m_2 - N_+}(z)| \le \vep\right) \le \bar C_{\ref{lem:anti-conc}} \cdot \left(\frac{\vep}{c_\star}\right)^{(1+\upeta)} \cdot \left(\log\left(\frac{c_\star}{\vep}\right)\right)^{\wh m_2 -N_+-1},
\]
where $c_\star$ is as in Lemma \ref{lem:det-0}(b). 
\end{cor}

\begin{proof}
Recall \eqref{eq:det-decompose-1}, and \eqref{eq:fancyK-1}-\eqref{eq:wh-gD-dfn}. Fix $s \in \{1,2\}$. We note that 
\begin{multline}\label{eq:XY-star}
\wh \det_{\wh m_2 - N_+}(z) = (-1)^{\sgn(\sigma_{X_\star}) \sgn(\sigma_{Y_\star})} \wh \gD(X_\star,Y_\star,z)  \cdot \det (Q[X_\star; Y_\star]) \\
+ \sum (-1)^{\sgn(\sigma_{X}) \sgn(\sigma_{Y})} \wh \gD(X,Y,z)   \cdot \det (Q[X; Y]),
\end{multline}
where $X_\star$ and $Y_\star$ are as in Lemma \ref{lem:det-0}(b), and the sum in the \abbr{RHS} of \eqref{eq:XY-star} is over subsets $X, Y \subset [N]$ such that $|X|=|Y|= \wh m_2 - N_+$ and $(X, Y) \ne (X_\star, Y_\star)$. It is easy to see to that $\wh \det_{\wh m_2 - N_+}(z) $ is a homogeneous polynomial in the entries of $Q$, which are jointly independent, such that the degree of each of its entry in that polynomial is at most one. Hence, Lemma \ref{lem:anti-conc} is applicable. By Lemma \ref{lem:det-0}(b) we have uniform lower bound on $\wh \gD(X_\star, Y_\star, z)$. Therefore, upon applying Lemma \ref{lem:anti-conc} with $k = \wh m_2 - N_+$ and the roles of $\{U_i\}_{i=1}^k$ being played by the diagonal entries of the sub matrix $Q[X_\star, Y_\star]$,  
we derive the desired anti-concentration bound. 
\end{proof}

 Corollary \ref{cor:lbd-dom-term} is not sufficient to yield Theorem \ref{thm:dom-lbd}. As discussed in Section \ref{sec:der-bd}, we also need a bound on the supremum of the derivative of ${\det}_{\wh m_2 - N_+}(z)$. This is derived below. 

\begin{lem}\label{lem:der-bd-dom}
 Let Assumption \ref{assump:mom} hold. Fix $\wt \vep_0$. There exists some constant $c_{\ref{lem:der-bd-dom}} >0$, depending only on $p, \wt \vep_0$, such that for all
$\vep_0, \vep_0' >0$ such that $\vep_0' /\vep_0$ is small enough,
\begin{equation}\label{eq:sec-mom-O1}
\limsup_{N \to \infty} \E\left[\sup_{z \in \left( \cT^{d, (s)}_{\vep_0', \vep_0}\setminus \cB_2^{\wt \vep_0}\right)^{c_{\ref{lem:der-bd-dom}}\vep_0'}} \left|\frac{d}{dz} \wh{\det}_{\wh m_2 - N_+}(z)\right|^2 \right] =O(1), \quad s\in \{1,2\}.
\end{equation}
\end{lem}


\begin{proof}
We borrow ingredients from the proof  of Lemma \ref{lem:der-bd}. 
Let $s =2$. The proof for $s=1$ being the same will be omitted. 

Our starting point is \eqref{eq:cauchy-der-l21}. Fix $k=\wh m_2 - N_+$. As in the proof of Lemma \ref{lem:der-bd} we set $\vep= \vep_{\ref{lem:tube-blow-up}} \vep_0''$ and fix $\beta \in (0,1/8)$. Set $\vep_0''=\vep_0'$. By \eqref{eq:T3beta} $\T^{4\beta}$ is contained in the union of $\cT^{d, (2)}_{2\vep_0', \vep_0/2}$ and $\cT^{d+{\sf g}_0,(1)}_{2\vep_0',\vep_0/2}$. Apply Lemma \ref{lem:sec-mom-large-k}(ii) for $\cT^{d, (2)}$ and $\cT^{d+{\sf g}_0, (1)}$ with $\vep_0'$ and $\vep_0$ being replaced by $2 \vep_0'$, and $\vep_0/2$, respectively to deduce that \eqref{eq:sup-exp} holds under the current setup. Therefore, the desired bound follows from \eqref{eq:cauchy-der-l21}. This finishes the proof. 
%
%
\end{proof}

\begin{rem}\label{rem:O1-vep}
From \eqref{eq:cauchy-der-l21}-\eqref{eq:sup-exp} we have that the \abbr{RHS} of \eqref{eq:sec-mom-O1} is $O( (\vep_0')^{-3})$. On the other hand Lemma \ref{lem:sec-mom-large-k}(ii), and hence \eqref{eq:sup-exp}, hold as soon as $\vep_0'/\vep_0 \le c$ where $c>0$ is some constant depending only on the degree of the Laurent polynomial. This allows us to use Theorem \ref{lem:der-bd-dom} for $\vep_0'= c \vep_0$ and extend the result for $\vep_0' < c \vep_0$, by noting that the sets $( \cT^{d, (s)}_{\vep_0', \vep_0}\setminus \cB_2^{\wt \vep_0})^{c_{\ref{lem:der-bd-dom}}\vep_0'}$ are increasing in $\vep_0'$. It in particular enables us to  claim that the \abbr{RHS} of \eqref{eq:sec-mom-O1} is $O(\vep_0^{-3})$. This observation will be used in the proof of Theorem \ref{thm:dom-lbd} to claim that the constant $C_{\ref{thm:dom-lbd}}$ does not depend on $\vep_0'$. 
\end{rem}


We are now ready to prove Theorem \ref{thm:dom-lbd}.

\begin{proof}[Proof of Theorem \ref{thm:dom-lbd}]
Fix $s \in \{1,2\}$. Similar to the proof of Theorem \ref{thm:prob-non-dom-small} we will also use a covering argument. 


Write  
\[
\wt \cF(z) := \left\{ |\wh{\det}_{\wh m_2 - N_+}(z)| \le \vep_0' \vep_\star^{-1} \right\} 
\]
 and 
 \[
 \wt \cF_0  := \left\{\sup_{z \in \left( \cT^{d, (s)}_{\vep_0', \vep_0}\setminus \cB_2^{\wt \vep_0}\right)^{c_{\ref{lem:der-bd-dom}}\vep_0'}} \left|\frac{d}{dz} \wh{\det}_{\wh m_2 -N_+}(z)\right| \ge \vep_\star^{-1} \right\},
\]
where $\vep_\star = (\vep_0')^{\upeta/4}$. By Lemmas \ref{lem:tube-blow-up}(i) and  \ref{lem:net-bd} there exists a net $\cN_\star$ of $\cT^{d, (s)}_{\vep_0', \vep_0}$ with mesh size ${c_{\ref{lem:der-bd-dom}}\vep_0'}/2$ such that $|\cN_\star| = O(\vep_0'^{-1})$. Therefore, applying Corollary \ref{cor:lbd-dom-term} and  Lemma \ref{lem:der-bd-dom} (see also Remark \ref{rem:O1-vep}) we have that
\begin{equation}\label{eq:prob-ubd-dom}
\prob\left( \cup_{z \in \cN_\star} \wt \cF(z) \cup \wt \cF_0 \right) =O\left( (\vep_0')^{\upeta/3}\right),
\end{equation}
for any $\vep_0'>0$ sufficiently small. On the other hand, for any $z \in \cN_\star$ we have that $D(z, c_{\ref{lem:der-bd-dom}}\vep_0'/2) \subset (\cT^{d, (s)}_{\vep_0', \vep_0}\setminus \cB_2^{\wt \vep_0})^{c_{\ref{lem:der-bd-dom}}\vep_0'/2}$. Therefore, by the first order Taylor series expansion and triangle inequality it follows that on the event $\cap_{z \in \cN} \wt \cF(z)^c \cup \wt \cF_0^c$, for any $z' \in \wt \cT_{\vep_0',\vep_0}^{d, (s)} \setminus \cB_2^{\wt \vep_0}$
\begin{align*}
|\wh{\det}_{\wh m_2 -N_+}(z')|  & \ge \inf_{z \in \cN_\star} |\wh{\det}_{\wh m_2 -N_+}(z)| - \dist(z', \cN_\star) \cdot \sup_{z \in \left( \cT^{d, (s)}_{\vep_0', \vep_0}\setminus \cB_2^{\wt \vep_0}\right)^{c_{\ref{lem:der-bd-dom}}\vep_0'}}\left|\frac{d}{dz} \wh{\det}_{\wh m_2 -N_+}(z)\right|  \\
& \ge \frac12 \vep_0' \cdot \vep_\star^{-1} = \frac{1}{2} (\vep_0')^{3\upeta/4}. 
\end{align*}
This together with \eqref{eq:prob-ubd-dom} completes the proof of the theorem.  
\end{proof}

\subsection{Step 4: combining bounds from Sections \ref{sec:high-mom}-\ref{sec:dom-lbd}}\label{sec:combine-away-spectral}
In this short section we combine the results obtained in the previous three sections and complete the proofs of Theorems \ref{thm:no-outlier} and \ref{thm:sep-spec-curve}. We start with the proof of Theorem \ref{thm:no-outlier}. The estimates derived Sections \ref{sec:high-mom}-\ref{sec:dom-lbd} are for the tubes which are defined through the roots $p_z(\cdot)=0$, while Theorem \ref{thm:no-outlier} is to be proved for regions that are defined via the spectral parameter $z$. Therefore, we will need to prove some additional geometric features of the tubes which will allow us to relate those to regions defined through $z$. These will also be used later in Section \ref{sec:bulk-eig}. Stating these results 
requires the following notation.

For any $\upepsilon >0$ and  set $\cB \subset \C$ we write $\cB^{-\upepsilon}:= ((\cB^c)^\upepsilon)^c$. 
For any collection of $({\sf g}_0+1)$ distinct indices ${\bm i} := \{i_1, i_2, \ldots, i_{{\sf g}_0+1}\} \subset [\wt m]$ and $c >0$ we set 
\begin{equation*}
\cW_{\bm i}(c) := \left\{ z \in\C: ||\eta_{i_j}(z)| -1 | + ||\eta_{i_j'}(z)| -1 | \le c, \mbox{ for all } j \ne j' \in [{\sf g}_0+1]\right\}.
\end{equation*}
and
\[
\cW(c) := \left\{z: \exists  \, \eta =\eta(z) \text{ such that } p_z(-\eta) =0 \text{ and } ||\eta| -1| \le c \right\}. 
\]

\begin{lem}\label{lem:tube-geo}
Fix $\wt \vep_0>0$. We have the following geometric properties:
\begin{enumerate}

\item[(i)] There exists a constant $C_{\ref{lem:tube-geo}} < \infty $ such that for any $\upepsilon >0$,
\[
(p(S^1))^{\upepsilon} \setminus \cB_2^{\wt \vep_0} \subset \cW(C_{\ref{lem:tube-geo}}\upepsilon)\setminus \cB_2^{\wt \vep_0}. 
\]

\item[(ii)] There exists a constant $c_{\ref{lem:tube-geo}} >0$ such that for any $c \le c_{\ref{lem:tube-geo}}$,
\begin{equation}\label{eq:cW-supset}
\left(\bigcup_{{\bm i} \subset [\wt m]: |{\bm i}|= {\sf g}_0+1} \cW_{{\bm i}}(c\wt \vep_0) \right) \cap (\cB_1^{\wt \vep_0} \cup \cB_2^{\wt \vep_0} )^c= \emptyset. 
\end{equation}

\item[(iii)] Fix $0 <\vep_0 <\vep_0'$ such that $\vep_0'+\vep_0 \le c_{\ref{lem:tube-geo}} \wt \vep_0$. Then, for any $\upepsilon >0$ such that $\upepsilon \le C_{\ref{lem:tube-geo}}^{-1}  \vep_0'$,
\[
(p(S^1))^{\upepsilon} \setminus (\cB_1^{\wt \vep_0} \cup \cB_2^{\wt \vep_0}) \subset \left( \bigcup_d \cT^d_{\vep_0',\vep_0}\right)\setminus (\cB_1^{\wt \vep_0} \cup \cB_2^{\wt \vep_0}).
\]
\item[(iv)] Let $\vep_0, \vep_0'$, and $\upepsilon$ be as in part (iii). Then,
\begin{equation}\label{eq:S-0-in-T}
(\cS_0^{\upepsilon} \setminus {\cS}_0)\setminus (\cB_1^{\wt \vep_0} \cup \cB_2^{\wt \vep_0})  \subset (\cT^{ {\sf g}_0, (1)}_{\vep_0', \vep_0} \cup \cT^{- {\sf g}_0, (1)}_{\vep_0', \vep_0})\setminus (\cB_1^{\wt \vep_0} \cup \cB_2^{\wt \vep_0}). 
\end{equation}

\item[(v)] Let $\vep_0, \vep_0'$, and $\upepsilon$ be as in part (iii). Then,
\begin{equation*}
({\cS}_0 \setminus \cS_0^{-\upepsilon}) \setminus ( \cB_1^{\wt \vep_0} \cup \cB_2^{\wt \vep_0}) \subset \cT^{0,(2)}_{\vep_0',\vep_0} \setminus \cB_2^{\wt \vep_0}. 
\end{equation*}

\end{enumerate}
\end{lem}

\begin{proof}
The proof of part (i) is immediate. Fix $z' \in (p(S^1))^{\upepsilon}$. Then there exists $\wh z \in p(S^1)$ such that $|z' - \wh z| \le \upepsilon$. This, in particular implies that there exists $i \in [\wt m]$ such that $\eta_i(\wh z) \in S^1$. Therefore, by the triangle inequality, 
\begin{equation}\label{eq:root-derivative}
||\eta_i( z')| -1| \le |\eta_i( z') - \eta_i(\wh z)| \le \max_{i \in [\wt m]} \sup_{z\in (p(S^1))^{2\upepsilon} \setminus \cB_2^{\wt \vep_0}} \left|\frac{d}{dz}\eta_i(z)\right| \cdot |z' -\wh z|.
\end{equation}
Since $p(S^1)$ is a bounded set, and the roots are analytic outside a neighborhood of $\cB_2$ we have that the supremum of the derivative of the roots is bounded  $(p(S^1))^{2\upepsilon} \setminus \cB_2^{\wt \vep_0}$. This yields part (i). 

Next we turn to prove (ii). Let us first prove \eqref{eq:cW-supset} for Laurent polynomials $p(\cdot)$ such that ${\sf g}(p) = {\sf g}_0=1$. To this end, fix an arbitrary constant $c>0$. Assume that \eqref{eq:cW-supset} does not hold for this chosen $c$. Then there exist a $z \notin \cB_1^{\wt \vep_0} \cup \cB_2^{\wt \vep_0}$ and a pair $i \ne j \in [\wt m]$ such that 
\(
||\eta_i(z)| -1| + |\eta_j(z)| -1| \le c\wt \vep_0. 
\)
Therefore, if \eqref{eq:cW-supset} does not hold for any $c >0$ then we obtain a pair $i \ne j \in [\wt m]$, a sequence $\{z_{n}\}$, and a subsequence of integers $\{n_k\}$ such that
\begin{equation}\label{eq:zeta-cW}
||\eta_i(z_{n_k})|-1| + ||\eta_j(z_{n_k})|-1| \le \wt \vep_0/n_k \quad \text{ and } \quad z_{n_k} \in \C \setminus (\cB_1^{\wt \vep_0} \cup \cB_2^{\wt \vep_0}), \qquad \forall \, k. 
\end{equation}
Since for each of these $z_{n_k}$ there exists a root of $p_{z_{n_k}}(\cdot)=0$ that is close to the unit circle,
it follows that $\{z_{n_k}\}$ is a bounded sequence 
(see also the proof of Lemma \ref{lem:tube-blow-up}(i)), which
thus possesses 
a converging subsequence, and hence may be assumed to
converge to some $z_\star \notin \cB_2^{\wt \vep_0}$. Now, by the continuity of the maps $z \mapsto \eta_i(z)$ and $z \mapsto \eta_j(z)$ we derive from \eqref{eq:zeta-cW} that \( |\eta_i(z_\star)| = |\eta_j(z_\star)| =1 \). As $z_\star \notin \cB_2^{\wt \vep_0}$ we further deduce that $\eta_i(z_\star) \ne \eta_j(z_\star)$. This, in particular, implies that there exist $\eta \ne \eta' \in S^1$ such that $p(-\eta)=p(-\eta')=z_\star$, and hence $z_\star \in \cB_1$. This is a contradiction. Therefore, we have \eqref{eq:cW-supset} for ${\sf g}_0=1$. 

To prove \eqref{eq:cW-supset} for a Laurent polynomial $p(\cdot)$ with ${\sf g}(p) ={\sf g}_0 >1$ we note that for such $p(\cdot)$ there exists a Laurent polynomial $\wh p(\cdot)$, with ${\sf g}(\wh p)=1$, such that $p(\eta)= \wh p(\eta^{{\sf g}_0})$ for $\eta \in \C$. This means that the roots of $p_z(\cdot)=0$ are obtained by taking the ${\sf g}_0$-th roots of those of $\wh p_z(\cdot)=0$. Furthermore, the set of intersection points and branch points of $p(\cdot)$ and $\wh p(\cdot)$ are identical. Since \eqref{eq:cW-supset} holds for $\wh p(\cdot)$, upon shrinking $c_{\ref{lem:tube-geo}}$ it continues to hold for $p(\cdot)$. 

The proof of (iii) is also straightforward. Indeed, by parts (i) and (ii) we have 
\begin{equation}\label{eq:geom-3}
(p(S^1))^{C_{\ref{lem:tube-geo}}^{-1}  \vep_0'} \setminus (\cB_1^{\wt \vep_0} \cup \cB_2^{\wt \vep_0}) \subset \cW(\vep_0') \setminus (\cB_1^{\wt \vep_0} \cup \cB_2^{\wt \vep_0})  
= \cW( \vep_0') \cap \left(\bigcup_{{\bm i} \subset [\wt m]: |{\bm i}|= {\sf g}_0+1} \cW_{{\bm i}}(c_{\ref{lem:tube-geo}}\wt \vep_0) \right)^c \setminus (\cB_1^{\wt \vep_0} \cup \cB_2^{\wt \vep_0}) 
\end{equation}
Note that by the definition of $\cW(\cdot)$ it follows that any $z$ belonging to rightmost set in \eqref{eq:geom-3} must have a root $-\eta(z)$ of $p_z(\cdot)=0$ such that $||\eta(z)| -1 | \le \vep_0'$. As ${\sf g}(p) = {\sf g}_0$ there are ${\sf g}_0$ such roots. Recalling the definition of $\cW_{{\bm i}}(\cdot)$ we deduce that distance of all other roots from $S^1$ must be at least $\vep_0$. Hence, $z \in \cT^d_{\vep_0',\vep_0}$ for some $d$. This yields part (iii). 

Turning to prove part (iv) we begin with the proof of the following intermediate result:
\begin{equation}\label{eq:S-0-in-T-1}
(\cS_0^{\upepsilon} \setminus {\cS}_0)   \setminus ( \cB_1^{\wt \vep_0} \cup \cB_2^{\wt \vep_0}) \subset (p(S^1))^{\upepsilon}\setminus \cS_0)\setminus ( \cB_1^{\wt \vep_0} \cup \cB_2^{\wt \vep_0}). 
\end{equation}
for any $\upepsilon <\wt \vep_0/2$.
To see \eqref{eq:S-0-in-T-1},
we note that given any $z$ belonging to the set on the \abbr{LHS} of \eqref{eq:S-0-in-T-1} there exists $z_\star \in \cS_0$ such that $|z - z_\star| \le \upepsilon$. Define $z_\star(t) := z_\star(1-t) +  t z$, $t \in [0,1]$. We claim that there exists $t_\star \in (0,1]$ such that $z_\star(t_\star) \in p(S^1)$. This will indeed show that $z \in (p(S^1))^{\upepsilon}$. 

To prove the existence of $t_\star$, as $z_\star \in \cS_0$, we see that, $|\eta_{N_-}(z_\star)| > 1 > |\eta_{N_-+1}(z_\star)|$. As $z \notin {\cS}_0$, we have either $|\eta_{N_-}(z)| \le 1$ or $|\eta_{N_-+1}(z)| \ge 1$. Assume $|\eta_{N_-}(z)| \le 1$. If an equality holds then we set $t_\star =1$. So, assume further 
that $|\eta_{N_-}(z)| < 1$.  
Since the map $z_0 \mapsto |\eta_{N_-}(z_0)|$ is well defined and continuous for $z_0 \notin \cB_2^{\wt \vep_0/2}$, the existence of $t_\star \in (0,1)$ such that $|\eta_{N_-}(z_\star(t_\star))|=1$, and hence $z_\star (t_\star) \in p(S^1)$, is immediate from the intermediate value theorem. Similarly for the case $|\eta_{N_-+1}(z)| \ge 1$ one can repeat the above argument to find $t_\star \in (0,1]$ such that $|\eta_{N_-+1}(z_\star(t_\star))|=1$ implying again $z_\star(t_\star) \in p(S^1)$. This yields \eqref{eq:S-0-in-T-1}.

Returning to the proof of \eqref{eq:S-0-in-T} we apply part (iii) to deduce that for any $\upepsilon >0$ such that   $\upepsilon \le C_{\ref{lem:tube-geo}}^{-1}  \vep_0'$ we have 
\begin{equation*}
(\cS_0^{\upepsilon} \setminus {\cS}_0)   \setminus ( \cB_1^{\wt \vep_0} \cup \cB_2^{\wt \vep_0}) \subset \left(\cup_d \cT^{d}_{\vep_0',\vep_0} \right) \setminus ( \cB_1^{\wt \vep_0} \cup \cB_2^{\wt \vep_0}).
\end{equation*}
Thus, upon noting that as ${\sf g}(p)={\sf g}_0$ the winding number should be a multiple of ${\sf g}_0$, it suffices to show that 
\begin{equation}\label{eq:S-0-in-T-4}
(\cS_0^{\upepsilon} \setminus {\cS}_0) \bigcap\left(\cup_{d: |d| \ge 2 {\sf g}_0} \cT^d_{\vep_0',\vep_0} \cup\cT^{\pm {\sf g}_0, (2)}_{\vep_0',\vep_0}\right) =\emptyset. 
\end{equation}
Assume that
there exists $z \in (\cS_0^{\upepsilon} \setminus {\cS}_0) \cap \cT^{d}_{\vep_0',\vep_0}$ for some $d$ such that $|d| \ge 2{\sf g}_0$. Again we need to split to the cases $|\eta_{N_-}(z)| \le 1$ and $|\eta_{N_-+1}(z)| \ge 1$. 
We first consider the case $|\eta_{N_-}(z)| \le 1$. 
Recall that 
$d = m_+ - N_+ = N_- - m_-$. By the definition of the tube $\cT^d_{\vep_0',\vep_0}$ we note that $d \ge {2 \sf g}_0$ implies that $|\eta_{N_-}(z)| \le 1-\vep_0$. On the other hand, $d \le -2 {\sf g}_0$ implies that $|\eta_{N_-}(z)|^{-1} \le 1-\vep_0$. However, we also recall from above that $|\eta_{N_-}(z)| \le 1$ implies that there exists $z_\star(t_\star) \in p(S^1)$ such that $|\eta_{N_-}(z_\star(t_\star))|=1$ and $|z - z_\star(t_\star)| \le \upepsilon$. As $\upepsilon \le C_{\ref{lem:tube-geo}}^{-1}  \vep_0'$, and $\vep_0' \le \vep_0/2$, this together with \eqref{eq:root-derivative} indeed yields a contradiction. A similar argument works for the case $|\eta_{N_-+1}(z)| \ge 1$. One can repeat the same argument again to further show that  $(\cS_0^{\upepsilon} \setminus {\cS}_0) \cap \cT^{\pm {\sf g}_0, (2)}_{\vep_0',\vep_0} = \emptyset$. We omit the details. 

Finally we proceed to prove part (v). Repeating an argument similar to the proof of \eqref{eq:S-0-in-T-1} we obtain that
\begin{equation}\label{eq:S-0-in-T-5}
({\cS}_0 \setminus \cS_0^{-\upepsilon}) \setminus ( \cB_1^{\wt \vep_0} \cup \cB_2^{\wt \vep_0}) \subset (p(S^1))^{\upepsilon}\cap \cS_0)\setminus ( \cB_1^{\wt \vep_0} \cup \cB_2^{\wt \vep_0}).
\end{equation}
Now by part (iii) the set on the \abbr{RHS} of \eqref{eq:S-0-in-T-5} must be contained in the union of the tubes. However, the only tube that has nonempty intersection with $\cS_0$ is $\cT^{0, (2)}_{\vep_0',\vep_0}$. 
This yields part (v) and hence the proof of the lemma is now complete. 
\end{proof}


We are ready to prove Theorem \ref{thm:no-outlier}. 

\begin{proof}[Proof of Theorem \ref{thm:no-outlier}]
Fix $\gamma' >1$ such that $\gamma - 1 < \gamma' < \gamma$. Set $\vep_0' = (\gamma'-1)\log N/N$, $\vep_0 = c_{\ref{lem:tube-geo}} \wt \vep_0/2$, some arbitrary constant $\vep_0''>0$ such that $\vep_0''/\vep_0$ sufficiently small, $\upepsilon_1 = C_{\ref{lem:tube-geo}}^{-1} \vep_0'$, and $\upepsilon_2 = C_{\ref{lem:tube-geo}}^{-1} \vep_0''$. By Lemma \ref{lem:tube-geo}(iv)-(v) we note that 
\begin{equation*}\label{eq:S-0-in-Tt}
(\cS_0^{\upepsilon_1} \setminus {\cS}_0)\setminus (\cB_1^{\wt \vep_0} \cup \cB_2^{\wt \vep_0})  \subset (\cT^{ {\sf g}_0, (1)}_{\vep_0', \vep_0} \cup \cT^{- {\sf g}_0, (1)}_{\vep_0', \vep_0})\setminus (\cB_1^{\wt \vep_0} \cup \cB_2^{\wt \vep_0}) 
\end{equation*}
and  
\begin{equation*}\label{eq:S-0-in-Tt1}
({\cS}_0 \setminus \cS_0^{-\upepsilon_2}) \setminus ( \cB_1^{\wt \vep_0} \cup \cB_2^{\wt \vep_0}) \subset \cT^{0,(2)}_{\vep_0'',\vep_0} \setminus \cB_2^{\wt \vep_0}. 
\end{equation*}
Therefore, by \cite[Theorem 1.1]{BZ} and a union bound we now derive that 
\begin{multline}\label{eq:no-outlier}
 \limsup_{N \to \infty} \prob  \left( \exists z \in \cS_0^{ \upepsilon_1}\setminus (\cB_1^{\wt \vep_0} \cup \cB_2^{\wt \vep_0}): \det(P_z^\delta)=0 \right)  \\
 \le  \limsup_{N \to \infty} \prob  \left( \exists z \in (\cS_0^{\upepsilon_1}\setminus \cS_0^{-\upepsilon_2}) \setminus (\cB_1^{\wt \vep_0} \cup \cB_2^{\wt \vep_0}): \det(P_z^\delta)=0 \right) \\
 \le \limsup_{N \to \infty} \left[ \prob  \left( \exists z \in \wh \cT^{\pm {\sf g}_0,(1)}_{\gamma', \vep_0} \setminus \cB_2^{\wt \vep_0}: \det(P_z^\delta)=0 \right) + \prob  \left( \exists z \in  \cT^{0,(2)}_{\vep_0'', \vep_0} \setminus \cB_2^{\wt \vep_0}: \det(P_z^\delta)=0 \right)\right]. 
\end{multline}
It remains to show that each of the terms in the \abbr{RHS} of \eqref{eq:no-outlier} equals zero in the limit. For $d\in \{0,{\sf g}_0\}$, upon recalling \eqref{eq;spec-rad-new}, this follows from Theorem \ref{thm:prob-non-dom-small}, Lemma \ref{lem:det-0}(b) (see also Remark \ref{rem:det-0}), and the triangle inequality.  To see that the same holds for $d= - {\sf g}_0$, as $P_N^{\sf T}$  is a Toeplitz matrix with symbol $\wt p(\cdot) = p(1/\cdot)$ and hence ${\rm ind}_{p(S^1)}(\cdot) = - {\rm ind}_{\wt p(S^1)}(\cdot)$,  we apply Theorem \ref{thm:prob-non-dom-small} and Lemma \ref{lem:det-0}(c) for $(P_z^\delta)^{\sf T}$. 
This completes the proof of the theorem. 
\end{proof}

We end the section with the proof of Theorem \ref{thm:sep-spec-curve}.

\begin{proof}[Proof of Theorem \ref{thm:sep-spec-curve}]
The case $d=0$ is covered by Theorem \ref{thm:no-outlier}. Consider the case $d >0$. Fix $s \in \{1,2\}$. Since $0 < \gamma - \gamma' <1$ we observe that, for $\vep_0''=N^{-(\gamma - \gamma')/8}$,
\[
\wh \cT_{\gamma',\vep_0}^{d, (s)} \setminus \cB_2^{\wt \vep_0} \subset  \cT_{\vep_0'',\vep_0}^{d, (s)} \setminus \cB_2^{\wt \vep_0},
\] 
for all large $N$. Therefore, applying Theorems \ref{thm:prob-non-dom-small} and \ref{thm:dom-lbd} (with $\vep_0'$ replaced by $\vep_0''$), and upon using the triangle inequality it is immediate that 
\[
\prob\left(\exists z \in \wh \cT_{ \gamma',\vep_0}^{d, (s)} \setminus \cB_2^{\wt \vep_0}: \det(P_z^\delta) =0 \right) \le 2 N^{-\upeta(\gamma-\gamma')/32},
\]
for all large $N$, yielding the desired result for $d >0$. To prove the same for $d <0$, as $\det (P_z^\delta)=\det((P_z^\delta)^{\sf T})$ and $P_N^{\sf T}$ is a Toeplitz matrix with symbol $p(1/\cdot)$, we work with $(P_z^{\delta})^{\sf T}$ and proceed same as above. 
This finishes the proof. 
\end{proof}

\begin{rem}\label{rem:Cgamma}
Theorem \ref{thm:sep-spec-curve} and Lemma \ref{lem:tube-geo}(ii) imply, 
upon recalling the definitions of the tube $\wh \cT_{ \gamma',\vep_0}^{d}$ and ${\sf g}_0$,  that with probability approaching one, for any eigenvalue $z \notin \cB_1^{\wt \vep_0} \cup \cB_2^{\wt \vep_0}$ one has that 
\[
p^{-1}(z) \cap (S^1+D(0, C_\gamma N^{-1} \log N)) = \emptyset,
\]
for some appropriately chosen $C_\gamma$ such that $\lim_{\gamma \to 1} C_\gamma =0$ and $\lim_{\gamma \to \infty} C_\gamma=\infty$.  
\end{rem}

\section{Location of the bulk of the eigenvalues}\label{sec:bulk-eig}
In this section we prove that away from the bad sets,
the bulk of the spectrum of $P^Q_{N, \gamma}$ is contained in tubes of width $O(\log N/N)$ around $p(S^1)$,  with high probability, see Theorem \ref{thm-thintube}. We also prove 
in Theorem \ref{thm-thintube2}
an  upper bound on the number of eigenvalues of $P^Q_{N, \gamma}$ residing in the ball $D(x, \upa \log N/N)$, where the distance of $x$ from the spectral curve is of the order $\log N/N$ and $\upa< \infty $ is some large constant. The latter estimate will we used in the proof of Corollary \ref{cor:main}. 

Before stating the results, recall the bad sets of Definition \ref{def-badintro}, and define for $\upepsilon, \upalpha, \wt\vep_0>0$,
\begin{equation}
\label{eq-defsec7}
\Omega_\upepsilon:= \{z\in \C: {\rm dist}(z,p(S^1))>\upepsilon\}, \quad  {\mathcal N}_\upalpha:=\{ z: z\in \Omega_{\upalpha \log N/N} \setminus (\cB_1^{\wt \vep_0} \cup \cB_2^{\wt \vep_0}) \; \mbox{an eigenvalue of 
    $P^Q_{N, \gamma}$}\} .
\end{equation}
The first main result of this section is the following.
\begin{thm} 
\label{thm-thintube}
Fix $\gamma>1$ and $\wt \vep_0 >0$. Let Assumptions \ref{assump:mom} and \ref{assump:anticonc} hold. Then 
there exists a constant $C_{\ref{thm-thintube}}<\infty$, depending only on $p(\cdot)$, $\gamma$, and $\wt \vep_0$, so that
for any $\upalpha >0$,  we have
  \begin{equation}
    \label{eq-2}
 \lim_{N \to \infty}  \prob\left( |{\mathcal N}_\upalpha|\leq \frac{C_{\ref{thm-thintube}}N}{\upalpha}\right) =1.
  \end{equation}
\end{thm}
 To prove Theorem \ref{thm-thintube} we will use Jensen's formula. For an analytic function $f$, let 
  $n_f(x,r)$ denote the number of roots in a ball of radius $r$ around $x$.
  If $f(x) \ne 0$, then we have
  \begin{equation}
    \label{eq-jensen}
    \int_0^r \frac{n_f(x,s)}{s} ds=\frac1{2\pi}\int_0^{2\pi} \log|f(x+re^{i\theta})| d\theta-\log |f(x)|.
  \end{equation}
  In particular, we obtain from \eqref{eq-jensen} the bound
  \begin{equation}
    \label{eq-jensen1}
    n_f(x,ur)\leq 
    \frac{1}{2(1-u)\pi}
    \int_0^{2\pi} \log|f(x+re^{i\theta})| d\theta-\frac{1}{1-u}\log |f(x)|,
  \end{equation}
  valid for $u\in (0,1)$.
  We will apply \eqref{eq-jensen1} with $f(z)=\det(P^Q_{N, \gamma}-zI_N)$.
Introduce the function
  \begin{equation}
    \label{eq-phiinfty}
    \phi_\infty(z):= \frac1{2\pi} \int_0^{2\pi} \log |z-p(e^{i\theta})|d\theta.
  \end{equation}
  To prove Theorem \ref{thm-thintube} we will need the following two lemmas. Recall that  $P_z^\delta = P_{N, \gamma}^Q - zI_N$.
  \begin{lem}[Upper bound on determinant]
    \label{lem-detUB} Consider the setup as in Theorem \ref{thm-thintube}.
Then
    there exist constants $0 < c_{\ref{lem-detUB}}, C_{\ref{lem-detUB}}< \infty$, depending only on $\gamma, \wt \vep_0$ and  $p(\cdot)$, such that 
    \[
    \prob(\cA_{UB}^\complement) \le N^{-3},
    \]
    for all large $N$, where
    \[
    \cA_{UB} :=  \bigcap_{z \in \Omega_{1 /N} \cap (p(S^1))^{c_{\ref{lem-detUB}} \wt \vep_0} \setminus (\cB_1^{\wt \vep_0} \cup \cB_2^{\wt \vep_0})} \left\{ \log| \det(P_z^\delta)|\leq N\phi_\infty(z)+ C_{\ref{lem-detUB}} \log N \right\}.
    \]   
\end{lem}
We also need the following complementary lower bounds.
\begin{lem}[Lower bound on determinant]
  \label{lem-detLB}
 Consider the  setup of  Theorem \ref{thm-thintube}. Then,
  there exists a  constant $C_{\ref{lem-detLB}} < \infty$ depending only on $\gamma, \wt \vep_0$, and  $p(\cdot)$, such that 
  \begin{equation}\label{eq:detLB}
  \max_{z \in (p(S^1))^{c_{\ref{lem-detUB}} \wt \vep_0} \setminus (\cB_1^{\wt \vep_0} \cup \cB_2^{\wt \vep_0})} \prob(\cA_{z, LB}^\complement) \le 1/N^3,
  \end{equation}
  where
  \begin{equation*}
      \cA_{z, LB}:= \left\{ \log| \det(P_z^\delta)|\geq N\phi_\infty(z)-C_{\ref{lem-detLB}}\log N\right\}.
  \end{equation*}
\end{lem}
 Using these two lemmas, whose proofs are postponed,
 we now prove Theorem \ref{thm-thintube}. 

\begin{proof}[Proof of Theorem \ref{thm-thintube} (assuming Lemmas \ref{lem-detUB} and \ref{lem-detLB})]
 Let $R:=2\max_{z\in S^1} |p(z)|$. For $i\in \N$ such that 
$i \le i_\star :=$ 
$ \lceil \log (c_{\ref{lem-detUB}} \wt \vep_0 N/ (4 \upa \log N)) /\log 2 \rceil$, set
  \begin{multline*}
  \Omega_{\upalpha,i} := \bigg\{z\in  D(0,R)\setminus (\cB_1^{\wt \vep_0} \cup \cB_2^{\wt \vep_0}): \\
  \left(\frac{3\upalpha}{4} +2^{i-3} \upalpha\right) \cdot \frac{\log N}{N} \leq \dist(z, p(S^1))\leq \left(\frac{3\upalpha}{4}   +2^{i-2}\upalpha\right) \cdot \frac{\log N}{N}\bigg\},
  \end{multline*}
  and define $\mathcal N_{\upalpha,i}:= \mathcal N_\upalpha \cap \Omega_{\alpha,i}$. From our choice of $i_\star$ it is clear that 
  \begin{equation}\label{eq:omega-away}
  \left(D(0,R) \setminus (\cB_1^{\wt \vep_0} \cup \cB_2^{\wt \vep_0})\right) \setminus \left(\cup_{i=1}^{i_\star} \Omega_{\upalpha, i}\right) \subset \left(D(0,R) \setminus (\cB_1^{\wt \vep_0} \cup \cB_2^{\wt \vep_0})\right) \cap \Omega_{c_{\ref{lem-detUB}}\wt \vep_0/32}. 
  \end{equation}
  By \cite[Theorem 1.1]{BZ} we have that
  \[
  \lim_{N \to \infty} \prob\left(|\cN_\upalpha \cap D(0,R)^\complement| >0\right) \le \lim_{N \to \infty} \prob\left(\exists z \notin D(0,R): z \mbox{ is an eigenvalue of } P^Q_{N, \gamma}\right) =0.
  \]
On the other hand, we have from
 \cite[Corollary 2.2]{SjVo19a} or \cite[Theorem 1.2]{BPZ1} that
    \begin{align*}
 &  \lim_{N \to \infty} \prob\left(|\cN_\upalpha \cap \Omega_{c_{\ref{lem-detUB}}\wt \vep_0/32}| \ge \frac{C_0' N}{2 \upa}\right) \\
 \le &  \lim_{N \to \infty} \prob\left(\left| \left\{z \in \Omega_{c_{\ref{lem-detUB}}\wt \vep_0/32}: z \mbox{ is an eigenvalue of } P^Q_{N, \gamma}\right\}\right| \ge  \frac{C_0' N}{2 \upa}\right)=0,
  \end{align*}
  for any $\upalpha >0$. 
Hence, in light of \eqref{eq:omega-away} it suffices to show that 
\begin{equation}\label{eq:cN-upa-i}
\liminf_{N \to \infty} \prob\left( |\cN_{\upalpha, i} |= O(2^{-i} \upalpha^{-1} N) \mbox{ for all } i \in [i_\star]\right) =1. 
\end{equation}
Turning to prove \eqref{eq:cN-upa-i}, we cover $\Omega_{\upalpha, i}$ by a collection of balls 
  $D_{i,j}:= D(z_{i,j},r_i)$, $j\in N_i$,
  with $z_{i,j} \in \Omega_{\upa, i}$ and $r_i= 2^{i-4}\upalpha \cdot (\log N/N)$. As $z_{i,j} \in \Omega_{\upa, i}$ we have that $\dist(z_{i,j}, \cB_1 \cup \cB_2) \ge \wt \vep_0$. For $i \le i_\star$ we also have that $r_i \le r_{i_\star} \le \wt c_{\ref{lem-detUB}}\vep_0/32$. 
  Therefore, there exists an absolute constant $u \in (0,1)$, such that for any $i \le i_\star$ and $j \in [N_i]$,
  \[
  \dist(D(z_{i,j}, u^{-1} r_i), \cB_1 \cup \cB_2) \ge \wt 3\vep_0/4.
  \]
  On the other hand, we note that 
  \[
  (3\upa /4) \cdot \log N/N \le \dist(D(z_{i,j}, u^{-1} r_i), p(S^1)) \le c_{\ref{lem-detUB}} \wt \vep_0/4.  
  \] 
  The last two observations together imply that 
    \begin{equation}\label{eq:choose-u}
  \bigcup_{i=1}^{i_\star} \bigcup_{j=1}^{N_i} D(z_{i,j}, u^{-1} r_i) \subset \Omega_{1/{(2N)}}\cap (p(S^1))^{c_{\ref{lem-detUB}} \wt \vep_0/4} \setminus (\cB_1^{\wt \vep_0/2} \cup \cB_2^{\wt \vep_0/2}). 
  \end{equation}
  This allows us to apply Lemmas \ref{lem-detUB} and \ref{lem-detLB}. Now, using
  \eqref{eq-jensen1}, we have with 
  \[n_{i,j}:= |\{z\in D_{i,j}:
  \mbox{\rm $z$ is an eigenvalue of $P^Q_{N, \gamma}$}\}|\]
  that  
  \[ n_{i,j}\leq \frac{1}{2(1-u)\pi}\int_0^{2\pi}
  \log |\det(P^\delta-(z_{i,j}+u^{-1}r_{i,j} e^{i\theta})I_N| d\theta -
  \frac{1}{1-u} \log |\det(P^Q_{N, \gamma} -z_{i,j}I_N)|,\]
  where $u \in (0,1)$ is as above. Let $\cA_{UB}$ be as in Lemma \ref{lem-detUB} with $\wt \vep_0$ replaced by $\wt \vep_0/2$. Then, on the event $\cA_{UB} \cap \cA_{z_{i,j}, LB}$, as $\phi_\infty(z)$ is harmonic off $p(S^1)$, by \eqref{eq:choose-u} we have that 
  \[ n_{i,j}\leq (1-u)^{-1} (C_{\ref{lem-detUB}}+C_{\ref{lem-detLB}}) \log N.
  \]

   By Lemma \ref{lem:net-bd} it follows that $N_{i}$, the number of balls of radius $r_i$ needed to cover 
  $\Omega_{\upalpha,i}$ satisfies the bound
  \begin{equation}
    \label{eq-NiUB}
    N_{i} \leq C 2^{-i} \cdot\upalpha^{-1} \cdot ({N}/{\log N}),
  \end{equation}
  for some universal constant $C < \infty$. Thus, on $\cA_\star := \cA_{UB} \cap \cap_{i,j} \mathcal A_{z_{i,j}, LB}$, 
   \[ 
   |\mathcal N_{\upalpha,i}| \leq \sum_{j \in N_i} n_{i,j} \leq (1-u)^{-1} (C_{\ref{lem-detUB}}+C_{\ref{lem-detLB}}) C N \cdot {2^{-i} \upalpha^{-1}}. 
   \]
Therefore \eqref{eq:cN-upa-i} and hence the theorem follows if
we can prove $\prob(\cA_\star)=1-o(1)$. Note however that
\begin{align*}
  \prob(\cA_\star) &
\geq 1- \prob(\mathcal A_{UB}^\complement)-\sum_{i=1}^{i_\star} \sum_{j=1}^{N_i} \prob(\mathcal A_{z_{i,j},LB}^\complement) \geq 1- N^{-3}- N^{-3} \left(\sum_{i \le i_\star} N_i\right)=1-o(1),
\end{align*}
where we used \eqref{eq-NiUB} and Lemmas \ref{lem-detUB} and \ref{lem-detLB} with $\wt \vep_0$ replaced by $\wt \vep_0/2$.
\end{proof}

Using the same ideas, we next provide a local upper bound on the number of eigenvalues.

\begin{thm}\label{thm-thintube2}
Consider the setup as in Theorem \ref{thm-thintube}. Fix $ 0<\ul{\upa} < \ol{\upa} <\infty$. 
 Then, there exists a constant $C_{\ref{thm-thintube2}}< \infty$ depending only on $\gamma, \wt \vep_0$, and  $p(\cdot)$,
 such that for any $z \in (p(S^1))^{\ol{\upa} \log N/N}$,
\[
\prob\left( \left|\cN_{\ul{\upa}, \ol{\upa}, z}\right| \ge {C_{\ref{thm-thintube2}} \ul{\upa}^{-2} \ol{\upa}^2 \cdot \log N}\right) \le N^{-2}, 
\]
for all large $N$, with 
\[
\cN_{\ul{\upa},\ol{\upa}, z} := \left\{ z' \in D(z, \ol{\upa} \log N/N) \cap \Omega_{\ul{\upa} \log N/N} \setminus (\cB_1^{\wt \vep_0} \cup \cB_2^{\wt \vep_0}): z' \mbox{ is an eigenvalue of } P^Q_{N, \gamma}\right\}.
\] 

\end{thm}

\begin{proof}
We follow the same strategy as in the proof of Theorem \ref{thm-thintube}. Fix $z \in (p(S^1))^{\ol{\upa} \log N/N}$. For $i \in [\ol{i}_\star]$, where $\ol{i}_\star := \lceil \log (8 \ol{\upa}/\ul{\upa}) /\log 2 \rceil$, we set
  \begin{multline*}
\ol{\Omega}_{\ul{\upalpha},\ol{\upa}, i} := \bigg\{z'\in  D(z, \ol{\upa} \log N/N) \cap \Omega_{\ul{\upa} \log N/N} \setminus (\cB_1^{\wt \vep_0} \cup \cB_2^{\wt \vep_0}): \\
  \left(\frac{3\ul{\upalpha}}{4} +2^{i-3} \ul{\upalpha}\right) \cdot \frac{\log N}{N} \leq \dist(z', p(S^1))\leq \left(\frac{3\ul{\upalpha}}{4}   +2^{i-2}\ul{\upalpha}\right) \cdot \frac{\log N}{N}\bigg\}. 
  \end{multline*}
Since $\dist(z, p(S^1)) \le \ol{\upa} \log N/N$, we notice that
\[
  \bigcup_{i=1}^{\ol{i}_\star} \ol{\Omega}_{\ul{\upa}, \ol{\upa}, i} \supset D(z, \ol{\upa} \log N/N) \cap \Omega_{\ul{\upa} \log N/N} \setminus (\cB_1^{\wt \vep_0} \cup \cB_2^{\wt \vep_0}). 
\]  
Thus, it suffices to bound the number of eigenvalues in $\cup_{i=1}^{\ol{i}_\star} \ol{\Omega}_{\ul{\upa}, \ol{\upa}, i}$.

As in the proof of Theorem \ref{thm-thintube} we cover $\ol{\Omega}_{\ul{\upa}, \ol{\upa}, i}$ by a collection of balls $D_{i,j} = D(z_{i,j}, \ol{r}_{i})$, $j \in \ol{N}_i$, $z_{i,j} \in \ol{\Omega}_{\ul{\upa}, \ol{\upa}, i}$ and $\ol{r}_i= 2^{i-4} \ul{\upa} \cdot (\log N/N)$. By our choice of $\ol{i}_\star$ we procure a $u \in (0,1)$ such that
\[
 D(z_{i,j}, u^{-1} \ol{r}_i) \subset (p(S^1))^{3 \ol{\upa} \log N/N} \cap \Omega_{1/(2N)} \setminus (\cB_1^{\wt \vep_0/2} \cup \cB_2^{\wt \vep_0}), \quad \mbox{ for all } j \in \ol{N}_i \mbox{ and } i \in [\ol{i}_\star].
\]
Hence, arguing similarly as in the proof of Theorem \ref{thm-thintube} we deduce that the number of eigenvalues of $P^Q_{N, \gamma}$ in $D_{i,j}$ is $O(\log N)$ for all $j \in \ol{N}_i$ and $i \in [\ol{i}_\star]$, on a set with probability at least $1 - 1/N^2$. 

To complete the proof we use a volumetric argument, yet again, to find that
\[
\ol{N}_i = O\left(2^{-2i} \cdot \left(\frac{\ol{\upa}}{\ul{\upa}}\right)^2\right). 
\]
This, together with the choice of $\ol{i}_\star$ indeed yields the desired bound on $\cN_{\ul{\upa}, \ol{\upa}, z}$. 
\end{proof}

We now turn to the proof of Lemma \ref{lem-detUB}. Recall $\phi_\infty(z)$, see \eqref{eq-phiinfty}, and define 
\[
\wt \det_k(z)  := \frac{\det_k(z)}{a_{-N_-}^N \prod_{i=1}^{m_-}\eta_i(z)^N} = \frac{\det_{k}(z)}{\exp(N\phi_\infty(z))}, 
\]
where the last equality follows upon recalling that $\{-\eta_i(z)\}_{i=1}^{m_-}$ are the roots of $p_z(\cdot)=0$ that are greater than or equal to one in modulus, and from the fact that
\[
\frac{1}{2\pi}\int_0^{2\pi} \log |\eta - e^{i\theta}| d\theta = \left\{\begin{array}{ll}
\log |\eta| & \mbox{ if } |\eta| \ge 1,\\
0 & \mbox{ otherwise}. 
\end{array}
\right.
\]
In the lemma below we derive a bound on the supremum of the second moment $\wt \det_k(\cdot)$.

\begin{lem}\label{lem:sec-mom-wt-det} Let  Assumption \ref{assump:mom} hold. 
Fix $\wt \vep_0, \vep_0, \vep_0' >0$ such that $\vep_0'/\vep_0$ is sufficiently small. Fix $k \in [N] \cup \{0\}$ and  $d \ge 0$. Then for any $\vep \le \vep_{\ref{lem:tube-blow-up}} \vep_0'/2$,
 we have
\[
\sup_{z \in \left(\cT^{d,(s)}_{\vep_0',\vep_0}\setminus \cB_2^{\wt \vep_0}\right)^\vep} \E \left[ \left| \wt \det_k(z)\right|^2\right] = O(1), \quad s \in \{1,2\}. 
\]
\end{lem}

\begin{proof}
By Lemma \ref{lem:tube-blow-up} it suffices to show that
\begin{equation}\label{eq:sec-mom-wt-det}
\sup_{z \in \cT^{d,(s)}_{\vep_0',\vep_0}} \E \left[ \left| \wt \det_k(z)\right|^2\right] = O(1). 
\end{equation}
We begin with the proof of  \eqref{eq:sec-mom-wt-det} for $k \in [N]$ and $s =1$. Recall from \eqref{eq:wh-m1} that for this choice of $s$,
 we have that $\wh m_1 = m_-$. Since $Q$ satisfies Assumption \ref{assump:mom}, we then observe that 
\begin{equation}\label{eq:sec-mom-wt-det-a}
\E\left[\left|\wt \det_{k}(z)\right|^2 \right] =  N^{-2 \gamma k}\sum_{\substack{X, Y \subset [N]\\ |X|=|Y|=k}} |\wh \gD(X,Y,z )|^2 \cdot  |\eta_{\wh m_1+1}(z)|^{2{\sf g}_0 N}  \cdot \E\left[\left|\det \left(Q\left[X; Y \right]\right)\right|^2 \right].
\end{equation}
By Lemma \ref{lem:comb-bound-tube-1}(i)-(ii), for any $k \in [N]$, as $|\eta_{\wh m_1+1}(\cdot)| \le 1$ on $\cT^{d, (1)}$, we have that 
\begin{equation}\label{eq:sec-mom-wt-det1}
|\wh \gD(X,Y,z )|^2 \cdot  |\eta_{\wh m_1+1}(z)|^{2{\sf g}_0 N} \le C^{\wt m (k+N_+)},
\end{equation}
for some constant $C< \infty$. Since the second moment of the determinant of any $k \times k$ matrix with entries satisfying Assumption \ref{assump:mom} is $k!$ and the number of ways one can choose two subsets of $[N]$ that are cardinality $k$ is $\binom{N}{k}^2$ the claimed upper bound, for $s=1$, now follows upon using that $\gamma >1$ and plugging the bound \eqref{eq:sec-mom-wt-det1} in \eqref{eq:sec-mom-wt-det-a}. 

To prove \eqref{eq:sec-mom-wt-det} for $s=2$ we observe that $\wh m_1 + {\sf g}_0 = m_-$, and thus $\wt \det_k(\cdot) = N^{-\gamma d} \wh \det_k(\cdot)$ on $\cT^{d, (2)}$. Hence, the bound for any $k \in [N]$ is immediate from Lemmas \ref{lem:sec-mom-small-k} and \ref{lem:sec-mom-large-k}(iii). 

To  prove \eqref{eq:sec-mom-wt-det} for $k=0$ we recall \eqref{eq:wh-m1} to notice that $\wh m_1 +{\sf g}_0 =m_-$ on $\cT^{d, (2)}$, while on $\cT^{d, (1)}$ we have $\wh m_1=m_-$. Therefore, as $|\wh \eta_{\wh m_1+1}(\cdot)| \le 1$ on $\cT^{d, (1)}$, we find that
\[
\left|\prod_{j=1}^{\wh m_1 +{\sf g}_0} \eta_j(\cdot)\right| \le \left|\prod_{j=1}^{m_-} \eta_j(\cdot)\right|
\]
on $\cT^{d, (s)}$ for any $s \in\{1,2\}$. This, in turn, implies that $|\wt \det_k(\cdot)| \le |\wh \det_k(\cdot)|$ on $\cT^{d, (s)}$ (recall \eqref{eq:fancyK-1} and \eqref{eq:hatP-k}). Thus, the bound for $k=0$ is immediate from Lemma \ref{lem:det-0}(i)-(ii).
This completes the proof of the lemma. 
\end{proof}

\begin{proof}[Proof of Lemma \ref{lem-detUB}]
Set $\vep_0 = c_{\ref{lem:tube-geo}}\wt \vep_0/2$. Fix $\vep_0' >0$ so that $\vep_0'/\vep_0$  is small enough for Lemma \ref{lem:sec-mom-wt-det} to hold. Now set $c_{\ref{lem-detUB}} = C_{\ref{lem:tube-geo}}^{-1}  c_{\ref{lem:tube-geo}}\cdot (\vep_0'/\vep_0)/2$.  Applying Lemma \ref{lem:tube-geo}(iii) and Markov's inequality we find that it suffices to show that for all $d \in \Z$
\begin{equation}\label{eq:detUB-mom-sup}
\max_{s \in \{1,2\}}\E\left[\sup_{z \in \cT^{d, (s)}_{\vep_0',\vep_0}\cap \Omega_{1/N} \setminus \cB_2^{\wt \vep_0} } |\wt \det_k(z)|^2 \right] = O(N^2). 
\end{equation}
The case $d < 0$ can be dealt by consider the transpose of $P_N$. Hence, we will prove \eqref{eq:detUB-mom-sup} only for $d \ge 0$.
Toward this end, we use Lemma \ref{lem:sec-mom-wt-det} and ideas from the proof of Lemma \ref{lem:der-bd}. 
Fix $d \ge 0$, $s \in \{1,2\}$, $\vep_0'' = 1/(4N)$, and $\beta \in (0,1/8)$. Let $\D, \vep, b$, and $\D_i^\beta$, for $i \in [b]$ be as in the proof of Lemma \ref{lem:der-bd}, where now the centers $z_i$ of the disks $\D_i^\beta$ are restricted to be inside  $\wt \T := \cT^{d, (s)}_{\vep_0',\vep_0}\cap \Omega_{1/N}$. Notice that $\cup_{i=1}^b \D_i^\beta \supset \wt \T$.

Applying Cauchy's integral formula for smooth functions, and proceeding as in the proof of \eqref{eq:cauchy-der-l2} we obtain that  
\begin{equation}\label{eq:cauchy-l2}
\E \left[ \sup_{z \in \D_i^\beta} \left| \Xi(z)\right|^2 \right] = O(1) \cdot \sup_{w \in \D_i^{3\beta}} \E\left[|\Xi(w)|^2\right],
\end{equation}
for any random holomorphic function $\Xi: \D \mapsto \C$. To complete the proof it remains to argue that the map $z \mapsto \wt \det_k(z)$ is analytic on $\D_i^{3\beta}$ for each $i \in [b]$. 

To this end, observe that $\D_i^{3\beta} \subset \Omega_{1/(2N)}$. Hence, the map $z \mapsto m_-(z)$ is constant on $\D_i^{3\beta}$. From the definition of the tubes, and Lemma \ref{lem:tube-blow-up}(ii) it further follows that 
\begin{equation}\label{eq:analytic}
|\eta_{m_-}(\cdot)|  \ge  |\eta_{m_-+1}(\cdot)| +\vep_0'/4 \quad \mbox{ on } \D_i^{3\beta} \subset \wt \T^{3\beta}.
\end{equation}
Therefore, arguing as in the proof of Lemma \ref{lem:der-bd} we deduce that the map $z \mapsto \wt \det_k(z)$ is indeed analytic on $\D_i^{3\beta}$. Thus, using the bound on $b$ (see \eqref{eq:bound-b}),  Lemma \ref{lem:sec-mom-wt-det}, and \eqref{eq:cauchy-l2} the proof of \eqref{eq:detUB-mom-sup} is completed. 
\end{proof}

We now turn to the proof of Lemma \ref{lem-detLB}. The key ingredient will be the anti-concentration bound derived in Lemma \ref{lem:anti-conc}. 
We need to argue that $\det(P_z^\delta)$ admits certain specific representation so that Lemma \ref{lem:anti-conc} is applicable. 
To carry out this step we need the following couple of notation. 

Fix $X_0, Y_0 \subset [N]$ such that $|X_0| = |Y_0| =k_0$ for some $k_0 \le N$. For $k \ge k_0$ define 
\begin{multline}\label{eq:det-decompose-1-new}
{\det}_k^{X_0, Y_0}(z) 
:= N^{-\gamma k}  \sum\nolimits_{\substack{X, Y \subset [N]\\ |X|=|Y|=k \\ X \supset X_0, Y \supset Y_0}} (-1)^{\sgn(\sigma_{X}) \sgn(\sigma_{Y})} \cdot (-1)^{\wh \sgn(\sigma_{X}) +\wh \sgn(\sigma_{Y})} \gD(X,Y)  \\
\cdot \det (Q[X\setminus X_0; Y\setminus Y_0]),
\end{multline}
where $\wh \sgn(X)$ and $\wh \sgn(Y)$ are the signs of the permutations on $X$ and $Y$ that place all elements of $X_0$ and $Y_0$ before those of $X \setminus X_0$ and $Y\setminus Y_0$, respectively, but preserves the order of the elements in each of those individual sets. Also define
\[
\wt \det_{k}^{X_0, Y_0}(z) := \frac{\det_k(z)}{\exp(N\phi_\infty(z))}. 
\]

\begin{lem}\label{lem:lbd-det}
Consider the setup as in Lemma \ref{lem-detLB}. Fix $d \ge 0$ and $s \in \{1,2\}$. Let $X_0, Y_0 \subset[N]$ with $|X_0| = |Y_0| =d$.  The following moment bounds hold for ${\det}_k^{X_0, Y_0}(z)$:

\begin{enumerate}
\item[(i)] For any $k$ such that $d \le k \le N$ we have
\[
 \sup_{z \in  \cT^{d, (s)}_{\vep_0',\vep_0}} \E\left[\left|\wt \det_{k}^{X_0, Y_0}(z)\right|^2 \right] \le C_{\ref{lem:lbd-det}}^{\wt m(k+N_+)} N^{-2\gamma k + 2(k-k_0)},
\]
where $C_{\ref{lem:lbd-det}}< \infty$ is some constant. 

\item[(ii)] Fix $h, K_0 \in \N$. Then, for any $k \in \N$ such that $d \le k \le K_0 - N_+$ we have
\begin{equation}\label{eq:high-mom-tt22}
 \sup_{z \in  \cT^{d, (s)}_{\vep_0',\vep_0}}\E\left[\left|\wt{\det}_k^{X_0, Y_0}(z)\right|^{2h} \right] \le C_{\ref{lem:lbd-det}}^{\wt m K_0 h} \cdot (K_0 h)^{8 K_0 h} \cdot \gC_{2K_0h}\cdot   N^{-2 \gamma h k + 2h (k-k_0)}.
\end{equation}
\end{enumerate}
\end{lem}

\begin{proof}
The proof of part (i) is similar in nature to that of Lemma \ref{lem:sec-mom-wt-det}. Hence, details are omitted. 


To prove part (ii) we employ the same combinatorial argument as in the proof Lemma \ref{lem:high-mom-small-k}. Indeed, similarly to \eqref{eq:high-mom-1}, one can obtain an analogous expression for the $(2h)$-th moment of the absolute value of $\wt \det_k^{X_0, Y_0}(z)$. 

Then, using that that entries of $Q$ are independent and possess zero mean, one observes that only partitions such that each block has size at least two need to be summed. This forces the number of such partitions to be at most $N^{2(k- d)h}$. Now, use the bound \eqref{eq:sec-mom-wt-det1} for $s=1$, while for $s=2$ use Lemma \ref{lem:comb-bound-tube-1}(i), and proceed as in the proof of Lemma \ref{lem:high-mom-small-k}. This yields the desired bound. 
Further details are omitted. 
\end{proof}

We now use Lemmas \ref{lem:anti-conc} and \ref{lem:lbd-det} to derive Lemma \ref{lem-detLB}. 

\begin{proof}[Proof of Lemma \ref{lem-detLB}]
Similar to the proof of Lemma \ref{lem-detUB} we notice that it suffices to prove the probability bound in \eqref{eq:detLB} for $z \in \cT^{d, (s)}_{\vep_0', \vep_0}$ for $d \ge 0$, $s \in \{1,2\}$, and some appropriately chosen $\vep_0'$ and $\vep_0$. 

Fix $d \ge 0$ and $s \in \{1,2\}$. Set $\wt X_\star = [N]\setminus [N-d]$ and $\wt Y_\star = [d]$.  
We claim that for any $k \ge d$
\begin{equation}\label{eq:detLB-1}
\wt{\det}_k(z) = \wt{\det}_k^{\wt X_\star, \wt Y_\star}(z) \cdot \det (Q[\wt X_\star; \wt Y_\star]) + \reallywidecheck{\det}_k(z), 
\end{equation}
where $\reallywidecheck{\det}_k(z)$ is a homogeneous polynomial of degree $k$ in the entries of $Q$ such that the degree of each individual entry is at most one, and the total degree of the entries of $Q[\wt X_\star; \wt Y_\star]$ is strictly less than $d$. To see this claim we recall from \eqref{eq:det-decompose-1} that 
\begin{eqnarray}\label{eq:det-decompose-new}
{\det}_k(z) 
&= &\sum_{\substack{X, Y \subset [N]\\ |X|=|Y|=k}} (-1)^{\sgn(\sigma_{X}) \sgn(\sigma_{Y})} \gD(X,Y)  \cdot N^{-\gamma k} \cdot \det (Q[X; Y]).
\end{eqnarray}
Now we split the sum into two parts depending on whether $X \supset \wt X_\star$ and $Y \supset \wt Y_\star$. If either $X \not\supset \wt X_\star$ or $Y \not\supset \wt Y_\star$ then the corresponding term in the \abbr{RHS} of \eqref{eq:det-decompose-new} is indeed a polynomial in the entries of $Q$ such that the total degree of the entries of $Q[\wt X_\star; \wt Y_\star]$ is strictly less than $d$. Now fix $X, Y \subset [N]$ such that $X \supset \wt X_\star$ and $Y \supset \wt Y_\star$. Expand $\det (Q[X; Y])$ by writing it as a sum over permutations $\uppi: X \mapsto Y$. The sum over permutations $\uppi$ such that $\uppi(\wt X_\star)= \wt Y_\star$ is indeed a product of $\det(Q[X \setminus \wt X_\star; Y\setminus \wt Y_\star])$ and $\det (Q[\wt X_\star; \wt Y_\star])$ (upto some signs). The sum over the rest of the permutations is again a polynomial such that the total degree of the entries of $Q[\wt X_\star; \wt Y_\star]$ is strictly less than $d$. Combining these observations and upon recalling the definition of $\wt \det_k^{\wt X_\star, \wt Y_\star}(z)$ we arrive at  \eqref{eq:detLB-1}. 

On the other hand, $\wt \det_k(z)$ is a polynomial of total degree less than $d$, for any $k <d$. Since $\det(P_z^\delta) = \sum_{k=0}^N \det_k(z)$, this together with \eqref{eq:detLB-1} imply that 
\[
\frac{\det(P_z^\delta)}{\exp(N\phi_\infty(z))} = \sum_{\cI \subset [d]} Z_{\cI} \prod_{i \in \cI} U_i ,
\]
where $\{U_i\}_{i=1}^d$ are the diagonal entries of the sub matrix $Q[\wt X_\star; \wt Y_\star]$, $\{Z_\cI, \cI \subset [d]\}$ is a collection of random variables that are independent of $\{U_i\}_{i=1}^d$, and 
\[
Z_{[d]} = \sum_{k \ge d} \wt{\det}_k^{\wt X_\star, \wt Y_\star}(z). 
\]
We now apply Lemma \ref{lem:anti-conc}, and obtain that
\begin{equation}\label{eq:detLB-2}
\prob\left(\frac{\det(P_z^\delta)}{\exp(N\phi_\infty(z))} \le N^{-(4+\gamma d)}\right)  \le N^{-4} + \prob\left(|Z_{[d]}| \le \frac{\ol{c}_\star}{2}\cdot N^{-\gamma d}\right),
\end{equation}
for all large $N$, and any $\ol{c}_\star >0$. It remains to bound the probability that $Z_{[d]}$ exceeds $(\ol{c}_\star/2) \cdot N^{-\gamma d}$ in absolute value. 

Turning to do that we set $K_0=\lceil \frac{5}{(\gamma -1)}\rceil +d +N_+$. Apply Lemma \ref{lem:lbd-det}(i) and Markov's inequality to deduce that 
\[
\prob\left( \left|\wt{\det}_k^{\wt X_\star, \wt Y_\star}(z)\right| \ge N^{-(2+\gamma d)}\right) \le N^{-5},
\]
for any $k \ge K_0 - N_+$. 
Apply Lemma \ref{lem:lbd-det}(ii) with $h=\lceil \frac{6}{(\gamma-1)}\rceil$ and Markov's inequality to further derive that
\[
\prob\left( \left|\wt{\det}_k^{\wt X_\star, \wt Y_\star}(z)\right| \ge N^{-\frac{(\gamma-1)}{2} -\gamma d}\right) \le N^{-5},
\]
for any $k \in [N]$ such that $d < k < K_0 - N_+$. Hence, by a union bound we obtain that 
\begin{equation}\label{eq:detLB-3}
\prob\left( N^{\gamma d} \left|\sum_{k > d} \wt{\det}_k^{\wt X_\star, \wt Y_\star}(z)\right| \ge N^{-1} + N^{-\frac{(\gamma-1)}{4}}\right) \le N^{-4}. 
\end{equation}
Equipped with \eqref{eq:detLB-2} and \eqref{eq:detLB-3}, and upon recalling the definition $Z_{[d]}$ we notice that to complete the proof of this lemma it is now enough to show that 
\begin{equation}\label{eq:detLB-4}
N^{\gamma d} \cdot \inf_{z \in \cT^{d, (s)}_{\vep_0', \vep_0} }\left| \wt{\det}_d^{\wt X_\star, \wt Y_\star}(z)\right| = \inf_{z \in \cT^{d, (s)}_{\vep_0', \vep_0} } \frac{| \gD(\wt X_\star, \wt Y_\star, z)|}{|a_{-N_-}|^N\prod_{j=1}^{m_-} |\eta_j(z)|^N}\ge \ol{c}_\star
\end{equation}
for some $\ol{c}_\star >0$. To prove \eqref{eq:detLB-4} we argue as in the proof of Lemma \ref{lem:det-0}(b). The only difference is that one needs to use \eqref{eq:analytic} instead of \eqref{eq:ratio-small-zeta}. The rest of the argument is the same. This completes the proof of this lemma.  
\end{proof}

We end this section with the proof of Theorem \ref{theo-location}. It is immediate from Theorems \ref{thm:no-outlier}, \ref{thm:sep-spec-curve}, and \ref{thm-thintube}.

\begin{proof}[Proof of Theorem \ref{theo-location}]
We use
$\wh\cN_{\wh \Omega}$ to denote the number of eigenvalues of $P_{N,\gamma}^Q$ in $\wh \Omega$. Fix $\wt \vep_0>0$ such that the Lebesgue measure of the set $p^{-1}(\cB_p^{\wt \vep_0}\cap p(S^1))$ is less than $\mu/(16 \pi)$ (recall Definition \ref{def-badintro}). Then, by \cite[Theorem 1.2]{BPZ1} (see also \cite[Theorem 2.1]{SjVo19a}) it follows that $\prob(\wh \cN_{\wh \Omega_1} \ge \mu N/4) \to 0$ as $N \to \infty$, where $\wh \Omega_1:=\cB_p^{\wt \vep_0}$. On the other hand, by Theorems \ref{thm:no-outlier},  \ref{thm:sep-spec-curve}, and Lemma \ref{lem:tube-geo}(iii)-(v) it follows that there exists some $0<\wh c_\gamma < \infty$ such that  $\prob(\wh \cN_{\wh \Omega_2} \, >0) \to 0$ as $N \to \infty$, where $\wh \Omega_2:= (p(S^1))^{\wh c_\gamma \log N/N} \setminus \wh \Omega_1$. 
Finally, upon choosing $\upa= 4 C_{\ref{thm-thintube}}/\mu$, and applying Theorem \ref{thm-thintube} with this $\upa$ we deduce that $\prob(\wh \cN_{\wh \Omega_3} \ge N \mu /4) \to 0$ as $N \to \infty$, where $\wh \Omega_3:= \Omega_{\upa \log N/N}$ (recall \eqref{eq-defsec7}). 

Setting $\vep_{\ref{theo-location}} = \wt \vep_0$ and $C_{\ref{theo-location}} = \upa \vee \wh c_\gamma^{-1}$ and taking a union bound over 
the events that $\wh \cN_{\wh \Omega_i}\ge N \mu/4$ 
for $i=1,2,3$, completes the proof. 
\end{proof}

\section{Resolvent estimates close to the spectral curve}
\label{sec-resolvent}
Given a symbol $p(\zeta)-z$, with $z\in \C$, as in 
\eqref{int1},  we will see that there is an important 
distinction between the regime where the spectral parameter $z$ is inside a loop of $p(S^1)$, 
i.e. where the winding number $\mathrm{ind}_{p(S^1)}(z)$ of the curve $p(S^1)$ around $z$ 
is non-zero, and the regime where $\mathrm{ind}_{p(S^1)}(z)=0$. In the former case,
we 
will provide in Section \ref{sec:QM}
a construction of quasimodes, which are approximate 
singular vectors, for finite Toeplitz matrices $P_N-z$. In regions where the
winding number 
is zero, we will be interested in obtaining resolvent estimates. The key difficulty in both regimes 
will be in obtaining sufficiently good estimates when $z$ is allowed to be at an $N$-dependent 
distance from $p(S^1)$.
\par
In this section will prove the following estimate on the resolvent of a parameter 
dependent Toeplitz matrix which, while being also of independent interest, 
will be an 
essential ingredient in proving a spectral gap for the small singular values, see Section 
\ref{sec:SmallSGValues}.
\begin{thm}\label{thm:resBound}
Let $\Omega'\Subset \C$ be a non-empty relatively compact open set. 
Let $N\in\N$, let $\Omega_N\Subset \Omega'$ be a family of non-empty relatively compact open  sets. 
Let $q_z$, $z\in \Omega'$, be a family of Laurent 
polynomials as in \eqref{lp1}, with $\mathfrak{N}_{z,\pm}=\mathfrak{N}_{\pm} \geq 0$ independent of 
$z$ in $\Omega_N$ and satisfying \eqref{lp1.1}, and  with 
coefficients $q_{z,n}\in \C$ for $-\mathfrak{N}_{-}\leq n \leq \mathfrak{N}_{+}$. Suppose that:
\begin{itemize}
	\item There exists a constant $0<C<\infty$ such that for all $z\in \Omega';$
		\begin{equation}\label{re0.1}
			 q_{z,\mathfrak{N}_{+}}\neq 0,\quad
			 | q_{z,-\mathfrak{N}_{-}}| \geq 1/C, \quad | q_{z,n}| \leq C, \text{ for } 
			-\mathfrak{N}_{-}\leq n \leq \mathfrak{N}_{+}.
	\end{equation}
	\item All roots $\zeta_{z}$ of $q_z(\zeta)$ are simple, for all $z\in \Omega'$.
	\item There exists a constant $0<C<\infty$ such that for all $z\in \Omega'$ and for any two 
	distinct roots $\zeta_z \neq \omega_z$ we have 
	\begin{equation}\label{re1.1}
		| \zeta_z - \omega_z| \geq 1/C, 
	\end{equation}
	and for any root $\zeta_z$
	\begin{equation}\label{re1.2}
		0 <  |\zeta_{z}|\leq  C.
	\end{equation}
	\item $0\notin q_z(S^1)$ and  there exists a constant $C_0>0$ such that for $N>0$ large enough 
	\begin{equation}\label{re0.3}
	\dist ( |\zeta_{z}|,1)  \geq C_0 \frac{\log N}{N}, \quad \text{for all } z\in \Omega_N, \text{ and all roots }
	\zeta_z.
	\end{equation}
	\item There exists an $m_0 \geq 0$ (unrelated to the constants $m_0$ in \eqref{at4.2} and \eqref{at4.1}) 
and constants $C_1>1,C_2>0$  such that for 
	all $N\in \N$ large enough and all $z\in \Omega_N$ we have that the roots of $q_z(\zeta)$ 
	inside $D(0,1)$, of total number $m_{z,+}\geq m_0$, satisfy
\begin{equation}\label{re0.3b}
	 |\zeta_1^+|\leq \cdots \leq |\zeta_{m_{z,+}-m_0}^+| < 1/C_1 <  1- 
	C_2 \frac{\log N}{N} \leq  |\zeta_{m_{z,+}-m_0+1}^+|\leq \cdots \leq |\zeta_{m_{z,+}}^+| .
	\end{equation}
       \item Let $\widehat{m}_0\in\{0,m_0\}$ and suppose that
       for all $N\in \N$ and 
       all $z\in \Omega_N$ we have $\mathrm{ind}_{q_z(S^1)}(0) = \widehat{m}_0$.
\end{itemize}
Then there exists a constant $0<\wt C< +\infty$ such that for $N>0$ large enough and all $z\in \Omega_N$,
\begin{equation}
\label{re0.4}
	\| P_N(q_z)^{-1} \| \leq \wt C\frac{N^{1+C_2 \Theta(\widehat m_0)}}{\log N}.
\end{equation}
\end{thm}
\begin{rem}
	It will be clear from the proof that \eqref{re0.4} holds when $\mathrm{ind}_{q_z(S^1)}(0) =0$ even without 
	the structural assumption \eqref{re0.3b}. Furthermore, notice that in Theorem \ref{thm:resBound} we 
	only consider the case of non-negative winding numbers. This is all that we need in the sequel, however 
	one could prove the same result in the case of negative winding numbers provided that one assume 
	a similar assumption as \eqref{re0.3b} for the roots in $\C\backslash \overline{D(0,1)}$. 
\end{rem}
\begin{proof}[Proof of Theorem \ref{thm:resBound}]
Throughout the proof, we will assume that $N>0$ is sufficiently large so that the assumptions of the theorem 
hold. Moreover, all $O(1)$ terms and constants are understood to be uniform in $N>0$ and $z\in \Omega_N$, 
without us mentioning it explicitly at each occurrence. In fact the error terms will only depend on the constants 
in the hypothesis of the theorem and global constants. 
\par
1. Let $z\in\Omega_N$. Since $q_z$ satisfies the assumptions of Case 1 of 
Lemma \ref{lem:root1}, we may order the roots of $q_z(\zeta)$ as in \eqref{at4.2.0}, and using also \eqref{re1.2} and \eqref{re0.3},
\begin{equation}\label{re2.1a}
0<  |\zeta_1^+|\leq \cdots \leq |\zeta_{m_{z,+}}^+| < 1 < |\zeta_1^-| \leq \cdots \leq |\zeta_{m_{z,-}}^-| \leq C. 
\end{equation}
\par
Since $\mathrm{ind}_{q_z(S^1)}(0) = \widehat{m}_0$ for all $z\in \Omega_N$, 
it follows from a similar computation as in \eqref{at9.0} that 
$m_{z,\pm}=\mathfrak{N}_{\pm}\pm\widehat{m}_0$ are independent of $N$ and $z$.  
Keeping in mind that $m_{z,\pm}$ are independent of $z$, but that the coefficients 
$q_{z,n}$ do depend on $z$, we suppress from now on the $z$ subscript of $q_{z}$, 
$q_{z,n}$ and $m_{z,\pm}$. Hence, we are interested in the symbol 
\begin{equation}\label{re2.1}
	q(\zeta) = \sum_{-(m_- +\widehat{m}_0)}^{m_+-\widehat{m}_0} q_j \zeta^{-j},
\end{equation}
where by \eqref{re0.1}
\begin{equation}\label{re2.1b}
q_{m_{+}- \widehat{m}_0} \neq 0, \quad 
| q_{- (m_{-}+ \widehat{m}_0)}| \geq 1/C
\end{equation}
2. We turn to estimating $\| P_N(q)^{-1} \|$. We begin by inverting 
$\mathrm{Op}(q): \ell^2(\Z) \to \ell^2(\Z)$, see \eqref{c1},  which is a convolution operator. We first 
search for a fundamental solution $E:\Z\to \C$ to 
 \begin{equation}\label{re0}
	\mathrm{Op}(q) E = \delta_0 \quad \text{on } \Z,
 \end{equation}
where $\delta_0(\nu) = \delta_{0, \nu}$, see Section \ref{sec:notation} for  the Dirac notation. Putting
\begin{equation}\label{re1}
		E(n) = \mathcal{F}^{-1}\! \left(\frac{1}{q}\right)\!(n) = 
		\frac{1}{2\pi} \int_{0}^{2\pi} \frac{1}{q(\e^{i\xi})}\e^{in\xi} d\xi. 
 \end{equation}
it follows from \eqref{a5} that $E$ solves \eqref{re0}. Thus, 
\begin{equation}\label{re2}
\begin{split}
		&\mathrm{Op}(q)E*v = v, \quad v\in \ell^2_{\mathrm{comp}}(\Z) \\
		&E* \mathrm{Op}(q) u = u, \quad u\in \ell^2_{\mathrm{comp}}(\Z),
\end{split}
 \end{equation}
where $*$ denotes the convolution on $\Z$ and $\ell^2_{\mathrm{comp}}$ denotes 
compactly supported functions in $\ell^2$. 
\par
 Notice that we may factor \eqref{re2.1} as 
\begin{equation}\label{re3}
	q(\zeta) 
	= q_{-(m_-+\widehat{m}_0)} \prod_{j=1}^{m_+-\widehat{m}_0}( 1 - \zeta^+_j/\zeta )
	 \prod_{j=m_+-\widehat{m}_0+1}^{m_+}( \zeta - \zeta^+_j ) \prod_{k=1}^{m_-} (\zeta - \zeta_k^-).
 \end{equation}
Performing the change of variables $\e^{i\xi} = \zeta \in S^1$ in \eqref{re1}, we 
get by shrinking $S^1$ to $0$, the residue theorem and \eqref{re3} that for $n\geq 1$
\begin{equation}\label{re4}
\begin{split}
		E(n) = \frac{1}{2\pi i} \int_{S^1} \frac{\zeta^{n-1}}{q(\zeta)} d\zeta
		= \sum_{r=1}^{m_+} 
		\frac{ (\zeta^+_r)^{n-1 + m_+-\widehat{m}_0}}
		{q_{-(m_-+\widehat{m}_0)}  \prod_{j\neq r}( \zeta^+_r - \zeta^+_j ) \prod_{k=1}^{m_-} (\zeta^+_r - \zeta_k^-)}
		,
\end{split}
 \end{equation}
 Similarly, we get that 
\begin{equation}\label{re4.1}
\begin{split}
		E(0) 
		= \sum_{r=1}^{m_+} 
		\frac{ (\zeta^+_r)^{-1 + m_+-\widehat{m}_0}}
		{q_{-(m_-+\widehat{m}_0)}  \prod_{j\neq r}( \zeta^+_r - \zeta^+_j ) \prod_{k=1}^{m_-} (\zeta^+_r - \zeta_k^-)}
		+ \frac{\delta_{m_+,\widehat{m}_0}}{q_{-(m_-+\widehat{m}_0)} \prod_{j=1}^{m_+}( - \zeta^+_j) \prod_{k=1}^{m_-} ( - \zeta_k^-)}.
\end{split}
 \end{equation}
By  deforming $S^1$ to $|\zeta|=R$, $R\to \infty$, we obtain similarly that
 for $n\leq -1$,
\begin{equation}\label{re5}
		E(n) 	= -\sum_{r=1}^{m_-} 
		\frac{ (\zeta^-_r)^{n-1+ m_+-\widehat{m}_0}}
		{q_{-(m_-+\widehat{m}_0)}  \prod_{j=1}^{m_+}( \zeta^-_r - \zeta^+_j) \prod_{j\neq r} (\zeta_r^- - \zeta_j^-)}.
 \end{equation}
Combining \eqref{re4}-\eqref{re5}, \eqref{re1.1}, 
\eqref{re2.1a}, \eqref{re0.1} and \eqref{re2.1b}, we have that 
\begin{equation}\label{re6}
		|E(n)| = O(1)\left\{\begin{array}{ll}
			|\zeta_{m_+}^+|^{n-1}, & n\geq 1,\\
			1, &n=0,\\
			|\zeta_1^-|^{n-1}, & n \leq -1.
		\end{array}\right.
 \end{equation}
Here, we 
replace $(\zeta_1^-)^{-1}$ resp. $\zeta_{m_+}$ with $0$ when 
$m_-=0$ resp. $m_+=0$. Thus, summing the geometric series and 
using \eqref{re0.3} we get that 
\begin{equation}\label{re6.1}
		\| E \|_{\ell^1} = O(1) \frac{N}{\log N},
 \end{equation}
and similarly that 
\begin{equation}\label{re6.2}
		\| E \|_{\ell^2} = O(1) \sqrt{\frac{N}{\log N }}.
 \end{equation}
Since $E$ is in $\ell^1$, see \eqref{re6.1},
the operator of convolution with $E$, denoted $E*$, 
is a bounded operator 
$\ell^2\to \ell^2$. Therefore,
\eqref{re2} can be extended to $u,v \in\ell^2(\Z)$, and $E*$ is the inverse of $\mathrm{Op}(q)$ 
on $\ell^2(\Z)$. 
\par
3. Next, for $v\in\ell^2(\Z)$ with support in $[0,N-1]$, we solve 
\begin{equation}\label{re6.3}
		\mathrm{Op}(q) \psi = 0 \quad \text{ on } \Z
 \end{equation}
with ``boundary'' conditions 
\begin{equation}\label{re7}
		\psi\!\upharpoonright_{[-(m_+-\widehat{m}_0),-1]\cup[N,N+m_-+\widehat{m}_0-1]} 
		= (E*v)\!\upharpoonright_{[-(m_+-\widehat{m}_0),-1]\cup[N,N+m_-+\widehat{m}_0-1]}.
 \end{equation}
We suppose here, and in the sequel, that $m_{\pm}\mp \widehat{m}_0>0$. The cases 
when $m_+=\widehat{m}_0$ or $m_-=\widehat{m}_0=0$ can be treated similarly, and we shall 
comment on these cases as we go along. In \eqref{re7} it is understood that when 
$m_+=\widehat{m}_0$, then we only have the boundary condition 
on $[N,N+m_-+\widehat{m}_0-1]$, and similarly when $m_-=\widehat{m}_0=0$. 
\par
Since $q(\zeta)$ satisfies the assumptions of Case 1 of Lemma \ref{lem:root1}, 
it follows from Propositions \ref{prop:expSol} and \ref{ev:prop1}, and the fact that all 
roots of $q(\zeta)$ are simple, that the general solution to \eqref{re6.3} is of the form 
\begin{equation}\label{re8}
		\psi(\nu) = 
		\sum_{j=1}^{m_+}a_{j}^+(\zeta_j^+)^\nu 
		+ \sum_{j=1}^{m_-}a_{j}^-(\zeta_j^-)^\nu, 
		\quad a^{\pm}_{j} \in \C. 
 \end{equation}
By \eqref{re7}, the coefficients are determined by 
\begin{equation}\label{re9}
	\begin{pmatrix}
		V_+ \Lambda_+^{\widehat{m}_0-m_+} & B_+\Lambda_-^{\widehat{m}_0-m_+} \\
		V_- \Lambda_+^N & B_- \Lambda_-^N\\ 
	\end{pmatrix}
	a
	=:
	 M a = c := (E*v)\!\upharpoonright_{[\widehat{m}_0-m_+,-1]\cup[N,N+m_-+\widehat{m}_0-1]}.
 \end{equation}
Here, the right hand side is seen as a vector in $\C^{m_++m_-}$, $a=(a_+ , a_-)^{\sf T}$, 
$a_\pm = (a^\pm_1,\dots,a^\pm_{m_\pm}) \in \C^{m_{\pm}}$, 
$\Lambda_{+}:=\mathrm{diag}(\zeta_1^+,\dots,\zeta_{m_{+}- \widehat{m}_0}^+)$, 
$\Lambda_{-}:=\mathrm{diag}(\zeta_{m_+-\widehat{m}_0+1}^+,\dots,\zeta_{m_{+}}^+, \zeta_1^-,\dots,\zeta_{m_{-}}^-)$, 
\begin{equation}
\label{eq-defABL}
	V_{\pm} := 
	\begin{pmatrix} 1 &  \dots & 1 \\ 
	 \zeta_1^+& \dots &  \zeta_{m_+-\widehat{m}_0}^+\\
	\vdots & \vdots & \vdots \\ 
	(\zeta_1^+)^{m_\pm\mp\widehat{m}_0-1} & \dots &  (\zeta_{m_+-\widehat{m}_0}^{+})^{m_\pm\mp\widehat{m}_0-1}
	 \end{pmatrix}, 
 \end{equation}
 and 
 \begin{equation}
\label{eq-defABL2}
	B_{\pm} := 
	\begin{pmatrix} 1 &  \dots & 1 & 1 &  \dots & 1 \\ 
	\zeta_{m_+-\widehat{m}_0+1}^+& \dots &  \zeta_{m_+}^+&  \zeta_1^-& \dots &  \zeta_{m_-}^-\\
	\vdots & \vdots & \vdots & \vdots & \vdots & \vdots \\ 
	(\zeta_{m_+-\widehat{m}_0+1}^+)^{m_\pm\mp\widehat{m}_0-1} & \dots &  (\zeta_{m_+}^{+})^{m_\pm\mp\widehat{m}_0-1} &(\zeta_1^-)^{m_\pm\mp\widehat{m}_0-1} & \dots &  (\zeta_{m_-}^-)^{m_\pm\mp\widehat{m}_0-1}
	 \end{pmatrix}.
 \end{equation}
 Here, we use the convention that when $\widehat{m}_0=0$ and $m_+>0$, then the first column of $B_\pm$ 
 is the one containing powers of $\zeta_1^-$. 
\begin{rem}\label{rem:re1}
	When $m_+=\widehat{m}_0$ then 
	\eqref{re9} is given by $B_-\Lambda_-^N a = Ma = (E*v)\!\upharpoonright_{[N,N+m_--1]}$, 
	and when $m_-=\widehat{m}_0=0$ then only the first sum in \eqref{re8} remains and 
	\eqref{re9} is given by $V_+\Lambda_+^{-m_+} a_+ = Ma = (E*v)\!\upharpoonright_{[-m_+,-1]}$. 
\end{rem}
By the Cauchy-Schwarz  inequality and \eqref{re6.2}, we have that for any $k\in \Z$ 
\begin{equation}\label{nnre2}
	| E*v(k) | = \left|  \sum_0^{N-1} E(k-m) v(m) \right| \leq \| E\|_{\ell^2(\Z)} \|v\|_{\ell^2(\Z)} 
	\leq O(1) \sqrt{\frac{N}{\log N }} \|v\|_{\ell^2(\Z)}.
\end{equation}
We put $L := L_+\oplus L_-$ with 
\begin{equation}\label{re12}
\begin{split}
	& L_+= \mathrm{diag}( G(\zeta_1^+), \dots ,G(\zeta_{m_+-\widehat{m}_0}^+)),  \\
	& L_-= \mathrm{diag}( G(\zeta_{m_+-\widehat{m}_0+1}^+), \dots ,G(\zeta_{m_+}^+), G(\zeta_1^-), \dots, G(\zeta_{m_-}^-)), 
	\quad G(x) = \left( \frac{1-|x|^{2N}}{1-|x|^2}\right)^{1/2},
\end{split}
\end{equation}
where we work with the same notational convention for $L_-$ as in \eqref{eq-defABL2}. 
\par
Suppose that there exists a constant $0<C< \infty$ such that 
for $N>0$ sufficiently large $M$ is bijective and such that for all $z\in \Omega_N$
\begin{equation}\label{nnre1}
	\| LM^{-1} \| \leq C\sqrt{\frac{N}{\log N}} N^{C_2 \Theta(\widehat{m}_0)},
\end{equation}
with $C_2$ as in \eqref{re0.3b}. We will prove this fact in Step 4 below. 
It then  follows from the bijectivity of $M$ that \eqref{re6.3} with boundary conditions \eqref{re7} 
has a unique solutions $\psi$ of the form \eqref{re8} with coefficients determined by 
\eqref{re9}. 
%
%
Hence, putting  
\begin{equation*}
	\widetilde{u} = ( E*v - \psi) \!\upharpoonright_{[0,N-1]},
\end{equation*}
we have that $
	P_N(q) \widetilde{u} = v,$
so $P_N(q)$ is surjective, and thus, being a square matrix, bijective. 
We get that for $N>0$ sufficiently large,
\begin{equation}
\label{eq-sout}
\begin{split}
		\| \psi \|_{\ell^2([0,N-1])} &\stackrel{\eqref{re8}}{\leq} 
		\sum_{j=1}^{m_+}|a_{j}^+| G(\zeta_j^+)
		+ \sum_{j=1}^{m_-}|a_{j}^-|G(\zeta_j^-) \\
		&\stackrel{\mbox{\tiny \rm H\"{o}lder}}{\leq} \sqrt{m_++m_-} \| L a \|_{\ell^2([1,m_++m_-])} \\ 
		& \stackrel{\eqref{re9}}{\leq} \sqrt{m_++m_-} \| LM^{-1} c \|_{\ell^2([1,m_++m_-])} \\ 
		&\stackrel{\eqref{nnre1}}{ \leq} O(1)\sqrt{\frac{N}{\log N }}N^{C_2 \Theta(\widehat{m}_0)}\,\|c \|_{\ell^2([1,m_++m_-])} 
 \stackrel{\eqref{nnre2}}{\leq} O(1) \frac{ N^{1+C_2 \Theta(\widehat{m}_0)} }{\log N}\|v\|_{\ell^2}.
\end{split}
 \end{equation}
 Using Young's convolution 
inequality and \eqref{re6.1}, we find that for $N>0$ sufficiently large 
\begin{equation*}
\begin{split}
		\| \widetilde{u} \|_{\ell^2([0,N-1])} &\leq \|E*v \|_{\ell^2([0,N-1])} 
		+ \| \psi \|_{\ell^2([0,N-1])}   \\ 
		& \leq \|E\|_{\ell^1} \|v\|_{\ell^2} + \| \psi \|_{\ell^2([0,N-1])} 
		 \stackrel{\eqref{re6.1},\eqref{eq-sout}}{\leq} O(1) \frac{ N^{1+C_2 \Theta(\widehat{m}_0)} }{\log N}\|v\|_{\ell^2},
\end{split}
 \end{equation*}
uniformly for $z\in\Omega_N$. This proves \eqref{re0.4}, provided that \eqref{nnre1} holds. 
\\
\par
4. It remains to prove \eqref{nnre1}. By \eqref{re2.1a}, we have that 
\begin{equation}\label{re10.1}
	\| V_\pm\|, \| B_{\pm}\| =O(1).
\end{equation}
 Since $V_{+},B_{-}$ are Vandermonde matrices, we have that  
 $\det V_+ = \prod_{1\leq i <j \leq m_+-\widehat{m}_0} (\zeta_i^+ - \zeta_j^+)$ and similar expression for 
 $\det B_-$ given by a finite product of differences of the roots 
 $\zeta_{m_+-\widehat{m}_0+1}^+, \dots, \zeta_{m_+}^+$ and 
 $\zeta_1^-,\dots, \zeta_{m_-}^-$ (resp. just  $\zeta_1^-,\dots, \zeta_{m_-}^-$ when 
 $\widehat{m}_0=0$). Therefore, \eqref{re1.1} implies that both matrices 
are bijective and that their determinants are uniformly bounded from below in modulus. 

Together with \eqref{re2.1a}, we deduce that 
\begin{equation}\label{re10}
	\| (V_+)^{-1}\|, \| (B_{-})^{-1}\| =O(1).
\end{equation}
 To construct the inverse of $M$, 
we proceed via the Schur complement formula and set 
\begin{equation*}
	\Gamma = B_-\Lambda_-^N - V_-\Lambda_+^{N+m_+-\widehat{m}_0}V_+^{-1}B_+\Lambda_-^{\widehat{m}_0-m_+}.
\end{equation*}
Using \eqref{re10}, \eqref{re10.1}, \eqref{re0.3}, \eqref{re2.1a}
and \eqref{re0.3b}, we have that
\begin{equation}\label{re11.2}
\begin{split}
	\| V_-\Lambda_+^{N+m_+-\widehat{m}_0}V_+^{-1}&B_+\Lambda_-^{\widehat{m}_0-m_+-N}B_-^{-1}\| \\
	&= O(1)\left\{ \begin{array}{ll}
	|\zeta_{m_+}|^{N+m_+} | \zeta_{1}^-|^{-(N+m_+)}, & \widehat{m}_0=0, \\
	|\zeta_{m_+-m_0}^+|^{N+m_+-m_0} | \zeta_{m_+-m_0+1}^+|^{m_0-(N+m_+)}, & \widehat{m}_0=m_0,
	\end{array}\right.
	 \\
	&=O(1)
	\left\{\begin{array}{ll}
	N^{-2C_0}, &\widehat{m}_0=0, \\
	\e^{-N\log C_1} N^{C_2}, & \widehat{m}_0=m_0
	\end{array}\right. \quad =: \varepsilon_N( \widehat{m}_0).
\end{split}
\end{equation}
 Hence, for $N>0$ large enough (depending only on the constants in the 
 assumptions of the theorem) we have that $|\varepsilon_N( \widehat{m}_0)|<1$, 
 and we assume from now on that this is the case. 
 By a Neumann series argument we see that 
\begin{equation}\label{nre1}
	\Gamma^{-1} = \Lambda_-^{-N}B_-^{-1}( 1 + O(1)\varepsilon_N( \widehat{m}_0)) . 
\end{equation}
Here the $O(1)$ term means an $m_-\times m_-$ matrix with operator norm 
bounded by a constant uniformly in $N$ and $z\in\Omega_N$. 
\par
A straightforward computation yields that 
\begin{equation}\label{nre2}
	M^{-1} = 
	\begin{pmatrix}
		\Lambda_+^{m_+-\widehat{m}_0} V_+^{-1} & -\Lambda_+^{m_+-\widehat{m}_0} V_+^{-1}B_+\Lambda_-^{\widehat{m}_0-m_+} \\
		0 & 1\\ 
	\end{pmatrix}
	\begin{pmatrix}
		1& 0\\
		-\Gamma^{-1}V_- \Lambda_+^{N_++m_+-\widehat{m}_0}  V_+^{-1}  & \Gamma^{-1}\\ 
	\end{pmatrix}.
 \end{equation}
Using \eqref{re12} and \eqref{nre2}, we obtain that 
\begin{align}\label{nre2.1}
	&LM^{-1} =\\
\nonumber &
	\begin{pmatrix}
		L_+\Lambda_+^{m_+-\widehat{m}_0} V_+^{-1}(1+
		B_+\Lambda_-^{\widehat{m}_0-m_+}
		\Gamma^{-1}V_- \Lambda_+^{N_++m_+-\widehat{m}_0}  V_+^{-1}),
		& -L_+\Lambda_+^{m_+-\widehat{m}_0} V_+^{-1}B_+\Lambda_-^{\widehat{m}_0-m_+} \Gamma^{-1}\\
		-L_-\Gamma^{-1}V_- \Lambda_+^{N_++m_+-\widehat{m}_0}  V_+^{-1} ,  & L_-\Gamma^{-1} 
	\end{pmatrix}.
 \end{align}
To prove \eqref{nnre1} note that it is sufficient to estimate the norm of 
each block of the matrix \eqref{nre2.1} separately. 
%
%
Since $|\zeta_{m_+}^{+}| \leq 1 \leq |\zeta_1^-|$ and \eqref{re0.3b} holds, it follows that 
\begin{equation}\label{nre3.1}
	 \| \Lambda_-^{-n} \| , \| \Lambda_+^n \| \leq O(1), 
\end{equation}
for any fixed $n\in \N$. Using that $|\zeta_{m_+}^{+}| \leq 1$ we find also that 
\begin{equation}\label{nre3.1b}
	 \| \Lambda_+^N\| \leq O(1). 
\end{equation}
However, in view of \eqref{re0.3b} and since $1 \leq |\zeta_1^-|$ we have that 
\begin{equation}\label{nre3.1c}
	 \| \Lambda_-^{-N}\| \leq 
	O(1)\left\{\begin{array}{ll}
		|\zeta_1^-|^{-N}, & \widehat{m}_0=0, \\
		|\zeta^+_{m_+-m_0+1}|^{-N}, &\widehat{m}_0=m_0
	\end{array}\right.
	\leq O(1)	N^{C_2\Theta(\widehat{m}_0)}
\end{equation}
\par
Continuing, we deduce from \eqref{re10}, \eqref{nre1} and \eqref{nre3.1c} that 
\begin{equation}\label{nre3.0}
	\| \Gamma^{-1} \| \leq O(1) \| \Lambda_-^{-N} B_-^{-1}\| \leq O(1)N^{C_2\Theta(\widehat{m}_0)}
\end{equation}
Recall \eqref{re12}. Then, by \eqref{re0.3}, 
\begin{equation}\label{nre3.2}
	\| L_+\| \leq  \max_{1\leq j\leq m_+-\widehat{m}_0}  G(\zeta_{j} ^+)= 
	 O(1) \sqrt{\frac{N}{\log N}}. \notag
\end{equation}
Combining this with \eqref{re10} and \eqref{nre3.1}, we see that
\begin{equation}\label{nre3}
	 \| L_+ \Lambda_+^{m_+}V_+^{-1}  \| =  O(1)	\sqrt{\frac{N}{\log N}}.
\end{equation}
Hence, using \eqref{nre3}, \eqref{re10},  \eqref{nre3.1}, \eqref{re10.1}, \eqref{nre3.0} and 
\eqref{nre3.1b}, we get that 
\begin{equation}\label{re13}
\| L_+\Lambda_+^{m_+} V_+^{-1}V_-\Lambda_-^{-m_+} 
\Gamma^{-1}B_+ \Lambda_+^{N+m_+}  V_+^{-1} \| = 
O(1) \sqrt{\frac{N}{\log N}}N^{C_2\Theta(\widehat{m}_0)}
\end{equation}
and 
\begin{equation}\label{re13b}
\|  L_+\Lambda_+^{m_+} V_+^{-1}V_-\Lambda_-^{-m_+} \Gamma^{-1} \| 
 	= O(1) \sqrt{\frac{N}{\log N}}N^{C_2\Theta(\widehat{m}_0)}.
\end{equation}
Using \eqref{re0.3} and \eqref{re0.3b}, we find 
that
\begin{equation}\label{re13.1a}
 \| L_-\Lambda_-^{-N}\| 
 	= \max\{   \max_j |\zeta_j^-|^{-N} G(\zeta_{j} ^-), 
	 \max_{ m_+-\widehat{m}_0+1 \leq j \leq m_+}  |\zeta_j^+|^{-N} G(\zeta_{j} ^+)\} = O(1) \sqrt{\frac{N}{\log N}}N^{C_2\Theta(\widehat{m}_0)}.
\end{equation}
Hence, similarly to \eqref{re13}, but using as well \eqref{re13.1a} and \eqref{nre1}, we get that 
\begin{equation}\label{re13.1}
 \| L_-\Gamma^{-1}B_+ \Lambda_+^{N+m_+} V_+^{-1} \|  =O(1) \| L_-\Lambda_-^{-N}\| =
 	O(1) \sqrt{\frac{N}{\log N}}N^{C_2\Theta(\widehat{m}_0)},
\end{equation}
and
\begin{equation}\label{re13.2}
 \|  L_-\Gamma^{-1} \| =O(1) \sqrt{\frac{N}{\log N}}N^{C_2\Theta(\widehat{m}_0)}.
\end{equation}
Hence, \eqref{nre2.1}, \eqref{nre3}-\eqref{re13b}, \eqref{re13.1}, and 
\eqref{re13.2} yield \eqref{nnre1}   for $N>0$ large enough. 
\begin{rem}
	Recall Remark \ref{rem:re1}. When $m_+=\widehat{m}_0$ we have that 
	$M^{-1} = \Lambda_-^{-N}B_-^{-1}$ and 
	we obtain \eqref{nnre1} by a similar estimate as in \eqref{re13.1}. 
	When $m_-=\widehat{m}_0=0$ 
	then $M^{-1}=\Lambda_+^{m_+}V_+^{-1}$ and we obtain \eqref{nnre1} by \eqref{nre3}.
\end{rem}
This completes the proof of Theorem \ref{thm:resBound}.
%
%
\end{proof}
\section{Quasimodes for banded Toeplitz matrices}\label{sec:QM}
In this section we will construct exponentially decaying quasimodes 
for the operator 
$P_N-z$, for $z$ \emph{close} to - however not on - the curve $p(S^1)$
while avoiding a small neighborhood of 
the set of bad points, see Definition \ref{def-badintro}. We will see that the sign of the winding 
number $d(z)$ will decide whether these quasimodes decay to the left ($d<0$) or to the right ($d>0$). 
The construction will be based on  exponential states $\mathfrak{z}^\pm$, see Proposition 
\ref{lem:TruncEV}, determined by the roots of the Laurent 
polynomial $p(\zeta)-z$, with a particular focus on the decay rates of these exponential states. 
We will separate between those states $\mathfrak{z}^\pm$ which decay at a fixed rate $r>0$ and 
those which decay at an $N$-dependent rate $r\asymp \log N/N$, where the former are determined 
by the roots of $p(\zeta)-z$ ``far'' away from $p(S^1)$ and the latter are determined by the roots 
``close'' to $p(S^1)$. 
 \\
 \par
Throughout this section we fix $z_0 \in p(S^1)$  such that 
\begin{equation}\label{qm1}
		dp \neq 0 \text{ on } p^{-1}(z_0).
\end{equation}
Let $\Omega\Subset \C$ be an open relatively compact simply connected 
neighbourhood of $z_0$ such that 
\begin{equation}\label{qm2}
		 |dp| \geq 1/C \text{ on } p^{-1}(z) \text{ for all }z\in \Omega.
\end{equation}
Note that by the implicit function theorem all roots $\zeta=\zeta(z)$ of $p(\zeta) -z$ 
are simple and depend smoothly on $z\in \Omega$. Hence, after possibly shrinking $\Omega$, 
there exists  $0<C<\infty$ such that for any two distinct roots 
$\zeta(z) \neq \omega(z)$ we have that 
\begin{equation}\label{qm3}
		 | \zeta - \omega | \geq \frac{1}{C},
\end{equation}
uniformly for $z\in\Omega$. Recall the notation of the roots from 
Lemma \ref{lem:root1}. 
\\
\par
Until further notice we will denote the winding number of the curve $p(S^1)$ around a point $z\in \C$, 
see also \eqref{at9.0}, by 
\begin{equation}\label{def:index}
	\wind \defeq \wind(z) \defeq \mathrm{ind}_{p(S^1)}(z).
\end{equation}
\begin{assu}\label{As:Omega}
Let $z_0\in p(S^1)$ be as in \eqref{qm1}, and let $\Omega$ be as in \eqref{qm2}, so that \eqref{qm3} holds. 
Let $\Omega_N \Subset \Omega \backslash p(S^1)$ be an open relatively 
compact simply connected non-empty $N$-dependent set such that the winding 
number $d(z)\neq 0$ is constant for all $z\in \Omega_N$, and such that there 
exist constants $C_{\alpha},C_{\beta}, C_{\gamma}>0$ and $m^0_{\pm}\geq 0$, 
independent of $N$, such that for $N>0$ large enough (depending only on these constants) 
and for all $z\in\Omega_N$ the following holds: 
\begin{itemize}
\item we have that 
 \begin{equation}\label{qm4.0}
		p^{-1}(z) \cap (S^1 + D(0,C_{\gamma}N^{-1}\log N)) = \emptyset;
\end{equation}
\item when $\wind>0$, then $0\leq m_+^0 \leq \wind \leq m_+$ and 
\begin{equation}\label{qm4.0b}
		0<  |\zeta_1^+|\leq \cdots \leq |\zeta_{m_{+}-m_+^0}^+| < \frac{1}{C_\beta} <  1- 
	C_\alpha \frac{\log N}{N} \leq  |\zeta_{m_{+}-m_+^0+1}^+|\leq \cdots \leq |\zeta_{m_{+}}^+|  <1;
\end{equation}
\item when $\wind<0$, then $0\leq m_-^0 \leq -\wind\leq m_-$ and 
\begin{equation}\label{qm4.0c}
		1<  |\zeta_1^-|\leq \cdots \leq |\zeta_{m_-^0}^-| < 1+ C_\alpha \frac{\log N}{N} <  C_\beta \leq  |\zeta_{m_{-}^0+1}^+|\leq \cdots \leq |\zeta_{m_{-}}^-|  < + \infty.
\end{equation}
\end{itemize}
In \eqref{qm4.0b} we work with the notation convention that when $m_+ =m_+^0$ then the inequalities before $1/C_\beta$ are void, and when $m_+^0=0$, then the inequalities on the right hand side of $1-C_{\alpha}\log N/N$ are 
void. A similar convention is to be understood for \eqref{qm4.0c}. 
\end{assu}
\begin{rem}
1. Unless otherwise stated, all future estimates will only depend on
$\Omega$ and the constants 
in Assumption \ref{As:Omega}, and therefore will be uniform in $\Omega_N$. This will be understood 
implicitly when we state that an estimate is uniform in $\Omega_N$. 
\par
2. We recall that we will work with 
$z\in \Omega(\vep_{\ref{theo-location}},C_{\ref{theo-location}},N)$,
see 
\eqref{eq-Omegadef} and the assumptions of Theorem \ref{thm:main}.
In particular, ${\rm dist}(z, p(S^1))<C_{\ref{theo-location}}\log N/N$.
Let $z_0$ be the point on $p(S^1)$ closest to $z$. Then, since $dp(z_0)\neq 0$
(as we work in $\mathcal{G}_{p,\vep}$), 
as well as the existence of $C_\alpha$ as in \eqref{qm4.0b} or \eqref{qm4.0c},
e.g.~that for $d>0$, $\zeta_{m_+}^+\geq 
1-C_\alpha\log N/N$.
To see the multiplicity count, i.e. that e.g.~$m_+^0\leq d$ when $d>0$,
note that away from the set $\cB_2$,
if ${\sf g}(p)=1$ then there is precisely one $\theta_0$ with $p(e^{i\theta_0})=z_0$, and since $dp|_{\theta_0}\neq 0$, there is exactly one
root $\zeta_{m_+}$ near $1$. When ${\sf g}(p)>1$, there are 
${\sf g}(p)$ such roots, while $|d|\geq {\sf g}(p)$ if $d\neq 0$,
so again the same conclusion holds. On the other hand, by Remark \ref{rem:Cgamma}
we have that for any eigenvalue $z$ in the good set, with probability approaching one, the assumption \eqref{qm4.0} holds with $C_\gamma$ as defined in Remark \ref{rem:Cgamma}.
\end{rem}
Our aim is to construct \textit{quasimodes} for $P_N-z$, $z\in\Omega_N$, these are 
approximate $\ell^2$-normalized eigenvectors $\psi\in\C^N$ in the sense that 
\begin{equation*}
	\| (P_N-z)\psi \| \to 0
\end{equation*}
when $N\to \infty$. Our focus will be on obtaining quantitative estimates for this decay. 
The eigenvectors of $P_{[0,\infty[}$ corresponding to the eigenvalue $z$, when 
$\wind > 0$, and 
shifted versions of the eigenvectors of $P_{]-\infty,0]}$ corresponding to the eigenvalue 
$z$, when $\wind <0$, are excellent candidates for quasimodes (after truncation). 
\par
We will construct these quasimodes as follows: 
Recall from Proposition \ref{at:prop3} that these eigenvectors are exponential states which are 
created from powers of the roots of $p(\zeta)-z$. Using assumption \eqref{qm4.0b} and \eqref{qm4.0c} 
we will take care (whenever possible) to distinguish between eigenvectors which are built up only by 
roots strictly away from the unit circle, and those eigenvectors which contain also a contribution 
from roots close to the unit circle. After truncation we orthonormalize the states within the two 
respective groups. As we shall see below, 
we will need to slightly modify the behaviour at the boundary of the second group to obtain a sufficiently fast 
decay. This way we will obtain an almost orthogonal $|\wind|$-dimensional system of 
quasimodes, see Proposition \ref{prop:QM} below.
\\
\par
Recall Proposition \ref{at:prop3}. To construct quasimodes out of \eqref{at:prop3.1} respectively 
\eqref{at:prop3.1b}, we begin with constructing a suitable system of solutions to \eqref{at14} when $N_+>0$ 
and its analog when $N_->0$ with $m_+,N_+,\zeta_j^+$ replaced by $m_-,N_-,\zeta_j^-$ and the corresponding vectors $\mathfrak{A}_j^-$ indexed over $1\leq \nu \leq N_-$.
\\
\par
We begin with considering the case $N_+>0$. Notice that this implies that $0\notin p^{-1}(\Omega)$. 
As discussed after \eqref{qm2}, the roots $\zeta_z$ of $p(\zeta)-z$ depend smoothly on $z\in \Omega$, 
so, after potentially slightly shrinking $\Omega$, there exists a constant $0<C<\infty$ such that any root 
$\zeta_z$ satisfies 
\begin{equation}\label{nqeq1}
	1/C \leq |\zeta_z|, \text{ for all } z \in \Omega. 
\end{equation}
Let $z\in \Omega_N$, suppose that $\wind>0$ and recall from \eqref{at9.0} that $\wind = m_+-N_+$. 
In view of \eqref{qm4.0b}, we find that $m_+-m_+^0 \geq N_+>0$. For the sake of the presentation 
we first suppose that $\wind>m_+^0 \ge 1$. Recall the matrix $\mathfrak{A}$ from \eqref{at14} and write 
\begin{equation*}
	\mathfrak{B}_+=(\mathfrak{A}_1^+,\dots,\mathfrak{A}_{N_+}^+) \in \C^{N_+\times N_+}, \quad 
	\mathfrak{C}_+=(\mathfrak{A}_1^+,\dots,\mathfrak{A}_{m_+-m_+^0}^+)\in \C^{N_+\times (m_+-m_+^0)}
\end{equation*}
By \eqref{nqeq1} we have that for all $j=1,\dots, m_+$ 
\begin{equation}\label{nqeq1.1}
	\| \mathfrak{A}_j^+\| \leq C^{N_+}\sqrt{N_+}.
\end{equation}
Notice that $\mathfrak{B}_+$ is a Vandermonde type matrix, so 
$|\det \mathfrak{B}_+| = \prod_{k=1}^{N^+}|\zeta_k^+|^{-N^+}\prod_{1\leq i<j\leq N_+}|\zeta_i^+ - \zeta_j^+|$. 
Thus, by \eqref{nqeq1}, \eqref{qm3}, we see that $\mathfrak{B}_+$ is invertible and, using \eqref{nqeq1.1}, we 
find 
\begin{equation}\label{nqeq1.2}
	\|\mathfrak{B}_+\|, \|\mathfrak{B}_+^{-1}\| =O(1), 
\end{equation}
where the constant depends only on $\Omega$. Hence, by \eqref{qm4.0b}, 
we see that $\mathfrak{C}_+$ has rank $N_+$ and has a kernel of dimension 
$\dim \mathcal{N}(\mathfrak{C}_+)=m_+-m_+^0-N_+=\wind-m_+^0> 0$. Let $x_1, \dots, x_{\wind-m_+^0}
 \in \C^{m_+-m_+^0}$ be an orthonormal system spanning $\mathcal{N}(\mathfrak{C}_+)$. Clearly, 
\begin{equation}\label{nqeq2}
	\C^{m_+} \ni a^+_j := x_j \oplus 0_{m_+^0} \in \mathcal{N}(\mathfrak{A}) \text{ for } j =1,\dots,\wind-m_+^0. 
\end{equation}
For $j=\wind-m_+^0+1,\dots, \wind$, we put 
\begin{equation}\label{nqeq3}
	a^+_j := (-\mathfrak{B}_+^{-1}\mathfrak{A}^+_{j+N_+}, \delta_{j})\in \C^{m_+},
\end{equation}
where $\delta_k\in \C^{m_+-N_+}$ is so that $\delta_k(n)=\delta_{k,n}$ for $n \in [m_+ - N_+]$. By construction 
we have that the $a^+_j $ are linearly independent and that they span $\mathcal{N}(\mathfrak{A})$. 
Writing $X_+=(x_1,\dots,x_{\wind-m_+^0})\in \C^{(m_+-m_+^0)\times (\wind-m_+^0)}$ and 
$\Delta_+=(\delta_{\wind-m_+^0+1},\dots, \delta_{\wind}) \in \C^{\wind\times m_+^0}$, we put
\begin{equation*}
	A^+_1 =\begin{pmatrix} X_+  \\ 
	0 \\ 
	\end{pmatrix}, \quad 
	A^+_2 =  \begin{pmatrix}  -\mathfrak{B}_+^{-1}(\mathfrak{A}^+_{m_+-m_+^0+1},\dots,\mathfrak{A}^+_{m_+}) \\ 
	\Delta_+ \\ 
	\end{pmatrix},
\end{equation*}
where on the left hand side $0 \in \C^{m_+^0\times (\wind-m_+^0)}$. Then, 
\begin{equation}\label{nqeq4}
	A^+ = (a^+_1, \dots, a^+_{\wind}) 
	=(A^+_1,A^+_2).
\end{equation}
When $m_+^0=0$, then we have no solutions \eqref{nqeq3} and $A^+ = A^+_1=X_+$. When $\wind=m_+^0$, 
then we have no solutions \eqref{nqeq2} and all vectors $a_j^+$ are of the form \eqref{nqeq3}, so 
\begin{equation*}
	A^+ =A_2^+=\begin{pmatrix}  -\mathfrak{B}_+^{-1}(\mathfrak{A}^+_{N_++1},\dots,\mathfrak{A}^+_{m_+}) \\ 1_{\wind} \\ 
	\end{pmatrix}.
\end{equation*}
Keeping in mind these special cases, we shall keep writing \eqref{nqeq4} for all cases of $\wind\geq m_+^0\geq 0$, 
with the convention that $A_1^+$ is void when $\wind=m_+^0$ and that $A^+_2$ is void when $m_+^0=0$. 
\\
\par
When $\wind<0$ (keeping in mind \eqref{at9.0}) and $N_->0$, we work under the 
assumption \eqref{qm4.0c} and we consider the 
analogue of $\mathfrak{A}$ from \eqref{at14} with $m_+,N_+,\zeta_j^+$ replaced by $m_-,N_-,\zeta_j^-$ 
and the corresponding vectors $\mathfrak{A}_j^-$ indexed over $1\leq \nu \leq N_-$. Notice that 
since $N_->0$, it follows that $\infty \notin p^{-1}(\Omega)$, so, after potentially slightly shrinking 
$\Omega$, we have that any root $\zeta_z$ of $p(\zeta)-z$ 
satisfies $|\zeta_z| \leq C$ for all $z\in\Omega$. Thus, the analogue of \eqref{nqeq1.1} holds for $\mathfrak{A}_j^-$. 
Putting $\mathfrak{B}_-=(\mathfrak{A}_{m_--N_-+1} ^-,\dots,\mathfrak{A}_{m_-}^-)$, we find by a 
Van der Monde argument that the analogue of \eqref{nqeq1.2} holds for $\mathfrak{B}_-$. 
Then, by the exact same steps as above, with the obvious modifications, we can construct $|\wind|$ linearly independent solutions $a^-_j$, to $\mathfrak{A}a^-=0$, such that 
\begin{equation}\label{nqeq5}
	A^- = (a^-_1, \dots, a^-_{\wind}) 
	= \begin{pmatrix}  -\mathfrak{B}_-^{-1}(\mathfrak{A}^-_{1},\dots,\mathfrak{A}^-_{m_-^0}) & X_- \\ \Delta_-
	& 0 \\ 
	\end{pmatrix}=(A^-_1,A^-_2),
\end{equation}
%
where $\delta_k\in \C^{m_--N_-}$ is so that $\delta_k(n)=\delta_{k,n}$ for $n \in [m_- - N_-]$, 
$\Delta_- = (\delta_{1},\dots, \delta_{m_-^0}) \in \C^{|\wind|\times m_-^0}$, 
and $X_-\in\C^{(|\wind|-m_-^0)\times (|\wind|-m_-^0)}$ is a matrix with orthonormal columns
spanning $\mathcal{N}((\mathfrak{A}_{m_-^0+1} ^-,\dots,\mathfrak{A}_{m_-}^-))$. 
\\
\par
In what follows we identify $\ell^2([0,N-1])\simeq \C^N$, so for the sake of consistency 
we index vectors and the columns of matrices starting from $0$. 
\begin{prop}\label{lem:TruncEV}
	Let $N \geq N_0 >1$ be integers, $z_0 \in p(S^1)$ be as in \eqref{qm1},  
	$\Omega\Subset \C$ be a sufficiently small open simply connected relatively 
	compact neighbourhood of $z_0$, satisfying \eqref{qm2}, and 
	$\Omega_N\Subset \Omega \backslash p(S^1)$ be as in Assumption \ref{As:Omega}. 
	Let $N>0$ be sufficiently large, depending only on $\Omega$ and the 
	constants in Assumption \ref{As:Omega} and the roots of $p(\zeta)-z$ be denoted as in Lemma \ref{lem:root1}. Write 
	\begin{equation}\label{eq:mthfrk-z1}
		\mathfrak{z}_j^{+}=  (1,\zeta_j^+,\dots, (\zeta_j^+)^{N-1})^t \in \C^N, \quad j=1,\dots, m_+, 
	\end{equation}
	and 
	\begin{equation}\label{eq:mthfrk-z2}
		\mathfrak{z}_j^{-}=  ((\zeta_j^-)^{1-N},(\zeta_j^-)^{2-N},\dots, 1)^t \in \C^N, 
		\quad j=1,\dots, m_-,
	\end{equation}
	with the convention that ${1}/{\infty} =0$. Moreover, let 
	$ \mathfrak{Z}^\pm = ( \mathfrak{z}_1^{\pm}, \dots,\mathfrak{z}_{m_\pm}^{\pm}) $. 
	\begin{itemize}
	\item Suppose that $\wind > 0$. When $N_+>0$, let 
	$A^+ =(a_1^+,\dots,a_{ \wind}^+)$ be as in \eqref{nqeq4}, and when 
	$N_+ \leq 0$ put $A^+  = I_{m_+}\in \C^{m_+\times m_+}$. Define 
	$I_{+,1} := [\wind-m_+^0]$ and $I_{+,2} := [\wind]\backslash [\wind-m_+^0]$. 
	Define 
		\begin{equation}\label{qm5}
		 	(\widetilde{u}_1^+,\dots,\widetilde{u}^+_{\wind})= \mathfrak{Z}^+ 
			A^+ (G_{+,1}^{-1/2} \oplus \mathfrak{L}_{+}^{-1}G_{+,2}^{-1/2})
		\end{equation}
	with  
	$\mathfrak{L}_{+}=\diag( \|\mathfrak{z}_{m_+-m_+^0+1}^{+}\|, \dots,\|\mathfrak{z}_{m_+}^{+}\|)$, 
	$G_{+,2}=(\mathfrak{Z}^+A^+_2 \mathfrak{L}_{+}^{-1})^*(\mathfrak{Z}^+A^+_2 \mathfrak{L}_{+}^{-1})$, 
	and $G_{+,1}=(\mathfrak{Z}^+A^+_1)^*(\mathfrak{Z}^+A^+_1)$, using the convention that 
	$ (\widetilde{u}_1^+,\dots,\widetilde{u}^+_{\wind}) = \mathfrak{Z}^+ A^+ G_{+,1}^{-1/2}$ when 
	$I_{+,2} = \emptyset$, and that 
	$(\widetilde{u}_1^+,\dots,\widetilde{u}^+_{\wind}) = \mathfrak{Z}^+ A^+  \mathfrak{L}_{+}^{-1}G_{+,2}^{-1/2}$ when 
	$I_{+,1} = \emptyset$. 
	Then, 
	uniformly for $z\in\Omega_N$,
	\begin{equation}\label{qm6.01}
		\langle \widetilde{u}_n^+ | \widetilde{u}_m^+\rangle = 
	\left\{
	\begin{array}{ll} \delta_{n,m}, &  (n,m) \in I_{+,1}\times I_{+,1} \cup I_{+,2}\times I_{+,2}
		\\
		 \delta_{n,m} + O((\log N / N)^{1/2}), & (n,m)\in I_{+,1}\times I_{+,2} \cup 
		 I_{+,2}\times I_{+,1},
		\end{array}\right.
	\end{equation}
	 and
	\begin{equation}\label{qm6}
		 	\| \mathbf{1}_{[N_0,N-1]} \widetilde{u}_j^+ \|\leq O(1)
	\left\{		\begin{array}{ll}
			\e^{-N_0 \log C_\beta }, & j \in  I_{+,1}, \\
			\e^{-\frac{N_0}{N} \log N^{C_{\gamma}}}, &  j \in  I_{+,2}.
			\end{array}
			\right.
	\end{equation}
	\item Suppose that $\wind < 0$. When $N_->0$, let 
	$A^- =(a_1^-,\dots,a_{ |\wind|}^-)$ be as in \eqref{nqeq5}, and when 
	$N_- \leq 0$ put $A^-  = I_{m_-}\in \C^{m_-\times m_-}$. 
	Define 
	$I_{-,2} := [m_-^0]$ and $I_{-,1}
	:= [-\wind]\backslash [m_-^0]$.  Define
		\begin{equation}\label{qm5.1}
		 	(\widetilde{u}_1^-,\dots,\widetilde{u}^-_{|\wind|})= \mathfrak{Z}^-
			A^- ( \mathfrak{L}_{-}^{-1}G_{-,1}^{-1/2} \oplus G_{-,2}^{-1/2} )
		\end{equation}
	with  
	$\mathfrak{L}_{-}=\diag( \|\mathfrak{z}_{m_+-m_+^0+1}^{+}\|, \dots,\|\mathfrak{z}_{m_+}^{+}\|)$, 
	$G_{-,1}=(\mathfrak{Z}^-A^-_1 \mathfrak{L}_{-}^{-1})^*(\mathfrak{Z}^-A^-_1 \mathfrak{L}_{-}^{-1})$ 
	and $G_{-,2}=(\mathfrak{Z}^-A^-_2)^*(\mathfrak{Z}^-A^-_2)$, using a similar notation convention 
	as above. 
	Then,  
	uniformly for $z\in\Omega_N$,
	\begin{equation}\label{qm6.02}
		\langle \widetilde{u}_n^- | \widetilde{u}_m^-\rangle = 
\left\{		\begin{array}{ll} \delta_{n,m}, &    (n,m) \in I_{+,1}\times I_{+,1} \cup I_{+,2}\times I_{+,2}
		  \\
				     \delta_{n,m} + O((\log N / N)^{1/2}), 
				     &    
				     (n,m)\in I_{-,1}\times I_{-,2} \cup 
		 I_{-,2}\times I_{-,1},
		\end{array}\right.
	\end{equation}
	 and 
	\begin{equation}\label{qm6.2}
		 	\| \mathbf{1}_{[0,N-N_0-1]} 
			\widetilde{u}_j^-  \|\leq O(1)
	\left\{		\begin{array}{ll}
			\e^{-\frac{N_0}{N} \log N^{C_{\gamma}}}, &
			j \in  I_{-,2},\\
			\e^{-N_0 \log C_\beta}, &   j \in  I_{-,1}.
			\end{array}\right.
	\end{equation}
	%
%
	\end{itemize}
\end{prop}
%
%

\begin{rem}\label{rem:pstate}
When $d >0$, the normalized versions of the vectors defined in \eqref{eq:mthfrk-z1} will be termed as the pure states associated with $z$. Similarly for $d<0$ the normalized version vectors \eqref{eq:mthfrk-z2} will be the pure states corresponding to $z$.

Thus, for $z$ satisfying \eqref{qm4.0b} we note that the pure states $\{\mathfrak{z}_j^+/\|\mathfrak{z}_j^+\|\}_{j=1}^{m_+-m_+^0}$ are completely localized and its entries decay exponentially in $N$, while the remaining pure states decay and localize at a scale $N/\log N$. Similar conclusion holds for $z$'s satisfying \eqref{qm4.0c}.
\end{rem}

\begin{rem}\label{rem:jb1}
For later use, let us record that for the Jordan block if $z$ is inside the the unit disc then it has one pure state $\mathfrak{z}_1^+$, and $\wt u_1^+ = \mathfrak{z}_1^+/\|\mathfrak{z}_1^+\|$.
\end{rem}

We postpone the proof of Proposition \ref{lem:TruncEV}
to after that of Proposition \ref{prop:QM}. 
Continuing, we notice that the $\widetilde{u}_j^\pm$ defined in \eqref{qm5} resp. \eqref{qm5.1}, 
are the truncated kernel elements of $P_{[0,+\infty[}-z$ resp. $P_{]-\infty,N-1]}-z$, separated into 
two groups according to their decay behaviour and orthonormalized within their respective groups.  
We could use them already as quasimodes for $P_N-z$. In fact, 
it follows from \eqref{qm6} with $N_0 = N- (N_+\lor 0 + N_- \lor 0)$, that when $\wind >0$ 
\begin{equation}\label{qn1}
	\|(P_N-z)\widetilde{u}^+_j\| \leq O(1) \| \mathbf{1}_{[N_0,N-1]} \widetilde{u}_j ^+\|
	\leq O(1) 
	\left\{	\begin{array}{ll}
			\e^{-N \log C_\beta },&  j \in I_{+,1}, \\
			N^{-C_{\gamma}}, &j \in I_{+,2},
			\end{array}\right.
\end{equation}
and similarly from \eqref{qm6.2} when $\wind<0$. 
Ultimately we want to use these quasimodes to obtain  spectral gaps
between the 
$(|\wind|-m_{\mathrm{sign}(\wind)}^0)$-th and the $(|\wind|-m_{\mathrm{sign}(\wind)}^0+1)$-th, and between the 
$|\wind|$-th and $(|\wind|+1)$-th singular value of $P_N-z$. We will see, that these quasimodes will be sufficiently good for our purposes when $\gamma > 2$. However, to obtain sufficiently fast decay 
when $\gamma >1$ we need to modify them.
\begin{prop}\label{prop:QM}
	Under the assumptions of Proposition \ref{lem:TruncEV}, we put $N_0 = (1-1/\log N)(N-1)$, and we 
	have the following:
	\par
	1. If $\wind >0$, let $\widetilde{u}^+_j$, $j=1,\dots, \wind$, be as in \eqref{qm5} and 
	define $\psi^+_j\in \C^N$, by   
	\begin{equation}\label{qn2}
		\psi^+_j(\nu) =
	\left\{	\begin{array}{ll}
		\widetilde{u}^+_j(\nu), & j \in I_{+,1}, \\
		 \mathbf{1}_{[0,N_0]} (\nu)\widetilde{u}^+_j (\nu)
		+ \mathbf{1}_{[N_0+1,N-1]}(\nu)
		 \frac{\cos ( \frac{\pi}{2} \frac{\nu}{N-1})}{\cos ( \frac{\pi}{2} \frac{N_0}{N-1})}
		 \widetilde{u}^+_j (\nu), & j \in I_{+,2}.
		 \end{array}\right.
	\end{equation}
	\par
	2. If $\wind < 0$, let $\widetilde{u}^-_j$, $j=1,\dots, -\wind$, be as in \eqref{qm5.1} and 
	define $\psi^-_j\in \C^N$, by   
	\begin{equation}\label{qn2.1}
		\psi^-_j(\nu) = 
		\left\{\begin{array}{ll}
		\mathbf{1}_{[0,N-N_0]} (\nu) 
		 \frac{\cos ( \frac{\pi}{2} \frac{N-1-\nu}{N-1})}{\cos ( \frac{\pi}{2} \frac{N_0}{N-1})}\widetilde{u}^-_j (\nu)
		+ \mathbf{1}_{[N-N_0+1,N-1]}(\nu)\widetilde{u}^-_j (\nu), & j \in I_{-,2}, \\ 
		\widetilde{u}^-_j(\nu), & j \in I_{-,1}. 
		\end{array}\right.
	\end{equation}
	Then, for $i,j=1,\dots, |\wind|$, uniformly in $z\in\Omega_N$,
	\begin{equation}\label{qn3}
		\langle \psi^{\mathrm{sign}(\wind)}_i | \psi^{\mathrm{sign}(\wind)}_j\rangle 
		= \delta_{i,j} + 
		\left\{\begin{array}{ll} 0, &\text{ when }  i,j \in I_{ \mathrm{sign}(\wind),1} \\
				  O( \log N) N^{-C_{\gamma}},&\text{ when }  i,j \in I_{ \mathrm{sign}(\wind),2} \\
				  O( \log N) N^{-C_{\gamma}} +O((\log N / N)^{1/2}),& \text{else},
		\end{array}\right. 
	\end{equation}
	and 
	\begin{equation}\label{qn4}
		\| (P_N -z) \psi_i^{\mathrm{sign}(\wind)} \| =  O(1)
		\left\{	  \begin{array}{ll} \e^{-N \log C_{\beta}}, & \text{ when }  i\in I_{ \mathrm{sign}(\wind),1} \\
				  \frac{\log N}{N} N^{-C_{\gamma}},& \text{ when }  i\in I_{ \mathrm{sign}(\wind),2}. \\
		\end{array}\right.
	\end{equation}
	%
\end{prop}
We refer to the vectors $\widetilde{\psi}^{\pm}_j$ as \textit{quasimodes}.
\begin{proof}
	We will consider only the case when $\wind >0$, the case when $\wind <0$ is similar. 
	Let $N>0$ be large enough so that the conclusions of Proposition \ref{lem:TruncEV} hold.
	\par
	1. First notice that 
	\begin{equation}\label{qn5}
		\cos \left( \frac{\pi}{2} \frac{N_0}{N-1}\right) =
		 \sin \frac{\pi }{2 \log N} = 
		 \frac{\pi }{2 \log N} ( 1 + O( (\log N)^{-2}).
	\end{equation}
	Let $j=\wind-m_+^0+1,\dots, \wind$, then
	\begin{equation*}
		\| \widetilde{u}_j^+ - \psi_j^+ \|^2
		 = \sum _{\nu =N_0+1}^{N-1} \left|1 -  \frac{\cos ( \frac{\pi}{2} \frac{\nu}{N-1})}{\cos ( \frac{\pi}{2} \frac{N_0}{N-1})}\right |^2 | \widetilde{u}^+_j(\nu)|^2 
		 \leq O(  (\log N)^2) \|  \mathbf{1}_{[N_0+1,N-1]}\widetilde{u}_j^+ \| ^2. 
	\end{equation*}
	Applying \eqref{qm6}, we get that 
	\begin{equation}\label{qn5.1}
		\| \widetilde{u}_j^+ - \psi_j^+ \|
		 \leq O( \log N) N^{-C_{\gamma}}. 
	\end{equation}
	\begin{rem}\label{rem:Modes1}
	For later use we note here that the analogue of \eqref{qn5.1} in the case when $\wind<0$ holds, i.e. 
	$$\| \widetilde{u}_j^- - \psi_j^- \| \leq O( \log N) N^{-C_{\gamma}}, \quad j \in I_{-,2}.$$
	\end{rem}
	This, together with the $\widetilde{u}_j^+$, $j=1,\dots, \wind$, 
	being normalized, see \eqref{qm6.01}, 
	yields that 
	\begin{equation*}
	         \| \psi _j^+\| = 1 + \left\{
		 \begin{array}{ll} 0, & \text{ when } j \in I_{+,1} \\
				  O( \log N) N^{-C_{\gamma}},& \text{ when }   j  \in I_{+,2}.
		\end{array}\right.
	\end{equation*}
	Using furthermore the fact that the $\widetilde{u}_j^+$ are almost orthogonal, see \eqref{qm6.01}, 
	we find that %
	\begin{equation*}
	\begin{split}
	         \langle \psi_i^+ | \psi_j^+\rangle  &=   \langle \psi_i^+ - \widetilde{u}_i^{+} | \psi_j\rangle 
	         +   \langle \widetilde{u}_i^{+} | \psi_j^+-\widetilde{u}_j^{+}\rangle  +
	           \langle \widetilde{u}_i^{+} | \widetilde{u}_j^{+} \rangle  \\
	           & = \delta_{i,j} + 
		   \left\{\begin{array}{ll} 0, & \text{ when }  i,j   \in I_{+,1}, \\
				  O( \log N) N^{-C_{\gamma}},& \text{ when } i, j  \in I_{+,2}\\
				  O( \log N) N^{-C_{\gamma}} +O((\log N /N)^{1/2}),&\text{else}.
		\end{array}\right. 
	\end{split}
	\end{equation*}
	\par
	2. It remains to prove \eqref{qn4}. Recall from \eqref{qm5} and Proposition \ref{at:prop3} that 
	the 
	$\widetilde{u}_j^+= ( \mathfrak{Z}^+ A^+ (G_{+,1}^{-1/2} \oplus \mathfrak{L}_{+}^{-1}G_{+,2}^{-1/2}))_{-,j}$ 
	are the truncated and normalized (in fact almost orthonormalized) kernel elements 
	of $P_{[0,+\infty[}-z$. Here, and in the sequel we denote by $A_{-,j} \in \C^n$ 
	the $j$-th column of some $n\times m$ matrix $A$ as a vector in $\C^n$. 
	Let $u_k\in \ell^2([0,+\infty[)$, $k=1, \dots, \wind$, be given by 
	\begin{equation}
	  \label{eq-pkak1}
	  u_k(\nu) = \left\{
	  \begin{array}{ll}
	    (\mathfrak{Z}^+A^+)_{\nu,k}, & 
	0\leq \nu \leq N-1,\\
	0, &\nu \geq N.
      \end{array}
      \right.
    \end{equation}
    Then, 
	\begin{equation*}
	((P_{[0,+\infty[}-z)u_k)(\nu) =0, \quad 0\leq \nu \leq N -1 - N_-\lor 0. 
	\end{equation*}
	By \eqref{at9.0} we see that $N_->0$ since we work here with $\wind>0$. 
	To ease the notation we write $\widetilde{N}_\pm = N_\pm \lor 0$, and 
	$\widetilde{a}_k = a_k - z\delta_{k,0}$ when $a_k$ exists, otherwise 
	$\widetilde{a}_k = - z\delta_{k,0}$, so that 
	\begin{equation*}
		 \sum_{k=-N_-}^{N_+} a_k \tau^k -z =  \sum_{k=-\widetilde{N}_-}^{\widetilde{N}_+} 
		 \widetilde{a}_k \tau^k.
	\end{equation*}
	Thus, for $0\leq \nu \leq N -1 - \widetilde{N}_-$, 
	\begin{equation}\label{qn6}
	\begin{split}
		(P_N-z)\widetilde{u}_j^+ (\nu) &= \sum_{k=-\widetilde{N}_-}^{\widetilde{N}_+} \widetilde{a}_k
		\mathbf{1}_{[0,N-1]} (\nu-k) \widetilde{u}_j^+ (\nu - k)\\
		&\stackrel{\eqref{qm5}}{=}  \sum_{k=-\widetilde{N}_-}^{\widetilde{N}_+} \widetilde{a}_k\mathbf{1}_{[0,N-1]} (\nu-k) 
		( \mathfrak{Z}^+ A^+ (G_{+,1}^{-1/2} \oplus \mathfrak{L}_{+}^{-1}G_{+,2}^{-1/2}))_{(\nu - k),j} \\ 
		&=\sum_{n=1}^{\wind}\sum_{k=-\widetilde{N}_-}^{\widetilde{N}_+} \widetilde{a}_k\mathbf{1}_{[0,N-1]} (\nu-k)
		 (\mathfrak{Z}^+A^+)_{(\nu-k),n}( G_{+,1}^{-1/2} \oplus \mathfrak{L}_{+}^{-1}G_{+,2}^{-1/2})_{n,j} \\
		 &\stackrel{\eqref{eq-pkak1}}{=}
		 \sum_{n=1}^{\wind}  (G_{+,1}^{-1/2} \oplus \mathfrak{L}_{+}^{-1}G_{+,2}^{-1/2})_{n,j} 
		((P_{[0,+\infty[}-z)u_{n})(\nu) =0.
	\end{split}
	\end{equation}
	%
	Thus, for $ 0 \leq \nu \leq N -1 - \widetilde{N}_-$ and $j\in I_{+,1}$ 
	\begin{equation*}
		(P_N-z)\psi_j^+ (\nu)  = (P_N-z)\widetilde{u}_j (\nu) =0.
	\end{equation*}
	Recall that $z\in\Omega\Subset \C$ is relatively compact. Hence, 
	applying \eqref{qm6} with $N_0 =  N-1-\widetilde{N}_--\widetilde{N}_+$, we see 
	that, uniformly in $z\in\Omega_N$,
	\begin{equation}\label{qn7.0}
	  \| (P_N -z) \psi_{j}^{+} \| = 
		 O(1)\|  \mathbf{1}_{[N-1-\widetilde{N}_--\widetilde{N}_+,N-1]} \psi_i^{+} \| 
		 =O(1) \e^{-N \log C_\beta }, ~j\in I_{+,1}, 
	\end{equation}
	\par 
	Until further notice let $j\in I_{+,2}$. Equation \eqref{qn6} implies that 
	for $ 0 \leq \nu \leq N_0 - \widetilde{N}_-$
	\begin{equation*}
	  (P_N-z)\psi_j^+ (\nu)  = (P_N-z)\widetilde{u}_j^+
	  (\nu) =0.
	\end{equation*}
	Next, by \eqref{qn5}, we have for
	$N_0 - (\widetilde{N}_+ + \widetilde{N}_-) < \nu 
	\leq N_0 + (\widetilde{N}_+ + \widetilde{N}_-) $ that
	\begin{equation}\label{qn7}
		\left| 1 -\frac{\cos ( \frac{\pi}{2} \frac{\nu}{N-1})}{\cos ( \frac{\pi}{2} \frac{N_0}{N-1})} \right| 
		\leq O( N^{-1} \log N ) | N_0 - \nu| = O( N^{-1} \log N ).
	\end{equation}
	Hence, for $N_0 - \widetilde{N}_- < \nu \leq N_0  + \widetilde{N}_+$, we get 
	by \eqref{qn6}, \eqref{qn7} that 
	\begin{equation}\label{qn8}
	\begin{split}
		(P_N-z)\psi_j^+ (\nu) &= \sum_{k=-\widetilde{N}_-}^{\widetilde{N}_+} \widetilde{a}_k\psi_j ^+(\nu - k) \\
		&=  \sum_{k=-\widetilde{N}_-}^{\widetilde{N}_+} \widetilde{a}_k \widetilde{u}_j^+(\nu - k)
	        -  \sum_{k=-\widetilde{N}_-}^{\widetilde{N}_+} \widetilde{a}_k \left( 1 -\frac{\cos ( \frac{\pi}{2} \frac{\nu-k}{N-1})}{\cos ( \frac{\pi}{2} \frac{N_0}{N-1})} \right) \widetilde{u}_j^+(\nu - k)\mathbf{1}_{[N_0+1,N-1]} (\nu-k)\\ 
		&=O( N^{-1} \log N ) \| \mathbf{1}_{I(\nu)} \widetilde{u}_j^+ \|,
	\end{split}
	\end{equation}
	where $I(\nu) := [ \nu -\widetilde{N}_+, \nu + \widetilde{N}_-]$. 
	For $N_0 +\widetilde{N}_+ < \nu \leq N-1-\widetilde{N}_-$ we 
	get by \eqref{qn6} and a similar estimate as in \eqref{qn7}, \eqref{qn8}, that 
	\begin{equation}\label{qn9}
	\begin{split}
		(P_N-z)\psi_j^+ (\nu) 
		&=  \frac{\cos ( \frac{\pi}{2} \frac{\nu}{N-1})}{\cos ( \frac{\pi}{2} \frac{N_0}{N-1})}
		 \sum_{k=-\widetilde{N}_-}^{\widetilde{N}_+} \widetilde{a}_k \widetilde{u}_j^+(\nu - k) 
	          -  \sum_{k=-\widetilde{N}_-}^{\widetilde{N}_+} \widetilde{a}_k \frac{\cos ( \frac{\pi}{2} \frac{\nu}{N-1})-\cos ( \frac{\pi}{2} \frac{\nu-k}{N-1})}{\cos ( \frac{\pi}{2} \frac{N_0}{N-1})} 
	        \widetilde{u}_j^+(\nu - k)\\ 
		&=O( N^{-1} \log N ) \| \mathbf{1}_{I(\nu)} \widetilde{u}_j^+ \|.
	\end{split}
	\end{equation}
	Now let $N-1 -\widetilde{N}_- < \nu \leq N-1$. For $-\widetilde{N}_- \leq k \leq \widetilde{N}_+$ 
	we have that 
	$\cos (\frac{\pi}{2}\frac{\nu-k}{N-1}) \asymp N^{-1}$, so by \eqref{qn5} we get that 
	\begin{equation}\label{qn10}
	\begin{split}
		(P_N-z)\psi_j^+ (\nu) &= 
		 \sum_{k=-\widetilde{N}_-}^{\widetilde{N}_+} \widetilde{a}_k 
		 \frac{\cos ( \frac{\pi}{2} \frac{\nu-k}{N-1})}{\cos ( \frac{\pi}{2} \frac{N_0}{N-1})}	
		 \widetilde{u}_j^+(\nu - k)   \mathbf{1}_{[N_0+1,N-1]}(\nu-k)\\
		&=O( N^{-1} \log N ) \| \mathbf{1}_{I(\nu)} \mathbf{1}_{[N_0+1,N-1]} \widetilde{u}_j ^+\|.
	\end{split}
	\end{equation}
	Combining \eqref{qn6}, \eqref{qn8}, \eqref{qn9} and \eqref{qn10} we obtain that 
	\begin{equation}\label{qn11}
	\begin{split}
		\| (P_N-z)\psi_j^+ \|^2 & \leq O (( N^{-1} \log N )^2) \sum_{\nu= N_0-\widetilde{N}_-}^{N-1} \| \mathbf{1}_{I(\nu)} \mathbf{1}_{[N_0-(\widetilde{N}_+ + \widetilde{N}_-) ,N-1]} \widetilde{u}_j ^+\|^{2} \\ 
		 &  \leq O(( N^{-1} \log N )^2) \| \mathbf{1}_{[N_0-(\widetilde{N}_+ + \widetilde{N}_-) ,N-1]} \widetilde{u}_j^+ \|^2.
	\end{split}
	\end{equation}
	Applying now \eqref{qm6} with $N_0=(1-1/\log N) N - (\widetilde{N}_+ + \widetilde{N}_-)$, 
	we get that 
	\begin{equation}\label{qn12}
		\| (P_N-z)\psi_j^+ \| =  O( N^{-1} \log N ) N^{-C_{\gamma}}
	\end{equation}
	which concludes the proof of the proposition. 
\end{proof}
\begin{proof}[Proof of Proposition \ref{lem:TruncEV}]
	We will only consider the case when $\wind >0$. The case $\wind <0$ can be proven similarly 
	with the obvious modifications. In what follows, we will assume that $N>0$ is sufficiently 
	large, depending only on $\Omega$ and the constants in Assumption \ref{As:Omega},  
	without us mentioning this at every occurrence. Similarly, we note that all constants $C>0$ and 
	error estimates 	will be uniform in $z\in\Omega_N$, and in fact they depend only on the above mentioned 
	parameters, so we shall only mention this at a few important points. 
	\par
	1. Since the roots of $p(\zeta)-z$ depend smoothly on $z\in\Omega$, see the discussion 
	above \eqref{qm3}, we have that either $0\in p^{-1}(z_0)$ or we may take $\Omega$ sufficiently 
	small so that $0\notin p^{-1}(\Omega)$. In both cases $0\notin p^{-1}(\Omega_N)$. 
	Hence, we have that the symbol $p(\zeta)-z$ satisfies for all $z\in\Omega_N$ the assumptions of 
	Case 1 of Lemma \ref{lem:root1}, see also \eqref{eqn:NC1}. We then know 
	that $m_+>0$ since $\wind>0$, see \eqref{at9.0}.
	\\
	\par
	2. Recall from \eqref{eq-defABL} the $m_+\times m_+$ matrices $V=V_+$ 
	and $\Lambda=\Lambda_+$ (with $\widehat{m}_0=0$). Then, for $r=1,2$ and $\psi_1 \in \C^{\wind-m_+^0}$, 
	$\psi_2 \in \C^{m_+^0}$, 
	\begin{equation}\label{qm9}
		 	\|\mathfrak{Z}^+A_r^+\psi_r \|^2 \geq 
			\sum_{j=0}^{\lfloor (N-2)/m_+\rfloor -1 } \|V \Lambda^{jm_+} A_r^+\psi_r\|^2.
	\end{equation}
	 Let 
	$s_1(V) \geq \cdots \geq s_{m_+}(V)\geq  0$ denote the singular values of $V$. Since 
	$|\det V| = \prod_1^{m_+} s_j(V) $, and $ s_1(V) = \|V\| \leq 
	m_+$, 
	it follows that 
	\begin{equation}\label{qm9.1}
	  s_{m_+}(V) \geq \frac{|\det V| }{m_+^{m_+}} = 
	  \frac{\prod_{1\leq i < j \leq m_+} |\zeta_i^+ - \zeta_j^+|}{m_+^{m_+}} \geq \frac{1}{C}.
	\end{equation}
Here, in the equality we used the formula for the determinant of the Van der Monde 
	matrix $V$. Moreover, the constant in the last inequality is uniform $N$ and in $z\in\Omega_N$ 
	by \eqref{qm3}. 
	Since $\Lambda$ is a diagonal matrix, we get from \eqref{qm9} that 
	\begin{equation}\label{qm10}
		 	\|\mathfrak{Z}^+A_r^+\psi_r \|^2 \geq s^2_{m_+}(V)
			\sum_{k=1}^{m_+}| (A_r^+\psi_r)(k)|^2 \sum_{j=0}^{
			\lfloor (N-2)/m_+\rfloor -1 } |\zeta^+_k|^{2jm_+}.
	\end{equation}
	Putting
	\begin{equation}\label{qm10.1}
		 	d_k^2 \defeq \sum_{j=0}^{\lfloor (N-2)/m_+\rfloor -1 } |\zeta^+_k|^{2jm_+} 
			=\frac{1-|\zeta^+_k|^{2\lfloor
			(N-2)/m_+\rfloor m_+}}{1-|\zeta^+_k|^{2m_+} },
	\end{equation}
	and $D=\diag(d_1,\dots, d_{m_+})$, we get from \eqref{qm10} that 
	\begin{equation}\label{qm11}
		 	\|\mathfrak{Z}^+A_r^+\psi_r \| \geq s_{m_+}(V) \| DA_r^+\psi_r\|.
	\end{equation}
Observe that 
	the smallest singular value of $D$ is 	$\min_k d_k\geq 1$. Recall \eqref{nqeq4}. Provided that 
	$\wind>m_+^0$, we notice that $A^+_1$ is an isometry since the columns of $X_+$ are orthonormal. 
	Thus, it follows from \eqref{qm11}, \eqref{qm9.1} that 
	\begin{equation}\label{qm12a}
	  \|\mathfrak{Z}^+A_1^+\psi_1 \| \geq  \frac{1}{C} \| \psi_1\|, 
	\end{equation}
	so the smallest singular value of $G_{+,1}$ is at least $ 1/C$, which, $G_{+,1}$ being self-adjoint, implies that 
	it is bijective with inverse bounded in norm by $C$. In particular, we have that 
	\begin{equation}\label{qm12aa}
	 \| A^+_1G_{+,1}^{-1/2}\|=\sqrt{C}.
	\end{equation}
	3. Working with $m_+^0\geq 1$, we turn to $G_{+,2}$ and we will show first 
	that for $n\geq \wind - m_+^0+1$ the $u^+_n$ 
	are, up to a small error, given by $\mathfrak{z}_{n}/\|\mathfrak{z}_{n}\|$. 
	First, we compute 
	\begin{equation}\label{qm12b}
	F_{n,m}\defeq 
	| \langle \mathfrak{z}^+_n \mid\mathfrak{z}^+_m \rangle | = \left| \sum_0^{N-1} (\zeta^+_n \overline{\zeta}^+_m)^\nu \right| 
	= \left(\frac{(1-|\zeta_n^+|^{2N})(1-|\zeta_n^+|^{2N}) +|(\zeta_n^+)^N-(\zeta_m^+)^N|^2}
	{(1-|\zeta_n^+|^2)(1-|\zeta_{m}^+|^2) +|\zeta_n^+-\zeta_m^+|^2}\right)^{1/2}.
	\end{equation}
	By \eqref{qm4.0b}, \eqref{qm4.0} and \eqref{qm3} we find that
	\begin{equation}\label{qm12c}
	  F_{n,m} = \left\{ \begin{array}{ll}
		\asymp 1, & \text{ when } 1 \leq n,m \leq m_+- m_+^0, \\ 
		\asymp 1, & \text{ when } n\neq m,\\ 
		\asymp \frac{ N}{\log N}, & \text{ when } m_+- m_+^0 + 1 \leq n=m \leq m_+.
	\end{array}\right.
	\end{equation}
	Second, notice that by \eqref{nqeq1.1}, \eqref{nqeq1.2}, we have that 
	$T:= -\mathfrak{B}_+^{-1}(\mathfrak{A}^+_{m_+-m_+^0+1},\dots,\mathfrak{A}^+_{m_+})$ has norm 
	$\|T\| =O(1)$, depending only on $\Omega$. Recalling the definition of $\Delta_+$ from the 
	discussion before \eqref{nqeq4}, we see that 
	\begin{equation}\label{qmn1}
			\mathfrak{Z}^+A_2^+\mathfrak{L}_+^{-1} = 
			\widetilde{\mathfrak{Z}}^+T\mathfrak{L}_+^{-1}  + \widehat{\mathfrak{Z}}^+\mathfrak{L}_+^{-1},
	\end{equation}
	where $\widetilde{\mathfrak{Z}}^+=( \mathfrak{z}_1^{+}, \dots,\mathfrak{z}_{N_+}^{+}) $ 
	and $\widehat{\mathfrak{Z}}^+=( \mathfrak{z}_{m_+-m_+^0+1}^{+}, \dots,\mathfrak{z}_{m_+}^{+})$. Using 
	\eqref{qm12c}, we get that 
	\begin{equation}\label{qmn2}
	\begin{split}
	G_{+,2} &= \mathfrak{L}_+^{-1}(\widehat{\mathfrak{Z}}^+) ^*\widehat{\mathfrak{Z}}^+\mathfrak{L}_+^{-1}
	+  \mathfrak{L}_+^{-1}\left( (\widehat{\mathfrak{Z}}^+)^*\widetilde{\mathfrak{Z}}^+T
	+(\widetilde{\mathfrak{Z}}^+T)^* \widehat{\mathfrak{Z}}^+
	+ ( \widetilde{\mathfrak{Z}}^+T)^*\widetilde{\mathfrak{Z}}^+T\right)\mathfrak{L}_+^{-1} 
	\\ &
	= 1 + O(1) \frac{\log N}{N} ,
	\end{split}
	\end{equation}
	where the $O(1)$ term denotes an $m_+^0\times m_+^0$ matrix with norm bounded by 
	$O(1)$, depending only on $\Omega$ and the constants in Assumption \ref{As:Omega}. 
	Hence, for $N>0$ sufficiently large (depending only 
	on $\Omega$ and the constants in Assumption \ref{As:Omega}), we have that $G_{+,2}$ is 
	invertible and 
	\begin{equation}\label{qmn3}
	G_{+,2}^{-1/2} = 1 + O(1) \frac{\log N}{N}.
	\end{equation}
	Hence, using again \eqref{qm12c} and denoting by $M_{-,n}$ the $n$-th column of a matrix, we 
	get that for $n\in I_{+,2}$ 
	\begin{equation}\label{qmn4}
	\begin{split}
		\widetilde{u}^+_n &= \left( \mathfrak{Z}^+A_2^+\mathfrak{L}_+^{-1}G_{+,2}^{-1/2} \right)_{-,n} \\
		&=\left(  \widetilde{\mathfrak{Z}}^+T\mathfrak{L}_+^{-1}G_{+,2}^{-1/2}   \right)_{-,n}  
		+\left(  \widehat{\mathfrak{Z}}^+\mathfrak{L}_+^{-1}G_{+,2}^{-1/2}  \right)_{-,n} 
		 = \frac{\mathfrak{z}_n^+}{\|\mathfrak{z}_n^+\|} + O(1)\sqrt{\frac{\log N}{N}},
	\end{split}
	\end{equation}
	where the $O(1)$ term denotes a vector in $\C^N$ whose $\ell^2$ norm is bounded by $O(1)$. 
	In particular, \eqref{qm12a} and \eqref{qmn3} imply that \eqref{qm5} is well-defined. 
	\begin{rem}\label{rem:Modes}
	For a later use we note that the analogue of \eqref{qmn4} in the case when $\wind<0$ holds, i.e. 
	$$\widetilde{u}^-_n =  \frac{\mathfrak{z}_n^-}{\|\mathfrak{z}_n^-\|} + O(1)\sqrt{\frac{\log N}{N}}, \quad 
	n \in I_{-,2}.$$
	\end{rem}
	\par
	4. Since $(\mathfrak{Z}^+A_1^+G_{+,1}^{-1/2})^*(\mathfrak{Z}^+A_1^+G_{+,1}^{-1/2}) = 1$, 
	and $(\mathfrak{Z}^+A_2^+\mathfrak{L}_+G_{+,2}^{-1/2})^*(\mathfrak{Z}^+A_2^+\mathfrak{L}_+G_{+,2}^{-1/2}) = 1$, we confirm the first line of \eqref{qm6.01}. To prove the second line, let $(n,m) \in I_{+,1} \times I_{+,2}$. 
	Equations \eqref{qmn4} and \eqref{qm5} then yield that 
	\begin{equation}\label{qmn5}
	\begin{split}
		\langle \wt u^+_n | \wt u^+_m \rangle &= 
		\left \langle  \mathfrak{Z}^+ (A^+_1G_{+,1}^{-1/2})_{-,n} \Big| \frac{\mathfrak{z}_{m}^+}{\|\mathfrak{z}_{m}^+\|} \right \rangle + O(1)\sqrt{\frac{\log N}{N}} \\
		& = \sum_{k=1}^{m_+-m_+^0} \langle \mathfrak{z}_k^+ | \mathfrak{z}_{m}^+ \rangle \|\mathfrak{z}_{m}^+\|^{-1}(A^+_1G_{+,1}^{-1/2})_{k,n} + O(1)\sqrt{\frac{\log N}{N}}  
		\stackrel{\eqref{qm12aa}, \eqref{qm12c}}{\leq} O(1)\sqrt{\frac{\log N}{N}}.
	\end{split}
	\end{equation}
	The other case of the second line of \eqref{qm6.01} follows from symmetry. 
	\\
	\par
	5. At last,
	we turn to proving \eqref{qm6}. Let $j\in I_{+,1}$. 
	Using H\"older's inequality, we compute
	\begin{equation}\label{qmn6}
	\begin{split}
		 	\| \mathbf{1}_{[N_0,N-1]} \widetilde{u}_j^+ \|^2 &= 
			\sum_{\nu = N_0}^{N-1} | (\mathfrak{Z}^+A^+ G_{+,1}^{-1/2})_{\nu,j}|^2 
			\leq 
			\sum_{\nu = N_0}^{N-1} \left |\sum_{\mu=1}^{m_+-m_0^+} 
			(\mathfrak{Z}^+)_{\nu,\mu} (A^+ G_{+,1}^{-1/2})_{\mu,j}\right|^2 \\
			&\leq \sum_{\nu = N_0}^{N-1} \left (\sum_{\mu=1}^{m_+-m_0^+} 
			|(\mathfrak{Z}^+)_{\nu,\mu} |^2\right) 
			\left (\sum_{\mu=1}^{m_+-m_0^+} 
			|(A^+ G_{+,1}^{-1/2})_{\mu,j}|^2\right).
	\end{split} 
	\end{equation}
	Since the right most term in the last line is controlled by the Hillbert-Schmidt norm of the corresponding 
	matrices, we get by using \eqref{qm12aa} and summing the geometric series that 
	\begin{equation}\label{qmn6.1}
	\begin{split}
		 	\| \mathbf{1}_{[N_0,N-1]} \widetilde{u}_j^+ \|^2 
			&\leq  \| A^+ G_{+,1}^{-1/2}\|^2_{\mathrm{HS}} \sum_{\nu = N_0}^{N-1} \left (\sum_{\mu=1}^{m_+-m_0^+} 
			|(\mathfrak{Z}^+)_{\nu,\mu} |^2\right) \\
			&\stackrel{\eqref{qm12aa}}{\leq} O(1) \sum_{\nu = N_0}^{N-1} \left (\sum_{\mu=1}^{m_+-m_0^+} 
			|\zeta_\mu^+|^{2\nu}\right) 
			= O(1)\sum_{\mu=1}^{m_+-m_0^+}  |\zeta_\mu^+|^{2N_0} \frac{1-|\zeta_\mu^+|^{2(N-N_0)}}{1-|\zeta_\mu^+|^{2}}.
	\end{split} 
	\end{equation}
	Since $|\zeta_{m_+-m_+^0}| \leq 1/C_\beta<1$, we have that
	for $\mu=1,\dots, m_+-m_+^0$,
	\begin{equation}\label{qmn7}
		\frac{1-|\zeta^+_\mu|^{2(N-N_0)}}{1-|\zeta^+_\mu|^{2}} \leq \frac{C_\beta}{C_\beta-1}.
	\end{equation}
	Similarly we see that 
	\begin{equation}\label{qmn7.1}
	|\zeta^+_{\mu}|^{2N_0}
	\leq \e^{-2N_0 \log C_\beta}.
	\end{equation}
	Combining this with \eqref{qmn6.1} and \eqref{qmn7} yields that 
	\begin{equation*}
	\| \mathbf{1}_{[N_0,N-1]} \widetilde{u}_j^+ \| \leq O(1)  \e^{-N_0 \log C_\beta},
	\end{equation*}
	uniformly in $z\in\Omega_N$, and we conclude the first line in 
	\eqref{qm6}.
	\\
	\par 
	It remains to prove the second line of \eqref{qm6}. Recall the definition of $A_2^+$ \eqref{nqeq4}, 
	and notice that 	
	\begin{equation}\label{qm12fff}
	\begin{split}
	A_2^+ \mathfrak{L}_+^{-1} &=
	\begin{pmatrix}  - \mathfrak{B}^{-1}(\mathfrak{A}^+_{m_+-m_+^0+1},\dots,\mathfrak{A}^+_{m_+}) \mathfrak{L}_+^{-1} \\ 
	\Delta_+ \mathfrak{L}_+^{-1} \\ 
	\end{pmatrix}\\
	&=\widetilde{\mathfrak{L}}_+^{-1} 
	\begin{pmatrix}  - \mathfrak{B}^{-1}(\mathfrak{A}^+_{m_+-m_+^0+1},\dots,\mathfrak{A}^+_{m_+}) \mathfrak{L}_+^{-1} \\ 
	\Delta_+ \\ 
	\end{pmatrix}\defeq \widetilde{\mathfrak{L}}_+^{-1} \widetilde{A}_2^+
	\end{split}
	\end{equation}
	where $\widetilde{\mathfrak{L}}_+:=1_{N_+} \oplus  
	\diag( \|\mathfrak{z}_{N_++1}^{+}\|, \dots,\|\mathfrak{z}_{m_+}^{+}\|)$. Furthermore, it follows from 
	\eqref{nqeq1.1}, \eqref{nqeq1.2} and \eqref{qm12c} that 
	\begin{equation}\label{qm12f}
		\| \mathfrak{B}^{-1}(\mathfrak{A}^+_{m_+-m_+^0+1},\dots,\mathfrak{A}^+_{m_+}) \mathfrak{L}_+^{-1}\| =
		O(1) \sqrt{ \frac{\log N}{N}},
	\end{equation}
	which implies that 
	\begin{equation}\label{qm12ff}
		\| \widetilde{A}_2^+ \| = O(1).
	\end{equation}
	Using the commutation relation \eqref{qm12fff}, we compute for $j\in I_{+,2}$ that 
	\begin{equation}\label{qmn8}
	\begin{split}
		 	\| \mathbf{1}_{[N_0,N-1]} \widetilde{u}_j^+ \|^2 &= 
			\sum_{\nu = N_0}^{N-1} | (\mathfrak{Z}^+A_2^+\mathfrak{L}_+^{-1}G_{+,2}^{-1/2})_{\nu,j}|^2 
			\leq 
			\sum_{\nu = N_0}^{N-1} \left |\sum_{\mu=1}^{m_+} 
			(\mathfrak{Z}^+\widetilde{\mathfrak{L}}_+^{-1})_{\nu,\mu} (\widetilde{A}_2^+G_{+,2}^{-1/2})_{\mu,j}\right|^2 \\
			&\leq \sum_{\nu = N_0}^{N-1} \left (\sum_{\mu=1}^{m_+} 
			|(\mathfrak{Z}^+\widetilde{\mathfrak{L}}_+^{-1})_{\nu,\mu} |^2\right) 
			\left (\sum_{\mu=1}^{m_+} 
			|(\widetilde{A}^+ G_+^{-1/2})_{\mu,j}|^2\right).
	\end{split} 
	\end{equation}
	Thus, by \eqref{qm12ff}, \eqref{qmn3}
	\begin{equation}\label{qmn8.1}
	\begin{split}
		 	\| \mathbf{1}_{[N_0,N-1]} \widetilde{u}_j^+ \|^2 
			&\leq \| \widetilde{A}^+ G_+^{-1/2}\|_{\mathrm{HS}}^2 
			\sum_{\nu = N_0}^{N-1} \left (\sum_{\mu=1}^{m_+} 
			|(\mathfrak{Z}^+\widetilde{\mathfrak{L}}_+^{-1})_{\nu,\mu} |^2\right) \\
			&\leq O(1)\sum_{\nu = N_0}^{N-1} \left (\sum_{\mu=1}^{N_+} |\zeta_\mu^+|^{2\nu}+
			\sum_{\mu=N_++1}^{m_+} \frac{|\zeta_\mu^+|^{2\nu}}{\| \mathfrak{z}_\mu^+\|^2}
			\right). 
	\end{split} 
	\end{equation}
	Using \eqref{qm12c}, we get by a computation similar to \eqref{qmn7} and \eqref{qmn7.1} that
	\begin{equation}\label{qmn8.12}
	\sum_{\nu = N_0}^{N-1} \left (\sum_{\mu=1}^{N_+} |\zeta_\mu^+|^{2\nu}+
			\sum_{\mu=N_++1}^{m_+-m_+^0} \frac{|\zeta_\mu^+|^{2\nu}}{\| \mathfrak{z}_\mu^+\|^2}
			\right) = O(1)  \e^{-N_0 \log C_\beta}.
	\end{equation}
	On the other hand, using \eqref{qm12b} and the fact that $|\zeta_{m_+}|<1$, 
	we find $\mu = m_+-m_+^0+1,\dots, m_+$ that 
	\begin{equation}\label{qmn8.13}
	\sum_{\nu = N_0}^{N-1}  \frac{|\zeta_\mu^+|^{2\nu}}{\| \mathfrak{z}_\mu^+\|^2}
	= |\zeta_\mu^+|^{2N_0} \frac{1-|\zeta_\mu^+|^{2(N-N_0)}}{1-|\zeta_\mu^+|^{2N}}
	\leq  |\zeta_\mu^+|^{2N_0}.
	\end{equation}
	By \eqref{qm4.0}, we see that 
	\begin{equation}\label{qmn8.13a}
	  |\zeta_\mu^+|^{2N_0} 
	 \leq \exp\left( 2N_0\log \left(1-C_{\gamma}\frac{\log N}{N}\right)\right) 
	 \leq  \exp\left( -2N_0C_{\gamma}\frac{\log N}{N}\right) 
	 \leq \e^{-\frac{2N_0}{N}\log N^{C_{\gamma}}},
	\end{equation}
	which, together with \eqref{qmn8.13}, \eqref{qmn8.12} and \eqref{qmn8.1} yields that 
	for $N>0$ sufficiently large 
	\begin{equation}\label{qmn8.14}
		 	\| \mathbf{1}_{[N_0,N-1]} \widetilde{u}_j^+ \|^2 \leq O(1) \e^{-\frac{2N_0}{N}\log N^{C_{\gamma}}},
	\end{equation}
	uniformly in $z\in\Omega_N$, and we conclude the first second line in \eqref{qm6}.
\end{proof}
\section{Singular values and vectors}
\label{sec-sing}
We provide in this section various estimates on singular values and vectors
of $P_N-z$. The section is divided to subsections, dealing respectively
with small singular values, singular vectors, large singular values, and the smallest singular value and Hilbert-Schmidt norm of the matrices appearing in the resolvent expansion of the matrices discussed in Section \ref{res}. 

\subsection{Small singular values}\label{sec:SmallSGValues}
We will work under the same assumptions as discussed in the beginning of 
Section \ref{sec:QM}. 
We will use the quasimodes constructed in Proposition \ref{prop:QM} 
to obtain an 
upper bound on the  $|\wind|$-th
smallest singular values of $P_N-z$ 
corresponding to the decay speed of the quasimodes. We will also show with the help 
of Theorem \ref{thm:resBound} that the $(|\wind|+1)$-th singular 
value of $P_N-z$ is significantly larger than the $|\wind|$-th one, 
which leads to a 
spectral gap. This is known as \textit{splitting phenomenon}, see for instance 
\cite[Section 9.2]{BoGr05}. This phenomenon has typically been investigated for $z$ being at 
a fixed distance from $p(S^1)$ or for fixed symbols. In 
that case it 
is known that 
the smallest $|\wind|$ 
singular values are of order $O(\e^{-N/C})$ 
whereas the $(|\wind|+1)$-th 
singular value is bounded from 
below by some small $N$ independent constant. What is new here is that we allow the 
spectral parameter $z$ to be at an $N$-dependent distance from $p(S^1)$,
which leads to a significantly smaller spectral gap. 
\par
Moreover, we will show that due to the assumptions \eqref{qm4.0b}, \eqref{qm4.0c}, 
there exists a second spectral gap between the first $(|\wind|-m_{\mathrm{sign}(\wind)}^0)$ 
singular values, which are exponentially small in $N$, and the $(|\wind|-m_{\mathrm{sign}(\wind)}^0+1)$-th 
singular value, which is at most polynomially small in $N$. 
\par
Both spectral gaps will be crucial in describing the localization of 
the eigenvectors of $P^{\delta}_N$. 
\begin{prop}\label{prop:SmallSG} 
	Let $z_0 \in p(S^1)$ be as in \eqref{qm1} and let 
	$\Omega\Subset \C$ be a sufficiently small open simply connected relatively 
	compact neighborhood of a point $z_0$ satisfying \eqref{qm2}. 
	Let $\Omega_N\Subset \Omega \backslash p(S^1)$ be as in Assumption 
	\ref{As:Omega}. For $z\in\Omega_N$ let $t_1 \leq \cdots \leq t_N$ denote
	the eigenvalues of $\sqrt{(P_N-z)^*(P_N-z)}$. 
	Then there exists a constant $0<C<\infty$ such that, for all $N$ large enough (depending only 
	on $\Omega$ and the constants in Assumption \ref{As:Omega}) and for all $z\in\Omega_N$ 
	 \begin{equation}\label{sgv1}
	 	0 \leq t_1 \leq\cdots\leq  t_{|\wind|} \leq C\frac{\log N}{N}N^{-C_{\gamma}}, 
	 \end{equation}
	 and
	  \begin{equation}\label{sgv2}
	 	 t_{|\wind|+1} \geq \frac{1}{C}\frac{\log N}{N}.
	 \end{equation}
	 Additionally, if $|\wind|-m_{\mathrm{sign}(\wind)}^0 >0$, then 
	 %
	  \begin{equation}\label{sgv1a}
	 	0 \leq t_1 \leq\cdots\leq t_{|\wind|-m_{\mathrm{sign}(\wind)}^0} \leq C\e^{-N \log C_\beta}, 
	 \end{equation}
	 and if $m_{\mathrm{sign}(\wind)}^0 >0$
	  \begin{equation}\label{sgv2a}
	 	\frac{1}{C}\frac{\log N}{N}N^{-C_\alpha} \leq t_{|\wind|-m_{\mathrm{sign}(\wind)}^0+1} \leq 
		\cdots \leq  t_{|\wind|} \leq C\frac{\log N}{N}N^{-C_{\gamma}}.
	 \end{equation}
\end{prop}
\begin{rem}
If in Assumption \ref{As:Omega} we have that  $|\zeta_{m_{+}-m_+^0+1}^+|=|\zeta_{m_{+}}^+|$ (in case $\wind>0$) or
$|\zeta_{m_{-}^0+1}^{-}|= |\zeta_{m_{-}}^-|$ (in case $d<0$) then clearly one can take $C_\alpha=C_\gamma+\epsilon$ for any $\epsilon>0$,
improving upon \eqref{sgv2a}. By \eqref{eq:ze-t-o-wt-ze}, this will always be the case for us.
\end{rem}

\begin{proof}
  Recall that for a square $N\times N$ matrix $A$, $t_j(A)$ 
  denote the eigenvalues of $\sqrt{A^*A}$ (ordered non-decreasingly) and $s_j(A)$
  denote the singular values (ordered non-increasingly), see 
  \eqref{qm24.0}. 
  By the singular value decomposition of the square matrix 
	$A$, we have that
	$t_j(A) = t_j(A^*)$ and 
	$s_j(A) = s_j(A^*)$.  Recall 
	the Ky Fan inequalities, see for instance 
	\cite{GoKr69}: for two square matrices $A,B$ we have 
	 \begin{equation}\label{qm24}
	 \begin{split}
		&s_{n+m-1}(A+B) \leq s_n(A) + s_m(B), \\ 
		&s_{n+m-1}(AB) \leq s_n(A) s_m(B).
	\end{split}
	 \end{equation}
	Further, if $A$ is an invertible $N\times N$ matrix, then 
	\begin{equation*}\label{qm24.1}
	s_n(A^{-1}) = \frac{1}{s_{N-n+1}(A)}, \quad n=1,\dots, N.
	\end{equation*}
	From \eqref{a6.1} we know that the symbol of the adjoint $\mathrm{Op}(p)^*$ 
	is given by $\overline{\widetilde{p}}(\omega) = \overline{p}(1/\omega)$. So 
	in view of Lemma \ref{lem:root1}, the roots of $\overline{p}(1/\omega)  -\overline{z}$, or 
	equivalently the roots of $p(1/\overline{\omega}) -{z}$ are given by 
	$\omega^{\mp}_j = 1/\overline{\zeta}_j^{\,\pm}$, 
	where $\zeta_{j}^{\pm}$ are the roots of $p(\zeta)-z$. Hence, the discussion and results 
	of Section \ref{Sec:AnaTO} valid for 
	$p(\zeta)-z$ remain valid for $\overline{p}(1/\omega) -\overline{z}$ with the roles of 
	$m_+,N_+$, and $m_-,N_-$ exchanged. Furthermore, the roles of the assumption \eqref{qm4.0b} 
	and \eqref{qm4.0c} are then interchanged. This argument will allow us to reduce the number 
	cases of symbols $p-z$ we need to consider, since we may always pass to studying the singular 
	values of the adjoint $(P_N-z)^*$. In particular, in the rest of the proof we may and will
	assume that $\wind>0$ for $z\in\Omega_N$.
	\par
	In what follows, we will work with $N>0$ sufficiently 
	large, depending only on $\Omega$ and the constants in Assumption \ref{As:Omega},  
	without us mentioning this dependence at every occurrence. Similar, we note that all 
	error estimates 	will be uniform in $z\in\Omega_N$, and in fact they depend only on the above mentioned 
	parameters, so we shall only mention this at a few important points. 
\\	
	\par
1. Let $z\in \Omega_N$. We begin with the upper bound on $t_{\wind}$.  
Let $\psi^+_j$, $j=1,\dots, \wind$ be as in \eqref{qn2}, 
and note that for $N>0$ sufficiently large, they are linearly independent. Indeed 
if they were not, then there would exist an $0\neq a=(a_1,\dots,a_{\wind}) \in \C^{\wind|}$ 
such that $a_1\psi^+_1+\dots+a_{\wind}\psi^+_{\wind} =0$.
Let $j_0$ be such that $|a_{j_0}|\geq |a_j|$ for $j\neq j_0$, with $|a_{j_0}|>0$. 
Since the 
$\psi^+_j$ are almost-orthonormal by \eqref{qn3}, we find that 
\begin{equation}\label{eqnAP0}
  1 +O( \log N) N^{-C_{\gamma}} = -\sum_{j\in [\wind], j\neq j_0} \frac{a_j}{a_{j_0}} \langle \psi^+_{j}| \psi^+_{j_0}\rangle =  
O( \log N) N^{-C_{\gamma}} +O((\log N /N)^{1/2} ),
\end{equation}
a contradiction. \\
\par
	Let $K_{\wind}\subset\C^N$ denote the linear subspace of $\C^N$ spanned by the $\psi_j^+$ and 
	notice that it has dimension $\wind$, for $N>0$ sufficiently large. For any $\psi \in K_{\wind}$, 
	we 
	write $\psi = \sum_{j=1}^{\wind} 
	a_j \psi^+_j$. If $\|\psi\| =1$, 
	then we get by \eqref{qn3} and the Cauchy-Schwarz  inequality that for $N>0$ sufficiently large, 
	 \begin{equation*}
	   1 = \sum_{j=1}^{\wind} |a_j|^2 \|\psi^+_j\|^2 + \
		\sum_{i\neq j } a_i \overline{a}_j \langle \psi^+_i | \psi^+_j \rangle  
		 = \| a\|^2 ( 1 + \hat\epsilon_N),
	 \end{equation*}
	 where $\hat\epsilon_N= O(1) (N^{-C_{\gamma}} \log N + (\log N/N)^{1/2})$ and the constant is uniform in $z\in\Omega_N$ and independent of $\psi$. 
	Thus, for $N>0$ large enough, 
	$\|a\|^2 = ( 1 + \hat \epsilon_N)$,
	uniformly in the choice of $z\in\Omega_N$ and $\psi$.
	The min-max principle and \eqref{qn4} 
	yield that 
	 \begin{equation*}
	 \begin{split}
	 t_{\wind}^2 &= \min\limits_{L:\dim L = \wind } 
	 \max\limits_{ \substack{\psi \in L \\ \|\psi\| =1}} \| (P_N-z)\psi\|^2 
	  \leq \max\limits_{\substack{\psi \in K_{\wind} \\ \|\psi\| =1}} \| (P_N-z)\psi\|^2 \\ 
	 & \leq \max\limits_{\substack{\psi \in K_{\wind} \\ \|\psi\| =1}} \left( \sum_1^{\wind} |a_j| \cdot \| (P_N-z)\psi^+_j\|\right)^2  
	 \stackrel{\eqref{qn4}}{\leq}
	 O(1)\left(  \frac{\log N}{N} N^{-C_{\gamma}} \right)^2,
	 \end{split}
	 \end{equation*}
	and we conclude the upper bound in \eqref{sgv1} and \eqref{sgv2a}.
	\par
	Suppose that $\wind>m_+^0$ and recall \eqref{qm4.0b} and \eqref{qm4.0c}. 
	Letting $K_{\wind-m_+^0}$ denote 
	the linear subspace of $\C^N$ spanned by 
	$\psi^+_1,\dots,\psi^+_{\wind-m_+^0}$, and recalling 
	from \eqref{qn3} that these vectors are orthonormal, 
	we deduce similarly as above  from the min-max principle and 
	\eqref{qn4}, that  
	 \begin{equation*}
	 t_{\wind-m_+^0}^2 
	  \leq \max\limits_{\substack{\psi \in K_{\wind-m_+^0} \\ \|\psi\| =1}} \| (P_N-z)\psi\|^2 
	  \leq \sum_1^{\wind-m_+^0} \| (P_N-z)\psi^+_j\|^2
	  \stackrel{\eqref{qn4}}{\leq}
	  O(\e^{-2N \log C_\beta}),
	 \end{equation*}
	and we conclude \eqref{sgv1a}.
	\\
	\par 
2. We may write $\Omega$ as the disjoint union 
\begin{equation}\label{nqm0.1}
	\Omega = \Omega'\, \dot{\cup} \,( p(S^1) \cap \Omega) \,\dot{\cup}\, \Omega'' \notag
\end{equation}
so that $\Omega_N \subset \Omega'$. In particular $\mathrm{ind}_{p(S^1)}(z) = \wind$ 
for all $z\in\Omega'$.
\par
The rest of the proof will be concerned with the lower bounds on $t_{\wind+1}$ and $t_{\wind-m_+^0+1}$ 
for $z\in\Omega_N$, 
and we will have to consider the two cases of $N_{\pm} >0$ or $N_+\leq 0$. 
Note that by \eqref{int0}, 
these are the only cases we need to consider,
since $\wind>0$ while 
by Remark 
\ref{rem1} and \eqref{at9.0}, we have that $\wind\leq 0$ if $N_-<0$.
We note here that in both cases 
$\infty \notin p^{-1}(\Omega)$, see \eqref{int0} and Lemma \ref{lem:root1}.
\par
Recall from the 
discussion after \eqref{qm2} that all roots $\zeta_z$ of $p(\zeta)-z$ 
depend smoothly on $z\in\Omega$. So, when either $N_\pm >0$, or $N_+ \leq 0$ and 
$0\notin p^{-1}(z_0)$, we see that $\{0,\infty\} \notin p^{-1}(\Omega)$ for $\Omega$ sufficiently 
small. On the other hand, when $N_+ \leq 0$ and $0 \in p^{-1}(z_0)$, then by the smooth 
dependence of the roots on $z$ that $0 \notin p^{-1}(\Omega')$, after possibly shrinking $\Omega$. 
Hence, for $\Omega$ sufficiently small, there exists a constant $C>0$ such 
that any root $\zeta_z$ satisfies 
\begin{equation}\label{nqm0}
	0<|\zeta_z| \leq C, \text{ for all } z \in \Omega'.
\end{equation}
Since the symbol $p(\zeta)-z$ for $z\in\Omega'$ satisfies the assumptions 
of Case 1 of Lemma \ref{lem:root1}, we may order its roots as in \eqref{at4.2.0}. So, in combination 
 with \eqref{nqm0} we get that 
\begin{equation}\label{nqm1}
0 <  |\zeta_1^+|\leq \cdots \leq |\zeta_{m_+}^+| < 1 < |\zeta_1^-| \leq \cdots \leq |\zeta_{m_-}^-| \leq C, 
\text{ for all } z \in \Omega'.
\end{equation}
By \eqref{at9.0} we know that $m_+\geq \wind> 0$, so we work in \eqref{nqm1} with the  
convention that when $m_-=0$ then only the estimates on the existing roots 
in \eqref{nqm1} apply. 
\\
\par
Recall Assumption \ref{As:Omega}, and as in Theorem \ref{thm:resBound}, let 
$\Theta$ denote the Heaviside function, and let 
 \begin{equation}\label{qm25a}
   \widehat{m}_0\in \left\{\begin{array}{ll}
		\{0,m_+^0\}, & m_+^0 >0\\
		\{0\}, &m_+^0=0.
	\end{array}\right.
 \end{equation}
To shorten the notation we will write $\Theta := \Theta(\widehat{m}_0)$ whenever convenient.  
The case $\widehat{m}_0=0$ will be used to prove the lower bound in \eqref{sgv2}, 
and the case $\widehat{m}_0=m_+^0$ for the lower bound in \eqref{sgv2a}. 
Write $\widetilde{N}_+=N_+\lor 0$. For $\zeta \in S^1$ and $z\in \Omega'$,  
we can write the symbol as  
 \begin{equation}\label{qm25}
	p(\zeta) -z = \zeta^{-\widetilde{N}_+} \left( \sum_{0}^{\widetilde{N}_++N_-} a_{N_+-j} \,\zeta^{j} -z\zeta^{\widetilde{N}_+} \right) 
	 = \zeta^{m_+ - \widetilde{N}_+-m^0_+\Theta} q(\zeta)
	\stackrel{\eqref{at9.0}}{= }\zeta^{\wind-m^0_+\Theta} q(\zeta),
 \end{equation}
 where 
 \begin{equation}\label{qm25.1}
	 q(\zeta) =\sum_{-m_--m^0_+\Theta}^{m_+-m^0_+\Theta} q_j \,\zeta^{-j}
 :=  a_{-N_-} 
	 \prod_{1}^{m_+-m^0_+\Theta}(1-\zeta^+_j/\zeta) 
	\prod_{m_+-m^0_+\Theta+1}^{m_+}(\zeta-\zeta^+_j)\prod_1^{m_-}(\zeta-\zeta^-_j )
 \end{equation}
 Here $q_{m_+-m^0_+\Theta} = a_{-N_-}\prod_{1}^{m_+}(-\zeta_j^+) \prod_1^{m_-}(-\zeta_j^-)$, 
 and $q_{-m_--m^0_+\Theta} = a_{-N_-}$. So, by \eqref{nqm1} and \eqref{int0}, we conclude 
 that there exists some constant $C>0$ such that for all $z\in\Omega'$ 
  \begin{equation}\label{qm25.1a}
	q_{m_+-m^0_+\Theta}\neq 0, \quad |q_{- m_--m^0_+\Theta}| \geq 1/C, \quad \notag
	|q_j| \leq C.
 \end{equation}
Equation \eqref{qm25.1} shows that there are no roots of $q$ on $S^1$ and, by performing a 
computation as in \eqref{at9.0}, that $\mathrm{ind}_{q(S^1)}(0)=\widehat{m}_0$ for all $z\in\Omega_N$. This, 
together with \eqref{nqm0}, \eqref{qm3}, \eqref{qm4.0} and \eqref{qm4.0b} shows that $q$, $\Omega_N$ 
and $\Omega'$ satisfy the assumptions of Theorem \ref{thm:resBound} with $\widehat{m}_0$ as in 
\eqref{qm25a}, $C_0=C_\alpha$, $C_1=C_\beta$, $C_2=C_\gamma$,
and therefore there
exists a constant $C>0$ such that for $N>0$ sufficiently large and all 
$z\in \Omega_N$, 
 \begin{equation}\label{nqm2}
   s_{N}(P_N( q)) \geq  \frac{1}{C} \frac{\log N}{ N} N^{-C_\alpha \Theta(\widehat m_0)}.
 \end{equation}
When $\wind = m_+^0>0$ then $s_{N-j+1} = t_{j}$, and
we immediately conclude the lower 
bound in \eqref{sgv2a} using that
$\widehat m_0 =m_+^0$. Hence, from now on we may 
and will assume that 
$\wind > m_+^0\Theta(\widehat m_0)$, see also Assumption \ref{As:Omega}.
\par
Recall the definition of $\chi_n$ from \eqref{eq-defchi}.
As noted after \eqref{a7.3} we have that 
$H( \chi_{n} ) = 0$ when $n\leq 0$ and that $H( \chi_{n} )$ has rank $n$, 
when $n> 0$. So by \eqref{a7} and \eqref{qm25} we get that 
 \begin{equation}\label{qm26}
   P_N( p -z )P_N( \chi_{m^0_+\Theta-\wind} ) = 
   P_N( q) - \Pi_N H(p -z ) H(\chi_{\wind-m^0_+\Theta} ) \Pi_N, 
 \end{equation}
where the second term on the right hand side has rank $\wind-m^0_+\Theta$. 
Since $s_1(P_N( \chi_{m^0_+\Theta-\wind} )) = \| P_N( \chi_{m^0_+\Theta-\wind} )\| = 1$, we 
get by \eqref{qm26} and the second Ky Fan
inequality in \eqref{qm24} that 
for any $n=1,\dots,N$, 
 \begin{equation}\label{qm27}
   s_n(P_N( p -z )) = s_n(P_N( p -z ))\, s_1(P_N( \chi_{m^0_+\Theta-\wind} ) ) 
   \geq s_n (P_N( q) - \Pi_N H(p -z ) H(\chi_{\wind-m^0_+\Theta} ) \Pi_N).
 \end{equation}
 By interleaving inequalities for singular values, see e.g. 
 \cite[Problem III.6.4]{Bhatia},
  for any matrices 
  $A,K\in \C^{N\times N}$
with $K$ of rank $m\in \{ 0,\dots, N-1\}$,
 \begin{equation}\label{qm27.1}
 s_{N-m}(A) 
	 \geq  s_{N}(A+K).
       \end{equation}
Applying this estimate with 
 \begin{equation*}
  A=P_N( q)- \Pi_N H(p -z ) H(\chi_{\wind-m^0_+\Theta} ) \Pi_N, 
  ~K=\Pi_N H(p -z ) H(\chi_{\wind-m^0_+\Theta} ) \Pi_N,
 \end{equation*}
and $m=\wind-m^0_+\Theta$, yields together with \eqref{qm27} and \eqref{nqm2} 
that for $N>0$ large enough (so that \eqref{nqm2} holds) 
 \begin{equation*}
   s_{N-\wind+m^0_+\Theta}(P_N( p -z )) 
   \geq  s_N (P_N( q)) 
   \geq \frac{1}{C}\frac{\log N}{N}N^{-C_\alpha \Theta(\widehat m_0)}, 
 \end{equation*}
for all $z\in\Omega_N$. Since $s_{N-j+1} = t_{j}$, we get \eqref{sgv2} by taking $\widehat{m}_0=0$, 
and we obtain 
the lower bound in \eqref{sgv2a}, when $\wind>m_+^0>0$, 
by taking $\widehat{m}_0=m_+^0$. 
This ends the proof of Proposition \ref{prop:SmallSG}.
\end{proof}
\subsection{Singular vectors}
\label{sec-singular}
Throughout this section we will continue to work with $z\in\Omega_N$ satisfying 
Assumption \ref{As:Omega} and we 
will denote the orthonormal eigenvectors of  $(P_N-z)^*(P_N-z)$ 
corresponding to the eigenvalues $0\leq t_1^2 \leq \cdots \leq t_{N}^2$ by 
$e_1,\dots, e_{N}$. Using the spectral gap from Proposition \ref{prop:SmallSG} 
we will show that the quasimodes from Proposition \ref{prop:QM} span, up to a small 
error, the same space as the eigenvectors $e_1,\dots, e_{|\wind|}$. More precisely 
we will show that the eigenvectors $e_1,\dots, e_{|\wind|-m_{\mathrm{sign}(\wind)}^0}$ are 
well approximated by a linear combination of the quasimodes 
$\psi_1,\dots, \psi_{|\wind|-m_{\mathrm{sign}(\wind)}^0}$, whereas the eigenvectors 
$e_{|\wind|-m_{\mathrm{sign}(\wind)}^0+1},\dots, e_{|\wind|}$ are well approximated by a linear combination 
of $\psi_{|\wind|-m_{\mathrm{sign}(\wind)}^0+1},\dots, \psi_{|\wind|}$.
\begin{prop}\label{prop:sgvSpan}
Consider the setup as in Proposition \ref{prop:SmallSG}. 
	Let $\Pi_{\kappa} = \mathbf{1}_{[0,\kappa^2]}((P_N-z)^*(P_N-z))$, $\kappa\geq 0$, $z\in \Omega_N$, be 
	the spectral projector onto the eigenspace of $(P_N-z)^*(P_N-z)$ corresponding to the 
	eigenvalues $0\leq t_1^2 \leq \dots \leq t_{j}^2\leq \kappa^2$. Let $\psi_j^\pm$ and $I_{\pm,1}, I_{\pm,2}$ 
	be as in Proposition \ref{prop:QM}, let $N>0$ be sufficiently large (depending only on 
	the $\Omega$ and the constants in Assumption \ref{As:Omega}), and put
	\begin{equation*}
	\widetilde{e}_j := 
		\frac{\Pi_{t_{|\wind|-m_{\mathrm{sign}(\wind)}^0}} \psi_j^{\,\mathrm{sign}(\wind)}}
	{ \|\Pi_{t_{|\wind|-m_{\mathrm{sign}(\wind)}^0}} \psi_j^{\,\mathrm{sign}(\wind)} \| }, \quad j \in I_{\mathrm{sign}(\wind),1},
	\end{equation*}
	and 
	\begin{equation*}
	\widetilde{e}_j := 
	\frac{\Pi_{t_{|\wind|}} \psi_j^{\,\mathrm{sign}(\wind)}}
	{ \|\Pi_{t_{|\wind|}} \psi_j^{\,\mathrm{sign}(\wind)} \| }	, \quad j  \in I_{\mathrm{sign}(\wind),2}.
	\end{equation*}
	Then, uniformly in $z\in \Omega_N$, 
	\begin{equation}\label{singv1}
	\| \widetilde{e}_j -\psi_j^{\,\mathrm{sign}(\wind)} \| = O(1)
	\left\{\begin{array}{ll} \frac{ N^{(1+C_\alpha)}}{\log N } \e^{-N \log C_\beta}, & j \in I_{ \mathrm{sign}(\wind),1}, \\
				 N^{-C_{\gamma}}\log N, & 
				  j\in I_{ \mathrm{sign}(\wind),2},
		\end{array}\right. \notag
	\end{equation}
	and 
	\begin{equation}\label{singv2}
	\langle \widetilde{e}_i  |  \widetilde{e}_j \rangle  = \delta_{i,j} 
	+ \left\{\begin{array}{ll}  \frac{ N^{(1+C_\alpha)}}{\log N } \e^{-N \log C_\beta}, & \text{ when }
		 i,j \in I_{ \mathrm{sign}(\wind),1}, \\
				  O( \log N) N^{-C_{\gamma}},& \text{ when }  i,j \in I_{ \mathrm{sign}(\wind),2}, \\
				  O( \log N) N^{-C_{\gamma}} +O((\log N /N )^{1/2}),& \text{else}.
		\end{array}\right. 
	\end{equation}
	\end{prop}
\begin{proof}
To ease the notation we drop the $\pm$ superscripts and let $\psi_j$ be either $\psi_j^+$ or $\psi_j^- $ 
depending on $\mathrm{sign}(\wind)$. Furthermore, we write $B_z := (P_N-z)^*(P_N-z)$ and we let $N\gg 1$. 
In what follows all constants will be uniform in $N\gg1$ and $z\in\Omega_N$ even when we do not 
state this fact explicitly. 
\par
By \eqref{qn4}, we have that 
\begin{equation}\label{p1}
	\langle B_z \psi_j | \psi_j \rangle = \| (P_N -z) \psi_j\|^2 
	= O(1) \left\{\begin{array}{ll} \e^{-2N \log C_\beta}, & j \in I_{ \mathrm{sign}(\wind),1} \\
				  \left(\frac{\log N}{N} N^{-C_{\gamma}}\right)^2, & 
				  j\in I_{ \mathrm{sign}(\wind),2}.
		\end{array}\right. 
\end{equation}
Let $\kappa= t_{|\wind|-m_{\mathrm{sign}(\wind)}^0}$ for $j \in I_{ \mathrm{sign}(\wind),1}$, 
and $\kappa=t_{|\wind|}$ for $ j \in I_{ \mathrm{sign}(\wind),2}$. Let 
$\Pi =\Pi_{\kappa}$ and $\Pi^\perp = \mathbf{1}_{]\kappa^2,\infty[}(B_z)$, and notice 
that both are selfadjoint projectors onto disjoint eigenspaces and that $\Pi + \Pi^\perp =1$. 
By the spectral theorem and the fact that $B_z\geq 0$, 
\begin{equation}\label{p2}
	\langle B_z \psi_j | \psi_j \rangle = \langle B_z \Pi \psi_j |\Pi  \psi_j \rangle 
	+ \langle B_z \Pi^\perp \psi_j | \Pi^\perp\psi_j \rangle 
	 \geq  \| \Pi^\perp\psi_j\|^2
	 \left\{ \begin{array}{ll} t_{|\wind|-m_{\mathrm{sign}(\wind)}^0+1}^2
	   , & j \in I_{ \mathrm{sign}(\wind),1}, \\
	 t_{|\wind|+1}^2,  &   j\in I_{ \mathrm{sign}(\wind),2}.
		\end{array} \right.
\end{equation}
Combining \eqref{p1}, \eqref{p2} and \eqref{sgv2}, \eqref{sgv2a} we get that 
%
\begin{equation*}
	 \| \Pi^\perp\psi_j\| 
	\stackrel{\eqref{p1}}{\leq}
\left\{\begin{array}{ll}  O(1) \left(  \e^{-N\log C_\beta} t_{|\wind|-m_{\mathrm{sign}(\wind)}^0+1}^{-1}  \right)
	\stackrel{\eqref{sgv2a}}{\leq} O(1) \frac{ N^{(1+C_\alpha)}}{\log N } \e^{-N\log C_\beta},&j\in I_{ \mathrm{sign}(\wind),1},\\
 O(1) \left(t_{|\wind|+1}^{-1}  \frac{\log N}{N} N^{-C_{\gamma}}\right)\stackrel{\eqref{sgv2}}{\leq}O(1) N^{-C_{\gamma}},&  j\in I_{ \mathrm{sign}(\wind),2}.
\end{array}
\right.
\end{equation*}
Hence, 
by \eqref{qn3},
\begin{equation*}
	| 1 - \|   \Pi \psi_j \| | \leq \left\{ \begin{array}{ll}
 \| \psi_j -  \Pi \psi_j\| = \|  \Pi^\perp\psi_j\| 
	= O(1)  \frac{ N^{(1+C_\alpha)}}{\log N } \e^{-N\log C_\beta}, & j\in I_{ \mathrm{sign}(\wind),1}\\
 \| \psi_j -  \Pi \psi_j\|  + O(1) N^{-C_{\gamma}} \log N 
	= O(1) N^{-C_{\gamma}} \log N, &  j\in I_{ \mathrm{sign}(\wind),2}
\end{array}
\right.
=: \varepsilon_j.
\end{equation*}
Thus, for $N>0$ sufficiently large, so that $\varepsilon_j \ll 1$, 
\begin{equation*}
	 \| \widetilde{e}_j - \psi_j \| = \left \| \frac{ \Pi \psi_j }{ 1 + O(\varepsilon_j)} - \psi_j \right \| 
	 \leq  O(\varepsilon_j),
\end{equation*}
which, together with \eqref{qn3} and 
\begin{equation*}
	\langle \widetilde{e}_i  |  \widetilde{e}_j \rangle = 
	\langle \psi_i  |  \psi_j \rangle 
	+ \langle \widetilde{e}_i - \psi_i  |  \psi_j \rangle 
	+\langle \psi_i  |  \widetilde{e}_j -\psi_j\rangle 
	+\langle \widetilde{e}_i -\psi_i  |  \widetilde{e}_j -\psi_j\rangle, 
\end{equation*}
lets us conclude the proofs of the proposition. 
\end{proof}
Notice that \eqref{singv2} implies by a similar argument as before \eqref{eqnAP0}, 
that for $N>0$ large enough, the vectors $\widetilde{e}_1,\dots,\widetilde{e}_{|\wind|}$ are 
linearly independent.
Hence, we conclude the following.
\begin{cor}\label{cor:sgvSpan2}
Under the assumptions of Proposition \ref{prop:sgvSpan}, we have that 
 the $\widetilde{e}_1,\dots,\widetilde{e}_{|\wind|-m_{\mathrm{sign}(\wind)}^0}$ 
span the range $\mathcal{R}(\Pi_{t_{|\wind|-m_{\mathrm{sign}(\wind)}^0}})$ and that 
$\widetilde{e}_1,\dots,\widetilde{e}_{|\wind|}$ span $\mathcal{R}(\Pi_{t_{|\wind|}})$. 
\end{cor}

\begin{rem}\label{rem:jb2}
By Remark \ref{rem:jb1}, \eqref{qn5.1}, and Proposition \ref{prop:sgvSpan} it follows that for the Jordan block with $z$ inside the unit disc such that \eqref{qm4.0} holds we have that 
\[
\left\| \frac{\mathfrak{z}_1^+}{\|\mathfrak{z}_1^+\|} - e_1\right\| = O(N^{-C_\gamma} \log N). 
\]
\end{rem}
In what follows we suppose that $N>0$ is sufficiently large, so that the conclusion of 
Proposition \ref{prop:sgvSpan} holds. Furthermore, all error terms will be uniform in 
$z\in\Omega_N$ even if we don't state this explicitly. 
\par
In view of Proposition \ref{prop:sgvSpan} and Corollary \ref{cor:sgvSpan2}, we know 
that there exist 
$a_1,\dots,a_{|\wind|-m_{\mathrm{sign}(\wind)}^0} \in  \ell^2( I_{\mathrm{sign}(\wind),1})$ 
and $b_{|\wind|-m_{\mathrm{sign}(\wind)}^0+1},\dots, b_{|\wind|} \in \ell^2( [|\wind|])$ such that 
\begin{equation}\label{eqAP1}
	e_j = \sum_{\nu  \in I_{\mathrm{sign}(\wind),1}} a_j(\nu) \widetilde{e}_\nu, ~~ j\in I_{\mathrm{sign}(\wind),1},
	\quad \text{and} \quad 
	e_j = \sum_1^{|\wind| } b_j(\nu) \widetilde{e}_\nu, ~~ j\in I_{\mathrm{sign}(\wind),2}. \notag
\end{equation}
Let $n,m\in I_{\mathrm{sign}(\wind),1}$. 
Using \eqref{singv2} and the fact that the $e_j$ are orthonormal, we get that 
\begin{equation}\label{eqAP2}
\begin{split}
	\delta_{n,m} &= \langle e_n | e_m \rangle  
	= \sum_{\nu,\mu} a_n(\nu)\overline{a}_m(\mu) \langle \widetilde{e}_\nu| \widetilde{e}_\mu \rangle 
	= \langle a_n | a_m \rangle + O\!\left(\|a\|^2 N^{(1+C_\alpha)}\e^{-N \log C_\beta}/\log N\right), 
\end{split} \notag
\end{equation}
where the constant in the error estimate is independent of $a$. We deduce that 
$\|a_{n}\|  = 1 + O( N^{(1+C_\alpha)} \e^{-N/C}/\log N)$, and 
\begin{equation}\label{eqAP3}
	 \langle a_n | a_m \rangle = \delta_{n,m} + O\left( N^{(1+C_\alpha)} \e^{-N \log C_\beta}/\log N\right).
\end{equation}
Similarly, we deduce from \eqref{singv2} that for $n,m\in I_{\mathrm{sign}(\wind),2}$ 
\begin{equation}\label{eqAP4}
	 \langle b_n | b_m \rangle = \delta_{n,m} +
	 O(N^{-C_{\gamma}}\log N) +O((\log N /N)^{1/2} ),
\end{equation}
where the constants in the error estimate are independent of $b$. Decompose 
\begin{equation*}
  b_j =\left\{\begin{array}{ll}
			\widetilde{b}_j \oplus \widehat{b}_j \in \ell^2( I_{+,1})\oplus \ell^2( I_{+,2}),
			 &\mathrm{sign}(\wind)>0,\\
			 \widehat{b}_j \oplus \widetilde{b}_j\in \ell^2( I_{-,2})\oplus \ell^2( I_{-,1}),&\mathrm{sign}(\wind)<0.
	\end{array}\right.
\end{equation*}
Exploiting the orthogonality of the $e_j$ and using \eqref{eqAP3}, \eqref{eqAP4}, we deduce
similarly as above that for $n\in I_{\mathrm{sign}(\wind),1}$ and $m\in I_{\mathrm{sign}(\wind),2}$,
\begin{equation}\label{eqAP5}
	 \langle a_n | \widetilde{b}_m \rangle =  O( N^{-C_{\gamma}}\log N) +O((\log N/ N )^{1/2}).
\end{equation}
By a similar argument as before \eqref{eqnAP0}, we deduce from \eqref{eqAP3} that for 
$N>0$ sufficiently large (depending only on the $\Omega$ and the constants in Assumption 
\ref{As:Omega}), the $a_j$ are linearly independent and span $\ell^2(I_{\pm,1})$. 
Thus, for each $\widetilde{b}_j$ there exists a $d_j \in \C^{|\wind|-m_{\mathrm{sign}(\wind)}^0}$ 
such that $\widetilde{b}_j = \sum_\nu d_j(\nu)a_\nu$. Hence, by \eqref{eqAP5}, \eqref{eqAP3}, 
\begin{equation}\label{eqAP6}
	\| \widetilde{b}_m\|^2 = \sum_{\nu,\mu} d_m(\nu) \overline{d_m(\mu)}  \langle a_\nu |a_\mu \rangle = 
	\| d_m\|^2\left( 1 + O\!\left( N^{(1+C_\alpha)} \e^{-N \log C_\beta}/\log N\right)\right).
\end{equation}
Combining this with \eqref{eqAP5}, \eqref{eqAP6}, \eqref{eqAP3},  we deduce that 
\begin{equation*}
\|d_k\|=  O( N^{-C_{\gamma}}\log N) +O((\log N/ N )^{1/2}) +  
O\!\left( N^{(1+C_\alpha)} \e^{-N \log C_\beta}/\log N\right),
\end{equation*}
where the error terms are independent of $a$, $b$, and $d$. Notice that the error term 
which is exponentially small in $N$ may be absorbed into the other two terms. Hence, 
\begin{equation}\label{eqAP7}
\|\widetilde{b}_m\|=  O( N^{-C_{\gamma}}\log N) +O((\log N/N )^{1/2} ). \notag
\end{equation}
Thus, for $j\in I_{\mathrm{sign}(\wind),2}$,
	\begin{equation*}
	e_j = \sum_{\nu \in I_{\mathrm{sign}(\wind),2}} \widehat{b}_j(\nu) \psi_\nu^{\,\mathrm{sign}(\wind)} +
	O( N^{-C_{\gamma}}\log N) +O((\log N /N)^{1/2}).
	\end{equation*}
In view of \eqref{qn5.1}, \eqref{qmn4} and Remarks \ref{rem:Modes} and \ref{rem:Modes1}, we 
may replace $\psi_\nu^{\,\mathrm{sign}(\wind)}$ in the above expression by 
${\mathfrak{z}_\nu^{\mathrm{sign}(\wind)}}/{\|\mathfrak{z}_\nu^{\mathrm{sign}(\wind)}\|}$, 
see also Proposition \ref{lem:TruncEV}.
\\
\par
Summing up what we have proven so far, we obtain in view of Propositions \ref{prop:sgvSpan} and \ref{prop:QM} the following. Recall 
that for $z\in\Omega_N$, 
	$e_1,\dots, e_{N}$ are  an orthonormal set of eigenvectors of $(P_N-z)^*(P_N-z)$ 
	corresponding to the eigenvalues $0\leq t_1^2 \leq \cdots \leq t_{N}^2$.
\begin{prop}\label{prop:sgvSpan3}
	Under the	assumptions of Proposition \ref{prop:sgvSpan},
	we have that for $j \in I_{\mathrm{sign}(\wind),1}$, 
	\begin{equation}
	  \label{eq-P10.4-1}
	e_j = \sum_{\nu \in  I_{\mathrm{sign}(\wind),1}}a_j(\nu)\psi_\nu^{\,\mathrm{sign}(\wind)} 
		+O(1) N^{(1+C_\alpha)}\e^{-N \log C_\beta}/\log N, \notag
	\end{equation}
	%
	where $a_j \in \ell^2(I_{\mathrm{sign}(\wind),1})$
	satisfy \eqref{eqAP3},  uniformly in $z\in\Omega_N$. 
	%
	Further, for $j\in I_{\mathrm{sign}(\wind),2}$,
	\begin{equation}
	  \label{eq-P10.4-2}
	e_j = \sum_{\nu \in I_{\mathrm{sign}(\wind),2}} b_j(\nu)
	\frac{\mathfrak{z}_\nu^{\mathrm{sign}(\wind)}}{\|\mathfrak{z}_\nu^{\mathrm{sign}(\wind)}\|} +
	O(  N^{-C_{\gamma}}\log N) +O((\log N/N )^{1/2} ).
	\end{equation}
	%
where $b_j \in \ell^2( I_{\mathrm{sign}(\wind),2})$
satisfy
	\begin{equation}\label{eq:ao1}
	 \langle b_n | b_m \rangle = \delta_{n,m} +
	 O( \log N N^{-C_{\gamma}}) +O((\log N /N)^{1/2} ),
	\end{equation}
	uniformly in $z\in\Omega_N$.
\end{prop}

\begin{rem}
  \label{rem-lasthour}
We recall, see points (i), (ii) of Lemma \ref{lem:tube-geo} and 
  \eqref{eq:ze-t-o-wt-ze}, that for $\nu\neq \nu' 
  \in I_{\mathrm{sign}(\wind),2}$
  we have that $|\zeta^{\mathrm{sign}(\wind)}_\nu \cdot \ol{\zeta^{\mathrm{sign}(\wind)}_{\nu'}}-1|$ is bounded
  away from $0$ by a constant $c>0$ independent of $N$, while 
  for $\nu=\nu' \in I_{\mathrm{sign}(\wind),2}$ it is of order $\log N/N$
  by \eqref{qm4.0}-\eqref{qm4.0c}. 
\end{rem}

Next, we turn to proving localization estimates of the singular vectors. 
\begin{thm}
  \label{thm:localization}
  Consider the setup and notation as in Proposition \ref{prop:SmallSG}.
	Then, there exist constants $0<C,C' < \infty$ 
	such that for $N>0$ sufficiently large (depending only on $\Omega$ and the constants in 
	Assumption \ref{As:Omega}) we have that for any $z\in\Omega_N$, for any discrete interval
	 $I =[\ell_0,\ell_1]\subset[0,N-1]$, with $|I| \geq 2(1+m_+^0)$, and for $ j \in I_{\mathrm{sign}(\wind),2}$
	\begin{equation}\label{eqApt1}
		  \| e_j \|^2_{\ell^2(I)} \leq 
		 C \min\left\{ |I| \frac{\log N }{ N} , \frac{1}{C'}\right\}\e^{-2C_{\gamma} \mathfrak{d} \frac{\log N}{N}} 
		 +C N^{-2C_{\gamma}}(\log N)^2 +C\frac{\log N}{N},
	\end{equation}
	and 
	\begin{equation}\label{eqApt1.1}
		  \| e_j \|^2_{\ell^2(I)} \geq 
		 \frac{1}{C}\min\left\{ |I| \frac{\log N }{ N} , \frac{1}{C'}\right\} \e^{-2C_\alpha \mathfrak{d} \frac{\log N}{N}} 
		 - C N^{-2C_{\gamma}}( \log N)^2-C\frac{\log N}{N},
	\end{equation}
	and for $j \in I_{\mathrm{sign}(\wind),1}$,
	\begin{equation}\label{eqApt2}
		 \| e_j \|^2_{\ell^2(I)}  \leq C\e^{-2 \mathfrak{d} \log C_\beta} + C N^{2(1+C_\alpha)}
		 \e^{-2N \log C_\beta} (\log N)^{-2} .
	\end{equation}
	Here, 
	\begin{equation}\label{eqApt3}
	  \mathfrak{d}:=\left\{\begin{array}{ll}
		 		\ell_0&\text{when } \wind>0,\\
				N-\ell_1&\text{when } \wind<0.
		 \end{array}\right. \notag
	\end{equation}
\end{thm}
\noindent
We remark that we do not use \eqref{eqApt1.1} in the rest of the paper
(instead, we use \eqref{eq-P10.4-2}). Its inclusion here is to contrast with
\eqref{eqApt1}.
\begin{proof}
We shall only consider the case when the winding number $\wind>0$, since the case when $\wind<0$ is 
similar. Throughout we will assume that $N>0$ is sufficiently large (depending only on $\Omega$ and 
the constants in Assumption \ref{As:Omega}) so that the conclusions of Propositions \ref{prop:sgvSpan3}, 
\ref{prop:sgvSpan}, \ref{prop:QM}, and \ref{lem:TruncEV} hold, and we have that all error terms are 
uniform in $z\in\Omega_N$ without us mentioning this explicitly. 
\\
\par
1. We begin by considering the case where $\wind- m_{+}^0>0$ and $j\in  I_{+,1}$. 
By Propositions \ref{prop:sgvSpan3} and \ref{prop:QM}, we find that
\begin{equation}\label{eqAP8}
		 \|e_j\|_{\ell^2(I)}^2 
		 \leq 2 \left\|\sum_{\nu \in I_{+,1}}a_j(\nu)\psi_\nu^+ \right\|^2_{\ell^2(I)}
		+O(1) \frac{ N^{2(1+C_\alpha)}}{(\log N)^2 } \e^{-2N \log C_\beta}.
\end{equation}
Recall from Proposition \ref{lem:TruncEV} that $\psi_\nu^+$ is given by the $\nu$-th column of 
$\mathfrak{Z}^+ A^+_1 G_{+,1}^{-1/2}$. Using the Cauchy-Schwarz inequality we get that 
\begin{equation}\label{eqAP9}
\begin{split}
	\left\|\sum_{\nu \in I_{+,1}}a_j(\nu)\psi_\nu \right\|^2_{\ell^2(I)}
	&= 
	\sum_{\mu \in I } \left| \sum_{\nu \in I_{+,1}}\sum_{\eta=1}^{m_+-m_+^0} 
	\mathfrak{Z}^+_{\mu,\eta} (A^+_1 G_{+,1}^{-1/2})_{\eta,\nu}a_j(\nu) \right|^2\\
	&= 
	\sum_{\mu \in I } \left| \sum_{\eta=1}^{m_+-m_+^0} \mathfrak{Z}^\pm_{\mu,\eta} (A^+_1 G_{+,1}^{-1/2}a_j)(\eta) \right|^2
\leq \| A^+_1 G_{+,1}^{-1/2}a_j\|^2 \sum_{\eta=1}^{m_+-m_+^0} \| \mathfrak{z}^+_\eta\|^2_{\ell^2(I)}.
\end{split}
\end{equation}
where $\mathfrak{z}^+_\eta$ is defined as in Proposition \ref{lem:TruncEV}. For $\eta=1,\dots, m_+$, we have 
that 
\begin{equation}\label{eqAP10}
\| \mathfrak{z}^+_\eta\|^2_{\ell^2(I)} = \sum_{\nu =\ell_0}^{\ell_1} |\zeta^+_{\eta}|^{2\nu} 
 = |\zeta_{\eta}^+|^{2\ell_0}  \frac{1-|\zeta_\eta^+|^{2|I|}}{1-|\zeta_\eta^+|^2}.
\end{equation}
By \eqref{qm4.0b} we have that $|\zeta_\eta^+| \leq 1/C_\beta$ for $\eta  =1,\dots, {m_+-m_+^0} $. Thus, 
\begin{equation}\label{eqAP10.1}
 |\zeta_{\eta}^+|^{2\ell_0} \leq \| \mathfrak{z}^+_\eta\|^2_{\ell^2(I)}\leq \frac{C_\beta^2}{C_\beta^2-1}
  |\zeta_{\eta}^+|^{2\ell_0}.
\end{equation}
Combining \eqref{qm12aa}, \eqref{eqAP3}, with \eqref{eqAP8} and \eqref{eqAP9}, we conclude 
\eqref{eqApt2}. 
\\
\par
2. Next, we let $j\in I_{+,2}$, so we may assume that $m_+^0>0$. 
 By Propositions \ref{prop:sgvSpan3} and \ref{prop:QM}, we find that 
\begin{equation}\label{eqAP11}
		 \|e_j\|_{\ell^2(I)}^2 \leq 
		2\left\|\sum_{\nu \in I_{+,2}}b_j(\nu)\frac{\mathfrak{z}_\nu^{+}}{\|\mathfrak{z}_\nu^{+}\|}  \right\|^2_{\ell^2(I)} + 
	\varepsilon_N,
\end{equation}
where, to ease the notation, we put 
$\varepsilon_N= O(( \log N)^2) N^{-2C_{\gamma}} +O(\log N /N)$. For the sake 
of the presentation we will keep this notation for the error term and we allow for the 
constant in the estimate to change while remaining uniform in $z\in\Omega_N$. 
\\
From Proposition \ref{prop:sgvSpan3}, we see that the coefficients $b_j$ are uniformly 
bounded in $\ell^2$ by $O(1)$. Thus by a similar computation as in \eqref{eqAP9}, \eqref{eqAP10.1}, 
we get together with \eqref{qmn3}, \eqref{qm12ff} that 
\begin{equation}\label{eqAP12}
		 \|e_j\|_{\ell^2(I)}^2 
		 \leq O(1)
		\sum_{\nu \in I_{+,2}} \frac{\| \mathfrak{z}^+_\nu\|^2_{\ell^2(I)}}{\| \mathfrak{z}^+_\nu\|^2}+ 
		\varepsilon_N.
\end{equation}
Recall \eqref{qm4.0b} and \eqref{qm4.0}. 
Let $\nu \in I_{+,2}$, and suppose first that $|I| \ll N/ \log N$, we find by Taylor expansion that 
\begin{equation}\label{eqAP15}
2|I|C_\alpha\frac{\log N}{N} + O\!\left( |I|\frac{\log N}{N}\right)^2 \leq 
1- |\zeta_\nu^+|^{2|I|}\leq 2|I|C_{\gamma}\frac{\log N}{N} + O\!\left( |I|\frac{\log N}{N}\right)^2.
\end{equation}
Thus, 
\begin{equation}\label{eqAP16}
\frac{C_\alpha}{C_\gamma} |I| \left(1 + O\!\left( |I|\frac{\log N}{N}\right)\right)
\leq
\frac{1-|\zeta_\nu^+|^{2|I|}}{1-|\zeta_\nu^+|^2} 
\leq 
\frac{C_\gamma}{C_\alpha} |I| \left(1 + O\!\left( |I|\frac{\log N}{N}\right)\right).
\end{equation}
Plugging the estimates \eqref{eqAP15}, \eqref{eqAP16} into \eqref{eqAP10}, and using 
\eqref{qm12c} as well, we get that there exists a constant $C>0$ such that, for $N>0$ large enough, 
\begin{equation}\label{eqAP19}
\frac{1}{C} |I| \frac{\log N }{ N} \e^{-2C_\alpha\ell_0(1+O(\frac{\log N}{N})) \frac{\log N}{N}}
\leq
\frac{\| \mathfrak{z}^+_\nu\|^2_{\ell^2(I)}}{\| \mathfrak{z}^+_\nu\|^2} 
\asymp  |I| \frac{\log N }{ N} | \zeta_{\nu}|^{2\ell_0} 
\leq C  |I| \frac{\log N }{ N} \e^{-2C_{\gamma}\ell_0 \frac{\log N}{N}}. 
\end{equation}
When $ N/(C\log N) \leq |I| \leq N$ for some $C\gg 1$, then we see by Taylor expansion that 
\begin{equation}\label{eqAP17}
\frac{1-|\zeta_\nu^+|^{2|I|}}{1-|\zeta_\nu^+|^2} 
 \asymp \frac{N}{\log N}. \notag
\end{equation}
Hence, by \eqref{qm12c} and \eqref{eqAP10}, we get that there exists a constant 
$C>0$ such that, for $N>0$ large enough, 
\begin{equation}\label{eqAP19.1}
\frac{1}{C} \e^{-2C_\alpha\ell_0(1+O(\frac{\log N}{N})) \frac{\log N}{N}}
\leq
\frac{\| \mathfrak{z}^+_\nu\|^2_{\ell^2(I)}}{\| \mathfrak{z}^+_\nu\|^2} 
\asymp  | \zeta_{\nu}|^{2\ell_0} 
\leq C   \e^{-2C_{\gamma}\ell_0 \frac{\log N}{N}}. 
\end{equation}
By \eqref{eqAP12}, \eqref{eqAP19} and \eqref{eqAP19.1} we conclude \eqref{eqApt1}.
\\
\par

3. Next we  prove the lower bound \eqref{eqApt1.1}, so we work still with 
$j\in I_{+,2}$ and $m_+^0>0$. Similar to \eqref{eqAP11}, we find by 
Propositions \ref{prop:sgvSpan3} and \ref{prop:QM} that 
\begin{equation*}
		 \|e_j\|_{\ell^2(I)}^2 \geq
		\frac12\left\|\sum_{\nu \in I_{+,2}}b_j(\nu)\frac{\mathfrak{z}_\nu^{+}}{\|\mathfrak{z}_\nu^{+}\|}  \right\|^2_{\ell^2(I)} - 
	\varepsilon_N.
\end{equation*}
Let 
\begin{equation*}
	V = 
	\begin{pmatrix} 1 &  \dots & 1  \\ 
	\zeta_{m_+-m^0_++1}^+& \dots &  \zeta_{m_+}^+\\
	\vdots & \vdots & \vdots \\ 
	(\zeta_{m_+-m^0_++1}^+)^{m_+-m^0_+-1} & \dots &  (\zeta_{m_+}^{+})^{m_+-m^0_+-1},
		 \end{pmatrix}
 \end{equation*}
$\Lambda=\mathrm{diag}(\zeta_{m_+-m_+^0+1}^+,\dots,\zeta_{m_{+}}^+)$ and let 
$\mathfrak{L}_+$ be as in Proposition \ref{lem:TruncEV}. By a similar argument as for 
\eqref{qm9.1}, we see that $s_{m_+-m_+^0}(V) \geq 1/C$ for some $C>0$, depending 
only on $\Omega$. Let $ |I| \geq 2(1+m_+^0)$. Then, 
\begin{equation}\label{eqAP20}
\begin{split}
		\left\|\sum_{\nu \in I_{+,2}}b_j(\nu)
		\frac{\mathfrak{z}^+_\nu}{\|\mathfrak{z}^+_\nu\|} \right\|^2_{\ell^2(I)} 
		&\geq \sum_{k=0}^{\lfloor (|I|-1)/m_+^0 \rfloor -1} \| V \Lambda^{\ell_0 + km_+^0}\mathfrak{L}_+^{-1}b\|^2
\geq \frac{1}{C^2} \sum_{k=0}^{\lfloor (|I|-1)/m_+^0 \rfloor -1} \|\Lambda^{\ell_0 + km_+^0}\mathfrak{L}_+^{-1}b\|^2 \\
&\geq \frac{1}{C^2} \sum_{\nu \in I_{+,2}} |b_j(\nu)|^2 
\sum_{k=0}^{\lfloor (|I|-1)/m_+^0 \rfloor -1} \frac{|\zeta_\nu^+|^{2(\ell_0 + km_+^0)}}{\|\mathfrak{z}^+_\nu\|^2}
\end{split}
\end{equation}
Let $\nu \in I_{+,2}$. A direct computation yields that 
\begin{equation}\label{eqAP21}
	\sum_{k=0}^{\lfloor (|I|-1)/m_+^0 \rfloor -1} |\zeta_\nu^+|^{2(\ell_0 + km_+^0)}
	=|\zeta_\nu^+|^{2\ell_0}\frac{1-|\zeta_\nu^+|^{2m_+^0( \lfloor (|I|-1)/m_+^0 \rfloor -1)}}{1-|\zeta_\nu^+|^{2m_+^0}}  
	\geq 
	|\zeta_\nu^+|^{2\ell_0}\frac{1-|\zeta_\nu^+|^{2(|I|-1 -m_+^0)}}{1-|\zeta_\nu^+|^{2m_+^0}}  
\end{equation}
Recall \eqref{qm4.0} and \eqref{qm4.0b}. 
Suppose first that $2(1+m_+^0) \leq |I| \ll N/ \log N$, then $|I|/2 \leq |I| - 1 -m_+^0 \leq |I|$, and 
 by Taylor expansion we find that 
\begin{equation*}
1- |\zeta_\nu^+|^{2( |I| - 1 -m_+^0)} \geq |I|C_\alpha \frac{\log N}{N} + O\!\left( |I|\frac{\log N}{N}\right)^2 .
\end{equation*}
and similarly 
\begin{equation*}
1- |\zeta_\nu^+|^{2m_+^0} \leq 2m_+^0 C_{\gamma} \frac{\log N}{N} + O\!\left( \frac{\log N}{N}\right)^2 .
\end{equation*}
Thus, 
\begin{equation}\label{eqAP23}
\frac{1-|\zeta_\nu^+|^{2( |I| - 1 -m_+^0)}}{1-|\zeta_\nu^+|^{2m_+^0}} 
\geq 
\frac{C_\alpha}{2m_+^0 C_{\gamma}} |I| \left(1 + O\!\left( |I|\frac{\log N}{N}\right)\right).
\end{equation}
Combining \eqref{eqAP23}, \eqref{eqAP21} and \eqref{eqAP20} with \eqref{qm12c}, 
yields that there exists a constant $C>0$ such that for $N>0$ large enough and all 
$z\in\Omega_N$,
\begin{equation}\label{eqAP24}
\sum_{k=0}^{\lfloor (|I|-1)/m_+^0 \rfloor -1} \frac{|\zeta_\nu^+|^{2(\ell_0 + km_+^0)}}{\|\mathfrak{z}^+_\nu\|^2}
\geq \frac{1}{C}  |I| \frac{\log N }{ N} | \zeta_{\nu}|^{2\ell_0}  
\geq
\frac{1}{C} |I| \frac{\log N }{ N} \e^{-2C_\alpha\ell_0(1+O(\frac{\log N}{N})) \frac{\log N}{N}}.
\end{equation}
When $ N/(C \log N )\leq |I| \leq N$ for some $C\gg1$, then we see by Taylor expansion that 
\begin{equation}\label{eqAP25}
\frac{1-|\zeta_\nu^+|^{2( |I| - 1 -m_+^0)}}{1-|\zeta_\nu^+|^{2m_+^0}} 
\asymp \frac{1}{C}  \frac{N}{\log N}. \notag
\end{equation}
We then deduce similarly to \eqref{eqAP24} that there exists a constant $C>0$ such 
that for $N>0$ large enough and all $z\in\Omega_N$, 
\begin{equation}\label{eqAP25.1}
\sum_{j=0}^{\lfloor (|I|-1)/m_+^0 \rfloor -1} \frac{|\zeta_\nu^+|^{2(\ell_0 + jm_+^0)}}{\|\mathfrak{z}^+_\nu\|^2}
\geq
\frac{1}{C}\e^{-2C_\alpha\ell_0(1+O(\frac{\log N}{N})) \frac{\log N}{N}}.
\end{equation}
From \eqref{eq:ao1} we know that for $N>0$ sufficiently large $\|b_j\| \geq 1/2$. Hence, 
\eqref{eqAP24}, \eqref{eqAP25.1} together with \eqref{eqAP20} imply \eqref{eqApt1.1}, which concludes 
the proof of the theorem. 
\end{proof}
\subsection{Large singular values}\label{sec:LargeSGValues}
The next result gives estimates on the Hilbert-Schmidt and trace norm of $E$ 
from \eqref{gp8}. 
\begin{prop}
  \label{prop:HS-bd}
	Let $z_0 \in p(S^1)$ be as in \eqref{qm1} and let 
	$\Omega\Subset \C$ be a sufficiently small open relatively 
	compact convex neighbourhood of $z_0$ satisfying \eqref{qm2}. 
	Let $N > 1$ be sufficiently large, and let 
	$\Omega_N\Subset \Omega \backslash p(S^1)$ 
	satisfy \eqref{qm4.0}. Let $\wind$ be as in \eqref{def:index} 
	and let $t_j$ be as in Proposition \ref{prop:SmallSG}, 
	then, uniformly in $N$ and $z\in \Omega_N$, 
	\begin{equation}\label{nls0.0}
		\sum_{j=|\wind|+1}^N t_j^{-2} = O\!\left( \frac{N^2}{\log N}\right),
	\end{equation}
	and 
	\begin{equation}\label{nls0.1}
		\sum_{j=|\wind|+1}^N t_j^{-1} = O( N {\log N}).
	\end{equation}
\end{prop}
\begin{proof}
  1.  Let $z_0$ and $\Omega$ as in the assumptions of the proposition and let $z\in\Omega_N$. 
	Let $\widetilde{N} =N_+\lor 0+N_- \lor 0$ and notice that $\widetilde{N} \geq |\wind|$, 
	see Lemma \ref{lem:root1} and \eqref{at9.0}. 
	By \eqref{c4}, \eqref{qm27.1} and \eqref{qm24.0}, we see that 
	\begin{equation}\label{nls0}
		t_j(P_N-z) \geq t_{j-\widetilde{N}}(P_{\Z/N\Z}-z), \quad j=1+\widetilde{N}, \dots, N.
	\end{equation}
	From \eqref{c4} we deduce that the set of eigenvalues of $(P_{\Z/N\Z}-z)^*(P_{\Z/N\Z}-z)$ 
	is given by 
	\begin{equation}\label{nls4}
		 \{ t_j^2(P_{\Z/N\Z}-z); j=1,\dots,N\} = 
		 \{ | p(\widehat{\zeta})-z|^2; \widehat{\zeta} \in S_N\}, 
	\end{equation}
	where $S_N \defeq \{ \e^{2\pi i j/N}; j=0,\dots, N-1\}$. 
	Note that $t_1(P_{\Z/N\Z}-z)\geq \dist(z,p(S^1)) >0$, 
	since $z\in\Omega_N$. Hence, combining \eqref{nls0} and \eqref{sgv2}, 
	we get for $\nu =1,2$,  
	\begin{equation}\label{nls2}
		\sum_{j=|\wind|+1}^N \frac{1}{t_j^{\nu}(P_N-z)} 
		\leq\sum_{j=1 + \widetilde{N}}^N 
		\frac{1}{t_j^{\nu}(P_N-z)} + \widetilde{N} \left( \frac{ CN}{\log N} \right)^{\nu} 
		\leq\sum_{j=1}^N 
		\frac{1}{t_j^{\nu}(P_{\Z/N\Z}-z)} + \widetilde{N}\left( \frac{C N}{\log N} \right)^{\nu}.
	\end{equation}
	It remains to deal with the traces $\tr ((P_{\Z/N\Z}-z)^*(P_{\Z/N\Z}-z))^{-\nu/2}$ for $\nu=1,2$. 
	\par
	2. We know by \eqref{int1} that there 
	are finitely many, say  $N_0< +\infty$, points in $p^{-1}(z_0)$. We will enumerate them 
	in the following way: for $j=1,\dots, N_0'$, with $N_0' \geq 1$, let $\zeta_j\in p^{-1}(z_0)$ 
	be so that $\zeta_j \in S^1$, and for $j=N_0'+ 1,\dots, N_0$, let $\zeta_j\in p^{-1}(z_0)$ 
	be so that $\zeta_j \notin S^1$. If $N_0'=N_0$ then the latter set of points is empty. 
	\par
	By \eqref{qm2} and the Implicit function theorem, it follows that, after potentially 
	shrinking $\Omega$ (while keeping it convex), there exist complex open neighbourhoods 
	$Z_j$ of $\zeta_j$, such that $p: Z_j \to \Omega$ is a diffeomorphism and such that 
	$Z_j\cap S^1=\emptyset$ for $j=N_0'+1,\dots,N_0$. Furthermore, after potentially further 
	shrinking $\Omega$ and the sets $Z_j$, we can arrange so that 
	$p(S^1\cap Z_j)= \Omega\cap p(S^1) =\Gamma$ for all $j=1,\dots,N_0'$. Let 
	$\Omega' \subset \Omega$ 
	be a smaller open convex neighbourhood of $z_0$ such that $\overline{\Omega'} \subset \Omega$, 
	and put $Z_j':= (p\!\upharpoonright_{Z_j\to \Omega})^{-1}(\Omega')\subset Z_j$. A convenient 
	choice is 
	\begin{equation}\label{nls5}
		\Omega' = \{ z\in \C; \dist(z,\Gamma') \leq 1/C_{\Omega'}, \pi(z) \in \Gamma'\}, \notag
	\end{equation}
	where $0< C_{\Omega'} < \infty$ is some sufficiently large constant, 
	$\Gamma' \subset \Gamma$ is a connected non-empty subsegment, containing $z_0$, 
	so that the endpoints of the segments $\Gamma$ and $\Gamma'$ are separated by some 
	fixed positive distance, i.e. for some large constant $C'>0$ we have that 
	$\dist(\partial \Gamma, \partial \Gamma')\geq 1/C'$. Furthermore, $\pi(z)$ denotes the point 
	in $p(S^1)$ with $|\Pi(z) - z| = \dist(z,p(S^1))$. Notice that this point is unique if 
	$C_{\Omega'}$ is large enough. 
	\\
	\par 
	We note that 
	there exists a constant $C>0$ 
	(depending only on $\Omega$ and $\Omega'$) 
	such that for all $z\in\Omega'$ and all $\zeta \in S^1\backslash \bigcup_{j=1}^{N_0'} (Z_j\cap S^1)$ 
	\begin{equation}\label{nls5.1}
		|p(\zeta) -z|\geq \dist(\partial\Omega,\partial\Omega')\geq 1/C. 
	\end{equation}
	Furthermore, observe that since two consecutive points in $S_N$ differ by an angle of $2\pi/N$, we have that 
	\begin{equation}\label{nls5.2}
		\# \{ S_N \cap Z_j \} = N \int_{ (p\upharpoonright_{Z_j \to \Omega})^{-1}(\Omega)} L_{S^1}(d\theta) + O(1),
		\quad j=1,\dots,N_0',
	\end{equation}
	where $L_{S^1}$ denotes the normalized Lebesgue measure on $S^1$. 
	\\
	\par
	Let $j\in \{1,\dots,N_0'\}$, $z\in\Omega_N' := \Omega_N\cap \Omega'$, 
	and $\zeta_0^j\in Z_j'$ be the unique point 
	in $Z_j'$ such that $p(\zeta_0^j)=\Pi(z)$. For $N>1$ sufficiently large 
	(depending only on $\Omega$ and $\Omega'$) there exists a point $\widehat{\zeta}^j_*\in Z_j \cap S_N$ which is 
	the closest point in $Z_j \cap S_N$ to $\zeta_0^j$. There may be two such points, in which case we  
	pick one of them. Then, 
	\begin{equation}\label{nls5.3}
	 | \zeta_0^j - \widehat{\zeta}_*^j| \leq |1 - \e^{2\pi i/N}| \leq O(N^{-1}).
	\end{equation}
	By Taylor expansion around $\zeta^j_0$, we find that 
	\begin{equation*}
		|p(\widehat{\zeta}^j_*) -z| \leq |p(\zeta^j_0) -z| + O(N^{-1}), 
	\end{equation*}
	where the constant is independent of $z$. In view of \eqref{nls4}, we have that
	\begin{equation*}
		 t_1(P_{\Z/N\Z}-z) = \min\limits_{\widehat{\zeta}\in S_N} |p(\widehat{\zeta}) -z|.
	\end{equation*}
	Setting $ \widetilde{t}:= | p(\zeta^j_0) -z |$, as $p(\zeta_0^j)= \Pi(z)$ and hence $|p(\zeta_0^j)-z| = {\rm dist}(z, p(S^1))$, we see that 
	\begin{equation}\label{nls6.1}
		 \widetilde{t} \leq t_1(P_{\Z/N\Z}-z) \leq \min\limits_{j=1,\dots,N_0'} |p(\widehat{\zeta}^j_*) -z|
		  \leq \max\limits_{j=1,\dots,N_0'}|p(\widehat{\zeta}^j_*) -z|
		\leq \widetilde{t} + O(N^{-1}),
	\end{equation}
	where the constant is independent of $z$. By Taylor expansion and \eqref{qm2} we get that 
	for $\widehat{\zeta}^j\in Z_j$
	\begin{equation}\label{nls6}
	\begin{split}
		 | p(\widehat{\zeta}^j) -z | 
		&= | p(\zeta^j_0) + (\partial_{\zeta}p)(\zeta^j_0) (\widehat{\zeta}^j -\zeta^j_0) 
		+ O(|\widehat{\zeta}^j -\zeta^j_0|^2)-z | \\
		&\geq \frac{1}{C} | \widehat{\zeta}^j -\zeta^j_0| ( 1 - C_1|\widehat{\zeta}^j -\zeta^j_0|) - \widetilde{t}.
	\end{split}
	\end{equation}
	Here the constant $0<C< \infty$ only depends on $\Omega$, see \eqref{qm2}, and the constant $0<C_1<\infty$ 
	depends only on $C$ and the second derivative of $p$ in the $Z_j$'s, which is uniformly bounded. 
	Suppose that 
	\begin{equation}\label{nls8}
		| \widehat{\zeta}^j -\zeta^j_0|  \leq  \frac{1}{2C_1}, \quad 
		|\widehat{\zeta}^j-\zeta^j_0| \geq 8C \widetilde{t}, 
	\end{equation}
	where $C,C_1$ are as in \eqref{nls6}. Hence, we find by \eqref{nls6}, \eqref{nls8} that 
	\begin{equation}\label{nls9}
		 | p(\widehat{\zeta}^j) -z | \geq \frac{1}{4C} | \widehat{\zeta}^j-\zeta^j_0|.
	\end{equation}
	\par
	3. We turn to proving the two estimates claimed in the proposition. We let $\Omega$ and $\Omega'$ 
	be as in Step 2. Recall \eqref{nls6.1} and note that if $\widetilde{t}\geq 1/C$ for some arbitrarily large 
	but fixed constant $C$, then 
	 \begin{equation*}
		\sum_{j=1}^N 
		\frac{1}{t_j(P_{\Z/N\Z}-z)} ,~~ \sum_{j=1}^N 
		\frac{1}{t_j^2(P_{\Z/N\Z}-z)} = O(N),
	\end{equation*} 
	and we conclude \eqref{nls0.0} and \eqref{nls0.1} in view of \eqref{nls2}. 
	\\
	\par
	Continuing, we may assume from now on that $\widetilde{t} \leq 1/C_0$, for some arbitrarily large but 
	fixed constant $C_0$. Since the roots of $p(\zeta)-z$ depend smoothly on  $z\in \Omega$, see the 
	discussion after \eqref{qm2}, it follows from \eqref{qm4.0} that there exists a constant $C>0$ such that 
	for all $N>0$,
	 \begin{equation*}
		\dist (\Omega_N,p(S^1)) \geq \frac{1}{C} \frac{\log N}{N},
	\end{equation*} 
	Hence, we may from now on assume that 
	 \begin{equation}\label{ls12}
		\frac{1}{C} \frac{\log N}{N} \leq \widetilde{t} \leq \frac{1}{C_0}. 
	\end{equation} 
	for some large $C>0$, independent of $N>0$ and $z\in\Omega_N$. Let $\nu \in\{1,2\}$. 
	Then, in view of \eqref{nls4}, \eqref{nls5.1} we get 
	\begin{equation}\label{nls11}
		\sum_{j=1}^N \frac{1}{t_j^\nu(P_{\Z/n\Z}-z)} 
		 = \sum_{j=1}^{N_0'} \sum_{\widehat{\zeta} \in Z_j\cap S_N} \frac{1}{| p(\widehat{\zeta})-z|^\nu} + O(N),
	\end{equation}
	where the constant in the error term depends only on $\Omega$ and $\Omega'$. 
	Using \eqref{nls5.2}, we see that by potentially shrinking $\Omega$, and the sets $Z_j$ and 
	$\Omega'$ accordingly, we may suppose that the first estimate in \eqref{nls8} holds for all 
	$\widehat{\zeta} \in Z_j\cap S_N$, $j=1,\dots, N_0'$. Next, pick a $j\in \{1,\dots,N_0'\}$, while keeping 
	in mind that $N_0'<\infty$, and split the corresponding inner sum on the right hand side of \eqref{nls11} into 
	two parts: one where the second inequality in \eqref{nls8} holds and one where it does not, so 
	\begin{equation}\label{nls13}
		\sum_{\widehat{\zeta} \in Z_j\cap S_N} \frac{1}{| p(\widehat{\zeta})-z|^\nu} 
		=
		\sum_{\substack{\widehat{\zeta} \in Z_j\cap S_N, \\ 
		|\widehat{\zeta} - \zeta_0^j| \geq 8C \widetilde{t}}} \frac{1}{| p(\widehat{\zeta})-z|^\nu} 
		+
		\sum_{\substack{\widehat{\zeta} \in Z_j\cap S_N, \\ 
		|\widehat{\zeta} - \zeta_0^j| < 8C \widetilde{t}}} \frac{1}{| p(\widehat{\zeta})-z|^\nu}.
	\end{equation}
	We begin with treating the second sum on the right hand side. Since the points in $S_N$ differ by an 
	angle of $2\pi / N$, we see that for $N>1$
	\begin{equation*}
		\#\{\widehat{\zeta} \in Z_j\cap S_N;|\widehat{\zeta} - \zeta_0^j| < 8C \widetilde{t}   \}  
		= O(N\widetilde{t}\,),
	\end{equation*}
	where the constant depends only on the symbol $p$, $\Omega$ and the constant $C$ from \eqref{nls6}. 
	Hence, by \eqref{nls4} and \eqref{nls6.1},
	\begin{equation}\label{nls13.1}
		\sum_{\substack{\widehat{\zeta} \in Z_j\cap S_N, \\ 
		|\widehat{\zeta} - \zeta_0^j| < 8C \widetilde{t}}} \frac{1}{| p(\widehat{\zeta})-z|^\nu}  =
		O(N\widetilde{t} ^{\,1-\nu}).
	\end{equation}
	\par
	Next, we deal with the first sum on the right hand side of \eqref{nls13}. Since for this term 
	\eqref{nls8} holds, we get by \eqref{nls9} that 
	\begin{equation}\label{nls13.2}
		\sum_{\substack{\widehat{\zeta} \in Z_j\cap S_N, \\ 
		|\widehat{\zeta} - \zeta_0^j| \geq 8C \widetilde{t}}} \frac{1}{| p(\widehat{\zeta})-z|^\nu} 
		\leq 
		\sum_{\substack{\widehat{\zeta} \in Z_j\cap S_N, \\ 
		|\widehat{\zeta} - \zeta_0^j| \geq 8C \widetilde{t}}} \frac{(4C)^\nu}{ |\widehat{ \zeta} -\zeta^j_0|^\nu} 
	\end{equation}
	By \eqref{nls5.3} and \eqref{ls12} we find that there exist constants 
	$0< C_2,C_2' < \infty$ such that 
	for $N>0$ large enough (depending only on the constants in the
	referenced inequalities) 
	we have that for all $\widehat{\zeta} \in Z_j\cap S_N$
	with $|\widehat{\zeta} - \zeta_0^j| \geq 8C \widetilde{t}$ and for all $z\in\Omega_N$, such that 
	\eqref{ls12} holds, 
	\begin{equation}\label{nls13.3}
		 |\widehat{ \zeta} -\zeta^j_0| \geq \frac{1}{C_2} |\widehat{ \zeta} -\widehat \zeta^j_*| \geq \frac{1}{C_2'}\widetilde{t}.
	\end{equation}
	By rotational invariance, we may assume that ${\widehat \zeta}^j_*=1$. Then, we have that 
	$\widehat{\zeta}_k=\e^{2\pi i k/N}$, $k=1,\dots, M_1$ are the points in $Z_k$ to the left of $1$ and 
	$\widehat{\zeta}_k=\e^{2\pi i k/N}$, $k=-1,\dots, -M_2$ are the points in $Z_k$ to the right of $1$. 
	 By \eqref{nls5.2} it is clear that after potentially further shrinking $\Omega$ and taking $N>0$ large enough, 
	 we have that $M_1,M_2 \leq N/C_3$, for $C_3 \gg 1$. Hence, we find that there exists a constant $0< C_4 <\infty$ 
	 such that 
	\begin{equation}\label{nls13.4}
		| \widehat{\zeta}_k - \zeta^j_*| =
		 |\e^{ \pm 2\pi i k/N} -1| \geq \frac{1}{C_4 N} |k|, \quad -M_2 \leq k \leq M_1, ~ k\neq 0.
	\end{equation}
	On the other hand, we see by \eqref{ls12} that if $|k| \leq B^{-1} \log N$, $B\gg 1$, then 
	\begin{equation}\label{nls13.5}
		| \widehat{\zeta}_k - {\widehat \zeta}^j_*| =	 |\e^{ \pm 2\pi i k/N} -1| \ll \widetilde{t}, 
	\end{equation}
	for $N>1$. Combining \eqref{nls13.4}, \eqref{nls13.5} with \eqref{nls13.2} and \eqref{nls13.3}, we conclude that 
	\begin{equation*}
		\sum_{\substack{\widehat{\zeta} \in Z_j\cap S_N, \\ 
		|\widehat{\zeta} - \zeta_0^j| \geq 8C \widetilde{t}}} \frac{1}{| p(\widehat{\zeta})-z|^\nu} 
		\leq \sum_{\substack{ k=-M_2 \\ | k|\geq B^{-1} \log N}}^{M_1} \frac{(4CC_2C_4N)^\nu}{|k|^\nu}  
		\leq\left\{\begin{array}{ll}  O(N \log N), & \nu =1 \\ O(N^2 (\log N)^{-1} ), & \nu =2. \end{array}\right.
	\end{equation*}
	Together with  \eqref{nls2}, \eqref{ls12}, \eqref{nls11},
	\eqref{nls13} and  \eqref{nls13.1}, 
	we conclude \eqref{nls0.0} and \eqref{nls0.1}. 
\end{proof}

\subsection{Smallest singular value and Hilbert-Schmidt norms}
\label{sec-ssvHS}
In this short section we provide bounds on minimal singular 
values and Hilbert-Schmidt norms of matrices appearing in the Grushin problem. 
Recall the definition of $E(\cdot)$ from \eqref{gp8}. 
\begin{lem}
\label{lem-10.7}
Consider the setup and notation as in Proposition \ref{prop:SmallSG}. Fix $M \in \N$ and $\gamma > 1$. Assume that the entries of $Q$ are i.i.d.~with zero mean, unit variance, and finite fourth moment. 
If $\gamma > 3/2$ then 
\[
\prob\left(\inf_{z \in \Omega_N} s_{\min} (I_N + N^{-\gamma} Q E(z)) \le 1/2\right) \le N^{-(\gamma -3/2)/2},
\]
for all large $N$. Otherwise, if $\gamma \in (1,3/2]$ then given any $\alpha_0 >0$ there exists $\beta = \beta(\alpha_0, \gamma) < \infty$,  such that  
\[
\prob\left(s_{\min} (I_N + N^{-\gamma} Q E(z)) \le N^{-\beta}\right) \le N^{-\alpha_0},
\]
uniformly for all $z \in \Omega_N$ and all large $N$. 
\end{lem}

\begin{proof}
1. We first  consider the easier case $\gamma > 3/2$. 
By Proposition \ref{prop:SmallSG}, applied with $M=|d|$, we have that $\|E(z)\| = O(N)$ uniformly for $z \in \Omega_N$. As the entries of $Q$ have zero mean and finite fourth moments,
applying \cite[Theorem 2]{La04} and Markov's inequality we obtain that 
\[
\prob \left( N^{-\gamma} \sup_{z \in \Omega_N} \| Q E(z) \| \ge 1/2\right) \le \prob \left(  \| Q \| \ge c N^{\gamma -1} \right) \le c^{-1} N^{-(\gamma -1)} \cdot \mathds{E} \|Q\| \le N^{-(\gamma -3/2)/2},
\]
for all large $N$, and some $c >0$. Thus, by the triangle inequality we derive that
\[
\prob\left(\inf_{z \in \Omega_N} s_{\min} (I_N + N^{-\gamma} Q E(z)) \le 1/2\right) \le\prob \left( N^{-\gamma} \sup_{z \in \Omega_N} \| Q E(z) \| \ge 1/2\right) \le N^{-(\gamma -3/2)/2}.
\]
This yields the result for $\gamma > 3/2$. 
\par
2. We turn to the case of $\gamma \in (1, 3/2]$. The proof uses
a change of basis, a lower bound on the smallest singular value of 
matrices with i.i.d.~entries shifted by a deterministic matrix, and an upper bound on the maximal singular value of a matrix with i.i.d.~entries. 

To carry out these steps we introduce additional
notation. Let  $\mathsf{E}(z)$ and $\mathsf{F}(z)$  be the two 
unitary matrices whose columns are $\{e_1(z), e_2(z), \ldots, e_N(z)\}$ and $\{f_1(z), f_2(z), \ldots, f_N(z)\}$, respectively (Recall the definitions of $e_i$'s and $f_j$'s from Section \ref{sec:GrPr1} applied to $P=P_N$). Further define $\mathsf{E}_1(z)$ and $\mathsf{E}_2(z)$ to be the
$N \times M$ and $N \times (N-M)$ matrices consisting of the first 
$M$ columns and the remaining $(N-M)$ columns of $\mathsf{E}(z)$, respectively. Similarly define $\mathsf{F}_1(z)$ and $\mathsf{F}_2(z)$. Finally, let $\mathsf{T}(z)$ be the diagonal matrices with entries $\{t_{M+1}(z)^{-1}, t_{M+2}(z)^{-1}, \ldots, t_N(z)^{-1}\}$, where $t_j(z)=t_j$, $j \in [N]$ are as in Proposition \ref{prop:SmallSG}.

Equipped with this  notation, upon recalling the definition of $E(z)$, we see that
\[
I_N + N^{-\gamma} Q E(z) =I_N + N^{-\gamma} Q  \mathsf{E}_2(z) \mathsf{T}(z) \mathsf{F}_2(z)^*.
\]
Fix $x \in S^{N-1}$. As the columns of $\mathsf{F}(z)$ form a basis of $\C^N$,
there exist $c_1 \in \C^{M}$ and $c_2 \in \C^{N-M}$ such that 
\[
x = \mathsf{F}_1(z) c_1 + \mathsf{F}_2(z) c_2 \qquad \text{ and } \qquad \|c_1\|^2+\|c_2\|^2=1.
\]
Therefore, using the orthonormality property of the columns of $\mathsf{F}(z)$ we derive that 
\begin{multline}\label{eq:block-mat-smin}
\|(I_N + N^{-\gamma} Q E(z) ) x\|^2 = \| \mathsf{F}_1(z) c_1 + \mathsf{F}_2(z) c_2+ N^{-\gamma} Q \mathsf{E}_2(z) \mathsf{T}(z) c_2\|^2\\
= \left\| \mathsf{F}_1(z) (c_1 + N^{-\gamma} \mathsf{F}_1(z)^* Q \mathsf{E}_2(z) \mathsf{T}(z) c_2)+ \mathsf{F}_2(z) (I_{N-M}+ N^{-\gamma} \mathsf{F}_2(z)^* Q \mathsf{E}_2(z) \mathsf{T}(z)) c_2\right\|^2\\
= \left\| c_1 + N^{-\gamma} \mathsf{F}_1(z)^* Q \mathsf{E}_2(z) \mathsf{T}(z) c_2\right\|^2+ \left\|(I_{N-M}+ N^{-\gamma} \mathsf{F}_2(z)^* Q \mathsf{E}_2(z) \mathsf{T}(z)) c_2\right\|^2\\
=\left\| \begin{bmatrix} I_M & N^{-\gamma} \mathsf{F}_1(z)^* Q \mathsf{E}_2(z) \mathsf{T}(z) \\ 0 & I_{N-M}+ N^{-\gamma} \mathsf{F}_2(z)^* Q \mathsf{E}_2(z) \mathsf{T}(z)\end{bmatrix} \cdot \begin{pmatrix} c_1 \\ c_2 \end{pmatrix}\right\|^2,
\end{multline}
where in the second step we have used that $\mathsf{F}_1(z) \mathsf{F}_1(z)^* + \mathsf{F}_2(z) \mathsf{F}_2(z)^*= I_N$ and in the third step we have again used the fact that the columns of $\mathsf{F}(z)=\begin{bmatrix} \mathsf{F}_1(z) &  \mathsf{F}_2(z) \end{bmatrix}$ form an orthonormal basis of $\C^N$. So \eqref{eq:block-mat-smin} shows that it is enough to find a lower bound on the smallest singular value of the $2 \times 2$ block matrix appearing in its \abbr{RHS}.

To this end, we note that any block matrix $A$ of the form $A= \begin{bmatrix} I & A_1 \\ 0 & A_2 \end{bmatrix}$ is invertible iff $A_2$ is, and in that case
\[
A^{-1}= \begin{bmatrix} I & -A_1 A_2^{-1} \\ 0 & A_2^{-1} \end{bmatrix}.
\]
Hence, using the triangle inequality we obtain that
\[
\|A^{-1}\| \leq 1 + (\|A_1\|+1) \cdot \|A_2^{-1}\|.
\]
Recalling the fact that $\|c_1\|^2+\|c_2\|^2=1$, using the above and \eqref{eq:block-mat-smin} we therefore deduce that
\begin{align}
(s_{\min}(I_N + N^{-\gamma} Q E(z)))^{-1}  & = \| (I_N + N^{-\gamma} Q E(z))^{-1}\| \label{eq:smin-simplify1}\\
& \le 1+ (N^{-\gamma} \| \mathsf{A}_1(z)\| +1) \cdot \|(I_{N-M}+ N^{-\gamma} \mathsf{A}_2(z))^{-1}\|, \notag
\end{align}
where for brevity we write
\[
\mathsf{A}_1(z) := \mathsf{F}_1(z)^* Q \mathsf{E}_2(z) \mathsf{T}(z) \qquad \text{ and } \qquad \mathsf{A}_2(z):= \mathsf{F}_2(z)^* Q \mathsf{E}_2(z) \mathsf{T}(z).
\]
Using \cite{La04}, Proposition \ref{prop:SmallSG}, and Markov's inequality again we further have that 
\begin{equation}\label{eq:prob-smax}
\prob \left( N^{-\gamma} \| \mathsf{A}_1(z)\| +1 \ge 2 N^{2\alpha_0 +1/2} \right) \le t_{M+1}(z)^{-1} \cdot N^{-(\gamma+2\alpha_0+1/2)} \cdot \mathds{E} \| Q\| = O(N^{-2\alpha_0}),
\end{equation}
where we have also used the facts that $t_{M+1}(z)^{-1}=O(N)$ and $\gamma >1$. Hence, by \eqref{eq:smin-simplify1}, we observe that in order to find a lower bound on the smallest singular value of $I_N + N^{-\gamma} Q E(z)$ it remains to find the same for $I_{N-M} + N^{-\gamma} \mathsf{A}_2(z)$. Upon using Weyl's inequality for singular values it is straightforward to notice that 
\begin{equation}\label{eq:smin-simplify3}
s_{\min} (I_{N-M} + N^{-\gamma} \mathsf{A}_2(z)) \ge s_{\min}(I_{N} + N^{-\gamma} \mathsf{A}(z)),
\end{equation}
where
\[
\mathsf{A}(z):= \mathsf{F}(z)^* Q \mathsf{E}(z) \widehat{\mathsf{T}}(z) \qquad \text{ and } \qquad \widehat{\mathsf{T}}(z) := \begin{bmatrix} I_M & 0 \\ 0 & \mathsf{T}(z) \end{bmatrix}.
\]
On the other hand
\begin{align}
s_{\min}(I_{N} + N^{-\gamma} \mathsf{A}(z))  & \ge s_{\min} (\widehat{\mathsf{T}}(z)) \cdot s_{\min}(\widehat{\mathsf{T}}(z)^{-1} + N^{-\gamma} \mathsf{F}(z)^* Q \mathsf{E}(z)) \notag\\
& = t_N(z)^{-1} \cdot N^{-\gamma} \cdot s_{\min} (N^\gamma \mathsf{F}(z) \widehat{\mathsf{T}}(z)^{-1} \mathsf{E}(z)^* + Q) \label{eq:smin-simplify2}
\end{align}
Since $P_N$ has a bounded norm and $\Omega_N$ is a bounded set we have that $\|\widehat{\mathsf{T}}(z)^{-1} \| = t_N(z) =O(1)$. As $\mathsf{E}(z)$ and $\mathsf{F}(z)$ are unitary matrices, an application of \cite[Theorem 2.1]{TaoVu08} together with \eqref{eq:smin-simplify2} now yield that
\begin{equation}\label{eq:prob-smin}
\prob\left( s_{\min}(I_{N} + N^{-\gamma} \mathsf{A}(z)) \le N^{-\beta_0}\right) \le N^{-2\alpha_0},
\end{equation}
for all large $N$, where $\beta_0 < \infty$ is some constant depending only on $\alpha_0$ and $\gamma$. Finally, using \eqref{eq:smin-simplify1} and \eqref{eq:smin-simplify3}, the probability bounds \eqref{eq:prob-smax} and \eqref{eq:prob-smin}, and a union bound we deduce that  
\[
\prob\left( s_{\min}(I_{N} + N^{-\gamma} QE(z)) \le \frac{1}{3} N^{-(2\alpha_0 +\beta_0+1/2)}\right) = O(N^{-2\alpha_0}),
\]
for all large $N$. This completes the proof of the lemma.
\end{proof}

Next we derive bounds on the Hilbert-Schmidt norms of some matrices. This will allow us to control the error terms in the resolvent expansion of various terms appearing in the Grushin problem.

\begin{prop}\label{prop:bd-terms-resolvent-exp}
 Consider the setup and notation as in Proposition \ref{prop:SmallSG}. Additionally assume that $\Omega$ is convex. Fix $\gamma>1$, $\kappa \in \N$ and $d \in [-\wt m, \wt m]\cap \Z \setminus \{0\}$. Let Assumption \ref{assump:mom} holds. 
 Then for any $\upalpha_0 >0$,  we have,
uniformly in $z \in \Omega_N$ and all $N$ large, that
  \begin{equation}\label{eq:res-error-1}
    \prob\left( N^{-\gamma \kappa}\|(E(z) Q)^\kappa E_+(z)\|_{{\rm HS}}\geq
    N^{-\frac{(\gamma -1) \kappa}{2}}\right) \le N^{-\upalpha_0},
  \end{equation}
   \begin{equation}\label{eq:res-error-3}
    \prob\left( N^{-\gamma \kappa}\|(E(z) Q)^\kappa E(z)\|_{{\rm HS}}\geq
    N^{1-\frac{(\gamma -1) \kappa}{2}}\right) \le N^{-\upalpha_0},
  \end{equation}
  \begin{equation}\label{eq:res-error-4}
     \prob\left( N^{-\gamma \kappa}\|E_-(z) (Q E(z))^{\kappa}\|_{{\rm HS}}\geq
     N^{ -\frac{(\gamma -1) \kappa}{2}}\right) \leq  N^{-\upalpha_0},
  \end{equation}
  and
   \begin{equation}\label{eq:res-error-2}
     \prob\left( N^{-\gamma (\kappa+1)}\|E_-(z) (Q E(z))^{\kappa} Q  E_+(z)\|_{{\rm HS}}\geq
     N^{-\gamma -\frac{(\gamma -1) \kappa}{2}}\right) \leq  N^{-\upalpha_0}. 
  \end{equation}
\end{prop}
To prove Proposition \ref{prop:bd-terms-resolvent-exp} we will rely on bounds on moments of Hilbert-Schmidt norm of matrices appearing in \eqref{eq:res-error-1}-\eqref{eq:res-error-2}, as well as bounds on $\|E_\pm(\cdot)\|_{{\rm HS}}$ and on the sum of the inverse of the non-zero singular values of $E(\cdot)$. 

\begin{prop}\label{prop:mom-HS}
Fix $\kappa \in \N$. Let $\{\sE^{(\ell)}\}_{\ell=1}^{\kappa}$ and $\{\sF^{(\ell)}\}_{\ell =1}^{\kappa}$ be two collections of $N \times N$ deterministic matrices, possibly with complex-valued entries, such that
\begin{equation}\label{eq:HS-mom-assump}
\max\left\{\max_{\ell=1}^{\kappa} \|\sE^{(\ell)}\|_{{\rm HS}}, \max_{\ell=1}^{\kappa} \|\sF^{(\ell)}\|_{{\rm HS}}\right\} \le 1.
\end{equation}
Define the product
\begin{equation}\label{eq:hat-sE-kappa}
\widehat{\sE}(\kappa) := \sE^{(1)} Q \sF^{(1)}\cdot \sE^{(2)} Q \sF^{(2)} \cdots  \sE^{(\kappa)} Q \sF^{(\kappa)}. 
\end{equation}
If the entries of $Q$ satisfies Assumption \ref{assump:mom} then for any $h_0 \in \N$
\[
\E \|\widehat{\sE}(\kappa)\|_{{\rm HS}}^{2h_0} \le C_{\ref{prop:mom-HS}} \cdot \mathfrak{C}_{\kappa h_0},
\]
 where $C_{\ref{prop:mom-HS}} <\infty$ is some constant depending only on $\kappa$ and $h_0$. 
\end{prop}

The proof of Proposition \ref{prop:mom-HS} is straightforward. It relies on bounds on moments of quadratic forms of independent random variables, derived in \cite{Whittle}. 

\begin{proof}[Proof of Proposition \ref{prop:mom-HS}]
We begin with an auxiliary computation. Let $A, B$ be deterministic matrices. Then
   \[W:=\|AQB\|_{{\rm HS}}^2 =\sum_{i,j,k,k',\ell,\ell'}
     A_{i,k}\ol{A_{i,k'}}Q_{k,\ell} \ol{Q_{k',\ell'}} B_{\ell,j} \ol{B_{\ell',j}}=
   \sum_{k,\ell,k',\ell'} Q_{kl} \ol{Q_{k',\ell'}} C_{k,\ell,k',\ell'},\]
   where 
   \[
   C_{k,\ell,k',\ell'}:=\sum_{i,j} A_{ik}\ol{A_{ik'}} B_{\ell,j}\ol{B_{\ell',j}}. 
   \]
  This, in particular, shows that $W$ is a quadratic form in the entries of $Q$. Hence, 
  using \cite[Theorem 2]{Whittle}\footnote{Whittle in \cite{Whittle} works with {\em real} linear and quadratic forms of {\em real} valued random variables. Upon splitting the sums into real and imaginary parts, analogous bounds can be derived in the complex case. Hence, without loss of generality we continue to use \cite{Whittle} for complex quadratic forms of complex valued random variables.} and Assumption \ref{assump:mom}, for any $s\geq 2$, we find that 
   \[ \E|W-\E W|^s\leq C(s)  \left( \sum_{k,k',\ell,\ell'} |C_{k,\ell,k',\ell'}|^2\right)^{s/2} \cdot \gC_{s}
   \leq C(s)  \|A\|_{{\rm HS}}^{2s} \|B\|_{{\rm HS}}^{2s} \cdot \gC_{s},\]
   for some absolute constant $C(s) < \infty$, depending only on $s$, where in the last step we used that 
   \[
    |C_{k,\ell,k',\ell'}|\leq \|a_k\| \cdot \|a_{k'}\| \cdot \|b_\ell\| \cdot \|b_{\ell'}\|,
    \] 
 and $a_k$ and $b_k$ are the $k$-th column and row of $A$ and $B$, respectively. 
   
  On the other hand we note that
   \[ |\E W| \leq \gC_1 \sum_{k,\ell} \|a_k\|^2 \|b_\ell\|^2
   =\|A\|_{{\rm HS}}^2 \|B\|_{{\rm HS}}^2\cdot \gC_1.
   \]
   Thus
   \begin{equation}\label{eq:W-quad-mom-bd}
   \E |W|^s \le 2^{s-1} (\E|W-\E W|^s + |\E W|^s) \le 2^s C(s) \gC_{s} \cdot  \|A\|_{{\rm HS}}^{2s} \|B\|_{{\rm HS}}^{2s},
   \end{equation}
   where we also used that by Jensen's inequality it follows $\gC_1^s \le \gC_{s}$. Next by H\"{o}lder's inequality we obtain that  
\[
\E \|\widehat{\sE}(\kappa)\|_{{\rm HS}}^{2h_0} \le \E \left( \prod_{j=1}^\kappa \|\sE^{(j)} Q \sF^{(j)}\|_{{\rm HS}}^{2h_0}\right) \le    \left(\prod_{j=1}^\kappa \E \|\sE^{(j)} Q \sF^{(j)}\|_{{\rm HS}}^{2\kappa h_0}\right)^{1/\kappa} \le \max_{j=1}^\kappa \E \|\sE^{(j)} Q \sF^{(j)}\|_{{\rm HS}}^{2\kappa h_0}. 
\]
Now the desired bound follows from \eqref{eq:W-quad-mom-bd}, upon setting $s=\kappa h_0$, $A=\sE^{(j)}$, and $B= \sF^{(j)}$, for $j \in [\kappa]$. This concludes the proof. 
 \end{proof}

Equipped with Proposition \ref{prop:mom-HS},
we now prove Proposition \ref{prop:bd-terms-resolvent-exp}. 
\begin{proof}[Proof of Proposition \ref{prop:bd-terms-resolvent-exp}] 
Recall that by Proposition \ref{prop:HS-bd}, there
exists a constant  $C_{\ref{prop:HS-bd}}$ so that
\begin{equation}\label{nls0.0a}
		\sum_{j=|d|+1}^N t_j(z)^{-2} \le  C_{\ref{prop:HS-bd}}\frac{N^2}{\log N} \qquad \mbox{ and } \qquad \sum_{j=|d|+1}^N t_j(z)^{-1} \le  C_{\ref{prop:HS-bd}} N {\log N}.
	\end{equation}
 For $j \in [\kappa]$ we let
 \begin{equation}
   \label{eq-defF}
 \sF^{(j)} := \frac{1}{\sqrt{C_{\ref{prop:HS-bd}} N \log N}} \sum_{i=|d|+1}^N \frac{1}{\sqrt{t_i(z)}}e_i(z) \circ e_i^*(z) \notag
\end{equation} 
 and 
 \begin{equation}
   \label{eq-defE}
 \sE^{(j+1)} := \frac{1}{\sqrt{C_{\ref{prop:HS-bd}} N \log N}}\sum_{i=|d|+1}^N \frac{1}{\sqrt{t_i(z)}}e_i(z) \circ f_i^*(z). \notag
\end{equation} 
 Further let 
 \begin{equation} 
   \label{eq-defEF}
 \sE^{(1)}:= |d|^{-1/2} E_-(z) \qquad \mbox{ and } \qquad \sF^{(\kappa+1)}:= |d|^{-1/2} E_+(z). \notag
\end{equation} 
 By Proposition \ref{prop:HS-bd} and \eqref{gp8}-\eqref{gp9} it follows that \eqref{eq:HS-mom-assump} holds with these choices of $\{\sE^{(j)}\}_{j=1}^{\kappa+1}$ and $\{\sF^{(j)}\}_{j=1}^{\kappa+1}$. On the other hand note that $E(z)  = C_{\ref{prop:HS-bd}} N \log N \cdot \sF^{(j)} \sE^{(j+1)}$ for $j \in [\kappa]$. Therefore, applying Proposition \ref{prop:mom-HS} we now derive that 
 \[
 \E \left[ \|\wt \sE\|_{{\rm HS}}^{2h_0}\right] \le d^{2h_0} \cdot (C_{\ref{prop:HS-bd}} N \log N)^{2h_0\kappa} \cdot  C_{\ref{prop:mom-HS}} \cdot \gC_{\kappa h_0}. 
 \]
 where $\wt \sE:=E_-(z) (Q E(z))^{\kappa} Q  E_+(z)$. 
 Hence applying Markov's inequality with $h_0 = \lceil \frac{2\upa_0}{\kappa (\gamma -1) }\rceil$ we obtain that
 \[
 \prob\left( N^{-\gamma (\kappa+1)} \|E_-(z) (Q E(z))^{\kappa} Q  E_+(z)\|_{{\rm HS}} \ge N^{-\gamma - \frac{(\gamma-1)\kappa}{2}}\right) \le N^{-(\gamma+1) \kappa h_0} \E \left[ \|\wt \sE\|_{{\rm HS}}^{2h_0}\right] \le N^{-\upa_0},
 \]
 for all large $N$. This concludes the proof of \eqref{eq:res-error-2}. To obtain \eqref{eq:res-error-4} we let
 \[
  \sE_\star:= \sE^{(1)} Q \sF^{(1)} \cdot \sE^{(2)} Q \sF^{(2)} \cdots \sE^{(\kappa-1)} Q \sF^{(\kappa-1)} \cdot \sE^{(\kappa)} Q \sF_\star^{(\kappa)},
 \]
 where $\{\sE^{(i)}\}_{i=1}^\kappa$ and $\{\sF^{(i)}\}_{i=1}^{\kappa-1}$ are as above, and $\sF^{(\kappa)}:=  C_{\ref{prop:HS-bd}}^{-1/2}N^{-1}({\log N})^{1/2} \cdot E(z)$. Notice that, by \eqref{nls0.0a}, we have $\|\sF_\star^{(\kappa)}\|_{\rm HS} \le 1$. Therefore, applying applying Proposition \ref{prop:mom-HS} again we obtain that 
 \[
 \E \left[ \| \sE_\star\|_{{\rm HS}}^{2h_0}\right] \le d^{h_0} \cdot (C_{\ref{prop:HS-bd}} N \log N)^{2h_0\kappa} \cdot  C_{\ref{prop:mom-HS}} \cdot \gC_{\kappa h_0}. 
 \]
 Now \eqref{eq:res-error-4} follows by choosing $h_0$ as 
 above and applying Markov's inequality. 
The proofs of \eqref{eq:res-error-1}-\eqref{eq:res-error-3} are similar.
For \eqref{eq:res-error-1} we take 
$\sE^{(1)}=E(z)/(C_{\ref{prop:HS-bd}} N\log N)$ 
and $\sF^{(\kappa+1)}=E_+(z)/\sqrt{|}d|$, while
for \eqref{eq:res-error-3} we take $\sF^{(\kappa+1)}=E(z)/(C_{\ref{prop:HS-bd}} N\log N)$. Further details are omitted.
 \end{proof}

 The next corollary will come handy in the proof of Theorem
 \ref{thm:main}.
 \begin{cor}
   \label{cor-control}
   In the setup of Proposition \ref{prop:bd-terms-resolvent-exp}, we have
   that for any $\upalpha_0>0$, for all $N$ large,
   \begin{equation}
     \label{eq-sec10new}
     \prob\big(N^{-\gamma}  \|  E(z) (I+N^{-\gamma} QE(z))^{-1} QE_+(z)\|_{\rm HS}
     \geq N^{-(\gamma-1)/4}\big)\leq N^{-\upalpha_0},
   \end{equation}
   \begin{equation}
    \label{eq-sec10new-1}
     \prob\big(N^{-2\gamma}  \|  E_-(z) (I+N^{-\gamma} QE(z))^{-1} QE(z) QE_+(z)\|_{\rm HS}
     \geq N^{-\gamma -(\gamma-1)/4}\big)\leq N^{-\upalpha_0},
   \end{equation}
   \begin{equation}
     \label{eq-sec10new-2}
     \prob\big(\|E_-(z) (I+N^{-\gamma}QE(z))^{-1}\|_{\rm HS}>2\sqrt{M}\big)\leq N^{-\upalpha_0},
   \end{equation}
   and  \begin{equation}
     \label{eq-sec10new-3}
     \prob\big(\|E(z)(I+N^{-\gamma}QE(z))^{-1}\|_{\rm HS}>2C_{\ref{prop:HS-bd}}N/\sqrt{\log N}\big)\leq N^{-\upalpha_0}.
   \end{equation}
 \end{cor}
 \begin{proof}
 The proof uses
   the resolvent expansion and Proposition \ref{prop:bd-terms-resolvent-exp}. First let us prove \eqref{eq-sec10new}.
   
  Using the resolvent expansion, we write,
   with $\hat Q=N^{-\gamma} Q$ and $k$ positive integer
   and omitting throughout the argument $z$,
   \begin{equation}
     \label{eq-sec10new1}
     (I+\hat Q E)^{-1}=\sum_{i=0}^{k-1} (-\hat QE)^i+ (-\hat QE)^k
     (I+\hat QE)^{-1}.
     \end{equation}
   Considering the first term in the right hand side of 
   \eqref{eq-sec10new1}, we have by \eqref{eq:res-error-1} that, for $i\geq 0$,
   \begin{equation}
     \label{eq-sec10new2}
     \| E(-\hat QE)^i\hat Q E_+\|_{\rm HS} \leq N^{-(\gamma-1)(i+1)}/2, 
   \end{equation}
   with probability $1-N^{-2\upalpha_0}$, for all $N$ large. 
   On the other hand, by \eqref{eq:res-error-3} and the fact that
   for any two matrices $A,B$ one has
   $\|BA\|_{\rm HS}\leq \|A\|\cdot \|B\|_{\rm HS}$, we 
   obtain that 
   \begin{equation}
     \label{eq-sec10new4}
     \| E(-\hat QE)^k(I+\hat QE)^{-1}  Q E_+\|_{\rm HS} 
     \leq  N^{1-\frac{(\gamma-1)k}{2}} 
     \cdot \|(I+\hat QE)^{-1}\| \cdot \|Q\| \cdot \|E_+\|,
   \end{equation}
   on a set of event with probability at least $1 - N^{-2\upa_0}$. 
   Now by the second part of
   Lemma \ref{lem-10.7}, \cite{La04} (which gives $\E\|Q\|=O(N^{1/2})$), 
   and \eqref{gp9}, we further obtain that $\|(I+\hat QE)^{-1}\| \cdot \|Q\| \cdot \| E_+\| \leq N^\beta$
   for some finite $\beta$, with probability at least $1-N^{-2\upalpha_0}$.
   By choosing $k$ large enough, we conclude 
   that on these events, the right hand side of \eqref{eq-sec10new4}
   can be made arbitrarily small. Combining this with \eqref{eq-sec10new1}
   and \eqref{eq-sec10new2}, we obtain \eqref{eq-sec10new}. 
   The proof of \eqref{eq-sec10new-1}, being similar, is omitted. 

   To see \eqref{eq-sec10new-2}, we use the resolvent
   expansion \eqref{eq-sec10new1} and write
   \begin{equation}
     \label{eq-evening1}
     E_-(I+N^{-\gamma} QE)^{-1}=
   E_-+\sum_{i=1}^{k-1} E_- (\hat QE)^i +
   E_-(\hat QE)^k (I+\hat QE)^{-1}. \notag
 \end{equation}
 Recall that $\|E_-\|_{\rm HS}=\sqrt{M}$. As above,
 the other terms are controlled 
 by \eqref{eq:res-error-4} and  Lemma \ref{lem-10.7}. This completes the proof of \eqref{eq-sec10new-2}. To prove \eqref{eq-sec10new-3} we observe that by \eqref{nls0.0a} one has that $\|E\|_{{\rm HS}} \le C_{\ref{prop:HS-bd}}N/\sqrt{\log N}$ (without loss of generality assume $C_{\ref{prop:HS-bd}} \ge 1$). Therefore, arguing similarly as above one derives \eqref{eq-sec10new-3}. Further details are omitted.
 \end{proof}
 
 \begin{rem}\label{rem:jb3}
 The reader may inspect the proofs of Proposition \ref{prop:bd-terms-resolvent-exp} and Corollary \ref{cor-control} to observe that \eqref{eq-sec10new} continues to hold even if we can replace the term $N^{-(\gamma-1)/4}$ by $N^{-(1-\upepsilon) \cdot(\gamma-1)}$ for any $\upepsilon >0$. This improvement will be necessary later in Section \ref{sec-mainproof} to argue that the eigenvectors of a randomly perturbed Jordan block is completely delocalized in sup-norm. 
 \end{rem}
 
 \begin{rem}\label{rem:unif}
In Sections \ref{sec:QM} and \ref{sec:SmallSGValues}-\ref{sec-ssvHS} we worked
with Assumption \ref{As:Omega}. Since $p(S^1)$ is compact,
it follows that Assumption \ref{As:Omega} holds uniformly in
$z \in \Omega_{\vep_{\ref{theo-location}}, C_{\ref{theo-location}}, N}$ (recall \eqref{eq-Omegadef}). This observation will be used in Section \ref{sec-mainproof}, in the
proofs of
Theorem \ref{thm:main} and Corollary \ref{cor:main}.
\end{rem}
 
\section{Proofs of Theorem \ref{thm:main} and Corollary \ref{cor:main}}
\label{sec-mainproof}
We start with the proof of Theorem \ref{thm:main}.
\subsection{Proof of Theorem \ref{thm:main}}
We begin by creating a net $\mathcal{N}_\gamma$ in 
$\Omega(\vep_{\ref{theo-location}},C_{\ref{theo-location}},N)$, of spacing 
$N^{-\theta}$ and cardinality at most
$O(N^{2\theta-1}\log N)$, with $\theta>\gamma \vee 4$.
Note that
for any $\hat 
z\in \Omega(\vep_{\ref{theo-location}},C_{\ref{theo-location}},N)$,
there exists a $z^*=z^*(\hat z)\in \mathcal{N}_\gamma$ with 
$|z^*-\hat z|\leq N^{-\theta}$.\\
\par

1. We begin with the proof of the first point in
Theorem \ref{thm:main}, and proceed as in the sketch of Section \ref{res}. 
Let $\hat z\in \Omega(\vep_{\ref{theo-location}},C_{\ref{theo-location}},N)$ 
be an eigenvalue of $P_{N,\gamma}^Q$, with associated right eigenvector
$v=v(\hat z)$. Recall that $\delta=N^{-\gamma}$.
For any $z
\in \mathcal{N}_\gamma$,
define $P^\delta_z\, :=P_{N,\gamma}^Q-z$,
and note that since
  $P^\delta_z-P^\delta_{\hat z} \, =(\hat z-z) I_N$, we have that
\begin{equation}
  \label{eq-deltaz}
  P^\delta_z v=(\hat z-z) v.
\end{equation}
Set $M=|{\rm ind}_{p(S^1)}(z)|=|\wind(z)|$ (this yields
$t_{M+1}\, \gtrsim \, \log N/N$ and
$\alpha$ bounded below by a constant multiple  of $\log N/N$, see Proposition \ref{prop:SmallSG}).

Consider the Grushin problem associated with $P^\delta_z$, see
Section \ref{sec-Grushin}, and recall \eqref{eq:algeb-id}.
We then obtain, similarly to 
\eqref{eq:ef-to-pm-1} and using \eqref{eq-deltaz}, that
\begin{eqnarray}
  \nonumber
\sum\nolimits_{i=M+1}^N (e_i(z)^* v) \cdot e_i(z)&=&   
  (I-E_+(z)R_+(z)) v\\
& =&(I-E_+^\delta(z)R_+(z)) v - E(z)(I+\delta QE(z)^{-1} \delta QE_+(z) R_+(z)v \nonumber\\
& =& E^\delta(z) (\hat z-z) v+
  E(z) (I+\delta QE(z))^{-1} \delta QE_+(z) R_+(z)v. 
\label{eq:ef-to-pm}
\end{eqnarray}
We next control the right hand side of \eqref{eq:ef-to-pm},
for $z=z^*$. As discussed in Section \ref{res}, the control is simpler
when $\gamma>3/2$. Indeed, in that case, 
by \eqref{gp9} and Proposition \ref{prop:SmallSG} we have that
$\|E(z)\|= O(N/\log N)$ uniformly in 
$\mathcal{N}_\gamma$ (see also Remark \ref{rem:unif}). Now by the first part of Lemma \ref{lem-10.7}, we have, 
with probability approaching $1$ as $N\to\infty$, that for all
$z\in \mathcal{N}_\gamma$, $\|(I+\delta QE(z))^{-1} \|\leq 2$. Therefore,
by \eqref{eq-Edelta}, $\|E^\delta(z)\|=O(N/\log N)$ uniformly in $\cN_\gamma$ with the same probability. 
Hence, for $z=z^*$,
\begin{equation}
  \label{eq-final1}
  \|E^\delta(z^*) (\hat z-z^*) v\|\leq O(N/\log N)\cdot N^{-\theta}=o(1).
\end{equation}
It remains to control the second term in \eqref{eq:ef-to-pm}.

By construction, $\|E_+\|, \|R_+\|\leq 1$, 
and $\|Q\|=O(N^{1/2+\upepsilon})$, for any $\upepsilon >0$, with
probability approaching $1$ as $N\to\infty$. Applying Lemma \ref{lem-10.7} again we
obtain that with $z=z^*$,
\begin{equation}
  \label{eq-final2}
  \|E(z^*) (I+\delta QE(z^*))^{-1} \delta QE_+(z^*) R_+(z^*)v\|\leq N^{-\gamma} O(N^{1/2}\cdot N/\log N)=o(1). 
\end{equation}
Together with \eqref{eq-final1},
we conclude from \eqref{eq:ef-to-pm}
that if $\gamma>3/2$ then for $z^*$ with $|z^*-\hat z|=O(N^{-\theta})$, 
 \begin{equation}
   \label{eq-order1}
\left\|   \sum_{i=M+1}^N (e_i(z^*)^* v) \cdot e_i(z^*)\right\| =o(1).
 \end{equation}

We turn next to the case $\gamma\in (1,3/2]$. In that case, 
we apply Corollary \ref{cor-control} and obtain that for all
$z\in \mathcal{N}_\gamma$, with probability larger than $1-N^{-\upalpha_0}$,
we have that
\begin{equation}
  \label{eq-final2a}
 \|E^\delta(z) \| =O(N/\sqrt{\log N}) \qquad \text{ and } \qquad \|E(z) (I+\delta QE(z))^{-1} \delta QE_+(z) R_+(z)v\|= o(1).
\end{equation}
Together with \eqref{eq-final1}, this yields \eqref{eq-order1}
also for $\gamma \in (1,3/2]$.

For $z \in \C$ we define  
\begin{equation}
  \label{eq-wdef}
  w_{z}:=E_+(z) R_+(z) v.
\end{equation}
We will show below  that $\|v - w_{z^*}\|=o(1)$ and $\|w_{z^*} - w_{\hat z}\|=o(1)$. This will yield \eqref{eq-main1} with $w=w_{\hat z}$. The localization estimate \eqref{eq-main2} then will
follow from 
\eqref{eqApt1} and \eqref{eqApt2} in
Theorem \ref{thm:localization}. The assertion that $w$ is a random linear combination of $\{e_j(\hat z)\}_{j \in [M]}$ follows from its definition in \eqref{eq-wdef}.

Turning to proving that $\|v - w_{z^*}\|=o(1)$, we note that,  
by 
\eqref{eq:ef-to-pm}, $v-w_{z^*}$ equals the left hand side of 
\eqref{eq:ef-to-pm}, whose norm is $o(1)$ by \eqref{eq-order1}.

  It remains to show that $\|w_{z^*} - w_{\hat z}\|=o(1)$. This follows from a standard perturbative argument. Indeed, let
  $B_z=(P_N-zI)^*(P_N-zI)$.
  Using Proposition \ref{prop:SmallSG}, the $M$-th singular value of
  both $B_{\hat z}$ and $B_{z^*}$ are $O(N^{-C_\gamma}\log N/N)$
  while the $(M+1)$-th is at least $C\log N/N$, 
  for appropriate $C,C_\gamma$ independent of $N$. Let
  $\Pi_z=\Pi_{z,M}$ be the spectral projector of $B_z$ on the
  first $M$ eigenvalues. Then, choosing $\Gamma=\partial D(0,C\log N/2N)$,
  we can write 
  \[\Pi_z=\frac{1}{2\pi i} \int_\Gamma (B_z-\lambda)^{-1} d\lambda,\]
  and an easy computation (using the Ky-Fan inequalities and the spectral gap)
  gives that $\|\Pi_{\hat z}-\Pi_{z^*}\|=O(N^{-\theta} N^2/\log N)=o(N^{-2})$
  since $\theta>4$. By \eqref{eq-wdef} we notice that $w_z=\Pi_z v$. 
  We therefore obtain that $\|w_{\hat z}-w_{z^*}\|= o(N^{-2}) =o(1)$, as claimed.
  This
completes the proof of
the first point of Theorem \ref{thm:main}, for all $\gamma>1$.
\\
\par

2. We next turn to the proof of the second part of Theorem \ref{thm:main}.
Fix
$z_0\in \Omega(\vep_{\ref{theo-location}},C_{\ref{theo-location}},N)$ 
and let $\hat z$ be an eigenvalue of $P_{N,\gamma}^Q$ as in the second 
part  of the theorem. Recall that $z^*\in \mathcal{N}_\gamma$ is such that
is such that $|\hat z-z^*|\leq N^{-\theta}$.

The starting point of the proof
is the observation that by \eqref{eq:algeb-id}, \eqref{eq-deltaz},
the definition of $E_{-+}^\delta(z)$, and the
resolvent expansion, we 
have that for any $z\in \C$,
\begin{eqnarray}\label{eq:coeff-eMnew}
&&E_-^\delta(z) (\hat z-z) v  = - E_{-+}^\delta(z) R_+(z) v \\
&& 
\!\!\!\!\!
\!\!\!\!\!
= - E_{-+}(z)R_+(z) v + \delta E_-(z) Q E_+(z) R_+(z) v -  
\delta^2 E_-(z) (I+\delta Q E(z))^{-1} Q E(z) Q E_+(z) R_+(z) v. \nonumber
\end{eqnarray}
(Compare with \eqref{eq:coeff-eM}.)
Using the definition of $E_{-+}(z)$ and of $R_+(z)$, we rewrite
the first two terms in the right hand side of \eqref{eq:coeff-eMnew} as
\begin{eqnarray}
  \label{eq:coeff-eM-domnew}
&&- E_{-+}(z)R_+(z) v + \delta E_-(z) Q E_+(z) R_+(z) v\\
&&\qquad= \sum_{i=1}^M t_i \cdot (e_i(z)^* v) \cdot \delta_i + 
\delta \sum_{i=1}^M \left[\sum_{j=1}^M (e_j(z)^* v) \cdot 
(f_i(z)^* Q e_j(z))\right] \delta_i. \nonumber 
\end{eqnarray}

Next, recall that $E_-^\delta=E_-(I+\delta QE)^{-1}$. Thus,
using \eqref{eq-sec10new-2}, we obtain that
\begin{equation}
  \label{eq-evening2}
  \|E_-^\delta(z) (\hat z-z) v\|\leq  2\sqrt{M}|\hat z-z|,
\end{equation}
with probability $1-N^{-2\upalpha_0}$, for any $\upa_0 >0$, when $N$ is
large enough.
By \eqref{eq-sec10new-1}, we also have with the same probability that
\begin{equation}
  \label{eq-evening3}
  \|\delta^2 E_-(z) (I+\delta Q E(z))^{-1} Q E(z) Q E_+(z) R_+(z) v\|\
  \leq N^{-\gamma-(\gamma-1)/4}.
\end{equation}
Upon choosing $\upa_0$ sufficiently large, by a union bound, \eqref{eq-evening2} and \eqref{eq-evening3}
hold simultaneously for all $z\in \mathcal{N}_\gamma$. Thus, we have
that with probability $1-N^{-3\upalpha_0/2}$,
\begin{equation}
  \label{eq-evening4}
\left\|\sum_{i=1}^M t_i \cdot (e_i(z)^* v) \cdot \delta_i + 
\delta \sum_{i=1}^M \left[\sum_{j=1}^M (e_j(z)^* v) \cdot 
(f_i(z)^* Q e_j(z))\right] \delta_i\right\|\leq 2\sqrt{M}|\hat z-z|+N^{-\gamma-(\gamma-1)/4}.
\end{equation}
\\
\par
3. Recall next Proposition \ref{prop:SmallSG}, and that $|\wind(z)|=M$
if $z\in \mathcal{N}_\gamma\cap  D(z_0,C_0\log N/N)$.
In what follows, we assume that $|z^*-\hat z|<N^{-\theta}$, so that the 
above conditions hold for $z=z^*$ and $\hat z$.
Let $M_0=|\wind(z)|-m^0_{\mbox{\rm sign}(\wind(z))}.$ Assume, for the time being, that $M_0 >0$. 
Recall from \eqref{sgv1a} that if $M_0>0$ then 
for $j\in [M_0]$, $|t_j(z)|\leq C e^{-cN}$
for some $c>0$.

Let $A=A(z)$ denote the $M_0\times M$ matrix 
with entries (in the $\delta_i$ basis)
$A_{ij}:=A_{ij}(z):=f_i(z)^* Q e_j(z)$ and, for $j\in[M]$, set
$\alpha_j:=e_j(z^*)^* v$. Write $\hat A$ for the 
$M_0\times M_0$ submatrix of $A$ consisting of its first $M_0$ columns. 
We obtain from \eqref{eq-evening4} and 
\eqref{sgv1a}
that 
\begin{equation}
  \label{eq-evening5}
  \left\|\sum_{i=1}^{M_0}\delta_i
  \sum_{j=1}^M  A_{ij}(z^*) \alpha_j\right\| \leq N^{\gamma}\big(e^{-cN}+ 2\sqrt{M}|z^*-\hat z|+
  N^{-\gamma-(\gamma-1)/4}\big)= o(1),
\end{equation}
since $\theta>\gamma$. 
We have the following lemma, whose proof is postponed.
\begin{lem}
  \label{lem-morningheri}
  For every $\eta>0$, there exist 
$0<c_\eta,C_\eta < \infty$ depending on $\eta$ and $C_0$ only,
so that, with probability at least $1-\eta$,
for all $z\in \mathcal{N}_\gamma \cap D(z_0,C_0\log N/N)$,
we have that  $s_{\min}(\hat A(z))>c_\eta$ and that $s_{\max}(A(z))<C_\eta$.
\end{lem}
Continuing with the proof of Theorem \ref{thm:main},
we obtain  from
\eqref{eq-evening5} that with $a=\sqrt{\sum_{j=1}^{M_0} |\alpha_j|^2}$,
\begin{equation}
  \label{eq-evening6}
 \left\| \sum_{i=1}^{M_0}\delta_j 
  \sum_{j=M_0+1}^M A_{ij} \alpha_j\right\| \geq a c_\eta -o(1).
\end{equation}
Recall now that the norm of the projection of $v$ on 
$\mbox{\rm span}(e_j, j\in [M])$ is $1-o(1)$
by the first part of the theorem. The orthogonality of the $e_j$'s then gives that $\|\alpha\|=1+o(1)$, where $\alpha= (\alpha_1, \alpha_2, \ldots, \alpha_M)$, 
and together with Lemma \ref{lem-morningheri}, \eqref{eq-evening6}, also that,
with $b=\sqrt{\sum_{j=M_0+1}^M |\alpha_j|^2}$,
\[ C_\eta  b\geq c_\eta a-o(1)= c_\eta \sqrt{1-b^2} -o(1).\]
It follows that, with $M_0>0$, for some $c_0>0$ independent of $N$, 
\begin{equation}
  \label{eq-lb}
 b^2=\sum_{j=M_0+1}^M |\alpha_j|^2 >c_0.
 \end{equation}
For  
$M_0=0$,
the lower bound \eqref{eq-lb} also holds, 
by the first part of Theorem \ref{thm:main}. 

We are now ready to complete the proof of \eqref{eq-main3}.
To this end, we assume in the sequel that 
$\wind>0$, the case $\wind<0$ requiring only notation changes. Set 
\begin{equation}\label{eq:defw''1}
w':=\sum_{j=M_0+  1}^M \alpha_j e_j \quad \text{ and } \quad w'':=\sum_{j=M_0+  1}^M \alpha_j \hat e_j, 
\end{equation}
with $\alpha_j=e_j(z^*)^* v$ as before and 
\begin{equation}\label{eq:defw''2}
\hat e_j:= \sum_{\nu=M_0+1}^M b_j(\nu) \frac{\mathfrak{z}_\nu^+}{\|\mathfrak{z}_\nu^+\|},
\end{equation}
 with $b_j(\nu)$'s as in \eqref{eq-P10.4-2}.

 Since $\|\alpha\|=1+o(1)$, using the triangle inequality, 
 the Cauchy-Schwarz inequality, and \eqref{eq-P10.4-2}, we observe that 
 \begin{equation}\label{eq:defw''3}
 \|w'\|^2_{\ell^2([\ell, \ell'])} \ge \frac12\|w''\|^2_{\ell^2([\ell, \ell'])} - \epsilon_N, 
 \end{equation}
  where $\epsilon_N=O( \log N/N)+O(N^{-2 C_\gamma} (\log N)^2)$.

  We next derive a lower bound on $ \|w'\|_{\ell^2([\ell,\ell'])}^{2}$. 
 Set 
$\zeta_{\nu\nu'}:=\zeta_\nu^+ \ol{\zeta_{\nu'}^+}$. 
Recall Proposition \ref{lem:TruncEV} and note that
\begin{equation}\label{eq:mfz-inp}
\langle \mathfrak{z}_\nu^+ \mid \mathfrak{z}_{\nu'}^+ \rangle_{\ell^2([\ell, \ell'])} = \zeta_{\nu\nu'}^\ell \frac{1 - \zeta_{\nu\nu'}^{\ell'-\ell+1}}{1 - \zeta_{\nu\nu'}}, \qquad \nu, \nu' \in [M] \setminus [M_0].
\end{equation}
Note that
 \[
 \frac{1}{(1-|\zeta_\nu^+|^2)\|\mathfrak{z}_\nu^{+}\|^2}=1+o(1) \qquad \text{ and } \qquad \|\mathfrak{z}_\nu^{+}\|^2\, \asymp\, N/\log N.
 \]
Thus, by Remark \ref{rem-lasthour}, from \eqref{eq:mfz-inp} we obtain that 
\begin{equation}\label{eq:mfz-inp1}
\frac{\langle \mathfrak{z}_\nu^+\mid \mathfrak{z}_{\nu'}^+ \rangle_{\ell^2([\ell, \ell'])}}{\|\mathfrak{z}_\nu^{+}\| \|\mathfrak{z}_\nu^{+}\|} = 
\left\{\begin{array}{ll} 
O(\log N/N) & \mbox{ if } \nu \ne \nu',\\
(1+o(1)) (|\zeta_\nu^+|^{2\ell} - |\zeta_\nu^+|^{2(\ell'+1)}) & \mbox{ if } \nu =\nu'.
\end{array}
\right.
\end{equation}
On the other hand, since $\nu,\nu'>M_0$, 
 for $q \in [N]$,
 $|\zeta_\nu^+|^{2q}= |\zeta_{\nu'}^+|^{2q}=
 e^{-c q \log N/N(1+o(1))}$ for some constant $c>0$ 
 independent of $\nu$, and therefore,
for $\ell, \ell' = O(N/\log N)$, 
 \begin{equation}\label{eq:mfz-inp2}
 (|\zeta_\nu^+|^{2\ell} - |\zeta_\nu^+|^{2(\ell'+1)}) \ge c' (\ell'-\ell) \cdot \frac{\log N}{N},
 \end{equation}
 for some $c'>0$. Hence, recalling \eqref{eq:defw''1}-\eqref{eq:defw''2},
 from \eqref{eq-P10.4-2}, 
 \eqref{eq:mfz-inp1} and 
 \eqref{eq:mfz-inp2}, we deduce that for any $\ell,\ell'\in [1,N]$,
\begin{multline}
  \label{eq-morningheri3}
  \|w''\|_{\ell^2([\ell,\ell'])}^2 = \sum_{j,j'=M_0+1}^M
   \ol{\alpha}_j {\alpha}_{j'}\sum_{\nu,\nu'=M_0+1}^M 
   \ol{b_j(\nu)} {b_{j'}(\nu')} \cdot    \frac{\langle \mathfrak{z}_\nu^+ \mid \mathfrak{z}_{\nu'}^+\rangle_{\ell^2([\ell,\ell'])}}
   {\|\mathfrak{z}_\nu^{+}\| \cdot \|\mathfrak{z}_{\nu'}^{+}\|}\\
  \ge  \sum_{j,j'=M_0+1}^M
     \ol{\alpha}_j {\alpha}_{j'}
       \sum_{\nu=M_0+1}^M \ol{b_j(\nu)} {b_{j'}(\nu')}  \cdot c' (\ell'-\ell) \cdot \frac{\log N}{N} + O(\log N/N)
       \\
   =     c' (\ell'-\ell) \cdot \frac{\log N}{N}
   \sum_{j,j'=M_0+1}^M
   \alpha_j \ol{\alpha}_{j'}\langle b_j|b_{j'}\rangle
   +O(\log N/N)\\
   \stackrel{\eqref{eq:ao1}}{=}
  c' (\ell'-\ell) \cdot \frac{\log N}{N}
   \left(o(1)+\sum_{j=M_0+1}^M |\alpha_j|^2\right)+O(\log N/N) 
   \stackrel{\eqref{eq-lb}}{\geq }
  \frac{c_0c' (\ell'-\ell)}{2} \cdot \frac{\log N}{N}
   +O(\log N/N).  
 \end{multline}
 In the range of $\ell,\ell'$ of the statement of Theorem \ref{thm:main},
 the first term in the right hand side of \eqref{eq-morningheri3}
 is of order larger that $\epsilon_N+O(\log N/N)$, and therefore, for such $\ell,\ell'$, by \eqref{eq:defw''3}
 \begin{equation}\label{eq:w-prime-lb}
 \|w'\|_{\ell^2([\ell,\ell'])} \ge c_1 (\ell'-\ell) \log N/N,
 \end{equation}
 for some $c_1 >0$. It remains to show that \eqref{eq:w-prime-lb} continues to hold with $w'$ replaced by $w=w_{\hat z}$. 
 
 To see this, note that
 by \eqref{eqApt2}, we have that $\|w_{z^*} -w'\|_{\ell^2([C \log N, N])} = O(N^{-2})$, 
 for some $C< \infty$. Recall from 
 above that $\|w_{\hat z} - w_{z^*}\|=o(N^{-2})$. Thus, by the triangle inequality we indeed have that \eqref{eq:w-prime-lb} holds for $w$ with $\ell \ge C \log N$ and $\ell'$ as in the statement of Theorem \ref{thm:main}. On the other hand, if $1 \le \ell \le C \log N$ we observe that $\ell' - \ell \le 2 (\ell' - C \log N)$ for any $\ell' \ge N^{c_0}$. Thus, upon shrinking $c_1$ in \eqref{eq:w-prime-lb}, that lower bound continues to hold for $w$ and such values of $\ell$.  
 The proof of \eqref{eq-main4} is now immediate from \eqref{eq-main1}. 
 The assertion about the constant $c_0$ follows from Remark \ref{rem:Cgamma} and the definition of $\epsilon_N$. 
 This finally completes the proof.
  \qed

\subsection{Proof of Lemma \ref{lem-morningheri}}
Throughout the proof, we assume for notational simplicity that $\wind>0$, the general case posing no 
new difficulties. We begin by rewriting $\hat A(z)$, see \eqref{eq-evening5}. Fix $M_1=m_+-m_+^0$
and note that $M_1\geq M_0$ with equality if $N_+=0$.
Let $E_0(z)$ denote the $N\times M_0$ matrix with columns $e_j(z)$, and $F_0(z)$ the $N\times M_0$ matrix with columns $f_j(z)$.
It follows from Lemma \ref{lem:TruncEV} and Proposition 
\ref{prop:sgvSpan} that 
the columns of $E_0(z)$ are almost orthonormal (with inner product $o(1)$ between different columns)
and with $|e_j(z)(k)|\leq e^{-ck}$, $k=1,\ldots,N$ for some $c>0$. 
Similarly, the columns of $F_0(z)$ are almost orthonormal,
with $|f_j(z)(k)|\leq e^{-c (N-k)}$, $k=1,\ldots,N$. 
Recall that $\hat A(z)=F_0(z)^* Q E_0(z)$.

We begin by fixing a particular $z$ as in the statement of the lemma, and write $\bar A=\hat A(z)$, $\bar E_0=E_0(z)$, $\bar F_0=F_0(z)$.
Note that
\[
\E\|\bar A\|_{{\rm HS}}^2 
= \sum_{i,j=1}^{M_0}\E | \bar A_{ij}|^2\leq \gC_2 \left( \sum_{i,j=1}^{M_0} \|e_i(z)\|^2 \cdot \|f_j(z)\|^2\right) \leq 2\gC_2 M_0^2,
\]
for all large $N$. Therefore,
one can find a constant
$C(\eta)$ so that 
\begin{equation}
  \label{eq-smaxub}
  \prob( s_{\max}(\bar A)>C(\eta)/2)\leq \eta/4.
\end{equation}

We next control $s_{\min}(\bar A)$.
We use the inequality
\begin{equation}
  \label{eq-sminlb}
  s_{\min}(\bar A)\geq \frac{|\det(\bar A)|}{(s_{\max}(\bar A))^{M_0-1}}.  
\end{equation}
To control the determinant in \eqref{eq-sminlb}, we 
begin by considering the matrix $\bar E_0$. Because the columns of $E_0$
decay exponentially and the columns of $E_0$ are almost orthogonal, 
there exists
a $K>0$ so that the $K\times M_0$ sub-matrix $\bar E_{0,K}$ 
consisting of the first $K$ rows of $\bar E_0$ 
has columns with inner product smaller than $1/M_0^2$, and therefore
singular values all larger than $1/2$. It follows from 
\cite[Lemma 4.3]{MIKOS} (see also \cite{GTZ})
that there exists an $M_0\times M_0$ sub-matrix of 
$\bar E_{0,K}$, denoted $\bar E_{0,K,M_0}$,
with singular values all larger than $1/2\sqrt{1+(K-M_0)M_0}$.
In particular, there exists a (nonrandom) constant $c_0$ so that 
$\mathrm{det}(\bar E_{0,K,M_0})\geq c_0$.
Similarly, with $\bar F_{0,K}$ denoting the sub-matrix of $F_0$ consisting from its last $K$ rows, 
there exists an $M_0\times M_0$ sub-matrix of $\bar F_{0,K}$, denoted
$\bar F_{0,K,M_0}$, with 
$\mathrm{det}(\bar F_{0,K,M_0})\geq c_0$.

Returning to the control of 
the determinant in \eqref{eq-sminlb}, we 
note that the latter is
a homogeneous polynomial in the entries of $Q$, and apply  Lemma
\ref{lem:anti-conc}, as follows.
Let $I=\{i_1\, < i_2 < \cdots < \,i_{M_0}\}\subset [K]$  
denote the
list of rows of $\bar E_0$ participating in $\bar E_{0,K,M_0}$. Similarly,
let $J=\{j_1\, < j_2 < \cdots < \, j_{M_0}\}\subset [N]\setminus [N-K]$
denote the list of rows of $\bar F_0$
participating in $\bar F_{0,K,M_0}$.
 Let $U_\ell=Q_{i_\ell,j_\ell}$, $\ell=1,\ldots M_0$,  then by the Cauchy-Binet formula,
\begin{equation}
  \label{eq-detindet}
  \det(\bar A)=  \sum_{\cI \subset [M_0]} Z_{\cI} \prod_{\ell \in \cI} U_{\ell}
\end{equation}
for some random variables $Z_{\cI}$ (which depend only on $Q_{a,b}$ with 
$(a,b)\not\in \cup_{\ell=1}^{M_0} (i_\ell,j_\ell)$).
Crucially,  $|Z_{[M_0]}|=|\det(\bar E_{0,K,M_0})\cdot
\det(F_{0,K,M_0})|\geq c_0^{2M_0}=:c_\star$.

By Lemma \ref{lem:anti-conc}, 
we obtain that for any $\vep\in (0,c_\star e^{-1})$,
\begin{equation}
  \label{eq-detindent1}
\prob\left( \det(\bar A) \le \vep\right) \le \bar C_{\ref{lem:anti-conc}} \cdot \left(\frac{\vep}{c_\star}\right)^{(1+\upeta)} \cdot \left(\log\left(\frac{c_\star}{\vep}\right)\right)^{M_0-1}. 
\end{equation}
Combining  \eqref{eq-detindent1} with \eqref{eq-sminlb} and
\eqref{eq-smaxub}, we obtain that there exists a $c_\eta>0$ so that
\begin{equation}
  \label{eq-finallocalsmin}
  \prob(s_{\min}(\bar  A)>2c_\eta) \geq 1-\eta/4.
\end{equation}

Next, let $z'\neq z\in  \mathcal{N}_\gamma \cap D(z_0,C_0\log N/N)$. Since $|dp(z)|$ is bounded below uniformly in a fixed small ($N$-independent)
neighborhood of $z_0$, we have that the roots of $p-z$ and $p-z'$ satisfy,
for all $j\leq N_-+N_+$,
that $|\zeta_j^{+}(z)-\zeta_j^{+}(z')|\leq C_1\log N/N$. Using 
Propositions \ref{lem:TruncEV}, \ref{prop:QM}, and  \ref{prop:sgvSpan},
this implies in turn that there exists a constant $C$ so that, for $j=1,\ldots,M_0$,
\begin{equation}
\label{eq-ejlip}
\| e_j(z)-e_j(z')\|\leq C|z-z'|+Ce^{-N/C} \le \, 2C |z- z'|.
\end{equation}
Indeed, Proposition \ref{lem:TruncEV} yields the first inequality
for the $\widetilde{u}_\ell^+$, $\ell=1,\ldots,M_1$, Proposition \ref{prop:QM}
 allows one to transfer this to the quasimodes $\psi_j^+$, and finally Proposition \ref{prop:sgvSpan}  transfers it to the $e_j$'s. The right most inequality in \eqref{eq-ejlip} follows from the fact that $\cN_\gamma$ is a net of mesh size $N^{-\theta}$. By the same argument the same bound continues to hold when $e_j$'s are replaced by $f_j$'s.
 
Denote $\mathcal{N}_{\gamma,z_0}:=\mathcal{N}_\gamma
\cap D(z_0,C_0\log N/N)$. Since $\hat A_{i,j}(\cdot)$, $i,j \in [M_0]$, are linear in the entries of $Q$,
applying \cite[Theorem 2]{Whittle} together with \eqref{eq-ejlip} we obtain that, for any $h \in \N$, 
 \[
\max_{i,j \in [M_0]}\max_{z' \in \cN_{\gamma,z_0}} \E|\hat A_{i,j}(z) - \hat A_{i,j}(z')|^{2h} \le C_h \cdot \left(\frac{\log N}{N}\right)^{2h},
 \]
 for some constant $C_h< \infty$, independent of $N$, where we also used that $|z-z'| = O(\log N/N)$ for all $z,z' \in \cN_{\gamma. z_0}$. Therefore, 
 choosing $h$ sufficiently large so that $|\cN_{\gamma, z_0}| N^{-h} =o(1)$, applying Markov's inequality,
taking a union bound over ${i,j \in [M_0]}$ and ${z' \in \cN_{\gamma,z_0}}$ and recalling Assumption \ref{assump:mom},
we deduce that, on an event of probability approaching one,
\begin{equation}
  \label{eq-chaser2}
  s_{\max}(\hat A(z)-\hat A(z'))\leq 
\|\hat A(z)-\hat A(z')\|_{\mathrm{HS}}\leq c_\eta,
\end{equation}
uniformly for all $z, z' \in \cN_{\gamma,z_0}$. 
By the Ky Fan inequalities, $s_{\min}(\hat A(z'))\geq s_{\min}(\hat A(z))-
s_{\max}(\hat A(z)-\hat A(z'))$. So we obtain the claim for $s_{\min}(\hat A(z'))$.

Finally, we turn to the control of  $s_{\max}( A(z))$, for all $z\in\mathcal{N}_{\gamma,z_0}$.
Let $\wt E(z)$ and $\wt E_0(z)$ be $N\times (M-M_0)$ matrix 
whose columns are $\mathfrak{z}_{j+M_0}^+(z)/\|\mathfrak{z}_{j+M_0}^+(z)\|$
(as in Proposition \ref{lem:TruncEV}) and $e_{j+M_0}(z)$, respectively, 
for $j=1,\dots,M-M_0$. As in Proposition \ref{prop:QM}, let 
$\wt F(z)$ be the $N \times M_0$ matrix with 
columns $\psi_i^{-}(z)$, $i=1,\dots,M_0$. Set $\wt A(z):= \wt F(z)^* Q \wt E(z)$. 
Recalling the definition of $A$ from the discussion above \eqref{eq-evening5}, 
we find that in the new notation, $A(z)= F_0(z)^* Q \wt E_0(z)$. 
Further let $\wt B(z)$ and $\wt C(z)$ be the matrices whose columns are $\{b_{j+M_0}\}$ and $\{a_j\}$, as defined in Proposition \ref{prop:sgvSpan3}. 
By Proposition \ref{prop:sgvSpan3} we have that 
\[
\|\wt E(z) \wt B(z) - \wt E_0(z)\|_{{\rm HS}} , \|\wt F(z) C(z) - F_0(z)\|_{{\rm HS}} = O(N^{-c}) \quad \text{ and } \quad \|\wt E_0(z)\|_{{\rm HS}}, \|F_0(z)\|_{{\rm HS}} \le 2M,
\]
for all $z \in \cN_{\gamma,z_0}$ and all large $N$,
where $c>0$ is some constant. 
Therefore, applying Proposition \ref{prop:mom-HS}, the triangle inequality, 
and proceeding as in the proof of \eqref{eq-chaser2}, we deduce that
\[
\max_{z \in \cN_{\gamma,z_0}}\| A(z) - \wt B(z)^* \wt A(z) \wt C(z)\| \le \max_{z \in \cN_{\gamma,z_0}}\| A(z) - \wt B(z)^* \wt A(z) \wt C(z)\|_{{\rm HS}} = o(1),
\]
on an event with probability approaching one, as $N\to+\infty$. 
Moreover, by \eqref{eqAP3} and \eqref{eq:ao1} we have that 
$\|\wt B(z)\|$ and $\|\wt C(z)\|$ are uniformly bounded for all 
$z \in \cN_{\gamma,z_0}$. Therefore, it suffices to bound $\sup_{z \in \cN_{\gamma,z_0}} \|\wt A(z)\|$, and hence in fact it is enough to prove that for fixed $i\in [M_0], j\in [M]\setminus [M_0]$, 
\begin{equation}
\label{eq-1603}
\lim_{x\to\infty} \limsup_{N\to\infty} \prob\left(\sup_{z\in\mathcal{N}_{\gamma,z_0}} |\wt A_{ij}(z)|>x\right)=0.
\end{equation}

Toward this goal, we note that 
for $z,z'\in \mathcal{N}_{\gamma,z_0}$, one has that for some constant $C_0 < \infty$, 
independent of $N$,
\[ 
\left\|\frac{\mathfrak{z}^+_j(z)}{\|\mathfrak{z}^+_j(z)\|}-\frac{\mathfrak{z}^+_j(z')}{\|\mathfrak{z}^+_j(z')\|}
\right\| \leq C_0|z-z'| \cdot \frac{N}{\log N}, \quad j \in [M]\setminus[M_0],
\]
and by Proposition \ref{prop:sgvSpan3} the bound \eqref{eq-ejlip} continues to hold with $e_j(\cdot)$ replaced by $\psi_j^-(\cdot)$ for $j \in [M_0]$. Therefore, upon using Assumption 
\ref{assump:mom} and applying \cite[Theorem 2]{Whittle}, we find that
for some constant $C_h$ independent of $N$,
\begin{equation}
  \label{eq-Cp}
  \E|\widetilde{A}_{i,j}(z)-
  \widetilde{A}_{i,j}(z')|^h\leq C_h  
\Big(\frac{ |z-z'| N}{\log N}\Big)^h.
\end{equation}

We now apply a Komogorov-type argument. Namely, for $k=0,1,\ldots$, let 
$\mathcal{N}_{\gamma,z_0,k}$ denote a minimal subset of 
$\mathcal{N}_{\gamma,z_0}$ so that for any
$z\in \mathcal{N}_{\gamma,z_0}$ there
exists 
$z' \in \mathcal{N}_{\gamma,z_0,k}$ so that
$|z-z'|\leq  2^{-k} \log N/N$. Note that  $|\mathcal{N}_{\gamma,z_0,k}|
=O(2^{2k})$. Now, fix a constant $x$ and introduce the event 
\[ \mathcal{A}_k=\{ \exists \, z,z'\in \mathcal{N}_{\gamma,z_0,k}:
|z-z'|\leq  2^{-k+1} \log N/N, |\widetilde{A}_{i,j}(z)-
\widetilde{A}_{i,j}(z')|> x/j^2\},\]
and set $\mathcal{B}=\cap_{k=1}^{k_0} \mathcal{A}_k^\complement$,
where $2^{k_0-1}\leq C_0 N/\log N\leq 2^{k_0}$. On the event $\mathcal{B}$ we
have that 
\[\max_{z\in \mathcal{N}_{\gamma,z_0}}|\widetilde{A}_{i,j}(z)|\leq 10 x. \]
On the other hand, by Markov's inequality and
  \eqref{eq-Cp}, for some constant $C_h'$ independent of $N$,
\begin{equation}
  \prob(\mathcal{B}^\complement)\leq C_h' x^{-h}
  \sum_{k=1}^{k_0} 2^{2k}2^{-hk}.
\end{equation}
Choosing $h=4$, this implies 
\eqref{eq-1603} and completes 
the proof of the lemma.
\qed

\subsection{Proof of Corollary \ref{cor:main}} 
We begin by applying Theorem \ref{theo-location} with $\mu$ replaced by $\mu/8$. This yields the existence of
constants $\vep_{\ref{theo-location}}, C_{\ref{theo-location}} $ such that 
  \begin{equation}
    \label{eq-finalloc-new}
    \prob\big(\mathcal{N}_{\Omega(\vep_{\ref{theo-location}},C_{\ref{theo-location}},N),N,\gamma}< (1-\mu/8) N\big)\to_{N\to\infty} 0.
  \end{equation}
We now claim the following.
Fix $\bar C < \infty$, $z_0 \in \Omega(\vep_{\ref{theo-location}},C_{\ref{theo-location}},N)$, and $\mu>0$. Then there exist $0<\mu_1, \mu_2 < \infty$ such that, for all large $N$, 
\begin{equation}\label{eq:cor-pre}
\prob\left(\exists i \in [N] \text{ such that } \supp_{\mu_1}(v(\lambda_i^N)) < \mu_2 N/\log N \text{ and } \lambda_i^N \in \D_{z_0} \right) \le \mu/(4 \bar C),
\end{equation}
where $\D_{z_0}:=D(z_0, C_{\ref{theo-location}} \log N/(2N)) \cap \Omega(\vep_{\ref{theo-location}},C_{\ref{theo-location}},N)$. To prove \eqref{eq:cor-pre} we borrow some notation from the proof of Theorem \ref{thm:main}. Let $\hat z \in \D_{z_0}$ and $z^*=z^*(\hat z) \in \cN_{\gamma,z_0}:= \cN_\gamma \cap D_{z_0}$ such that $|\hat z - z_*| \le N^{-\theta}$. 
Recall from the proof of Theorem \ref{thm:main} that $w_z=\sum_{j=1}^M \alpha_j e_j(z)$, for $z \in \C$, $w'= w'(z^*) = \sum_{j=M_0+1}^M \alpha_j e_j(z^*)$, and $w=w_{\hat z}$. Also recall that $\|w-w_{z^*}\|=o(1)$. Applying Lemma \ref{lem-morningheri}, with $\eta = \mu/(4\bar C)$, and upon choosing $\mu_1$ sufficiently small, we obtain that, simultaneously for all $z^*$ in $\cN_{\gamma, z_0}$, and hence simultaneously for all $\hat z \in \D_{z_0}$, and any $I \subset [N]$ such that $\|v\|_{\ell^2(I)} >1 -\mu_1$, 
\[
\sum_{j=M_0+1}^M |\alpha_j| \|e_j\|_{\ell^2(I)} \ge \|w'\|_{\ell^2(I)} \stackrel{\eqref{eq-main1}}{\ge} \|v\|_{\ell^2(I)} - \|w_{z^*}-w'\| -o(1) \ge 1 -\mu_1/2 - \sqrt{\sum_{j=1}^{M_0} |\alpha_j|^2} \stackrel{\eqref{eq-lb}}{\ge} \mu_1,
\]
with probability at least $1 - \mu/(4\bar C)$, where $v=v(\hat z)$ is the right eigenvector corresponding to $\hat z$. On the other hand by \eqref{eqApt1} it follows that for any such $I$ we must have that $|I| \ge \mu_2 N /\log N$, for some $\mu_2>0$. This, indeed proves \eqref{eq:cor-pre}.

We continue with the proof of
the corollary. 
Introduce the following notation.
For $z_0 \in \C$ set
\[
\cJ_{z_0}:= \left\{ \#\{i \in [N]: \lambda_i^N \in \D_{z_0}\} \le C_{\ref{thm-thintube2}} C_{\ref{theo-location}}^4 \log N \right\},
\]
\[
\cC_{z_0}:= \left\{\exists i \in [N] \text{ such that } \supp_{\mu_1}(v(\lambda_i^N)) < \mu_2 N/\log N \text{ and } \lambda_i^N \in \D_{z_0} \right\},
\]
and for $i \in [N]$
\[
\mathbb{I}_i:= {\bf I}(\supp_{\mu_1}(v(\lambda_i^N)) < \mu_2 N/\log N). 
\]
Let $\mathcal{R}$ be a net of $\Omega(\vep_{\ref{theo-location}},C_{\ref{theo-location}},N)$ of mesh size $C_{\ref{theo-location}} \log N/(2N)$. It is clear that $|\mathcal{R}| \le C C_{\ref{theo-location}}^{-1} N/\log N$, for some $C< \infty$. Note also that by Theorem \ref{thm-thintube2},
for any $z_0 \in \mathcal{R}$,
\begin{equation}\label{eq:cJ-z}
\prob(\cJ_{z_0}^c) \le N^{-2}.
\end{equation}
We can now complete the proof. Indeed, for all large $N$, 
\begin{multline*}
\sum_{i=1}^N \E[\mathbb{I}_i] \stackrel{\eqref{eq-finalloc-new}}{\le} \sum_{i} \E[\mathbb{I}_i \cdot {\bf I}(\lambda_i^N \in \Omega(\vep_{\ref{theo-location}},C_{\ref{theo-location}}, N)] + N\mu/4 \\
\le \sum_{z_0 \in \mathcal{R}} \E \left[\sum_{i: \lambda_i^N \in \D_{z_0}} \mathbb{I}_i\right]  + N\mu/4 \le N \sum_{z_0 \in \mathcal{R}} \prob(\cJ_{z_0}^c) + \sum_{z_0 \in \mathcal{R}} \E \left[\sum_{i: \lambda_i^N \in \D_{z_0}} \mathbb{I}_i \cdot {\bf I}(\cJ_{z_0})\right]  + N \mu/4 \\
\stackrel{\eqref{eq:cJ-z}}{\le} C_{\ref{thm-thintube2}} C_{\ref{theo-location}}^4 \log N \sum_{z_0 \in \mathcal{R}} \prob(\mathcal{C}_{z_0}) + N \mu/2 \stackrel{\eqref{eq:cor-pre}}{\le} N \mu,
\end{multline*}
where in the penultimate step we used hat $|\mathcal{R}| =o(N)$, and in last step we chose $\ol{C}=C C_{\ref{thm-thintube2}} C_{\ref{theo-location}}^3$, and used the bound on $|\mathcal{R}|$. This completes the proof.
\qed

\begin{rem}[Sup-norm delocalization]\label{rem:sup-v-nogap}
We borrow notation from the proof of Theorem \ref{thm:main}. Let $P_N$ be the Jordan block. By Remark \ref{rem:c-gamma-explain}, for any fixed $\gamma' < \gamma $, all eigenvalues of $P_{N, \gamma}^Q$ are inside the disc $D(0, 1- (\gamma'-1) \log N/N)$ with probability approaching one. Therefore, for any eigenvalue $\hat z$ and its approximating net point $z^*$ we can take $C_\gamma= \gamma'-1$ (see \eqref{qm4.0}). On the other hand, by Remark \ref{rem:jb3} we observe that the $o(1)$ term in \eqref{eq-final2a} can be replaced by $O(N^{-(\gamma'-1)})$ and hence the same can be done for \eqref{eq-order1}. Furthermore, in this case we have only one pure state, recall Remarks \ref{rem:pstate} and \ref{rem:jb1}. Thus $M=1$. Therefore, by Remark \ref{rem:jb2} we next derive that 
\[
 \left\|\frac{\mathfrak{z}^+_1}{\|\mathfrak{z}^+_1\|} - v \right\|_\infty \le \left\|\frac{\mathfrak{z}^+_1}{\|\mathfrak{z}^+_1\|} - v \right\|  \le \|v - \alpha_1 e_1\| + |1 - |\alpha_1|| + \left\|\frac{\mathfrak{z}^+_1}{\|\mathfrak{z}^+_1\|} - e_1 \right\|= O(N^{-(\gamma'-1)} \log N).
\]
Since, $\gamma >3/2$, upon choosing $\gamma'$ appropriately, it is now immediate that 
\[
\sqrt{\frac{\log N}{N}} \asymp \frac12 \cdot \frac{\|\mathfrak{z}^+_1\|_\infty}{\|\mathfrak{z}^+_1\|}\le \|v\|_\infty \le 2  \frac{\|\mathfrak{z}^+_1\|_\infty}{\|\mathfrak{z}^+_1\|}  \asymp \sqrt{\frac{\log N}{N}}. 
\]
\end{rem}

\providecommand{\bysame}{\leavevmode\hbox to3em{\hrulefill}\thinspace}
\providecommand{\MR}{\relax\ifhmode\unskip\space\fi MR }
\providecommand{\MRhref}[2]{%
  \href{http://www.ams.org/mathscinet-getitem?mr=#1}{#2}
}
\providecommand{\href}[2]{#2}


\begin{thebibliography}{99}
    \bibitem{Al}P.~Alexandersson.
\newblock {\em Schur polynomials, banded Toeplitz matrices and Widom's formula},
\newblock Electronic Journal of Combinatorics, {\bf 19}(4), P22 (2012).  

\bibitem{ADK}J.~Alt, R.~Ducatez, and A.~Knowles. 
\newblock {\em Delocalization transition for critical Erd\H{o}s R\'enyi graphs}, 
\newblock arXiv preprint arXiv:2005.14180 (2020). 

\bibitem{An17}
N.~Anantharaman, \emph{Quantum ergodicity on regular graphs}, Communications in Mathematical Physics,
  \textbf{535}, 633--690  (2017).

\bibitem{AnMa15}
N.~Anantharaman and E.~Le Masson, \emph{Quantum ergodicity on large regular
  graphs}, Duke Mathematical Journal,~\textbf{164}, 723--765  (2015).

\bibitem{AnSa19}
N.~Anantharaman and M.~Sabri, \emph{Quantum ergodicity on graphs:~From spectral
to spatial delocalization}, Annals of Mathematics,~\textbf{189}(3), 753--835  (2019).
  
\bibitem{An58}
P.~W. Anderson, \emph{{Absence of Diffusion in Certain Random Lattices}}, Physical 
  Review, \textbf{109}(5), 1492--1505  (1958).

  \bibitem{BPZ} A.~Basak, E.~Psaquette, and O.~Zeitouni, 
    \emph{Regularization of non-normal matrices by Gaussian noise - the banded Toeplitz and twisted Toeplitz cases}, Forum of Mathematics, Sigma 7, paper e3 (2019).

     \bibitem{BPZ1} A.~Basak, E.~Psaquette, and O.~Zeitouni,
    \emph{Spectrum of random perturbations of Toeplitz matrices with finite symbols}, Transactions of the  American Mathematical Society {\bf 373}, 4999--5023 (2020).

  \bibitem{BZ}A.~Basak and O.~Zeitouni, {\em Outliers of random perturbation of Toeplitz matrices with finite symbols}, Probability Theory and Related Fields, {\bf 178}, 771--826 (2020).
  \bibitem{Be20} Benigni, L.,
     {\em Eigenvectors distribution and quantum unique ergodicity for
              deformed {W}igner matrices},
   {Ann. Inst. Henri Poincar\'{e} Probab. Stat.},
   {\bf 56},
      (2020).
\bibitem{BLo}L.~Benigni and P.~Lopatto, 
\newblock {\em Optimal delocalization for generalized Wigner matrices},
\newblock arXiv preprint arXiv:2007.09585 (2020).
    
  \bibitem{Bhatia}
    R.~Bhatia.
    \emph{Matrix Analysis}.
    Springer-Verlag, Berlin (1997).
    
    \bibitem{BU}D.~Borthwick and A.~Uribe,
    \newblock {\em On the Pseudospectra of Berezin--Toepolitz Operators}, 
    \newblock Methods and Applications of Analysis, {\bf 10}(1), 31--66 (2003). 
    
    \bibitem{BoGr04}A.~B\"ottcher and S.~M.~Grudsky, 
    \newblock {\em Asymptotically good pseudomodes for Toeplitz matrices and Wiener-Hopf operators}, 
    \newblock In Operator Theoretical Methods and Applications to Mathematical Physics, pp.~175-188, Birkh\"auser, Basel (2004). 
    
 \bibitem{BoGr05}
A.~B\"ottcher and S.~M.~Grudsky, \emph{Spectral Properties of Banded Toeplitz matrices}, SIAM (2005).

    
\bibitem{BoYa17}
P.~Bourgade and H.-T. Yau, \emph{The eigenvector moment flow and local quantum
  unique ergodicity}, Communications in Mathematical Physics, \textbf{350},  231--278  (2017).
    

\bibitem{BD}D.~Bump and P.~Diaconis,
\newblock {\em Toeplitz minors},
\newblock Journal of Combinatorial Theory, Series A, {\bf 97}, 252--271 (2002). 

\bibitem{Cornelius}
  E.~F. Cornelius,
  \newblock {\em Identities for complete homogeneous systems},
  \newblock Journal of Algebra, Number Theory and Applications,
  {\bf 21}, 109--116 (2011).
  
  \bibitem{D1}E.~B.~Davies, 
  \newblock {\em Pseudo--spectra, the harmonic oscillator and complex resonances},
  \newblock  Proceedings of the Royal Society of London. Series A:~Mathematical, Physical and Engineering Sciences, {\bf 455}, 585--599 (1982).
  
  \bibitem{D2} E.~B.~Davies,
  \newblock {\em Semi-classical states for non-self-adjoint Schr\"odinger operators}, 
  \newblock Communications in Mathematical Physics, {\bf 200}(1), 35--41 (1999).
  
  \bibitem{D3} E.~B.~Davies,
  \newblock {\em Pseudospectra of differential operators},
  \newblock Journal of Operator Theory, {\bf 43}(2), 243--262 (2000).

\bibitem{DaHa09}E.~B.~Davies and M.~Hager,
\newblock {\em Perturbations of Jordan matrices},
\newblock Journal of Approximation Theory, {\bf 156}, 82--94 (2009). 

\bibitem{NSjZw04}
N.~Dencker, J.~Sj{\"o}strand, and M.~Zworski, \emph{{Pseudospectra of
  semiclassical (pseudo-) differential operators}}, Communications on Pure and
  Applied Mathematics \textbf{57} (2004), no.~3, 384--415.
  
  \bibitem{Dr18}
A.~Drouot, \emph{Resonances for random highly oscillatory potentials}, J. of
  Mathematical Physics \textbf{59} (2018), no.~101506, 34 pp.
  
   \bibitem{DyZw19}
S.~Dyatlov and M.~Zworski, \emph{{Mathematical Theory of Scattering
  Resonances}}, American Mathematical Society, 2019.
  
  \bibitem{En83}
V.~Enss, \emph{{Introduction to Asymptotic Observables for Multi-Particle
  Quantum Scattering}}, Lecture Notes in Mathematics, Springer, \textbf{1218}
  (1983), 61--92.
  
\bibitem{EKYY13}
L.~Erd{\H o}s, A.~Knowles, H.-T.~Yau, and J.~Yin, \emph{Spectral statistics of
  erd{\H o}s-r{\'e}nyi graphs I: Local semicircle law}, Annals of Probability
  \textbf{41}(3B), 2279--2375 (2013).

\bibitem{ESY09b}
L.~Erd{\H o}s, B.~Schlein, and H.-T.~Yau, \emph{Local semicircle law and
complete delocalization for {Wigner} random matrices}, Communications in Mathematical Physics,
  \textbf{287},  641--655 (2009).

\bibitem{ESY09a}
L.~Erd{\H o}s, B.~Schlein, and H.-T.~Yau, \emph{Semicircle law on short scales and delocalization of
  eigenvectors for Wigner random matrices}, Annals of Probability \textbf{37}(3),
   815--852  (2009).
  
  \bibitem{Ga12}
J.~Galkowski, \emph{{Nonlinear Instability in a Semiclassical Problem}}, Communications in 
Mathematical Physics, {\bf 316}, 705--722  (2012).

\bibitem{GT}D.~Garc\'{i}a-Garc\'{i}a and M.~Tierz,
\newblock {\em Toeplitz minors and specializations of skew Schur polynomials},
\newblock Journal of Combinatorial Theory, Series A, {\bf 172}, article 105201 (2020). 


    

\bibitem{GoKr69}
I.C. Gohberg and M.G. Krein,
\newblock {\em Introduction to the theory of linear
  non-selfadjoint operators}, 
  \newblock Translations of mathematical monographs, vol.~18,
  AMS (1969).

\bibitem{GS}I.~Goldsheid and S.~Sodin,
\newblock {\em Lower bounds on Anderson-localised eigenfunctions on a strip}, 
\newblock arXiv preprint arXiv:2012.03017 (2020).
    
\bibitem{GTZ} S.~A.~Goreinov, E.~E.~Tyrtyshnikov, and N.~L.~Zamarashkin,
  \newblock {\em A theory of pseudo-skeleton approximations},
  Linear Algebra and its Applications, {\bf 261}, 1--21 (1997).
  
  \bibitem{Gr}V.~V.~Grushin.
  \newblock {\em Les probl\`emes aux limites d\'eg\'en\'er\'es et les op\'erateurs pseudodiff\'erentiels}, 
  \newblock Actes du Congr\`es International des Math\'ematiciens, Tome 2, Gauthier-Villars, 737--743 (1970).

    \bibitem{HaSj08}
M.~Hager and J.~Sj{\"o}strand, \emph{{Eigenvalue asymptotics for randomly
  perturbed non-selfadjoint operators}}, Mathematische Annalen 
  \textbf{342}, 177--243  (2008).
  
    \bibitem{HeSj86}
B.~Helffer and J.~Sj\"ostrand, \emph{{Resonances en limite semiclassique}},
  Bull. Soc. Math. France \textbf{114} (1986), no.~24--25.
%
\bibitem{H73}L.~H\"{o}rmander,
  \newblock {\em An introduction to complex analysis in several variables}, 
  \newblock Elsevier  (1973).
  
  \bibitem{Ho85}
L.~H\"{o}rmander, \emph{{The Analysis of Linear Partial Differential Operators IV}},
  {Grundlehren der mathematischen Wissenschaften}, vol. 275, Springer-Verlag,
  1985.
  
\bibitem{KK19}
  S.~Kalmykov and L.~V.~Kovalev,
  \newblock {\em Self intersections of {L}aurent polynomials and the density of {J}ordan curves},
  \newblock {Proceedings of the American Mathematical Society}, to appear (2021).
  
    \bibitem{Kl16}
F.~Klopp, \emph{{Resonances for large one-dimensional "ergodic" systems}},
  Anal. PDE \textbf{9} (2016), no.~2, 259--352.

  \bibitem{La04}R.~Lata{\l}a,
  \newblock {\em Some estimates of norms of random matrices},
  \newblock {Proceedings of the American Mathematical Society}, {\bf 133}(5), 1273--1282 (2004).
  
  \bibitem{LO}K.~Luh and S.~O'Rourke,
  \newblock {\em Eigenvector delocalization for non-Hermitian random matrices and applications},
  \newblock Random Structures \& Algorithms, {\bf 57}(1), 169--210 (2020). 
  
  \bibitem{LAKS}A.~De Luca, B.~L.~Altshuler, V.~E.~Kravtsov, and A.~Scardicchio, 
  \newblock {\em Anderson Localization on the Bethe Lattice:~Nonergodicity of Extended States}, 
  \newblock Physical Review Letters {\bf 113}(4), 046806  (2014).
  
  \bibitem{LT}A.~Lytova and K.~Tikhomirov,
  \newblock {\em On delocalization of eigenvectors of random non-Hermitian matrices},
  \newblock Probability Theory and Related Fields, {\bf 177}, 465--524 (2020). 
  
  \bibitem{Macd}I.~G.~Macdonald,
  \newblock {\em Symmetric functions and Hall polynomials}, second edition,
  \newblock The Oxford University Press (1995).  
  
  \bibitem{MM}E.~A.~Maximenko and M.~A.~Moctezuma-Salazar,
  \newblock {\em Cofactors and eigenvectors of banded Toeplitz matrices:~Trench formulas via skew Schur polynomials},
 \newblock Operators and Matrices, {\bf 11}(4), 1149--1169 (2017). 
 
   \bibitem{Me95}
R.~B. Melrose, \emph{{Geometric Scattering Theory}}, Cambridge University
  Press, 1995.

 \bibitem{MIKOS} A.~Mikhalev and I.~V.~Oseledets,
   \newblock{\em Rectangular maximum-volume sub-matrices and their applications},
   \newblock Linear Algebra and its Applications, {\bf 538}, 187--211 (2018).
   
   \bibitem{NoZw07}
S.~Nonnenmacher and M.~Zworski, \emph{Distribution of resonances for open
  quantum maps.}, Comm. Math. Phys (2007), no.~269, 311--365.

\bibitem{NoZw09}
S.~Nonnenmacher and M.~Zworski, \emph{Quantum decay rates in chaotic scattering}, Acta Math.
  \textbf{203} (2009), no.~2, 149 -- 233.
  
  \bibitem{NoSjZw14}
S.~Nonnenmacher, J.~Sj{\"o}strand, and M.~Zworski, \emph{Fractal weyl law for
  open quantum chaotic maps}, Annals of Math. \textbf{179} (2014), no.~1,
  179--251.
   
   \bibitem{Pr08}
K.~Pravda-Starov, \emph{{Pseudo-spectrum for a class of semi-classical
  operators}}, Bull. Soc. Math. France \textbf{136} (2008), no.~3, 329--372.
  
\bibitem{RaZw}
A.~Raphael and M.~Zworski, \emph{{Pseudospectral effects and basins of
  attraction}}, unpublished notes (2005).
%

\bibitem{RT}L.~Reichel and L.~N.~Trefethen, 
\newblock{\em Eigenvalues and pseudo-eigenvalues of Toeplitz matrices}, 
\newblock Linear algebra and its applications,{\bf 162}, 153--185 (1992).


\bibitem{RV}M.~Rudelson and R.~Vershynin,
\newblock {\em Delocalization of eigenvectors of random matrices with independent entries},
\newblock Duke Mathematical Journal, {\bf 164}(13), 2507--2538 (2015). 

\bibitem{RVnogap}M.~Rudelson and R.~Vershynin,
\newblock {\em No-gaps delocalization for general random matrices},
\newblock Geometric and Functional Analysis, {\bf 26}, 1716--1776 (2016). 

\bibitem{SaSc05}
B.~Sandstede and A.~Scheel, \emph{{Basin boundaries and bifurcations near
  convective instabilities:~a case study}}, Journal of Differential Equations,
  {\bf 208}(1), 176--193 (2005).
  
\bibitem{Sa92}
P.~Sarnak, \emph{Arithmetic quantum chaos, in \emph{The Schur Lectures} (tel
  aviv)}, Israel Math. Conf. Proc \textbf{8}, 183--236 (1992).

  \bibitem{SiSo87}
L.~M. Sigal and A.~Soffer, \emph{{The N-particle Scattering Problem: Asymptotic
  Completeness for Short Range Systems}}, Annals of Math. \textbf{126} (1987),
  35--108.
  
\bibitem{Sj73}
 J.~Sj{\"o}strand, \emph{Operators of principle type with interior boundary conditions}, Acta Math. 
 \textbf{130} (1973), 1--51.
 %
   \bibitem{Sj14}
J.~Sj{\"o}strand, \emph{{Weyl law for semi-classical resonances with randomly perturbed
  potentials}}, M{\'e}moires de la SMF \textbf{136} (2014).
%
\bibitem{SjVo16}J.~Sj{\"o}strand and M.~Vogel,
\newblock {\em Large bidiagonal matrices and random perturbations},
\newblock Journal of Spectral Theory, {\bf 6}(4), 977--1020 (2016). 

  \bibitem{SjVo19b}
J.~Sj{\"o}strand and M.~Vogel, \emph{General Toeplitz matrices subject to
  Gaussian perturbations}, Annales Henri Poincar\'{e}, {\bf 22}, 49--81 (2021).
  
\bibitem{SjVo19a}
J.~Sj{\"o}strand and M.~Vogel, \emph{Toeplitz band matrices with small random perturbations},
  Indagationes Mathematicae, {\bf 32}(1), 275--322 (2021).
  
\bibitem{SjZw91}
J.~Sj{\"o}strand and M.~Zworski, \emph{{Complex scaling and the distribution of
  scattering poles}}, Journal of the American Mathematical Society 
  \textbf{43}(4),
  729--769
  (1991). 
  
  \bibitem{SZ}J.~Sj{\"o}strand and M.~Zworski,
  \newblock {Elementary linear algebra for advanced spectral problems},
  \newblock Annales de l'institut Fourier, {\bf 57}(7), 2095--2141 (2007).

  \bibitem{Stan2}R.~P.~Stanley,
  \newblock {\em Enumerative Combinatorics}, Volume 2. 
  \newblock Cambridge University Press (1999).
 \bibitem{TaoVu08}T.~Tao and V.~Vu,
  \newblock {\em Random Matrices:~The Circular Law},
  \newblock Communications in Contemporary Mathematics, {\bf 10}(2), 261--307 (2008).
  \bibitem{Tr97}
L.N. Trefethen, \emph{{Pseudospectra of linear operators}}, SIAM Rev.
  \textbf{39},  383--406 (1997).
  
\bibitem{TC}L.~N.~Trefethen and S.~J.~Chapman,
\newblock {\em Wave packet pseudomodes of twisted Toeplitz matrices}, 
\newblock Communications on Pure and Applied Mathematics, {\bf 57}(9), 1233--1264 (2004).

\bibitem{TrEm05}
L.~N.~Trefethen and M.~Embree, \emph{{Spectra and Pseudospectra: The Behavior
  of Nonnormal Matrices and Operators}}, Princeton University Press (2005).
  

  \bibitem{Vo19c}
M.~Vogel, \emph{Almost sure Weyl law for quantized tori}, 
\newblock Communications in Mathematical Physics, {\bf 378}(2), 1539--1585  (2020). 

  \bibitem{W72}H.~Whitney,
\newblock {\em Complex Analytic Varieties}. 
\newblock Addison-Wesley Pub.~Co.~Reading, Massachusetts (1972).

\bibitem{Whittle}P.~Whittle,
\newblock {\em Bounds for the Moments of Linear and Quadratic Forms in Independent Variables}.
\newblock Teoriya Veroyatnostei i ee Primeneniya, {\bf 5}(3), 331--335 (1960).

\bibitem{Zw01}
M.~Zworski, \emph{{A remark on a paper of E.B. Davies}}, Proc. A.M.S. (2001),
  no.~129, 2955--2957.

\end{thebibliography}
\end{document}